\documentclass{amsart}
\usepackage{amsmath,amsthm,amssymb}
\usepackage{enumerate}
\usepackage{graphicx}
\usepackage{cite}
\usepackage{comment}
\usepackage{oands}
\usepackage{tikz}
\usepackage{changepage}
\usepackage{bbm}
\usepackage{mathtools}
\usepackage[margin=1.15in]{geometry}
\usepackage{hyperref}
\usepackage{appendix}
\usepackage[pagewise,mathlines]{lineno}

\theoremstyle{plain}
\newtheorem{thm}{Theorem}[section]
\newtheorem{cor}[thm]{Corollary}
\newtheorem{lem}[thm]{Lemma}
\newtheorem{prop}[thm]{Proposition}

\newtheorem{notation}[thm]{Notation}

\theoremstyle{definition}
\newtheorem{defn}[thm]{Definition}
\newtheorem{remark}[thm]{Remark}

\def\@rst #1 #2other{#1}

\hypersetup{
    colorlinks=false,
    linktocpage,
    }

\newcommand{\dsb}{\begin{adjustwidth}{2.5em}{0pt}
\begin{footnotesize}}
\newcommand{\dse}{\end{footnotesize}
\end{adjustwidth}}

\newcommand{\ssb}{\begin{adjustwidth}{2.5em}{0pt}}
\newcommand{\sse}{\end{adjustwidth}}

\newcommand{\aryb}{\begin{eqnarray*}}
\newcommand{\arye}{\end{eqnarray*}}
\def\alb#1\ale{\begin{align*}#1\end{align*}}
\newcommand{\eqb}{\begin{equation}}
\newcommand{\eqe}{\end{equation}}
\newcommand{\eqbn}{\begin{equation*}}
\newcommand{\eqen}{\end{equation*}}

\newcommand{\BB}{\mathbbm}
\newcommand{\ol}{\overline}
\newcommand{\ul}{\underline}
\newcommand{\op}{\operatorname}

\newcommand{\frk}{\mathfrak}
\newcommand{\eqD}{\overset{d}{=}}
\newcommand{\ep}{\epsilon}
\newcommand{\rta}{\rightarrow}

\newcommand{\wt}{\widetilde}
\newcommand{\wh}{\widehat} 
\newcommand{\mcl}{\mathcal}

\newcommand*\tc[1]{\tikz[baseline=(char.base)]{\node[shape=circle,draw,inner sep=1pt] (char) {#1};}}
\newcommand*\tb[1]{\tikz[baseline=(char.base)]{\node[shape=rectangle,draw,inner sep=2.5pt] (char) {#1};}}

\newcommand*\patchAmsMathEnvironmentForLineno[1]{  \expandafter\let\csname old#1\expandafter\endcsname\csname #1\endcsname
  \expandafter\let\csname oldend#1\expandafter\endcsname\csname end#1\endcsname
  \renewenvironment{#1}     {\linenomath\csname old#1\endcsname}     {\csname oldend#1\endcsname\endlinenomath}}\newcommand*\patchBothAmsMathEnvironmentsForLineno[1]{  \patchAmsMathEnvironmentForLineno{#1}  \patchAmsMathEnvironmentForLineno{#1*}}\AtBeginDocument{\patchBothAmsMathEnvironmentsForLineno{equation}\patchBothAmsMathEnvironmentsForLineno{align}\patchBothAmsMathEnvironmentsForLineno{flalign}\patchBothAmsMathEnvironmentsForLineno{alignat}\patchBothAmsMathEnvironmentsForLineno{gather}\patchBothAmsMathEnvironmentsForLineno{multline}}

\title[Scaling limits for the FK model III]{Scaling limits for the critical Fortuin-Kastelyn model on a random planar map III: finite volume case}
 
\author{Ewain Gwynne}
 
\author{Xin Sun}


\address{Department of Mathematics\\
  Massachusetts Institute of Technology\\
  Cambridge, MA 02139}
\email{ewain@mit.edu \\  xinsun89@math.mit.edu}

\begin{document}

\begin{abstract}
We prove scaling limit results for the finite-volume version of the inventory accumulation model of Sheffield (2011), which encodes a random planar map decorated by a collection of loops sampled from the critical Fortuin-Kasteleyn (FK) model. In particular, we prove that the random walk associated with the finite-volume version of this model converges in the scaling limit to a correlated Brownian motion $\dot Z$ conditioned to stay in the first quadrant for two units of time and satisfy $\dot Z(2)  = 0$. We also show that the times which describe complementary connected components of FK loops in the discrete model converge to the $\pi/2$-cone times of $\dot Z$. Combined with recent results of Duplantier, Miller, and Sheffield, our results imply that many interesting functionals of the FK loops on a finite-volume FK planar map (e.g.\ their boundary lengths and areas) converge in the scaling limit to the corresponding ``quantum" functionals of the CLE$_\kappa$ loops on a $4/\sqrt\kappa$-Liouville quantum gravity sphere for $\kappa \in (4,8)$. Our results are finite-volume analogues of the scaling limit theorems for the infinite-volume version of the inventory accumulation model proven by Sheffield (2011) and Gwynne, Mao, and Sun (2015). 
\end{abstract}

\maketitle

\tableofcontents

\section{Introduction}
\label{sec-intro}

Let $\Theta = \{\tc H , \tc C , \tb H , \tb C , \tb F\}$. The set of finite words consisting of elements of $\Theta$, modulo the relations
\eqb \label{eqn-theta-relations}
\tc C \tb C = \tc H \tb H = \tc C \tb F = \tc H \tb F = \emptyset ,\qquad \tc C \tb H = \tb H \tc C ,\qquad \tc H \tb C = \tb C \tc H 
\eqe 
forms a semigroup, which was first introduced by Sheffield in~\cite{shef-burger}. Given a word $x$ consisting of elements of $\Theta$, we write $\mcl R(x)$ for its reduction modulo the relations~\eqref{eqn-theta-relations}, with all burgers to the right of all orders. Following~\cite{shef-burger}, we think of elements of $\Theta$ as representing a hamburger, a cheeseburger, a hamburger order, a cheeseburger order, and a flexible order (i.e.\ a request for the ``freshest available" burger of either type), respectively. A word $x$ consisting of elements of $\Theta$ (read from left to right) represents a sequence of burgers being produced and orders being placed. Whenever an order is placed, it is fulfilled by the first burger of the appropriate type to the left of this order which has not yet been consumed. The reduced word $\mcl R(x)$ represents the set of unfulfilled orders (which were placed at a time when no suitable burgers were available) and the set of unconsumed burgers. 
 
The reason for our interest in the semigroup $\Theta$ is that there is a bijection, described in \cite[Section 4.1]{shef-burger}, between words $\dot X$ consisting of $2n$ elements of $\Theta$, with the property that $\mcl R(\dot X) = \emptyset$; and triples $(M , e_0, S)$ consisting of a planar map $M$ with $ n$ edges, an oriented root edge $e_0$ of $M$, and a distinguished subset of the set of edges of $M$. This bijection generalizes a bijection due to Mullin~\cite{mullin-maps} (see also~\cite{bernardi-maps} for a more explicit description) and is essentially equivalent to the construction of~\cite[Section 4]{bernardi-sandpile} for a fixed choice of planar map $M$. The edge set $S$ gives rise to a collection of loops $\mcl L$ on $M$ (described by sequences of edges in a certain quadrangulation $Q = Q(M)$ associated with $M$) which form interfaces between edges in $S$ and edges of the dual map $M^*$ which do not cross edges of $S$.

For $p\in [0,1]$, define a probability measure on $\Theta$ by
\eqb \label{eqn-theta-probs}
\BB P\left(\tc{H}\right) = \BB P\left(     \tc{C} \right) = \frac14 ,\quad  \BB P\left( \tb{H} \right) = \BB P\left( \tb{C} \right) = \frac{1-p}{4} ,\quad \BB P\left(\tb{F} \right) =  \frac{p}{2} .
\eqe 
Let $X = \dots X_{-1}  X_0  X_1 \dots$ be a bi-infinite word whose symbols are iid samples from the probability measure~\eqref{eqn-theta-probs}. If $p\in (0,1/2)$ and we sample a word $\dot X$ according to the conditional law of $X_1\dots X_{2n}$ given $\{\mcl R(X_1\dots X_{2n})=\emptyset\}$, then (as explained in~\cite[Section 4.2]{shef-burger}) the law of the triple $(M,e_0, \mcl L)$ is given by the uniform measure on such triples weighted by $q^{\#\mcl L/2}$, where $q = 4p^2/(1-p)^2 \in (0,4)$. This implies that the conditional law of $\mcl L$ given $M$ is that of the \emph{critical Fortuin-Kasteleyn (FK) cluster model with parameter $q $} on $M$~\cite{fk-cluster}, which is closely related to the $q$-state Potts model for integer values of $q$ (see \cite{kager-nienhuis-guide, grimmett-fk} and the references therein for more on the FK model). We call a pair $(M,\mcl L)$ sampled according to this probability measure a \emph{(critical) FK planar map of size $ n$} and the triple $(M,e_0 , \mcl L)$ a \emph{rooted (critical) FK planar map of size $ n$}. 

As alluded to in~\cite{shef-burger}, there is also an infinite-volume version of the above bijection, which relates infinite-volume FK planar maps and bi-infinite words $X$ with elements sampled independently according to the probabilities~\eqref{eqn-theta-probs}. See~\cite{chen-fk,blr-exponents} for more details. 

The law of of the FK planar map $(M , \mcl L)$ is conjectured to converge in the scaling limit at $n\rta\infty$ to the law of a \emph{conformal loop ensemble} ($\op{CLE}_\kappa$)~\cite{shef-cle,shef-werner-cle,ig1,ig2,ig3,ig4} on top of an independent \emph{$\gamma$-Liouville quantum gravity surface} \cite{shef-kpz,shef-zipper,wedges}, where $\kappa \in (4,8)$ and $\gamma \in (\sqrt 2 , 2)$ satisfy
\eqb \label{eqn-p-kappa}
q = \frac{4p^2}{(1-p)^2}= 2 + 2\cos(8\pi/\kappa) \qquad \gamma = \frac{4}{\sqrt\kappa} .
\eqe 
In the special case $p=1/3$, the marginal law of the planar map $M$ (without the collection of loops $\mcl L$) is uniform on all such maps. Uniform random planar maps and their scaling limits have been studied extensively. In particular, a uniformly chosen random quadrangulation with $2n$ edges converges in law in the Gromov-Hausdorff topology to a continuum random metric space called the \emph{Brownian map}~\cite{legall-uniqueness,miermont-brownian-map}. See~\cite{miermont-survey,legall-sphere-survey} and the references therein for more details. In~\cite{qle,sphere-constructions,tbm-characterization,lqg-tbm1,lqg-tbm2,lqg-tbm3}, Miller and Sheffield construct a metric on LQG for $\gamma = \sqrt{8/3}$ under which it is isometric to the Brownian map. 
 
In \cite[Theorem 2.5]{shef-burger}, it is shown that a certain non-Markovian random walk associated with the bi-infinite word $X$ sampled according to the probabilities~\eqref{eqn-theta-probs} converges in the scaling limit to a pair of correlated two-sided Brownian motions $Z = (U,V)$, started from $Z(0) = 0$ and satisfying
 \eqb \label{eqn-bm-cov}
\op{Var}(U(t) ) = \frac{1-p}{2} |t| \quad \op{Var}(V(t)) = \frac{1-p}{2} |t| \quad \op{Cov}(U(t) , V(t) ) = \frac{p}{2} |t| ,\qquad \forall t\in\BB R. 
\eqe
On the other hand, it was recently shown by Duplantier, Miller, and Sheffield~\cite{wedges} that a whole-plane $\op{CLE}_\kappa$~\cite{werner-sphere-cle,mww-nesting} on a certain infinite-volume Liouville quantum gravity (LQG) surface called a \emph{$\gamma$-quantum cone} can be encoded by a correlated two-dimensional Brownian motion $Z = (U,V)$ via a continuum analogue of the bijection of \cite[Section 4.1]{shef-burger}. This encoding is called the \emph{peanosphere construction} and can be interpreted as a mating of two correlated continuum random trees~\cite{aldous-crt1,aldous-crt2,aldous-crt3}. Thus \cite[Theorem 2.5]{shef-burger} can be interpreted as the statement that the contour functions for infinite-volume FK planar maps converge in the scaling limit to the contour function of a $\op{CLE}_\kappa$ on a $\gamma$-quantum cone. In other words, one has convergence of infinite-volume FK planar maps toward a CLE$_\kappa$-decorated $\gamma$-quantum cone the so-called \emph{peanosphhere topology}.  
In~\cite[Theorem 1.9]{gms-burger-cone}, this convergence statement is strengthened by proving that the times corresponding to complementary connected components FK loops in Sheffield's bijection converge in the scaling limit to the $\pi/2$-cone times of the correlated Brownian motion $Z$, which encode complementary connected components of $\op{CLE}_\kappa$ loops in the construction of \cite{wedges}. 

In \cite[Theorem 1.1]{sphere-constructions}, the authors prove that one can encode a $\op{CLE}_\kappa$ on a finite-volume LQG surface called a \emph{quantum sphere} via a constant multiple of a correlated Brownian motion $\dot Z$ with variances and covariances as in~\eqref{eqn-bm-cov} conditioned to stay in the first quadrant for two units of time and satisfy $\dot Z(2) = 0$ (this conditioning is made precise in \cite[Section 3]{sphere-constructions}; see also Section~\ref{sec-bm-cond} of the present paper). 
Hence, it is natural to expect that the random walk associated with a word $\dot X = X_1\dots X_{2n}$ sampled from the inventory accumulation model conditioned on the event that $\mcl R(\dot X) = \emptyset$ (which encodes a finite-volume FK planar map) converges in the scaling limit to a correlated Brownian motion $\dot Z$ conditioned to stay in first quadrant for two units of time and satisfy $\dot Z(2) = 0$. This statement implies the convergence of finite-volume FK planar maps toward CLE$_\kappa$ on an independent LQG sphere in the peanosphere topology.

In this paper we will prove the above scaling limit statement, thereby extending \cite[Theorem 2.5]{shef-burger} to the finite volume case. We will also obtain an exact analogue of \cite[Theorem 1.9]{gms-burger-cone}, which will imply, among other things, that the boundary lengths and areas of complementary connected components of FK loops on a random planar map on the sphere converge in the scaling limit to the quantum lengths and areas of complementary connected components of $\op{CLE}_\kappa$ loops on a $\gamma$-quantum sphere. This latter result answers \cite[Question 13.3]{wedges} in the finite-volume case. Furthermore, the results of this paper will be used by the first author and J. Miller in~\cite{gwynne-miller-cle} to prove a stronger scaling limit result for FK planar maps (i.e.\ the convergence of the law of the entire topological structure of the FK loops on an FK planar map to that of CLE$_\kappa$ on an independent $\gamma$-quantum sphere). The scaling limit result of~\cite{gwynne-miller-cle}, in turn, will be used in the subsequent work~\cite{gwynne-miller-inversion} to prove that the law of whole-plane $\op{CLE}_\kappa$ is invariant under inversion for $\kappa \in (4,8)$ (see~\cite{werner-sphere-cle} for a proof in the case when $\kappa \in (8/3,4]$). 

The vast majority of the arguments in the present paper use only basic properties of the hamburger-cheeseburger model and the results of~\cite{shef-burger,gms-burger-cone,gms-burger-local}, and can be read without any knowledge of CLE$_\kappa$ or Liouville quantum gravity. 
However, unlike the proofs of the estimates and scaling limit results for the hamburger-cheeseburger model proven in prior works~\cite{shef-burger,gms-burger-cone,gms-burger-local,blr-exponents}, our proofs (in particular the arguments in Section~\ref{sec-big-loop}) will make some use properties of FK planar maps, the bijection of \cite{shef-burger}, and the relationship between these objects and $\op{CLE}_{\kappa}$ on an independent Liouville quantum gravity cone established in~\cite{wedges,gms-burger-cone,gwynne-miller-cle}. We also use the scaling limit results of~\cite{shef-burger,gms-burger-cone}, the estimates of~\cite{gms-burger-local}, and at one point the infinite-volume version of the scaling limit results of~\cite{gwynne-miller-cle}. 

\begin{remark} \label{remark-implication}
The paper~\cite{gwynne-miller-cle} (currently in preparation) introduces a topological structure called a \emph{lamination} associated with a collection of non-crossing loops, an area measure, and a boundary length measure, which (under certain hypotheses) determines the collection of loops and the two measures modulo an ambient homeomorphism of $\BB C$. It is then proven that the lamination of the FK loops on an FK planar map converges in law to the lamination of a CLE$_\kappa$ on an independent Liouville quantum gravity surface with parameter $\gamma$, where $\kappa$, $\gamma$, and $p$ are related as in~\eqref{eqn-p-kappa}. This is accomplished for both infinite-volume and finite-volume FK planar maps (which correspond to the case where the quantum surface is a $\gamma$-quantum cone or a quantum sphere, respectively; see~\cite{wedges} for definitions). In both cases, the scaling limit result is proven by explicitly describing the lamination of an FK planar map (resp. a CLE$_\kappa$ on an LQG surface) in terms of the word in Sheffield's bijection (resp. the Brownian motion in the peanosphere construction of~\cite{wedges,sphere-constructions}); then applying~\cite[Theorem 2.5]{shef-burger} and~\cite[Theorem 1.9]{gms-burger-cone} (in the infinite-volume case) or Theorems~\ref{thm-main} and~\ref{thm-cone-limit-finite} of the present paper (in the finite-volume case), plus some additional estimates (some of which use the results of~\cite{gms-burger-local}). The only significant difference between the proofs in the two cases is which of these pairs of theorems is applied. 
In particular, the infinite-volume version of the result of~\cite{gwynne-miller-cle} is proven independently of the present paper. This point is crucial, as at one step of the proof of the main results of the present paper (in particular the proof of Lemma~\ref{prop-big-F-constant}), we need to apply the infinite-volume version of the main result of~\cite{gwynne-miller-cle} in order to estimate the probability of a certain event defined in terms of an infinite-volume FK planar map. We note, however, that the proof of inversion invariance of whole-plane $\op{CLE}_{\kappa}$ in~\cite{gwynne-miller-inversion} requires the finite-volume version of~\cite{gwynne-miller-cle}. In summary, we have the following implication relations between various results (here an arrow from one result to another means that the proof of the second result uses the first result). 
\begin{center}
\includegraphics{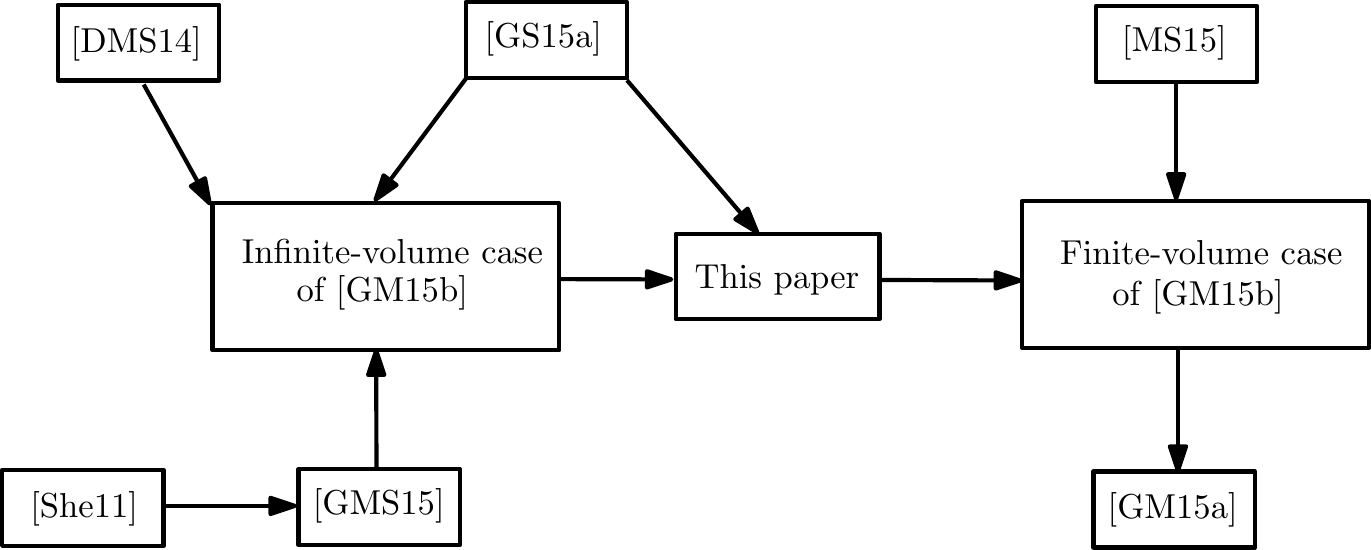} 
\end{center}
\end{remark}

\bigskip

\noindent{\bf Acknowledgments}
We thank Nina Holden, Cheng Mao, Jason Miller, and Scott Sheffield for helpful discussions. We thank the Isaac Newton Institute for its hospitality during part of our work on this project. The first author was supported by the U.S. Department of Defense via an NDSEG fellowship. The second author was partially supported by NSF grant DMS-1209044. 
 
\subsection{Notation and preliminaries}
\label{sec-burger-prelim}
In this section we will introduce some notation which will remain fixed throughout the paper. This notation is in agreement with that used in~\cite{gms-burger-cone,gms-burger-local}. 

\subsubsection{Basic notation} 

Here we record some basic notations which we will use throughout this paper. 

\begin{defn}
Let $X$ be a random variable taking values in a countable state space $\Omega$. A \emph{realization} of $X$ is an element $x\in\Omega$ such that $\BB P(X=x) > 0$. 
\end{defn}

\begin{notation} \label{def-discrete-intervals}
For $a < b \in \BB R$, we define the discrete intervals $[a,b]_{\BB Z} := [a, b]\cap \BB Z$ and $(a,b)_{\BB Z} := (a,b)\cap \BB Z$. 
\end{notation}

\begin{notation}\label{def-asymp}
If $a$ and $b$ are two quantities, we write $a\preceq b$ (resp. $a \succeq b$) if there is a constant $C$ (independent of the parameters of interest) such that $a \leq C b$ (resp. $a \geq C b$). We write $a \asymp b$ if $a\preceq b$ and $a \succeq b$. 
\end{notation}

\begin{notation} \label{def-o-notation}
If $a$ and $b$ are two quantities which depend on a parameter $x$, we write $a = o_x(b)$ (resp. $a = O_x(b)$) if $a/b \rta 0$ (resp. $a/b$ remains bounded) as $x \rta 0$ (or as $x\rta\infty$, depending on context). We write $a = o_x^\infty(b)$ if $a = o_x(b^s)$ for each $s > 0$ (if $b$ is tending to $0$) or for each $s < 0$ (if $b$ is tending to $\infty$). The regime we are considering will be clear from the context.
\end{notation}

Unless otherwise stated, all implicit constants in $\asymp, \preceq$, and $\succeq$ and $O_x(\cdot)$ and $o_x(\cdot)$ errors involved in the proof of a result are required to satisfy the same dependencies as described in the statement of said result.

\subsubsection{Inventory accumulation model}

Let $p\in (0,1/2)$. We will always treat $p$ as fixed and do not make dependence on $p$ explicit.

Let $\Theta = \{\tc H , \tc C , \tb H , \tb C , \tb F\}$ be the generating set for the semigroup described above. Given a word $x$ consisting of elements of $\Theta$, we denote by $\mathcal R(x)$ the word reduced modulo the relations~\eqref{eqn-theta-relations}, with all burgers to the right of all orders, as above. We also write $|x|$ for the number of symbols in $x$ (regardless of whether or not $x$ is reduced). 
 
Let $X = \dots X_{-1} X_0 X_1 \dots$ be a bi-infinite word with each symbol sampled independently according to the probabilities~\eqref{eqn-theta-probs}. For $a \leq b\in \BB R$, let
\eqb \label{eqn-X(a,b)}
X(a,b) := \mcl R\left(X_{\lfloor a\rfloor} \dots X_{\lfloor b\rfloor} \right) .
\eqe 
We adopt the convention that $X(a,b) = \emptyset$ if $b < a$.  
 
By \cite[Proposition 2.2]{shef-burger}, it is a.s.\ the case that the ``infinite reduced word" $X(-\infty,\infty)$ is empty, i.e.\ each symbol $X_i$ in the word $X$ has a unique match which cancels it out in the reduced word.

\begin{notation}\label{def-match-function}
For $i \in \BB Z$ we write $\phi(i)$ for the index of the match of $X_i$. 
\end{notation}

\begin{notation} \label{def-theta-count}
For $\theta\in \Theta$ and a word $x$ consisting of elements of $\Theta$, we write $\mcl N_{\theta}(x)$ for the number of $\theta$-symbols in $x$. We also let
\eqb \label{eqn-theta-count-whole}
d(x)  := \mcl N_{\tc H}(x) - \mcl N_{\tb H}(x) ,\quad 
d^*(x)  := \mcl N_{\tc C}(x) - \mcl N_{\tb C}(x),\quad 
D(x)  := \left(d(x) , d^*(x)\right) 
\eqe 
and
\eqb \label{eqn-theta-count-reduced}
\frk h(x) := \mcl N_{\tb H}\left(\mcl R(x)\right) ,\quad \frk c(x) := \mcl N_{\tb C}\left(\mcl R(x)\right) ,\quad \frk o(x) := \mcl N_{\tb H}\left(\mcl R(x)\right) + \mcl N_{\tb F}\left(\mcl R(x)\right) + 1.
\eqe
\end{notation}

We note that the definitions of $\frk h(x)$ and $\frk c(x)$ in~\eqref{eqn-theta-count-reduced} differ from the definitions in~\cite[Definition 3.1]{gms-burger-local} (the definitions in this paper concern orders, whereas the definitions in~\cite{gms-burger-local} concern burgers). 
  
For $i\in\BB Z$, we define $Y_i = X_i$ if $X_i \in \{\tc{H} , \tc{C} , \tb{H} ,  \tb{C}\}$; $Y_i = \tb{H}$ if $X_i = \tb F$ and $X_{\phi(i)} = \tc{H}$; and $Y_i = \tb{C}$ if $X_i = \tb F$ and $X_{\phi(i)} = \tc{C}$. For $a\leq b \in \BB R$, define $Y(a,b)$ as in~\eqref{eqn-X(a,b)} with $Y$ in place of $X$.  
 
For $n \geq0$, define $ d(n) =   d(Y(1,n))$ and for $n<0$, define $  d(n) = -  d(Y( n+1 , 0))$. Define $d^*(n)$ similarly. Extend each of these functions from $\BB Z$ to $\BB R$ by linear interpolation. 
Let 
\eqb \label{eqn-discrete-path}
  D(t) := (d(t) , d^*(t)) .
\eqe  
For $n\in\BB N$ and $t\in \BB R$, let 
\eqb \label{eqn-Z^n-def}
U^n(t) := n^{-1/2} d (n t) ,\quad V^n(t) := n^{-1/2} d^*(n  t) , \quad Z^n(t) := (U^n(t) , V^n(t) ) .
\eqe 
It is easy to see that the condition that $X(1,2n) =\emptyset$ is equivalent to the condition that $Z^n([0,2]) \subset [0,\infty)^2$ and $Z^n(2) = 0$. 

Let $Z= (U,V)$ be a two-sided two-dimensional Brownian motion with $Z(0) = 0$ and variances and covariances as in~\eqref{eqn-bm-cov}.
It is shown in \cite[Theorem 2.5]{shef-burger} that as $n\rta \infty$, the random paths $Z^n$ defined in~\eqref{eqn-Z^n-def} converge in law in the topology of uniform convergence on compacts to the random path $Z$ of~\eqref{eqn-bm-cov}.

There are several stopping times for the word $X$, read backward, which we will use throughout this paper. Namely, let
\eqb \label{eqn-J-def}
J := \inf\left\{j \in \BB N \,:\, \text{$X(-j,-1)$ contains a burger}\right\} ,
\eqe
so that $\{J > n\}$ is the event that $X(-n,-1)$ contains no burgers. For $m\in\BB N$, let
\eqb \label{eqn-J^H-def}
J_m^H := \inf\left\{j \in \BB N \,:\, \mcl N_{\tc H}\left(X(-j,-1)\right) = m\right\} ,\quad L_m^H := d^*\left(X(-J_m^H,-1)\right) 
\eqe
be, respectively, the $m$th time a hamburger is added to the stack when we read $X$ backward and the number of cheeseburgers minus the number of cheeseburger orders in $X(-J_m^H,-1)$. Define $J_m^C$ and $L_m^C$ similarly with the roles of hamburgers and cheeseburgers interchanged.
 
Let 
\eqb \label{eqn-cone-exponent}
 \mu := \frac{\pi}{2\left( \pi - \arctan \frac{\sqrt{1-2p} }{p} \right) }  =  \frac{\kappa}{8} \in (1/2,1) ,
\eqe 
with $p$ and $\kappa$ related as in~\eqref{eqn-p-kappa}. The number $\mu$ arises frequently in the study of the hamburger-cheeseburger model with parameter $p$. For example, it is shown in \cite{gms-burger-cone} that the probability that the correlated Brownian motion $Z$ stays in the $\delta^{1/2}$-neighborhood of the  first quadrant for 1 unit of time is proportional to $\delta^\mu$ (Lemma 2.2; see also \cite{shimura-cone}); the probability that $X(1,n)$ contains no orders is regularly varying with exponent $\mu$ (Proposition 5.1); and $X(1,n)$ typically contains at most $n^{-(1-\mu) + o_n(1)}$ flexible orders (Corollary 5.2). It is also shown in \cite[Theorem 1.10]{gms-burger-local} that the probability that $X(1,2n) = \emptyset$ is proportional to $n^{-1-2\mu + o_n(1)}$. See~\cite{blr-exponents} for some additional calculations relating to the exponent $\mu$.

\subsection{Statements of main results}
\label{sec-main-results}

The first main result of this paper is the following finite-volume analogue of \cite{shef-burger}.  

\begin{thm} \label{thm-main}
Let $\dot Z = (\dot U , \dot V)$ be a correlated two-dimensional Brownian motion as in~\eqref{eqn-bm-cov} started from 0 conditioned to stay in the first quadrant during the time interval $[0,2]$ and satisfy $\dot Z(2) = 0$. Also define the paths $Z^n$ for $n\in\BB N$ as in~\eqref{eqn-Z^n-def}. As $n\rta\infty$, the conditional law of $Z^n|_{[0,2]}$ given $\{X(1,2n) =\emptyset\}$ converges to the law of $\dot Z$ (with respect to the topology of uniform convergence). 
\end{thm}

The correlated two-dimensional Brownian motion $\dot Z$ conditioned to stay in the first quadrant during the time interval $[0,2]$ and satisfy $\dot Z(2) = 0$ can be viewed as the two-dimensional analogue of a Brownian excursion. The process $\dot Z$ is defined and rigorously constructed in \cite[Section 3]{sphere-constructions}. See Section~\ref{sec-bm-cond} for a review of the definition of this process as well as an alternative (more explicit) construction. We remark that the one-dimensional distributions of correlated Brownian motion conditioned to stay in the first quadrant (as well as those of a more general class of conditioned Brownian motions) are computed in~\cite[Theorem 6]{dw-cones}. 

As noted above, \cite[Theorem 1.1]{sphere-constructions} implies that our Theorem~\ref{thm-main} can be interpreted as the statement that the contour function of a finite-volume FK planar map converges in the scaling limit to the analogous function for a $\op{CLE}_\kappa$ on a $\gamma$-quantum sphere. 
We will prove Theorem~\ref{thm-main} only in the case where $p \in (0,1/2)$ (i.e.\ $\kappa \in (4,8)$), which corresponds to a limiting Brownian motion with strictly positive correlation. The version of Theorem~\ref{thm-main} for $p=0$ (i.e.\ $\kappa = 8$), in which case $D(\cdot)$ is a two-dimensional simple random walk and the limiting object is a pair of independent Brownian excursions, follows from basic scaling limit results for simple random walks. See~\cite[Theorem 4]{dw-limit} for a much more general statement.
 
We remark that there are several other results in the literature concerning scaling limits of random walks conditioned to stay in a cone (and possibly on their endpoint); see, e.g.~\cite{shimura-cone-walk,garbit-cone-walk,dw-cones,dw-limit}. Theorem~\ref{thm-main} extends~\cite[Theorem 4]{dw-limit} to a certain non-Markovian random walk. 

We will also obtain a finite-volume analogue of the main result of \cite{gms-burger-cone}. To state this result, we need to recall some definitions from \cite{gms-burger-cone}.

\begin{defn}\label{def-cone-time}
A time $t$ is called a \emph{(weak) $\pi/2$-cone time} for a function $Z = (U,V) : \BB R \rta \BB R^2$ if there exists $t'< t$ such that $U_s \geq U_{t }$ and $V_s \geq V_{t }$ for $s\in [t'   , t ]$. Equivalently, $Z([t'   , t ])$ is contained in the ``cone" $Z_{t } + \{z\in \BB C : \op{arg} z \in [0,\pi/2]\}$. We write $  v_Z(t)$ for the infimum of the times $t'$ for which this condition is satisfied, i.e. $  v_Z(t)$ is the entrance time of the cone. The \emph{$\pi/2$-cone interval} corresponding to the time $t$ is $[v_Z(t) , t]$. We say that $t$ is a \emph{left (resp. right) $\pi/2$-cone time} if $V_t = V_{v_Z(t)}$ (resp. $U_t = U_{v_Z(t)}$). Two $\pi/2$-cone times for $Z$ are said to be in the \emph{same direction} if they are both left or both right $\pi/2$-cone times, and in the \emph{opposite direction} otherwise. For a $\pi/2$-cone time $t$, we write $ u_Z(t)$ for the supremum of the times $t^* < t$ such that
\eqbn
\inf_{s\in [t^* , t]} U_s  < U_t  \quad \op{and} \quad \inf_{s\in [t^* , t]} V_s  < V_t  .
\eqen
That is, $ u_Z(t)$ is the last time before $t$ that $Z$ crosses the boundary line of the cone which it does not cross at time $ v_Z(t)$.  
\end{defn} 

If $i \in \BB Z$ is such that $X_i = \tb F$, then $i/n$ and $(i-1)/n$ are (weak) $\pi/2$-cone time for $Z^n$ with $v_{Z^n}((i-1)/n) = \phi(i)/n$.

\begin{defn}\label{def-maximal}
A $\pi/2$-cone time for $Z$ is called a \emph{maximal $\pi/2$-cone time} in an (open or closed) interval $I\subset \BB R$ if $[ v_Z(t) , t] \subset I$ and there is no $\pi/2$-cone time $t'$ for $Z$ such that $[v_Z(t') , t']\subset I$ and $[ v_Z(t) , t] \subset (v_Z(t') , t')$. An integer $i\in \BB Z$ is called a \emph{maximal flexible order time} in an interval $I\subset \BB R$ if $X_i = \tb F$, $\{\phi(i) , \dots , i\} \subset I$, and there is no $i' \in \BB Z$ with $\phi(i') = \tb F$, $\{\phi(i) , \dots , i\} \subset \{\phi(i')+1 , \dots , i'-1\}$, and $\{\phi(i') , \dots, i'\}\subset I$. 
\end{defn}

We can now state our second main result, which is an exact analogue of \cite[Theorem 1.9]{gms-burger-cone} in the setting where we condition on $\{X(1,2n)=\emptyset\}$. 
 
\begin{thm} \label{thm-cone-limit-finite}
Let $\dot Z = (\dot U , \dot V)$ be a correlated two-dimensional Brownian motion as in~\eqref{eqn-bm-cov} conditioned to stay in the first quadrant during the time interval $[0,2]$ and satisfy $\dot Z(2) = 0$. Let $\mcl T$ be the set of $\pi/2$-cone times for $\dot Z$. For $n\in\BB N$, let $\dot X^n = \dot X_1^n \dots \dot X_{2n}^n$ be sampled according to the law of the word $X_1\dots X_{2n}$ conditioned on the event $\{X(1,2n) = \emptyset\}$ and let $\dot Z^n :[0,2]\rta\BB R^2$ be the path~\eqref{eqn-Z^n-def} corresponding to $\dot X^n$. Let $\mcl I_n$ be the set of $i\in [1,2n]_{\BB Z}$ such that $\dot X_i^n = \tb F$ and for $n\in\BB N$ let $\mcl T_n = \{n^{-1} (i-1) \,:\, i\in \mcl I_n\}$. Fix a countable dense set $\mcl Q \subset \BB R$. There is a coupling of the sequence $(\dot X^n)$ with the path $\dot Z$ such that the following holds a.s.
\begin{enumerate}
\item $\dot Z^n \rta \dot Z$ uniformly. \label{item-cone-limit-Z}
\item $\mcl T$ is precisely the set of limits of convergent sequences $(t_{n_j}) \in \mcl T_{n_j}$ satisfying $\liminf_{j\rta\infty} (t_{n_j} - v_{\dot Z^{n_j}}(t_{n_j})) > 0$ as $(n_j)$ ranges over all strictly increasing sequences of positive integers. \label{item-cone-limit-equals}
\item For each sequence of times $t_{n_j} \in \mcl T_{n_j}$ as in condition~\ref{item-cone-limit-equals}, we have $\lim_{j \rta\infty} v_{\dot Z^{n_j}}(t_{n_j}) = v_{\dot Z}(t)$, $\lim_{j\rta\infty} u_{\dot Z^{n_j}}(t_{n_j}) = u_{\dot Z}(t)$, and the direction of the $\pi/2$-cone time $t_{n_j}$ is the same as the direction of $t$ for sufficiently large $j$. \label{item-cone-limit-times}
\item Suppose given an open interval $I \subset [0,2]$ with endpoints in $\mcl Q$ and $a \in I \cap \mcl Q$. Let $t$ be the maximal (Definition~\ref{def-maximal}) $\pi/2$-cone time for $\dot Z$ in $I$ with $a\in [ v_{\dot Z}(t),t]$. For $n\in\BB N$, let $i_n$ be the maximal flexible order time (with respect to $\dot X^n$) $i$ in $n I$ with $a n \in [\phi(i)   , i ]$ (if such an $i$ exists) and $i_n = \lfloor a n \rfloor$ otherwise; and let $t_n = n^{-1} i_n$. Then a.s.\ $t_n\rta t$. \label{item-cone-limit-maximal}
\item For $a \in (0,2)$ and $a\in (0,2)$, let $\tau^{a,r}$ be the smallest $\pi/2$-cone time $t$ for $\dot Z$ such that $t\geq a$ and $t - v_{\dot Z}(t) \geq r$. For $n\in\BB N$, let $\iota_n^{a,r}$ be the smallest $i\in\BB N$ such that $X^n_i = \tb F$, $i \geq a n$, and $i - \phi(i) \geq r n - 1$; and let $\tau_n^{a,r} = n^{-1} \iota_n^{a,r}$. We have $\tau^{a,r}_n \rta \tau^{a,r}$ for each $(a,r) \in \mcl Q\times (\mcl Q\cap (0,\infty))$. \label{item-cone-limit-stopping}
\end{enumerate}
\end{thm}

Theorem~\ref{thm-cone-limit-finite} will turn out to be a straightforward consequence of our proof of Theorem~\ref{thm-main} combined with \cite[Theorem 1.9]{gms-burger-cone}. 

The $\pi/2$-cone times of the correlated Brownian motion $\dot Z$ encode the $\op{CLE}_\kappa$ loops on a $\gamma$-quantum sphere; whereas the times $i$ for which $X_i = \tb F$ encode the FK loops on the FK planar map $(M,\mcl L)$ (see \cite{wedges,gms-burger-cone,gwynne-miller-cle} for more details regarding this correspondence). Therefore, Theorems~\ref{thm-main} and~\ref{thm-cone-limit-finite} together imply that many interesting functionals of the FK loops (e.g.\ the boundary lengths and areas of their complementary connected components and the adjacency graph on the set of loops) converge in the scaling limit to the corresponding functionals of the $\op{CLE}_\kappa$ loops on a $\gamma$-quantum sphere. This answers \cite[Question 13.3]{wedges} in the finite volume case.

\subsection{Outline}
\label{sec-outline}

In this subsection we will give an overview of the proofs of our main results and some motivation for our method. We will also describe the content of the remainder of this article. 

\begin{remark}
Sections~\ref{sec-bm-cond} through~\ref{sec-conclusion} can be read largely independently of each other. Only a few main results from each section are needed in subsequent sections, and these main results are stated near the beginning of each section. 
\end{remark} 

In Section~\ref{sec-bm-cond}, we will describe how to make sense of Brownian motion conditioned on various zero-probability events. In particular, we will give a new construction of the limiting object in Theorem~\ref{thm-main} (which was originally constructed in~\cite[Section 3]{sphere-constructions}) which gives an explicit formula for the law of the conditioned Brownian path at each time.

In the remainder of the paper, we turn our attention to the proofs of Theorems~\ref{thm-main} and~\ref{thm-cone-limit-finite}. 
See Figure~\ref{fig-whole-path} for an illustration of our argument.

\begin{figure}[ht!]
 \begin{center}
\includegraphics{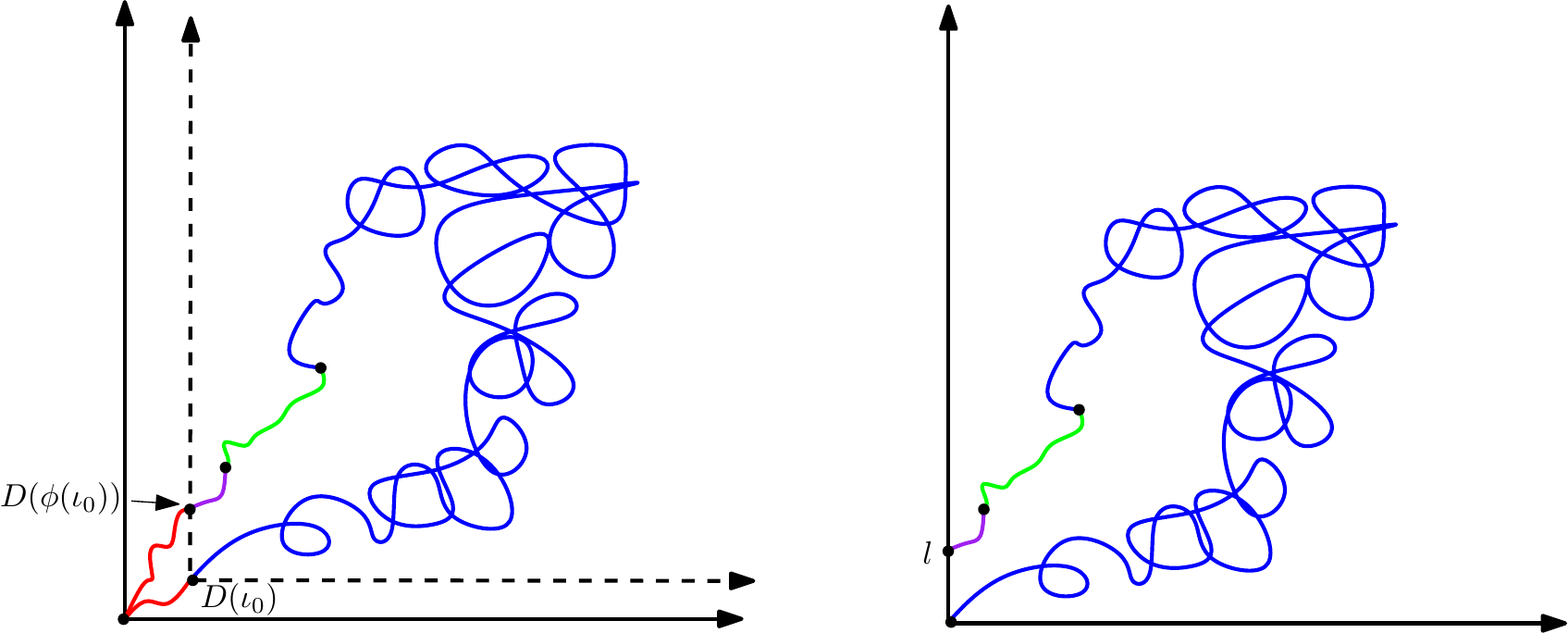} 
\caption{An illustration of the proof of Theorems~\ref{thm-main} and~\ref{thm-cone-limit-finite}. The left figure shows that path $D$ of~\eqref{eqn-discrete-path} on the event $\{X(1,2n) = \emptyset\}$. The first step, accomplished in Section~\ref{sec-big-loop}, is to prove the existence of a large interval $[\phi(\iota_0) ,\iota_0]_{\BB Z}$ with $\iota_0 - \phi(\iota_0) \approx 2n- n^\xi$ for $\xi$ slightly smaller than 1. The image of this interval under the path $D$ is the left curve minus the two red segments. We then condition on the two red segments and consider the problem of showing convergence conditioned on the event that $X(1,2n)$ contains $l \asymp n^{\xi/2}$ cheeseburger orders and no other symbols (right figure). We analyze this situation by first growing the curve backward up to time $-(2n-\delta n)$ for fixed $\delta>0$ (blue segment). Conditional on a realization of the blue curve, we continue until we get to a certain time $j$ at distance of order $n^\xi$ from $2n$, with the property that each coordinate of $D(X(-j,-1)$ lies at distance of order $n^{\xi/2}$ from the origin (green segment). Estimates for this segment of the curve are proven in Sections~\ref{sec-pi-reg} and~\ref{sec-end-local}. Finally, we grow the purple segment of the curve until the path exits the first quadrant at $l$. We analyze this part of the curve using the results of \cite{gms-burger-local} together with the regularity conditions on the event of Section~\ref{sec-big-loop}.}\label{fig-whole-path}
\end{center}
\end{figure}

The basic idea of the proofs of Theorems~\ref{thm-main} and~\ref{thm-cone-limit-finite} is as follows. Let $\delta >0$ be small but independent of $n$. Suppose we read the word $X$ backward and we are given a realization $x$ of $X_{-m} \dots X_{-1}$ for some $m  \in [2n-\delta n , 2n]_{\BB Z}$ such that the reduced word $\mcl R(x)$ contains no burgers. Let $\frk h (x) = \mcl N_{\tb H}\left(\mcl R(x)\right)$ and $\frk c(x) = \mcl N_{\tb C}\left(\mcl R(x)\right)$ be the number of hamburger orders and cheeseburger orders, respectively, in the reduced word corresponding to this realization, as in Definition~\ref{def-theta-count}. Ideally, we would like to establish the following (although we will actually prove something similar but more complicated).
\begin{itemize} 
\item $ \BB P\left(X(-2n,-1) =\emptyset \,|\, X_{-m} \dots X_{-1} =  x\right) $ is bounded above by a constant times $n^{-1-\mu}$ times a slowly varying function of $n$, no matter how pathological the realization $x$ is. 
\item For a typical realization $x$, $ \BB P\left(X(-2n,-1) =\emptyset \,|\, X_{-m} \dots X_{-1} =  x\right) $ is bounded below by a constant times $n^{-1-\mu}$ times this same slowly varying function of $n$.
\item Conditional on $\{X(-2n,-1) =\emptyset\}$ and $X_{-m} \dots X_{-1} =  x$, it is unlikely that $\sup_{i\in [m,2n]_{\BB Z}} |X(-j,-m-1)|$ is larger than $n^{1/2}$ times some quantity which tends to 0 as $\delta \rta 0$. 
\item $\BB P\left(X(-2n,-1) =\emptyset \,|\, X_{-m} \dots X_{-1} =  x\right)$ depends ``continuously" on $(\frk h(x) , \frk c(x))$, in the sense that
\[ 
\frac{\BB P\left(X(-2n,-1) =\emptyset \,|\, X_{-m} \dots X_{-1} =  x\right)}{\BB P\left(X(-2n,-1) =\emptyset \,|\, X_{-m} \dots X_{-1} =  x' \right)}
\]
 is close to 1 when $| \frk h(x) - \frk h(x')|$ and $|\frk c(x) - \frk c(x')|$ are bounded above by a small multiple of $n^{1/2}$. 
\end{itemize}
Using Bayes rule and the above statements, one obtains that the conditional law of $X_{-m}\dots X_{-1}$ given $\{X(-2n,-1) = \emptyset\}$ is close to its conditional law given only that $X(-m,-1)$ contains no burgers; and that the path $Z^n(-\cdot)$ unlikely to move very much between times $n^{-1} m$ and $2$ when we condition on $\{X(-2n,-1) =\emptyset\}$ and $X_{-m} \dots X_{-1} =  x$. This allows us to compare the limit of the conditional law of $Z^n$ given $\{X(-2n,-1) =\emptyset\}$ to the limit of the conditional law of $Z^n$ given that $ X(-2n,-1)  $
contains no orders. This latter scaling limit is obtained in \cite[Theorem 4.1]{gms-burger-cone}. See~\cite{gms-burger-local} for a similar argument when we condition the tip of the path $Z^n$ to be in the interior of the first quadrant, rather than at the origin; and~\cite{cc-pos-bridge,dw-limit} for similar arguments in the case of random walks with iid increments. 

In order to use the above approach, one requires ``local estimates", i.e.\ estimates for the probability that the word $X(-2n,-m-1)$ contains a \emph{particular} number of symbols of a given type, under various conditionings. Several such estimates are proven in \cite{gms-burger-local}, and we will expand on these estimates in this paper. 

However, local estimates are not sufficient for our purposes, for the following reason. The event $\{X(-2n,-1) = \emptyset\}$ is not determined solely by the number of symbols of each type in $X(-2n,-m-1)$ and $X(-m,-1)$. Rather, due to the presence of $\tb F$-symbols, this event depends in a complicated way on the precise ordering of the symbols in these two words. We know by \cite[Corollary 5.2]{gms-burger-cone} that with overwhelming probability, the number of $\tb F$-symbols in $X(-m,-1)$ is of order at most $n^{1-\mu + o_n(1)} = o_n(n^{1/2})$. Nevertheless, fluctuations of order $n^{1-\mu + o_n(1)}$ due to these flexible orders can still effect whether $X(-2n,-1) =\emptyset$. Even inductive arguments as in~\cite[Section 5]{gms-burger-local} produce estimates which are off by a  factor of $n^{o_n(1)}$, which is not sufficiently precise for our purposes. 

We are unable to address the combinatorics of how the words $X(-2n,-m-1)$ and $X(-m,-1)$ match up in the presence of $\tb F$-symbols (indeed, we expect that it is not possible to do so in a sufficiently precise way for our purposes). Instead, we use various methods to avoid dealing with the matches of flexible orders directly. 

In Section~\ref{sec-big-loop}, we prove a proposition which reduces the problem of showing convergence conditioned on $\{X(-2n,-1)= \emptyset\}$ to the problem of showing convergence conditioned on the event that $X(-2n,-1)$ contains a specific number (say $l$) of cheeseburger orders, with $l \asymp n^{\xi/2}$ for $\xi$ a little smaller than 1, and no other symbols; plus some further regularity conditions. This problem turns out to be more tractable than the original problem, since we can arrange that with high probability there is a $k\in\BB N$ with $k \asymp n^\xi$ with the property that all of the flexible orders in $X(-2n+k,-1)$ are matched to hamburgers in $X(-2n,-2n+k-1)$ (so we only need to estimate about the total number of symbols of each type in $X(-2n,-2n+k-1)$, not their ordering). 

In order to reduce our problem in this manner, we will exploit a symmetry of finite-volume FK planar maps which does not have a straightforward description in terms of the hamburger-cheeseburger model. Namely, we will use that the law of a finite-volume rooted FK planar map is invariant under uniform re-rooting. For this argument we read the word $X$ forward, rather than backward. The scaling limit results of \cite{gms-burger-cone,gwynne-miller-cle} allow us to compare the local behavior of the FK loops on a finite-volume FK planar maps to the local behavior of $\op{CLE}_\kappa$ loops on a $\gamma$-quantum cone when $n$ is large. Using this, we will argue that if $n$ is large and we condition on $\{X(1,2n) = \emptyset\}$, then with high probability the following is true. Let $(M,e_0,\mcl L)$ be the rooted FK planar map associated with $X_1,\dots,X_{2n}$. Also let $(Q,\BB e_0)$ be the associated quadrangulation as in~\cite[Section 4.1]{shef-burger}. Then times $i\in [1,2n]_{\BB Z}$ correspond to edges $\lambda(i)$ of $Q$ which are contained in a complementary connected component of some loop with area $\asymp n^\xi$ and boundary length $\asymp n^{\xi/2}$. By the explicit description of the loops of $\mcl L$ in terms of the $\tb F$-symbols in the word $X_{1} \dots X_{2n}$, if $i \in [1,2n]_{\BB Z}$ is such a time and we re-choose the root edge of $e_0$ of the FK planar map in such a way that $\lambda(i)$ becomes the root edge of $Q$, then there is a time $\iota_0^n \in [2n-n^\xi , 2n]_{\BB Z}$ such that $X_{ \iota_0^n} = \tb F$, $X_{\phi(\iota_0^n)} = \tc H$, $\phi(\iota_0^n) \in [0,n^\xi]_{\BB Z}$, and $|X(\phi(\iota_0^n) , \iota_0^n)| \asymp n^{\xi/2}$.
By invariance of the rooted FK planar map $(M,e_0, \mcl L)$ under uniform rerooting, this shows that such a time $\iota_0^n$ exists with high probability when we condition on $\{X(1,2n)=\emptyset\}$. The restriction of $Z^n$ to $[n^{-1}\phi(\iota_0^n) , n^{-1} \iota_0^n] $ is likely to be close to $Z^n$ itself, and the conditional law of $X_{ \phi(\iota_0^n)} \dots X_{\iota_0^n}$ given $\phi(\iota_0^n)$, $\iota_0^n$, and $|X(\phi(\iota_0^n) , \iota_0^n)|$ is the same as its conditional law given that $X(\phi(\iota_0^n ) , \iota_0^n)$ contains exactly $|X(\phi(\iota_0^n) , \iota_0^n)|$ cheeseburger orders and no other symbols. This accomplishes the desired reduction. 

In Section~\ref{sec-pi-reg}, we will establish the existence of a time $\wt\pi_n$ with the following properties. On the event of Section~\ref{sec-big-loop}, each cheeseburger in $X(-J,-\wt\pi_n-1)$ is matched to a cheeseburger order in $X(-\wt\pi_n,-1)$ (with $J$ as in~\eqref{eqn-J-def}); and, although $\wt\pi_n$ is not a stopping time for the word $X$, read backward, the conditional law of $X_{-J}\dots X_{-\wt\pi_n-1}$ is ``close" to the conditional law we would get if $\wt\pi_n$ were a stopping time. The time $\wt\pi_n$ corresponds to the last point of the green segment and the first point of the purple segment in Figure~\ref{fig-whole-path}. The aforementioned properties of $\wt\pi_n$ will enable us to apply the estimates of \cite{gms-burger-local} to the purple segment in Figure~\ref{fig-whole-path}.

In Section~\ref{sec-end-local}, we will use the results of Section~\ref{sec-pi-reg} together with the estimates of \cite{gms-burger-local} to estimate the conditional probability that the path $D$ exits the first quadrant at the marked point $l$ in Figure~\ref{fig-whole-path} given a realization of the blue part of the path. These estimates will be used together with Bayes' rule to compare the conditional law of the blue part of the path in Figure~\ref{fig-whole-path} given the event depicted on the right side of the figure to its conditional law given only that it stays in the first quadrant until time $m$; this latter conditional law can in turn be estimated using \cite[Theorem 4.1]{gms-burger-cone}. Roughly speaking, we obtain that the former law is absolutely continuous with respect to the latter, with Radon-Nikodym derivative bounded independently of $n$ and of $l\asymp n^{\xi/2}$; see Propositions~\ref{prop-E^l-abs-cont} and~\ref{prop-end-box} for precise statements. 

In Section~\ref{sec-conclusion}, we will combine the results of Sections~\ref{sec-bm-cond},~\ref{sec-big-loop}, and~\ref{sec-end-local} to prove Theorems~\ref{thm-main} and~\ref{thm-cone-limit-finite}. The results of Sections~\ref{sec-end-local} and~\ref{sec-big-loop} almost immediately imply tightness of the conditional laws of $Z^n|_{[0,2]}$ given $\{X(1,2n) = \emptyset\}$. Hence to prove Theorem~\ref{thm-main} we just need to show that any subsequently limiting law is that of the path $\dot Z$ in Theorem~\ref{thm-main}. 

The estimates of Section~\ref{sec-end-local} are not quite enough for this purpose, as these estimates are off by constant factors, so can at best tell us that a subsequential limiting law is mutually absolutely continuous with respect to the law of $\dot Z$ when restricted to $[0,t]$ for $t<2$. To get around this difficulty, we use an argument which is similar to the argument of the earlier sections of the paper, but with a small multiple of $n^{1/2}$ in place of $n^{\xi/2}$. Namely, we use the results of Sections~\ref{sec-big-loop} and~\ref{sec-end-local} to show that with high conditional probability given $\{X(1,2n)=\emptyset\}$, there is a ``macroscopic $\tb F$-interval" in $[1,2n]_{\BB Z}$ satisfying certain conditions, i.e., an $i\in [1,2n]_{\BB Z}$ such that $X_{ i} = \tb F$, $X_{\phi(i)} = \tc H$, $2n-i$ and $\phi(i)$ are small (but $n$-independent) multiples of $n^{1/2}$; $|X(\phi(i) , i)|$ is a small (but $n$-independent) multiple of $n^{1/2}$; and several regularity conditions are satisfied. We then condition on $X_1\dots X_{\phi(i)-1}$, $X_i\dots X_{2n}$, and $|X(\phi(i) , i)|$ for such an $i$ and show that, roughly speaking, the conditional law of the restriction of $Z^n$ to $[\phi(i)/n, i/n]$ is close to what we expect it to be in the limit. This will yield Theorem~\ref{thm-main}. Our second main result Theorem~\ref{thm-cone-limit-finite} is extracted in Section~\ref{sec-cone-proof} from the same estimates used to prove Theorem~\ref{thm-main} together with the results of~\cite{gms-burger-cone}.

\section{Conditioned Brownian motion}
\label{sec-bm-cond}

\subsection{Definitions and characterization lemmas}
\label{sec-bm-def}

In this section we will consider a correlated two-dimensional Brownian motion $Z = (U,V)$ as in~\eqref{eqn-bm-cov} conditioned on various zero-probability events. 

We start by considering the law of $Z$ conditioned to stay in the first quadrant, which is also studied in~\cite{shimura-cone,gms-burger-cone}.
In \cite[Theorem 2]{shimura-cone}, Shimura proves, for each $\theta \in (0,2\pi)$, the existence of a probability measure on the space of continuous functions $[0,1] \rta \BB R^2$ which is interpreted as the law of a standard two-dimensional Brownian motion (started from 0) conditioned to stay in the cone $\{z\in\BB C \,:\, 0 \leq \op{arg} z  \leq \theta\}$ for one unit of time. By applying an appropriate linear transformation to a path with this law, we obtain a random path $\wh Z$ which we interpret as the correlated two-dimensional Brownian motion $Z$ in~\eqref{eqn-bm-cov} conditioned to stay in the first quadrant for one unit of time. By scaling, one obtains the law of $Z$ conditioned to stay in the first quadrant for a general $t >0$ units of time.
 The following lemma, which is \cite[Lemma 2.1]{gms-burger-cone}, uniquely characterizes the law of $\wh Z$.

\begin{lem}[\cite{gms-burger-cone}] \label{prop-bm-meander}
Let $\wh Z = (\wh U , \wh V)  : [0,1]\rta\BB R^2$ be a random path distributed according to the conditional law of $Z $ (started from 0) given that it stays in the first quadrant for one unit of time. Then $\wh Z$ is a.s.\ continuous and satisfies the following conditions.
\begin{enumerate}
\item For each $t \in (0,1]$, a.s.\ $\wh U(t)>0$ and $\wh V(t)>0$. \label{item-bm-meander-pos}
\item For each $t \in (0,1)$, the regular conditional law of $\wh Z|_{[ t , 1]}$ given $\wh Z|_{[0,t]}$ is that of a Brownian motion with covariances as in~\eqref{eqn-bm-cov}, starting from $\wh Z(t)$, parametrized by $[t,1]$, and conditioned on the (a.s.\ positive probability) event that it stays in the first quadrant.  \label{item-bm-meander-markov}
\end{enumerate}
If $\wt Z    : [0,1]\rta\BB R^2$ is another random a.s.\ continuous path satisfying the above two conditions, then $\wt Z\eqD \wh Z$.
\end{lem}

We will also need to consider the law of the correlated Brownian motion $ Z$ of~\eqref{eqn-bm-cov} (started from 0) conditioned to stay in the first quadrant for one unit of time and return to the origin at time 0. This law is described in \cite[Theorem 3.1]{sphere-constructions}, and (by that theorem) is uniquely characterized as follows.

\begin{lem} \label{prop-bm-excursion}
Let $\dot Z = (\dot U , \dot V)  : [0,1]\rta\BB R^2$ be sampled from the conditional law of $Z|_{[0,1]}$ (started from 0) given that it stays in the first quadrant and returns to 0 at time 1. Then $\dot Z$ is a.s.\ continuous and satisfies the following conditions.
\begin{enumerate}
\item For each $t \in [0,1]$, a.s.\ $\dot U(t) \not=0$ and $\dot V(t) \not=0$. \label{item-bm-excursion-pos}
\item For each $0<s<t<1$, the regular conditional law of $\dot Z|_{[s, t]}$ given $\dot Z|_{[0,s]}$ and $\dot Z|_{[t,1]}$ is that of a Brownian bridge with covariances as in~\eqref{eqn-bm-cov} from $\dot Z(s)$ to $\dot Z(t)$, parametrized by $[s,t]$, and conditioned on the (a.s.\ positive probability) event that it stays in the first quadrant. \label{item-bm-excursion-markov}
\end{enumerate}
If $\ddot Z   : [0,1]\rta\BB R^2$ is another random continuous path satisfying the above two conditions, then $\ddot Z\eqD \dot Z$.
\end{lem}

Lemma~\ref{prop-bm-excursion} is stated for the time interval $[0,1]$, but by Brownian scaling the analogous statement holds for $[0,T]$ for any fixed $T>0$. In particular, for $T=2$ we get the limiting law in Theorem~\ref{thm-main}. We note that Brownian motion conditioned to stay in a cone and return to the origin at time 1 (as well as more general conditioned Brownian motions) is also studied in~\cite{dw-cones,dw-limit} as a limit of conditioned random walks with iid increments.

\subsection{A more explicit construction of a correlated Brownian loop in the first quadrant}
\label{sec-bm-construction}

In this subsection we will give an alternative construction of correlated Brownian motion $\dot Z$ conditioned to stay in the first quadrant for one unit of time and satisfy $\dot Z(1) = 0$. Our construction gives a stronger convergence statement and more information about the limiting law than the construction in \cite[Section 3]{sphere-constructions}. 

To state our result, we need the following notation.  Let $\wh Z  = (\wh U ,\wh V )$ be a random path with the law of $Z$ started from 0 and conditioned to stay in the first quadrant until time $1$, as in Section~\ref{sec-bm-def}. 
For $t > 1$, let $\wh f_t$ be the density with respect to Lebesgue measure of the law of $\wh Z^t(t)$, where $\wh Z^t$ is a correlated Brownian motion as in~\eqref{eqn-bm-cov} started from 0 and conditioned to stay in the first quadrant until time $t$. By \cite[Equation 3.2]{shimura-cone} and Brownian scaling,
\eqb \label{eqn-f-density-def}
\wh f_t(z) = \frac{\op{det} A }{ 2^\mu \Gamma(\mu)  t^{1 + \mu}  } |Az|^{2\mu} e^{-|Az|^2/2t} \sin\left(2\mu \op{arg} (Az)\right)
\eqe 
where
\eqb \label{eqn-bm-matrix}
A := \sqrt{\frac{2(1-p)}{1-2p }} \left(  \begin{array}{cc}
1 &  - \frac{p}{1-p}   \\ 
0  &  \frac{\sqrt{1-2p}}{1-p}
\end{array}        \right) ,
\eqe
so that $A \wh Z$ is a standard two-dimensional Brownian (i.e.\ zero covariance and unit means) motion conditioned to stay in a cone. 
Note that our $\mu$ is equal to $1/2$ times the exponent $\mu$ of \cite{shimura-cone}. 

The main result of this section is the following.

\begin{prop} \label{prop-bm-limit}
Suppose we are in the setting described just above. Fix $C>1$ and for $\delta>0$, let
\eqb \label{eqn-bm-B-def}
\mcl B^\delta(C) := \left\{\wh Z(1) \in \left[C^{-1}\delta^{1/2} , C\delta^{1/2}\right]  \right\} .
\eqe  
\begin{enumerate}
\item As $\delta \rta 0$, the conditional laws of $\wh Z$ given $\mcl B^\delta(C)$ converges weakly (with respect to the uniform topology) to the law of a correlated Brownian motion $\dot Z$ with covariance as in~\eqref{eqn-bm-cov} conditioned to stay in the first quadrant until time 1 and satisfy $\dot Z(1) = 0$. \label{item-bm-limit-weak}
\item For each fixed $t\in (0,1)$, the conditional law of $\wh Z|_{[0,t]}$ given $\mcl B^\delta(C)$ converges in the strong topology to the law of $\dot Z|_{[0,1]}$ as $\delta\rta 0$, in the sense that for each Borel measurable subset $\mcl U$ of the set of paths $[0,t]\rta \BB R^2$, 
\eqb \label{eqn-bm-limit-strong} 
\BB P\left(\wh Z|_{[0,t]} \in \mcl U \,|\, \mcl B^\delta(C) \right) \rta \BB P\left(\dot Z|_{[0,t]} \in \mcl U \right) .
\eqe 
\label{item-bm-limit-strong}
\item For each $t \in (0,1)$, the laws of $\dot Z|_{[0,t]}$ and $\wh Z|_{[0,t]}$ are mutually absolutely continuous, and the Radon-Nikodym derivative of the law of $\dot Z|_{[0,t]}$ with respect to the law of $\wh Z|_{[0,t]}$ is given by $g_t(\dot Z(t))$, where
\eqb \label{eqn-g-density-def}
g_t(z) :=  \frac{c_t  \wh f_{1-t}(z) }{ \BB P^z\left(Z([0,1-t]) \subset (0,\infty)^2\right)  }
\eqe  
with $c_t$ a normalizing constant, $\wh f_{1-t}(z)$ as in~\eqref{eqn-f-density-def}, and $\BB P^z$ the law of a correlated Brownian motion as in~\eqref{eqn-bm-cov} started from $z$. 
\label{item-bm-limit-density}
\end{enumerate}
\end{prop}

For the proof of Proposition~\ref{prop-bm-limit}, we use the following notation. 
\begin{itemize}
\item For $z\in \BB R^2$, let $\BB P^z$ be as in the statement of Proposition~\ref{prop-bm-limit} and let $f_t^z$ be the density of $Z(t)$ under $\BB P^z$ (with respect to Lebesgue measure). Let $\BB E^z$ be the corresponding expectation.
\item Let $G_t $ be the event that $Z$ stays in the first quadrant until time $t$. 
\item For $z,w\in \BB R^2$ and $t\in\BB R$, let $\BB P_t^{z,w}$ be the law of a correlated two-dimensional Brownian bridge from $z$ to $w$ in time $t$, with covariance as in~\eqref{eqn-bm-cov}.
\item For $t \in [0,1]$ let $\wh f_t^1$ be the density of the unconditional law of $\wh Z(t)$ with respect to Lebesgue measure. 
\end{itemize}

\begin{lem} \label{prop-bm-rn-deriv}
Fix $C>1$ and $t\in (0,1)$. For $\delta>0$, let $\mcl B^\delta(C)$ be defined as in~\eqref{eqn-bm-B-def}. For each bounded measurable function $\psi : (0,\infty)^2 \rta (0,\infty)^2$, we have 
\eqbn
\lim_{\delta \rta 0} \BB E\left(\psi(\wh Z(t) ) \,|\, \mcl B^\delta(C)\right) = \frac{   \int_{(0,\infty)^2} \psi(z) g_t(z) \wh f_t^1(z) \, dz   }{   \int_{(0,\infty)^2} g_t(z) \wh f_t^1(z)  \, dz   } ,
\eqen
with $g_t$ as in~\eqref{eqn-g-density-def} and $\wh f_t^1$ defined just above. 
\end{lem}
\begin{proof} 
For the proof we use the notation introduced just above. 
For each $t\in (0,1)$, the conditional law of $\wh Z|_{[t,1]}$ given $\wh Z|_{[0,t]}$ is $\BB P^{\wh Z(t)}\left(\cdot \,|\, G_{1-t} \right)$. Therefore, 
\begin{align} \label{eqn-Z-flip1}
\BB E \left(  \psi(\wh Z(t))  \BB 1_{ \mcl B^\delta(C) }  \,|\, \wh Z|_{[0,t]} \right)  
&= \frac{  \BB P^{\wh Z(t)}\left(  Z(1-t) \in \left[C^{-1}\delta^{1/2} , C\delta^{1/2}\right]^2  ,\,      G_{1-t} \right)  \psi(\wh Z(t))   }{ \BB P^{\wh Z(t)}\left(  G_{1-t} \right)  } \notag \\ 
&=   \frac{    \BB E^{\wh Z(t)}\left( \BB P^{\wh Z(t)}\left(G_{1-t} \,|\, Z(1-t) \right) \BB 1_{\left(  Z(1-t) \in \left[C^{-1}\delta^{1/2} , C\delta^{1/2}\right]^2 \right)} \right)  \psi(Z(t))   }{X_t(\wh Z(t))}
\end{align}
where
\eqb \label{eqn-cond-constant}
X_t (z) : =  \BB P^z\left(  G_{1-t} \right)  .
\eqe  
For each $z \in \BB R^2$, the regular conditional law of $Z|_{[0,1-t]}$ given $Z(1-t)$ under $\BB P^z$ is $\BB P^{z,Z(1-t)}_{1-t}$. Furthermore, the density of the law of $Z(1-t)$ under $\BB P^z$ with respect to Lebesgue measure is given by $f_{1-t}^z $. Therefore, taking expectations in~\eqref{eqn-Z-flip1} yields
\begin{align} \label{eqn-Z-flip2}
  \int_{(0,\infty)^2} \int_{\left[C^{-1}\delta^{1/2} , C\delta^{1/2}\right]^2} 
  \BB P_{1-t}^{z,w}\left(  G_{1-t}    \right) f_{1-t}^z\left(w \right)     \psi(z) \wh f_t^1(z) X_t(z)^{-1} 
   \, dw \, dz .
\end{align}
By reversibility of Brownian bridge, we have $\BB P_{1-t}^{z,w}\left(  G_{1-t}    \right)  = \BB P_{1-t}^{w,z}\left(  G_{1-t}    \right)$. Furthermore, the Gaussian density $f_{1-t}^z(w)$ is symmetric in $z$ and $w$. By combining these observations with Fubini's theorem, we see that~\eqref{eqn-Z-flip2} equals
\begin{align}  
& \int_{\left[C^{-1}\delta^{1/2} , C\delta^{1/2}\right]^2}  \int_{(0,\infty)^2} \BB P_{1-t}^{w,z}\left(  G_{1-t}    \right) f_{1-t}^w\left(z \right)     \psi(z) \wh f_t^1(z) X_t(z)^{-1}  \, dz  \, dw \notag \\
&\qquad =   \int_{\left[C^{-1}\delta^{1/2} , C\delta^{1/2}\right]^2}  \BB E^w \left(  \psi(Z(1-t)) \wh f_t^1(Z(1-t)) X_t(Z(1-t) )^{-1}  \BB 1_{G_{1-t}}         \right) \, dw . \notag
\end{align}
For $w\in (0,\infty)^2$, let $A$ be as in~\eqref{eqn-bm-matrix} and let
\eqbn
\Delta(w) := |Aw|^{2\mu} \sin\left(2\mu \op{arg} A w\right) .
\eqen
It is clear from \cite[Equation 3.2]{shimura-cone} (which gives an explicit formula for the density of $A\wh Z(t)$) that $\wh f_t^1(z) X_t(z)^{-1}$ is a bounded function of $z$. By~\cite[Corollary to Lemma 4]{shimura-cone} and the discussion immediately thereafter, for $w\in \left[C^{-1}\delta^{1/2} , C\delta^{1/2}\right]^2$ we have
\alb
\BB E^w \left(  \psi(Z(1-t)) \wh f_t^1(Z(1-t)) X_t(Z(1-t) )^{-1} \BB 1_{G_{1-t}}  \right) = a_0 \Delta(w)\left(  \int_{(0,\infty)^2} \psi(z) \wh f_t^1(z) X_t(z)^{-1}  \wh f_{1-t}(z) \, dz + o_\delta(1) \right)
\ale
where $a_0$ is a finite positive constant depending only on $p$ and the $o_\delta(1)$ depends only on $C$, $t_0$, and $\psi$ (and in particular is uniform for $w \in \left[C^{-1}\delta^{1/2} , C\delta^{1/2}\right]^2$). Therefore,
\begin{align} \label{eqn-bm-B^delta-int}
\BB E \left(  \psi(\wt Z(t))  \BB 1_{ \mcl B^\delta(C) }  \right) & = a_0 \int_{\left[C^{-1}\delta^{1/2} , C\delta^{1/2}\right]^2} \int_{(0,\infty)^2}   \Delta(w)   \psi(z) \wh f_t^1(z)  X_t(z)^{-1} \wh f_{1-t}(z)  \, dz   \, dw + o_\delta(\delta^{1+\mu })  \notag   \\
& = a_0  \left( \int_{(0,\infty)^2} \psi(z) \wh f_t^1(z) g_t(z) \, dz     \right) \left( \int_{\left[C^{-1}\delta^{1/2} , C\delta^{1/2}\right]^2}   \Delta(w)       \, dw  \right) + o_\delta(\delta^{1+\mu })  .
\end{align} 
In particular, by taking $\psi \equiv 1$ we get
\eqb \label{eqn-bm-B^delta-prob}
\BB P \left(   \mcl B^\delta(C)    \right) = a_0 \left( \int_{(0,\infty)^2}   \wh f_t^1(z) g_t(z)  \, dz     \right) \left( \int_{\left[C^{-1}\delta^{1/2} , C\delta^{1/2}\right]^2}   \Delta(w)       \, dw  \right)  + o_\delta(\delta^{1+\mu }) .
\eqe 
By combining~\eqref{eqn-bm-B^delta-int} and~\eqref{eqn-bm-B^delta-prob}, we obtain the statement of the lemma. 
\end{proof}

\begin{lem} \label{prop-bm-cond-reg}
Fix $C>1$ and $\ep  > 0$. There exists $t_\ep \in (0,1)$ and $\delta_\ep > 0$, depending only on $C$ and $\ep$, such that for $\delta \in (0,\delta_\ep]$ we have
\eqbn
\BB P\left( \sup_{t\in [t_\ep,1]} |\wh Z(t)| \leq \ep \,|\, \mcl B^\delta(C)\right) \geq 1-\ep .
\eqen
\end{lem}
\begin{proof}
For $\ep >0$ and $t\in (0,1)$, let 
\eqbn
F_{1-t}^\ep := \left\{\sup_{s\in [0,1-t ]} |Z(s)|  > \ep \right\} .
\eqen
By a similar argument to the one leading to~\eqref{eqn-Z-flip2} of Lemma~\ref{prop-bm-rn-deriv}, we have 
\begin{align} \label{eqn-Z-reg1}
\BB P \left(    \sup_{t\in [t,1]} |\wh Z(t)| > \ep ,\,       \mcl B^\delta(C)    \right) &=    \int_{\left[C^{-1}\delta^{1/2} , C\delta^{1/2}\right]^2}  \BB E^w \left(     \wh f_t^1(Z(1-t)) X_t(z)^{-1} \BB 1_{ F_{1-t}^\ep \cap  G_{1-t}  }         \right) \, dw \notag \\
 &=    \int_{\left[C^{-1}\delta^{1/2} , C\delta^{1/2}\right]^2}  \BB E^w \left(    \wh f_t^1(Z(1-t)) X_t(z)^{-1} \BB 1_{ F_{1-t}^\ep} \,|\,  G_{1-t}   \right) \BB P^w\left(G_{1-t}\right)     \, dw .
\end{align}
with $X_t(z)$ as in~\eqref{eqn-cond-constant}. Furthermore, by~\eqref{eqn-Z-flip2} with $\psi \equiv 1$, 
\begin{align} \label{eqn-Z-reg2}
\BB P \left(     \mcl B^\delta(C)    \right) &=      \int_{\left[C^{-1}\delta^{1/2} , C\delta^{1/2}\right]^2}  \BB E^w \left(    \wh f_t^1(Z(1-t)) \,|\,  G_{1-t}         \right) \BB P^w\left(G_{1-t} \right)     \, dw .
\end{align}
By \cite[Theorem 2]{shimura-cone}, for each fixed $t>0$
\alb
&\BB E^w \left(    \wh f_t^1(Z(1-t)) X_t(z)^{-1} \BB 1_{ F_{1-t}^\ep} \,|\,  G_{1-t}   \right) \rta \BB E\left(\wh f_t^1(\wh Z(1-t)) X_t(z)^{-1}\BB 1_{F_{1-t}^\ep}\right)  \\
&\BB E^w \left(    \wh f_t^1(Z(1-t)) X_t(z)^{-1}  \,|\,  G_{1-t}   \right) \rta \BB E\left(\wh f_t^1(\wh Z(1-t))X_t(z)^{-1} \right)  
\ale
as $\delta\rta 0$, uniformly over all $w\in \left[C^{-1}\delta^{1/2} , C\delta^{1/2}\right]^2$. By continuity of $\wh Z$, it follows that for each $\ep > 0$ and $\alpha>0$, there exists $t_\ep \in (0,1)$ and $\delta_\ep  >0$, depending only on $C$, $\ep$, and $\alpha$ such that for $\delta \in (0,\delta_\ep]$ and $w \in  \left[C^{-1}\delta^{1/2} , C\delta^{1/2}\right]^2$, we have
\eqb \label{eqn-Z-reg-limit}
\BB E^w \left(    \wh f_{t_\ep}^1(Z(1-t_\ep)) X_{t_\ep}(z)^{-1} \BB 1_{ F_{1-t_\ep}^\ep} \,|\,  G_{1-t_\ep}   \right) \leq \alpha \BB E^w \left(    \wh f_t^1(Z(1-t_\ep))   \,|\,  G_{1-t_\ep}   \right) .
\eqe 
By \cite[Equation 4.2]{shimura-cone} and scale invariance, for $w\in \left[C^{-1}\delta^{1/2} , C\delta^{1/2}\right]^2$ we have
\eqbn
\BB P^w\left(G_{1-t_\ep}\right)  \asymp (1-t_\ep)^{-\mu} \delta^\mu,
\eqen
with the implicit constant depending only on $C$. It therefore follows from~\eqref{eqn-Z-reg1},~\eqref{eqn-Z-reg2}, and~\eqref{eqn-Z-reg-limit} that for $\delta \in (0,\delta_\ep]$ we have
\eqbn
\BB P \left(  \sup_{t\in [t_\ep,1]} |\wh Z(t)| > \ep \,|\,    \mcl B^\delta(C)    \right) \preceq \alpha 
\eqen
with the implicit constant depending only on $C$. We now conclude by choosing $\alpha$ smaller than $\ep$ divided by this implicit constant.
\end{proof}

\begin{proof}[Proof of Proposition~\ref{prop-bm-limit}]
Let $\mcl U$ be as in the statement of the proposition. For $z\in (0,\infty)^2$, the regular conditional law of $\wh Z|_{[0,t]}$ given $\{\wh Z(t) =z\}$ and the event $\mcl B^\delta(C)$ is that of a correlated Brownian motion as in~\eqref{eqn-bm-cov} conditioned to stay in the first quadrant for $t$ units of time and end up at $z$ (see \cite{gms-burger-local} for a definition of this object). By~\cite[Lemma 1.11]{gms-burger-local} and Lemma~\ref{prop-bm-excursion}, the same is true of the conditional law of $\dot Z$ given $\{\dot Z(t) = z\}$. Therefore, if $\delta < 1-t$ then
\eqbn
 \BB P\left(\wh Z|_{[0,t]} \in \mcl U \,|\, \wh Z(t) = z ,\, \mcl B^\delta(C) \right) = \BB P\left(\dot Z|_{[0,t]} \in \mcl U \,|\, \dot Z(t) = z\right) .
\eqen
Call this quantity $\psi(z)$. 
By Lemma~\ref{prop-bm-rn-deriv} applied with this choice of $\psi$, it follows that 
\begin{align} \label{eqn-prob-in-U-conv}
 \lim_{\delta \rta 0} \BB P\left( \wh Z|_{[0,t]} \in \mcl U    \,|\, \mcl B^\delta(C)\right) 
=\lim_{\delta \rta 0} \BB E\left(    \psi(\wh Z(t))    \,|\, \mcl B^\delta(C)           \right)   
 =  \frac{   \int_{(0,\infty)^2} \psi(z) g_t(z) \wh f_t^1(z) \, dz   }{   \int_{(0,\infty)^2} g_t(z) \wh f_t^1(z)  \, dz   }  .
\end{align} 
By~\eqref{eqn-prob-in-U-conv}, Lemma~\ref{prop-bm-cond-reg}, and the Arz\'ela-Ascoli theorem, the family of conditional laws of $\wh Z$ given $\mcl B^\delta(C)$ for $\delta\in (0,1)$ is tight. Let $\ddot Z$ be discributed according to any weak subsequential limit of these laws. By~\eqref{eqn-prob-in-U-conv}, for each $t\in (0,1)$, the law of $\ddot Z|_{[0,1]}$ is obtained by re-weighting the unconditional law of $\wh Z|_{[0,t]}$ by $g_t(\ddot Z(t))$. Since this reweighting depends only on $\wh Z(t)$, it follows that for each such $t$ and a.e.\ $z\in (0,\infty)^2$, the regular conditional law of $\ddot Z|_{[0,t]}$ given $\{\ddot Z(t) = z\}$ is the same as the regular conditional law of $\wh Z|_{[0,t]}$ given $\{\wh Z(t)=z\}$. By condition~\ref{item-bm-meander-markov} of Lemma~\ref{prop-bm-meander}, for each $s\in (0,t)$, the regular conditional law of $\wh Z|_{[s,t]}$ given $\wh Z(s)$ and $\wh Z(t)$ is that of a Brownian bridge from $\wh Z(s)$ to $\wh Z(t)$ conditioned to stay in the first quadrant. Therefore the same is true for $\ddot Z$. In other words, $\ddot Z$ satisfies condition~\ref{item-bm-excursion-markov} of Lemma~\ref{prop-bm-excursion}. It is clear from~\eqref{eqn-prob-in-U-conv} that condition~\ref{item-bm-excursion-pos} of the same lemma is also satisfied. Thus $\ddot Z \eqD \dot Z$.
This proves the first assertion of the proposition. The other two assertions are clear from~\eqref{eqn-prob-in-U-conv}. 
\end{proof}

\begin{remark} \label{remark-bm-limit-Z^delta}
For $\delta>0$, let $\wh Z^\delta$ be a correlated Brownian motion as in~\eqref{eqn-bm-cov} conditioned to stay in the first quadrant for $1-\delta$ units of time, so that $\wh Z^\delta|_{[1-\delta,1]}$ evolves as a Brownian motion as in~\eqref{eqn-bm-cov} with no conditioning. Also let 
\eqb \label{eqn-bm-wt-B-def}
\wt{\mcl B}^\delta(C) := \left\{\wh Z^\delta(1-\delta) \in \left[C^{-1}\delta^{1/2} , C\delta^{1/2}\right]  \right\} .
\eqe 
Essentially the same argument used to prove Proposition~\ref{prop-bm-limit} shows that the statement of the proposition remains true if we substitute $\wh Z^\delta$ for $\wh Z$ and $\wt{\mcl B}^\delta(C)$ for $\mcl B^\delta(C)$. This fact will be used in Section~\ref{sec-macroscopic-cone} below. 
\end{remark}

\section{Existence of a large loop}
\label{sec-big-loop}

\subsection{Statement and relationships between events}
\label{sec-big-loop-setup}

In this section we will prove a proposition which will eventually allow us to reduce the proof of Theorem~\ref{thm-main} to a convergence statement for the path $Z^n$ conditioned on the event that $X(1,n)$ contains a particular number $l \asymp n^{\xi/2}$ of hamburger orders and no other symbols (plus some regularity conditions. 

Throughout this section we will use the following notation. Fix $\xi  \in (0,1)$. For $n\in\BB N$ and $\ep_0   > 0$, let $\iota_0^n = \iota_0^n(\ep_0)$ be the largest $i \in [1,2n]_{\BB Z}$ for which the following holds.
\begin{enumerate}
\item $X_{i}   = \tb F$ and $X_{\phi(i)} = \tc H$.\label{item-big-loop-match}
\item $2n-  n^\xi \leq   i - \phi(i) \leq 2n-   \ep_0 n^\xi $.\label{item-big-loop-upper} 
\item $\ep_0 n^{\xi/2 } \leq |X(\phi(i) , i)| \leq \ep_0^{-1}  n^{\xi/2 }$.  \label{item-big-loop-length}
\item The smallest $i \in [\iota_0^n+1, 2n]_{\BB Z}$ with $X_i = \tb F$ and $\phi(i)  < \phi(\iota_0^n)$ (equivalently $\phi(i) < \iota_0^n$) satisfies $X_{\phi(i)} = \tc C$.     \label{item-big-loop-loop}
\end{enumerate}
If no such $i$ exists, we set $\iota_0^n = 0$. 

Also let $\iota_1^n = \iota_1^n(\ep_0 )$ be the largest $i\in [\phi(\iota_0^n) , \iota_0^n]_{\BB Z}$ for which $X_{i}   = \tb F$, $X_{\phi(i)} = \tc C$, and $i - \phi(i) \geq  \frac12 \left( \iota_0^n - \phi( \iota_0^n) \right)$. We set $\iota_1^n = 0$ if no such $i$ exists. 

\begin{remark}
If $\iota_0^n >  0$, $\iota_1^n < \iota_0^n$, and $\{X(1,2n) = \emptyset\}$, then under the bijection of \cite[Section 4.1]{shef-burger}, $X_{\phi(\iota_0^n)} \dots X_{\iota_0^n}$ corresponds to the largest complementary connected component of a certain loop in the FK planar map corresponding to $X_1\dots X_{2n}$; and  $X_{\phi(\iota_1^n)} \dots X_{\iota_1^n }$ corresponds to a complementary connected component of one of the outermost loops contained in this complementary connected component. See also~\cite{gwynne-miller-cle}. 
\end{remark}

The main result of this section is the following proposition.

\begin{prop} \label{prop-big-loop}
For $\ep_0 , \ep_1> 0$, let $\mcl G_n   = \mcl G_n (\ep_0 , \ep_1  )$ be the event that $\iota_0^n= \iota_0^n(\ep_0) >0$, $\iota_1^n = \iota_1^n(\ep_0) >0$, $X(1,2n)= \emptyset$, and the following additional conditions hold. 
\begin{enumerate}  
\item $\iota_1^n - \phi(\iota_1^n) \geq \iota_0^n - \phi(\iota_0^n) - \ep_1^{-1}\left(\iota_0^n - \phi(\iota_0^n) \right)^{\xi }  $. \label{item-subloop-area}
\item $|X(\phi(\iota_1^n) , \iota_1^n)| \geq \ep_1 \left(\iota_0^n - \phi(\iota_0^n) \right)^{\xi/2}$. \label{item-subloop-length}
\end{enumerate}
For each $q \in (0,1)$, there exists $\ep_0 , \ep_1 > 0$ and $n_* \in \BB N$ such that for $n\geq n_*$, 
\eqbn
\BB P\left(\mcl G_n  \,|\, X(1,2n) = \emptyset\right) \geq 1-q .
\eqen
\end{prop}
 
The reason for introducing the times $\iota_0^n$ and $\iota_1^n$ is that conditioning on $\phi(\iota_0^n) , \iota_0^n$, and $|X(\phi(\iota_0^n) , \iota_0^n)|$ will allow us to reduce the problem of studying the conditional law of $X_1\dots X_{2n}$ given $\{X(1,2n) = \emptyset\}$ to the more tractable problem of studying its conditional law given the following event. 

\begin{defn} \label{def-end-event}
Let $J$ be the smallest $j\in\BB N$ for which $X(-j,-1)$ contains a burger. For $n\in\BB N$ and $l\in\BB N$, let $\wh{\mcl E}_n^l  $ be the event that $J=n$, $X_{-J} = \tc H$, and $X(-n,-1)$ contains exactly one hamburger, $l$ cheeseburger orders, and no other symbols. For $\ep_1 > 0$, let $\mcl E_n^l  = \mcl E_n^l(\ep_1)$ be the event that $\wh{\mcl E}_n^l$ occurs and there exists $j\in [1,J]_{\BB Z}$ such that the following conditions hold.
\begin{enumerate}
\item $j \geq n - \ep_1^{-1} n^{\xi }$. 
\item There are at least $\ep_1 n^{\xi/2}$ hamburger orders to the left of the leftmost flexible order in $X(-j,-1)$.  
\item Each cheeseburger in $X(-J , -j-1)$ is matched to a cheeseburger order in $X(-j,-1)$. 
\end{enumerate}
Let $\pi_n = \pi_n(\ep) $ be the smallest $j \in [1,J]_{\BB Z}$ satisfying the above conditions (if it exists) and otherwise let $\pi_n := J $. 
\end{defn} 

\begin{remark} \label{remark-pi-not-stopping}
The time $\pi_n$ of Definition~\ref{def-end-event} is \emph{not} a stopping time for the word $X$, read backward. Rather, for any realization $x$ of $X_{-\pi_n} \dots X_{-1}$ for which $\pi_n < J$, the conditional law of $X_{-J} \dots X_{-|x|}$ given $\{X_{-\pi_n} \dots X_{-1} = x\}$ is the same as its conditional given $\{X_{-|x|} \dots X_{-1} = x\}$ and the additional event that each cheeseburger in $X(-J,-|x|-1)$ is matched to a cheeseburger order in $X(-|x|,-1)$. 
\end{remark}

\begin{lem} \label{prop-G'-event}
Fix $\ep_0 , \ep_1 > 0$. Let $\iota_0^n = \iota_0^n(\ep_0)$ and $\iota_1^n = \iota_1^n(\ep_1)$ be defined as above and let $\mcl G_n = \mcl G_n(\ep_0,\ep_1)$ be the event of Proposition~\ref{prop-big-loop}. Also let $\mcl G_n' = \mcl G_n'(\ep_0 ,\ep_1)$ be the event that $\iota_0^n > 0$, $ X(1,2n)=\emptyset $, and there exists $i\in [\phi(\iota_0^n) , \iota_0^n]_{\BB Z}$ such that the following conditions hold:
\begin{enumerate}
\item $i \leq \phi(\iota_0^n) +  \ep_1^{-1} \left(\iota_0^n - \phi(\iota_0^n) \right)^{\xi }$. \label{item-almost-E^l-area}
\item There are at least $\ep_1 \left(\iota_0^n - \phi(\iota_0^n) \right)^{\xi/2}$ hamburger orders to the left of the leftmost flexible order in $X(i,\iota_0^n)$.  \label{item-almost-E^l-length}
\item Each cheeseburger in $X(\phi(\iota_0^n) , i-1 )$ is matched to a cheeseburger order in $X(i,\iota_0^n)$. \label{item-almost-E^l-maximal}
\end{enumerate}
Then $\mcl G_n  \subset \mcl G_n'$, so for each $q \in (0,1)$, there exists $\ep_0 , \ep_1 > 0$ and $n_* \in \BB N$ such that for $n\geq n_*$, 
\eqbn
\BB P\left(\mcl G_n' \,|\, X(1,2n) = \emptyset\right) \geq 1-q .
\eqen
\end{lem}
\begin{proof} 
Observe that if $\mcl G_n$ occurs, then the conditions for the time $i$ in the definition of $\mcl G_n'$ are satisfied with $i = \phi(\iota_1^n)$. Indeed, the conditions~\ref{item-subloop-area} and~\ref{item-subloop-length} in the definition of $\mcl G_n$ imply the conditions~\ref{item-almost-E^l-area} and~\ref{item-almost-E^l-length} in the definition of $\mcl G_n'$. The condition~\ref{item-almost-E^l-maximal} follows from maximality of $\iota_1^n$ and the fact that $X(\phi(\iota_0^n) , \iota_0^n)$ contains no cheeseburgers. 
\end{proof}

The reason for our interest in the event $\mcl G_n'$ of Lemma~\ref{prop-G'-event} is as follows.

\begin{lem} \label{prop-E^l-equiv}
Let $\ep_0,\ep_1> 0$. Let $\iota_0^n = \iota_0^n(\ep_0)$ be as in the beginning of this subsection and let $\mcl G_n' = \mcl G_n'(\ep_0,\ep_1)$ be the event of Lemma~\ref{prop-G'-event}. Suppose given realizations $x$ of $X_1\dots X_{\phi(\iota_0^n)-1}$, $x'$ of $X_{\iota_n} \dots X_{2n}$, and $l$ of $|X(\phi(\iota_0^n) , \iota_0^n)|$ for which 
\eqbn
\BB P\left(X_1\dots X_{\phi(\iota_0^n)-1} =x ,\, X_{\iota_n} \dots X_{2n} = x' ,\, |X(\phi(\iota_0^n) , \iota_0^n)| =l,\, \mcl G_n' \right) > 0  .  
\eqen
Let $m := 2n - |x'| - |x| - 1$. 
The conditional law of $X_{\phi(\iota_0^n) } \dots X_{\iota_0^n-1}$ given 
\[
\{X_1\dots X_{\phi(\iota_0^n)-1} =x ,\, X_{\iota_n} \dots X_{2n} = x' ,\, |X(\phi(\iota_0^n) , \iota_0^n)| =l,\, \mcl G_n' \} 
\]
is the same as the conditional law of the word $X_{-m} \dots X_{-1}$ given the event $\mcl E_{m }^l = \mcl E_{m}^l(\ep_1)$ of Definition~\ref{def-end-event}. 
\end{lem}
\begin{proof}
Fix $x,x'$, and $l$ as in the statement of the lemma. Also let $n_1 := |x |+1$ and $n_2 := 2n-|x'|$ be the corresponding realizations of $\phi(\iota_0^n)$ and $\iota_0^n$, respectively. We first argue that the conditional law of $X_{\phi(\iota_0^n) } \dots X_{\iota_0^n-1}$ given 
\eqb \label{eqn-E^l-realization}
\{X_1\dots X_{\phi(\iota_0^n)-1} =x ,\, X_{\iota_n} \dots X_{2n} = x' ,\, |X(\phi(\iota_0^n) , \iota_0^n)| =l ,\, X(1,2n) =\emptyset\} 
\eqe 
is the same as the conditional law of $X_{-m} \dots X_{-1}$ given the event $\wh{\mcl E}_n^l$ of Definition~\ref{def-end-event}. 

To see this, we first observe that the event~\eqref{eqn-E^l-realization} is equal to the intersection of the following events:
\alb
E_1 &:= \left\{ \text{$X_{n_1} = \tc H$, $\phi(n_1 ) = n_2$, and $X(n_1,n_2-1)$ contains one $\tc H$, $l$ $\tb C$'s, and no other symbols}\right\}\\ 
E_2 &:= \left\{X_1\dots X_{n_1-1} =x ,\, X_{n_2} \dots X_{2n} = x'\right\} \\
E_3 &: = \left\{\text{There is no $i \in [n_2 + 1 , 2n]_{\BB Z}$ satisfying the defining conditions of $\iota_0^n$}\right\} \\
E_4 &:= \left\{ \text{The smallest $i \in [n_2+1, 2n]_{\BB Z}$ with $X_i = \tb F$ and $\phi(i) < n_2$ satisfies $X_{\phi(i)} = \tc C$.}   \right\} \\
E_5 &:= \left\{X(1,2n) =\emptyset \right\}.
\ale
The events $E_1$ and $E_2$ are independent. The event $E_3 \cap E_4 \cap E_5$ depends only on $X(n_1 ,n_2-1 )$, $X_{n_2+1}\dots X_{2n}$, and $X_{1}\dots X_{n_1-1}$, so by our choice of $x$, $x'$, and $l$, this event always occurs when $E_1\cap E_2$ occurs. It follows that the conditional law of the word $X_{n_1} \dots X_{n_2-1}$ given the event~\eqref{eqn-E^l-realization} is the same as its conditional law given $E_1$, which by translation invariance is the same as the conditional law of $X_{-m} \dots X_{-1}$ given $\wh{\mcl E}_m^l$. 

It is clear from the definitions of $\mcl G_n'$ and $\mcl E_{m}^l$ that if we further condition on $\mcl G_n'$, the the conditional law of $X_{n_1} \dots X_{n_2-1}$ under this conditioning is the same as the conditional law of $X_{-m} \dots X_{-1}$ given $ \mcl E_{m}^l$. 
\end{proof}

\subsection{Defining events in terms of the FK planar map} 
\label{sec-big-loop-events}
 
On the event $\{X(1,2n) =\emptyset\}$, let $(M^n , e_0^n, \mcl L^n)$ be the rooted FK planar map associated with $X_1\dots X_{2n}$ via Sheffield's bijection. Let $(Q^n ,\BB e_0^n)$ be the quadrangulation associated with $M^n$ as in \cite[Section 4.1]{shef-burger}, i.e.\ the set of vertices of $Q^n$ is the union of the set of vertices of $M^n$ and the set of faces of $M^n$; two vertices of $Q^n$ are joined by an edge if and only if these two vertices correspond to a face of $M^n$ and a vertex adjacent to this face; and $\BB e_0^n$ is the unique edge of $Q^n$ whose primal endpoint is adjacent to $e_0^n$ and which is the first edge counterclockwise from $e_0^n$ with this property. Note that $Q^n$ has $2n$ edges.

The idea of the proof of Proposition~\ref{prop-big-loop} is as follows. We will define events $A_i^n = A_i^n(\ep_0 , \ep_1)$ for $i\in [1, n]_{\BB Z}$ (which depend only on the ``local" behavior of the word near time $i$) and show that, roughly speaking, the following holds. 
\begin{enumerate}
\item If $\ep_0$ and $\ep_1$ are small, then with high probability, $A_i^n $ occurs for most $i\in [1,n]_{\BB Z}$, even if we condition on $\{X(1,2n) = \emptyset\}$. 
\item If $A_i^n $ occurs, $X(1,2n) = \emptyset$, and we choose a new root edge for the FK planar map in such a way that the $i$th edge hit by the exploration path in Sheffield's bijection becomes the root edge of $Q^n$, then $\mcl G_n $ occurs for the word corresponding to the re-rooted map. 
\end{enumerate}
The statement of the proposition will follow from this together with invariance of the law of $(M^n , e_0^n, \mcl L^n)$ under uniform re-rooting (see Lemma~\ref{prop-uniform-root} below). The operation of choosing a new root edge does not have a nice description in terms of the word $X_1\dots X_{2n}$ (indeed, we are not aware of a simple criterion to determine whether two words $x$ and $x'$ consisting of elements of $\Theta$ which have the same length and both reduce to the empty word correspond under Sheffield's bijection to the same FK planar map with two different choices of root edge). Hence to exploit re-rooting invariance of $(M^n , e_0^n, \mcl L^n)$ we need to study FK loops rather than burgers and orders. 

In the remainder of this subsection we will construct events $\wt A_i^n$ which will eventually be used to define the events $A_i^n$ after we zoom in on an interval of length $n^\xi$. We will define these events in terms of the infinite-volume rooted FK planar map associated with the bi-infinite word $X$, which we denote by $(M^\infty  , e_0^\infty , \mcl L^\infty)$. This FK planar map can be obtained from the word $X$ via essentially the same bijection described in \cite[Section 4.1]{shef-burger}. See~\cite{blr-exponents,chen-fk} for more details regarding the infinite-volume version of this bijection. 
  
Let $(Q^\infty ,\BB e_0^\infty)$ be the associated rooted quadrangulation (defined in the same manner as the rooted quadrangulation $(Q^n,\BB e_0^n)$ at the beginning of this subsection but with $(M^\infty, e_0^\infty)$ in place of $(M^n,e_0^n)$). Let $\lambda$ be the path which explores $(Q^\infty , \BB e_0^\infty)$, i.e.\ $\lambda(i)$ is the edge of $Q^\infty$ corresponding to the symbol $X_i \in \Theta$ under Sheffield's bijection. 

For $i \in \BB Z$ and $j\in\BB N$, let $\{\ell_{i,j} \}_{ j\in\BB N}$ be the ordered sequence of loops in $\mcl L^\infty$ which surround $\lambda(i)$, so that each $\ell_{i,j}$ is a cyclically ordered set of edges in $Q^\infty$ and $\ell_{i,j+1} $ disconnects $\ell_{i,j} $ from $\infty$ for each $j\in\BB N$.

We want to consider complementary connected components of the loops $\ell_{i,j}$, which we define as follows. For the definition we let $S$ and $S^*$ be the primal and dual FK edge sets for the map $M^\infty$.

\begin{defn} \label{def-fk-component}
Let $\ell \in \mcl L^\infty$ be an FK loop. Let $A$ and $A^*$ be the clusters of edges in $S$ and $S^*$ which are separated by $\ell$ (so that $A$ and $A^*$ are connected). A \emph{primal (resp. dual) complementary connected component} of $\ell$ is a set of edges $U\subset Q^\infty$ such that the following is true. There exists a simple cycle $C$ of $S$ (resp. $S^*$) which is contained in $A$ (resp. $A^*$) such that $U$ is the set of edges of $Q^\infty$ disconnected from $\ell$ by $C$; and there is no set $U'$ of edges of $Q^\infty$ satisfying the above property and properly containing $U$. In this case we write $C = \partial U$.  
\end{defn}

\begin{defn} \label{def-discrete-bdy}
Suppose $U$ is a complementary connected component of a loop in $\mcl L^\infty$. We write $\op{Area} U$ for the number of edges in $U$ and $\op{len} \partial U$ for the number of edges in $\partial U$. 
\end{defn}
 
For $i \in \BB Z$ and $j\in\BB N$, let $M_{i,j}^0 $ (resp. $  M_{i,j}^\infty $) be the set of edges in $Q^\infty$ which belong to the same complementary connected component of the loop $\ell_{i,j}$ as $\lambda(i)$ (resp. the set of edges in $Q^\infty$ which are disconnected from $\infty$ by $\ell_{i,j}$). 
 See Figure~\ref{fig-big-loop} for an illustration.

\begin{figure}[ht!]
 \begin{center}
\includegraphics{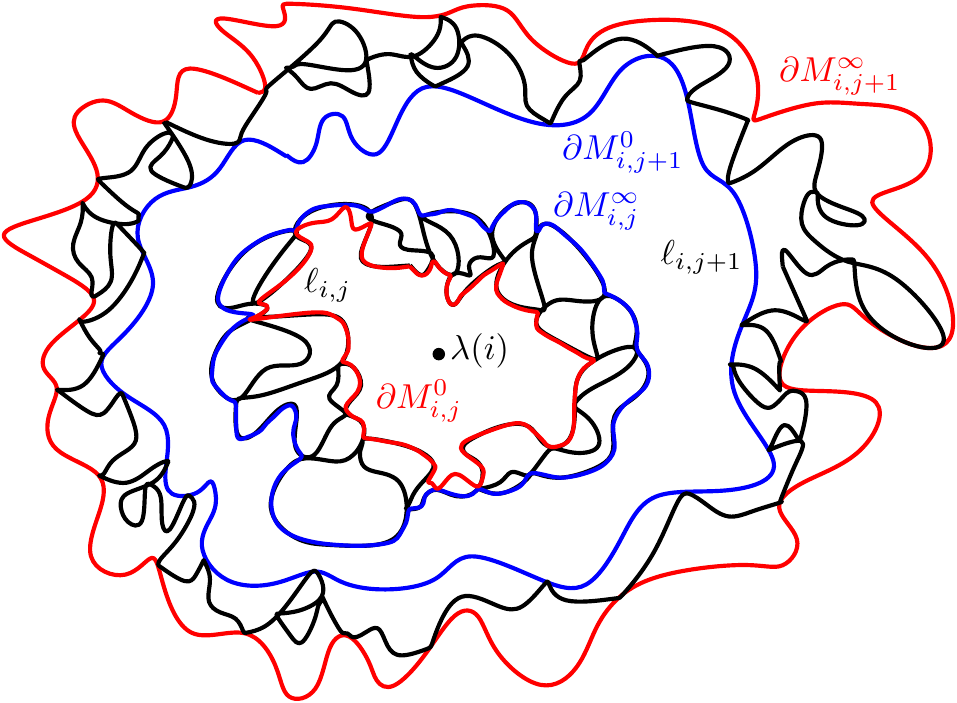} 
\caption{An illustration of two consecutive loops surrounding $\lambda(i)$ in the infinite-volume rooted FK planar map $(M^\infty , e_0^\infty , \mcl L^\infty)$ together with the components $M_{i,j}^0$ and $M_{i,j+1}^0$ containing $\lambda(i)$ and the complements $M_{i,j}^\infty$ and $M_{i,j+1}^\infty$ of the unbounded components. The boundaries of $M_{i,j+1}^0$ and $M_{i,j}^\infty$ (resp. $M_{i,j}^0$ and $M_{i,j+1}^\infty$) are either both formed by edges of $S$ or both formed by edges of $S^*$, so these pairs of boundaries are shown in the same color.
}\label{fig-big-loop}
\end{center}
\end{figure} 
 
Let $\wt\theta_{i,j}$ (resp. $\theta_{i,j}$) be the time at which $\lambda$ starts (resp. finishes) tracing $\ell_{i,j}$, so that $[\wt\theta_{i,j} , \theta_{i,j}-1]_{\BB Z}$ is the intersection of $\BB Z$ with the smallest interval which contains the set of $k\in \BB Z$ such that $\lambda(k) \in \ell_{i,j}$. It follows from the construction in~\cite[Section 4.1]{shef-burger} that $X_{\theta_{i,j}} = \tb F$ and $\phi(\theta_{i,j}) = \wt\theta_{i,j}$. 

Fix $\ep_0,\ep_1>0$. For $n\in\BB N$ and $i\in [1,n]_{\BB Z}$, let 
$  J_i^n(\ep_0)$ be the smallest $j \in\BB N$ for which $\ell_{i,j}$ is traced by $\lambda$ in the counterclockwise direction (equivalently $X_{\theta_{i,j}} =\tc H$), 
\alb
\op{Area} M_{i,j}^\infty \geq \ep_0 n ,\quad \op{and}\quad \ep_0 n^{1/2} \leq \op{len}\partial  M_{i,j}^\infty \leq \ep_0^{-1} n^{1/2}  .
\ale
Here we use the notation of Definition~\ref{def-discrete-bdy}. Set $J_i^n(\ep_0) = 0$ if no such $j$ exists. 
Write
\begin{align} \label{eqn-loops-at-J}
&\ul \theta_i^n := \theta_{i, J_i^n(\ep_0)} ,\quad \ul M_i^n := M_{i,J_i^n(\ep_0)}^\infty ,\quad \ul \ell_i^n := \ell_{i,J_i^n(\ep_0)} \notag \\
&\ol \theta_i^n := \theta_{i, J_i^n(\ep_0)+1} ,\quad \ol M_i^n := M_{i,J_i^n(\ep_0)+1}^\infty ,\quad \ol \ell_i^n := \ell_{i,J_i^n(\ep_0)+1}  ,
\end{align} 
where here we use the convention that $\theta_{i,0} = i$ and $M_{i,0}^\infty = \ell_{i,0} = \emptyset$. 
Then $\ul \ell_i^n$ is the innermost counterclockwise loop surrounding $\lambda(i)$ with area at least $\ep_0 n$ and outer boundary length between $\ep_0 n^{1/2}$ and $\ep_0^{-1} n^{1/2}$; and $\ol \ell_i^n$ is the loop immediately outside $\ul \ell_i^n$.  
 
Let $\wt A_i^n = \wt A_i^n(\ep_0, \ep_1 )$ be the event that the following is true.
\begin{enumerate} 
\item $J_i^n(\ep_0) > 0$.
\item $ \op{len}\left(\partial \ol M_i^n \right) \geq \ep_1  n^{1/2} $.  
\item $\op{Area} \ol M_i^n \leq \ep_1^{-1} \op{Area} \ul M_i^n$. 
\item $[\phi(\ol\theta_i^n) , \ol\theta_i^n ]_{\BB Z}\subset [1,n]_{\BB Z}$. \label{item-wt-A-contained-event}
\end{enumerate}

\subsection{Proof of Proposition~\ref{prop-big-loop}}
\label{sec-big-loop-proof}

In this subsection we will prove Proposition~\ref{prop-big-loop} conditional on a technical lemma which we state below, and prove in Section~\ref{sec-big-loop-word}. Throughout, we continue to use the notation introduced in the preceding two subsections.

\begin{lem} \label{prop-big-F-constant}
Define the events $\wt A_i^n = \wt A_i^n(\ep_0,\ep_1)$ as in Section~\ref{sec-big-loop-events}. For each $i\in[1,n]_{\BB Z}$, the event $\wt A_i^n$ is measurable with respect to $X_1\dots X_n$. Furthermore, for each $q\in (0,1)$, there exists $n_* \in\BB N$, $\ep_0 , \ep_1  > 0$, such that the following is true. 
For $n\in\BB N$, let $\wt S^n  = \wt S^n (\ep_0 , \ep_1  )$ be the set of $i\in [1,n]_{\BB Z}$ for which $\wt A_i^n $ occurs. For each $n\geq n_*$, 
\eqbn
\BB P\left(\# \wt S^n \geq (1-q) n \right) \geq 1-q .
\eqen
\end{lem}

Lemma~\ref{prop-big-F-constant} will be proven in Section~\ref{sec-big-loop-word} below using the infinite-volume version of the results in~\cite{gwynne-miller-cle} (c.f. Remark~\ref{remark-implication}). In the remainder of this subsection, we assume Lemma~\ref{prop-big-F-constant} and use it to prove Proposition~\ref{prop-big-loop}. 
 
By \cite[Theorem~1.10]{gms-burger-local}, $\BB P\left(X(1,2n) = \emptyset\right)$ decays slower than some negative power of $n$. Hence to transfer the statement of Lemma~\ref{prop-big-F-constant} to a statement about conditional probabilities $\{X(1,2n) =\emptyset\}$, it suffices to produce an event whose probability decays faster than any power of $n$. We accomplish this by dividing $[1,2n]_{\BB Z}$ into many intervals of length $\asymp n^{1-\xi}$ for fixed $\xi \in (0,1)$ and independently applying Lemma~\ref{prop-big-F-constant} in each interval. 

\begin{lem} \label{prop-big-F-exponential}
 Fix $\xi \in (0,1)$, $\ep_0, \ep_1   >0$. For $n\in\BB N$ and $k\in \BB N$, let 
 \eqb \label{eqn-k-interval-def}
 I_n^k := \left[ (k-1) n^{\xi } +1 , k n^{\xi }\right]_{\BB Z} .
 \eqe 
 For $i \in [0,2n]_{\BB Z}$ let $k$ be chosen so that $i\in I_n^k$ and let $A_i^n=   A_i^n(\ep_0,\ep_1  )$ be defined in the same manner as the event $\wt A_i^n(\ep_0 ,\ep_1)$ as above but with $ n^\xi$ in place of $n$ and $I_n^k$ in place of $[1,n]_{\BB Z}$. Let $S^n = S^n(\ep_0,\ep_1 )$ be the set of $i\in [1,  2n-n^\xi  ]_{\BB Z}$ for which $A_i^n$ occurs. For each $q\in (0,1)$, there exists $n_* \in\BB N$ and $\ep_0 , \ep_1 > 0$ such that
 \eqbn
\BB P\left(\#   S^n  \geq (2- q) n \right) = 1- o_n^\infty(n) .
\eqen
\end{lem}
\begin{proof}
By the first assertion of Lemma~\ref{prop-big-F-constant}, the random variables $\#(I_n^k \cap S^n)$ for $k\in\BB N$ are iid. By Lemma~\ref{prop-big-F-constant}, for any $\alpha>0$, we can find $  n_* \in\BB N$, and $\ep_0 , \ep_1   > 0$ such that for each $k\in\BB N$ and $n\geq\wt n_*$, 
\eqbn
\BB P\left(\#  ( S^n \cap I_n^k) \geq (1- \alpha)  n^\xi \right) \geq 1- \alpha .
\eqen
By Hoeffding's inequality for Bernoulli sums, except on an event of probability $o_n^\infty(n)$, the number of $k \in [1, 2n^{1-\xi}]_{\BB Z}$ for which $\#(I_n^k \cap S^n)   \geq (1- \alpha)  n^\xi$ is at least $1-2\alpha$. If this is the case, then $\# S^n \geq  2(1-\alpha)(1-2\alpha + o_n(1)) n$. We conclude by choosing $\alpha$ small enough that $ 2(1-\alpha)(1-2\alpha) >  2-q$ and $n$ sufficiently large. 
\end{proof}
 
Before we can prove Proposition~\ref{prop-big-loop}, we first need the following elementary observation about the bijection of \cite{shef-burger}.

\begin{lem} \label{prop-uniform-root}
Let $(M^n ,e_0^n ,\mcl L^n)$ be a rooted FK planar map with $n$ edges and let $(Q^n ,\BB e_0^n)$ be the rooted quadrangulation associated with $(M^n ,e_0^n ,\mcl L^n)$ as described in the beginning of Section~\ref{sec-big-loop-events}. The conditional law of $\BB e_0^n$ given $(M^n,\mcl L^n)$ is given by the uniform measure on the $2n$ edges of $Q^n$. 
\end{lem}
\begin{proof}
As explained in \cite[Section 4.2]{shef-burger}, the law of the triple $(M^n, e_0^n, \mcl L^n)$ conditioned on $\{X(1,2n) =\emptyset\}$ is equal to the uniform distribution on such triples for which $M^n$ has $n$ edges weighted by $q^{\#\mcl L}$, where $q = 4p^2/(2p-1)^2$. In particular, this weighting does not depend on $e_0^n$, so the conditional law given $(M^n,\mcl L^n)$ of the oriented root edge $e_0^n$ of $M^n$ is uniform over all $2n$ choices of oriented edges in $M^n$. Each choice of root edge $e_0^n$ for $M^n$ determines a unique root edge $\BB e_0^n$ for $Q^n$, namely the edge of $Q^n$ whose primal endpoint is the initial endpoint of $ \BB e_0^n$ and which is the next edge clockwise from $\BB e_0^n$ among all edges of $M^n$ and $Q^n$ that start at that endpoint; conversely, each edge of $Q^n$ arises from a unique edge of $M^n$ in this manner. The statement of the lemma follows.
\end{proof}

\begin{proof}[Proof of Proposition~\ref{prop-big-loop}]
On the event $\{X(1,2n)=\emptyset\}$, let $(M^n,e_0^n,\mcl L^n)$ be the rooted FK planar map associated with the word $X_1\dots X_{2n}$and let $(Q^n,\BB e_0^n)$ be the rooted quadrangulation constructed from $M$, as described at the beginning of Section~\ref{sec-big-loop-events}. Let $\lambda^n$ be the discrete exploration path which traces the edges of the quadrangulation $Q^n$; i.e.\, for $i\in [1,2n]_{\BB Z}$, $\lambda^n(i)$ is the $i$th edge of $Q^n$ visited by the exploration path (called $e_i$ in \cite{shef-burger}); and $\lambda^n(0) = \lambda^n(2n) := \BB e_0^n$ is the root edge of $Q^n$. 

Given $q\in (0,1)$, let $n_*\in\BB N$ and $\ep_0, \ep_1   >0$ be chosen so that the conclusion of Lemma~\ref{prop-big-F-exponential} is satisfied. For $i \in [1,2n-1 ]_{\BB Z}$ let $ A_i^n = A_i^n(\ep_0,\ep_1) $ and $S^n = S^n(\ep_0,\ep_1)$ be defined as in that lemma. By \cite[Theorem~1.10]{gms-burger-local}, $\BB P\left(X(1,2n) = \emptyset\right) \geq n^{-1-2\mu  + o_n(1)}$, with $\mu$ as in~\eqref{eqn-cone-exponent}. It therefore follows from Lemma~\ref{prop-big-F-exponential} that
\eqbn
\BB P\left(\#   S^n  \geq (2- q) n \,|\, X(1,2n) = \emptyset \right) = 1- o_n^\infty(n) .
\eqen

For $i \in [1,2n]_{\BB Z}$ such that $A_i^n$ occurs, let $I_n^k$ be the interval as in~\eqref{eqn-k-interval-def} which contains $i$. Let $\ul \ell_i^n$ and $\ol \ell_i^n$ be the loops in $\mcl L^n$ surrounding $\lambda^n(i)$ and contained in $\lambda^n(I_n^k)$ as in Section~\ref{sec-big-loop-events} but with $I_n^k$ in place of $[1,n]_{\BB Z}$ and the finite-volume bijection used in place of the infinite-volume one. Also define $\ul M_i^n$ and $\ol M_i^n$ as in that section. For $i\in [1, 2n-n^\xi   ]_{\BB Z}$, on the event $A_i^n$ we have the following.
\begin{enumerate} 
\item $\ul \ell_i^n$ is a counterclockwise loop, $\ol \ell_i^n$ is a clockwise loop which disconnects $\ul\ell_i^n$ from the root edge, and there is no loop in $\mcl L^n$ surrounding $i$ which is disconnects $\ul \ell_i^n$ from $\ol \ell_i^n$. \label{item-A-loops}
\item $\ep_0 n^\xi \leq \op{Area} \ul M_i^n \leq  n^\xi$.   \label{item-A-area}    
\item $\ep_0  n^{\xi/2} \leq \op{len}\left(\partial \ul M_i^n \right) \leq \ep_0^{-1}  n^{\xi/2}$.  \label{item-A-length} 
\item $\op{Area} \ol M_i^n \leq \ep_1^{-1} \op{Area} \ul M_i^n$. \label{item-A-subloop-area}
\item $ \op{len}\left(\partial \ol M_i^n \right) \geq \ep_1   n^{\xi/2}$.  \label{item-A-subloop-length} 
\item There is no counterclockwise loop disconnected from the root edge by by $\ol \ell_i^n$ such that the set of points disconnected from the root edge by this loop has area at least $\ep_0 n^\xi$ and boundary length between $\ep_0  n^{\xi/2}$ and $ \ep_0^{-1}  n^{\xi/2}$. \label{item-A-maximal}
\end{enumerate}
For $i\in [1, 2n-1]_{\BB Z}$, let $\wh A_i^n$ be the event that there exist loops $\ul \ell_i^n$ and $\ol\ell_i^n$ in $\mcl L^n$ surrounding $\lambda^n(i)$ such that the above five conditions hold.   
Let $\wh S^n$ be the set of $i\in [1,2n-1]_{\BB Z}$ such that $\wh A_i^n$ occurs. Then $S^n\subset \wh S^n$, so 
\eqb \label{eqn-loops-exist-set}
\BB P\left(\#  \wh S^n  \geq (2- q) n \,|\, X(1,2n) = \emptyset \right) = 1- o_n^\infty(n) .
\eqe 

If $\wh A_i^n$ occurs and we re-root the quadrangulation $Q^n$ so that $\lambda^n(i)$ becomes the root edge, then the complements of $\ul M_i^n$ and $\ol M_i^n$ become complementary connected components of loops which contain the root edge whose area is close to $2n$. Furthermore, the orientations of the loops $\ol\ell_i^n$ and $\ul\ell_i^n$ are flipped. Consequently, there exist loops $\ul \ell_0^n$ and $\ol \ell_0^n$ in $\mcl L^n$ such that the following is true.
\begin{enumerate} 
\item $\ul \ell_0^n$ is a clockwise loop, $\ol \ell_0^n$ is a counterlockwise loop disconnected from the root edge by $\ul\ell_0^n$, and there is no loop in $\mcl L^n$ surrounding $i$ which is contained between $\ul \ell_0^n$ and $\ol \ell_0^n$. \label{item-A0-loops}
\item Let $\ul M_0^n$ be the complementary connected component of $\ul \ell_0^n$ with the largest area. Then $\ul M_0^n$ lies on the opposite side of $\ul\ell_0^n$ from the root edge and $2n -n^\xi \leq \op{Area} \ul M_0^n \leq 2n- \ep_0 n^\xi$.   \label{item-A0-area}    
\item $\ep_0  n^{\xi/2} \leq \op{len}\left(\partial \ul M_0^n \right) \leq \ep_0^{-1}  n^{\xi/2}$.  \label{item-A0-length} 
\item Let $\ol M_0^n$ be the complementary connected component of $\ol \ell_0^n$ with the largest area. Then $\ul M_0^n$ lies on the opposite side of $\ol\ell_0^n$ from the root edge and $\op{Area} \ol M_0^n \geq 2n -  \ep_1^{-1} \left(2n -  \op{Area} \ul M_0^n \right)$. \label{item-A0-subloop-area}
\item $\op{len}\left(\partial \ol M_0^n \right) \geq \ep_1   n^{\xi/2}$.  \label{item-A0-subloop-length}   
\item There is no clockwise loop which surrounds $\ol \ell_0^n$ from the root edge such that the area of its largest complementary connected component is at most $2n - \ep_0 n^\xi$ and the boundary length of its largest complementary connected component is between $\ep_0  n^{\xi/2}$ and $ \ep_0^{-1}  n^{\xi/2}$. \label{item-A0-maximal}
\end{enumerate}
Let $\wh A_0^n$ be the event that there exist loops $\ul \ell_0^n$ and $\ol\ell_0^n$ such that the above 6 conditions hold. 

By Lemma~\ref{prop-uniform-root}, if we condition on the un-rooted loop-decorated quadrangulation $(M^n, \mcl L^n)$, the law of the root edge $\BB e_0^n$ of $Q^n$ is uniform among the edges of $Q^n$. Since the events $\wh A_i^n $ depend only on $(M^n , \mcl L^n)$ (not on the choice of root edge) we have 
\eqbn
\BB P\left(\wh A_0^n  \,|\,  (M^n , \mcl L^n) \right) = \frac{\# \wh S^n(\ep)}{2n } .
\eqen 
By averaging over all possible realizations of $(M^n , \mcl L^n)$, and using~\eqref{eqn-loops-exist-set}, we get 
\eqb \label{eqn-sample-root}
\BB P\left(\wh A_0^n \,|\, X(1,2n) = \emptyset   \right) \geq 1- q . 
\eqe 

We will finish the proof by showing that for large enough $n$, we have $\wh A_0^n \subset \mcl G_n(\ep_0 , \ep_1')$ for $\ep_1' > 0$ depending only on $\ep_1$. Indeed, suppose $\wh A_0^n$ occurs, and let $\ul\ell_0^n$ and $\ol\ell_0^n$ be the loops satisfying the conditions in the definition of $\wh A_0^n$. Let $\wt\iota_0^n$ (resp. $\wt\iota_1^n$) be largest $i\in [1,2n]_{\BB Z}$ such that $\lambda^n(i-1) \in  \ul M_0^n$ (resp. $\lambda^n(i-1) \in \ol M_0^n $). 
We claim that $\wt\iota_0^n = \iota_0^n$ and $\wt\iota_1^n = \iota_1^n$ on the event $\wh A_0^n$. 
Our proof of this claim relies on some basic facts about the times during which the path $\lambda^n$ is tracing the edges of a complementary connected component $U$ of a loop $\ell \in \mcl L^n$ which lies on the opposite side of $\ell$ as the root edge. These facts can be deduced from the construction of Sheffield's bijection~\cite[Section 4.1]{shef-burger}, and are explained in more detail in~\cite{gwynne-miller-cle}. The facts are as follows.
\begin{enumerate}
\item If $i_*$ is the largest $i \in [1,n]_{\BB Z}$ for which $\lambda^n(i_*-1) \in U$, then $ X_{i_*} = \tb F$ and $\phi(i_*)$ is the time at which $\lambda^n$ begins filling in $U$. We have $X_{\phi(i_*)} = \tc C$ (resp. $X_{\phi(i_*)} = \tc H$) if $\ell$ is traced by $\lambda^n$ in the counterclockwise (resp. clockwise) direction. Furthermore, we have $U = \lambda^n ([\phi(i_*)  , i_*-1]_{\BB Z})$. \label{item-component-match}
\item We have $\op{Area} U = i_* - \phi(i_*)$ and $\op{len}\left(\partial U \right) = |X(\phi(i_*) , i_*) | + 1$. \label{item-component-area}
\item If $i_*'$ is the smallest $i \geq i_* +1$ with $X_{i} = \tb F$ and $\phi(i)  \leq \phi(i_*) -1$, then $X_{\phi(i_*')} \not= X_{\phi(i_*)}$. Furthermore, $i_*$ (resp. $\phi(i_*')$) is the time at which $\lambda^n$ finishes (resp. begins) tracing the loop $\ell$. \label{item-component-next}
\item Conversely, if $i_* \in [1,2n]_{\BB Z}$ such that $X_{i_*} = \tb F$ and, with $i_*'$ as above, we have $X_{\phi(i_*')} \not= X_{\phi(i_*)}$, then $\lambda^n([\phi(i_*) , i_*-1]_{\BB Z})$ is a complementary connected component $U$ of a loop $\ell \in \mcl L^n$ which lies on the opposite side of $\ell$ as the root edge. \label{item-component-converse}
\end{enumerate}

Since $\wt\iota_0^n$ is the time immediately after $\lambda^n$ finishes filling in a complementary connected component of clockwise loop which lies on the opposite side of this loop from the root edge, fact~\ref{item-component-match} above implies that $X_{\wt\iota_0^n} = \tb F$, $X_{\phi(\wt\iota_0^n)} = \tc H$, and $\ul M_0^n = \lambda^n\left([\phi(\wt\iota_0^n) , \wt\iota_0^n]_{\BB Z}\right)$. Therefore, condition~\ref{item-big-loop-match} in the definition of $\iota_0^n$ holds with $i = \wt \iota_0^n$. Furthermore, fact~\ref{item-component-area} implies that
\alb
\op{Area} \ul M_0^n = \wt\iota_0^n - \phi(\wt\iota_0^n) \quad\op{and}\quad
\op{len}\left(\partial \ol M_0^n \right)  = |X(\phi(\wt\iota_0^n),\wt\iota_0^n)|  +1 .
\ale
It follows from conditions~\ref{item-A0-area} and~\ref{item-A0-length} in the definition of $\wh A_0^n$ that conditions~\ref{item-big-loop-upper} and~\ref{item-big-loop-length} in the definition of $\iota_0^n$ hold with $i =  \wt\iota_0^n$. 
By fact~\ref{item-component-next}, we also have that condition~\ref{item-big-loop-loop} in the definition of $\iota_0^n$ holds with $i = \wt \iota_0^n$. Therefore, $\wt\iota_0^n \leq \iota_0^n$ on $\wh A_0^n$. 

Conversely, conditions~\ref{item-big-loop-match} and~\ref{item-big-loop-loop} in the definition of $\iota_0^n$ together with fact~\ref{item-component-converse} imply that $\lambda^n([\phi(\iota_0^n) , \iota_0^n])$ is a connected component of a loop of $\mcl L^n$ which lies on the opposite side of this loop from the root edge. Furthermore, by conditions~\ref{item-big-loop-upper} and~\ref{item-big-loop-length} in the definition of $\iota_0^n$ together with fact~\ref{item-component-area}, this component has area between $2 n - n^\xi$ and $2n - \ep_0 n^\xi$ and boundary length between $\ep_0  n^{\xi/2}$ and $ \ep_0^{-1}  n^{\xi/2}$. This contradicts condition~\ref{item-A0-maximal} in the definition of $\wh A_0^n$. Thus $\iota_0^n = \wt\iota_0^n$ on $\wh A_0^n$. 

Now we consider $\wt \iota_1^n$. By fact~\ref{item-component-match} above, we have $X_{\wt\iota_1^n} = \tb F$, $X_{\phi(\wt\iota_1^n)} = \tc C$, and $\ol M_0^n = \lambda^n\left([\phi(\wt\iota_1^n) , \wt\iota_1^n]_{\BB Z}\right)$. Furthermore, by condition~\ref{item-A0-loops} in the definition of $\wh A_0^n$, there is no $i \in (\wt\iota_1^n , \iota_0^n)_{\BB Z}$ with $X_i = \tb F$, $X_{\phi(i)} = \tc C$, and $\phi(i) < \phi(\wt\iota_1^n)$ (by fact~\ref{item-component-converse}, the largest such $i$ would correspond to the exit time of $\lambda^n$ from a complementary connected component of a loop lying between $\ul \ell_0^n$ and $\ol\ell_0^n$).  

By conditions~\ref{item-A0-subloop-area} and~\ref{item-A0-subloop-length} in the definition of $\wh A_0^n$ together with fact~\ref{item-component-area}, we have
\alb
 |X(\phi(\wt\iota_1^n) , \wt\iota_1^n)|  &=  \op{len}\left(\partial \ol M_0^n \right)-1 \geq \ep_1 n^{\xi/2}-1 \geq \ep_1' \left( \iota_0^n - \phi( \iota_0^n) \right)^{\xi/2} \\
 \wt\iota_1^n - \phi(\wt\iota_1^n) & = \op{Area} \ol M_0^n \geq 2n -  \ep_1^{-1} \left(2n -  \op{Area} \ul M_0^n \right) \\
& \geq 2n - \ep_1^{-1} n^\xi  \geq  \iota_0^n - \phi( \iota_0^n)  -    (\ep_1')^{-1} \left( \iota_0^n - \phi( \iota_0^n) \right)^\xi
\ale
for $\ep_1'$ slightly smaller than $\ep_1/2$. The second inequality implies in particular that $\wt\iota_1^n - \phi(\wt\iota_1^n) \geq \frac12 \left( \iota_0^n - \phi( \iota_0^n) \right)$, so $\wt\iota_1^n = \iota_1^n$. The two inequalities together imply that $\mcl G_n(\ep_0,\ep_1')$ occurs.
\end{proof}

\subsection{Proof of Lemma~\ref{prop-big-F-constant}} 
\label{sec-big-loop-word}

In this subsection, we will prove Lemma~\ref{prop-big-F-constant} and thereby complete the proof of Proposition~\ref{prop-big-loop}. This section is the only place in the paper where the (infinite-volume version of the) main result of~\cite{gwynne-miller-cle} is needed. Recall from Remark~\ref{remark-implication} that the proof of the infinite volume version of the main result of~\cite{gwynne-miller-cle} requires only the results of~\cite{shef-burger,gms-burger-cone,gms-burger-local,wedges}, \emph{not} the results of the present paper.
 
\begin{proof}[Proof of Lemma~\ref{prop-big-F-constant}]
We first prove the measurability statement.
It follows from Sheffield's bijection (see also~\cite{gwynne-miller-cle}) that for $i ,j \in\BB N$, the time $\theta_{i,j}$ when $\lambda$ finishes tracing the loop $\ell_{i,j}$ can be described as follows: $\theta_{i,j}$ is the $j$th smallest $k > i$ such that $X_k = \tb F$, $\phi(k) < i$, and $X_{\phi(k)} \not= X_{\phi(\wt k)}$, for $\wt k$ the largest $k' \in [i,k)_{\BB Z}$ satisfying $X_{\wt k} = \tb F$ and $X_{\phi(\wt k)} < i$. From this, we infer that for each $k_1 , k_2 \in \BB N$ with $k_1 < i < k_2$, the event $\{\phi(\theta_{i,j}) =  k_1 , \, \theta_{i,j} = k_2\}$ is determined by $X_{ k_1} \dots X_{k_2}$. Furthermore, the part of the infinite-volume FK planar map traced by the path $\lambda$ during the time interval $[\phi(\ol \theta_i^n) , \ol \theta_i^n]_{\BB Z}$ is determined by the word $X_{\phi(\ol \theta_i^n)} \dots X_{ \ol\theta_i^n}$. It therefore follows from the definition of $\wt A_i^n$ that this event is determined by $X_1\dots X_n$ (here we note that $J_i^n(\ep_0)$ is the \emph{smallest} $j\in\BB N$ for which certain conditions hold). 
 
To prove the claimed estimate for $\# \wt S^n$, let $\kappa \in (4,8)$ and $\gamma \in (\sqrt 2 , 2)$ be determined by $p$ as in~\eqref{eqn-p-kappa}. Let $\Gamma$ be a whole-plane $\op{CLE}_{\kappa}$ (as defined in~\cite{mww-extremes,werner-sphere-cle}) and let $(\BB C , h , 0, \infty)$ be a $\gamma$-quantum cone independent from $\Gamma$ (as defined in~\cite[Section 4.3]{wedges}). 
We will first estimate the probability of a continuum analogue of the events $\wt A_i^n$ which is defined in terms of $\Gamma$ and $h$. We will then apply the infinite-volume version of the scaling limit results in~\cite{gwynne-miller-cle} to convert this into an estimate for $\# \wt S^n$ (c.f.\ Remark~\ref{remark-implication}). 

To this end, fix $\alpha>0$ to be chosen later, depending only on $q$. 
Let $\eta$ be a whole-plane space-filling $\op{SLE}_\kappa$ process from $\infty$ to $\infty$ which traces the loops in $\Gamma$ (as described in~\cite[Sections 1.2.3 and 4.3]{ig4} and~\cite[Footnote 9]{wedges}), parametrized by quantum mass with respect to $h$ relative to time 0 (so in particular $\eta(0) =0$). For $t\in (0,1)$ and $\ep_0 > 0$, let $\ul \ell^t = \ul\ell^t(\ep_0)$ be the innermost clockwise loop $\ell$ in $\Gamma$ surrounding $\eta(t)$ such that the set of points disconnected from $\infty$ by $\ell$ has quantum area at least $\ep_0$ and quantum boundary length between $\ep_0$ and $\ep_0^{-1}$ (both measured with respect to $h$). Let $\ol \ell^t$ be the next outermost loop in $\Gamma$ surrounding $\eta(t)$. Let $A_t = A_t(\ep_0,\ep_1)$ be the continuum analogue of the event $\wt A_{\lfloor t n \rfloor}^n$, i.e.\ the event that the following is true.
\begin{enumerate}
\item The boundary of the set of points disconnected from $\infty$ by $\ol\ell^t$ has quantum length at least $\ep_1 n^{1/2}$. 
\item The quantum area of the set of points disconnected from $\infty$ by $\ol\ell^t$ is at most $\ep_1^{-1}$ times the quantum area of the set of points disconnected from $\infty$ by $\ul\ell^t$.
\item $\eta$ traces all of $\ol\eta^t$ during the time interval $[0,1]$. 
\end{enumerate}
For each $t\in \BB R$, $\eta$ a.s.\ traces arbitrarily small $\op{CLE}_\kappa$ loops surrounding $\eta(t)$ whose outer boundaries have finite positive quantum length in arbitrarily small intervals of time surrounding $t$. Since $\eta(\cdot - s) \eqD \eta$ for each fixed $s\in\BB R$~\cite[Lemma 9.3]{wedges}, it follows that there exists and $\ep_0 , \ep_1   > 0$ (depending only on $\alpha$) such that for each $t\in [\alpha  , 1 -\alpha ] $ we have $\BB P(A_t) \geq 1-\alpha$. (We remark that the event $A_t$ can equivalently be described in terms of the $\pi/2$-cone times of a correlated Brownian motion as in~\eqref{eqn-bm-cov}, which gives the left and right quantum boundary lengths of $\eta$ at time $t$ with respect to $h$, in a manner which is directly analogous to the description of the event $\wt A_i^n$ in terms of the word $X$; see \cite[Theorem 1.13]{wedges} as well as~\cite{gwynne-miller-cle}.)
By the infinite-volume version of the main theorem of~\cite{gwynne-miller-cle}, we have
\eqbn
\lim_{n\rta\infty} \BB P\left(\wt A_{\lfloor t n \rfloor}^n \right) =  \BB P(A_t) \geq 1- \alpha , \quad \forall t\in [\alpha,1-\alpha] .
\eqen 
By dominated convergence, 
\eqbn
\lim_{n\rta\infty} n^{-1} \BB E\left( \#\wt S^n \right)  = \lim_{n\rta\infty} n^{-1} \sum_{i=1}^n \BB P\left(\wt A_i^n \right)   \geq 1 -4\alpha  .
\eqen
Hence we can find $n_* = n_*(\alpha,\ep_0,\ep_1) \in \BB N$ such that for $n\geq n_*$, it holds that $ \BB E\left( \#\wt S^n \right) \geq (1-4\alpha) n$. So, for
$n\geq n_*$ and $q\in(0,1)$ we have 
\eqbn
\BB E\left(  \#\wt S^n  \right) \leq (1-q)n \left(1- \BB P\left( \#\wt S^n   \geq (1-q) n \right) \right)   + n  \BB P\left( \#\wt S^n \geq (1 -q) n . \right) .
\eqen
By re-arranging this inequality, then choosing $\alpha$ sufficiently small depending on $q$ ($\alpha \leq q^2/4$ will suffice), we obtain the statement of the lemma.
\end{proof}

\section{Existence of a time with enough cheeseburger orders}
\label{sec-pi-reg}

The goal of this section is to address the following technical difficulty. Let $\pi_n =\pi_n(\ep)$ be the time of Definition~\ref{def-end-event}. The time $\pi_n$ is not a stopping time for the word $X$, read backward. However, we do have the following. Let $x$ be a realization of $X_{-\pi_n}\dots X_{-1}$. The conditional law of $X_{-J}\dots X_{-|x|-1}$ given $\{X_{-\pi_n} \dots X_{-1} = x\}$ is the same as its conditional law given that $\{X_{-|x|} \dots X_{-1} = x\}$ and each cheeseburger in $X(-J,-|x|-1)$ is matched to a cheeseburger order in $\mcl R(x)$ (c.f. Remark~\ref{remark-pi-not-stopping}). 

We want to say that this conditional law is similar to the conditional law of $X_{-J}\dots X_{-|x|-1}$ given only $\{X_{-|x|} \dots X_{-1} = x\}$. For this, we need a lower bound on the conditional probability given $\{X_{-|x|} \dots X_{-1} = x\}$ that there is no cheeseburger in $X(-J,-|x|-1)$ matched to a flexible order in $\mcl R(x)$. This probability is uniformly positive provided there are a large number of cheeseburger orders to the left of the leftmost flexible order in $\mcl R(x)$. But, we do not know that this is the case, so instead we will construct a time $\wt\pi_n \geq \pi_n$ which has similar properties to $\pi_n$ but has the additional property that $X(-\wt\pi_n , -1)$ is likely to contain a large number (of order $n^{\xi/2}$) cheeseburger orders to the left of its leftmost flexible order. The precise properties of the time $\wt\pi_n$ are described in Lemma~\ref{prop-wt-pi-exists} below. 

Recall the definitions of $J_m^H$, $L_m^H$, $J_m^C$, and $L_m^C$ from~\eqref{eqn-J^H-def} and the discussion just below. We also introduce the following additional notation. 

\begin{defn} \label{def-c_f}
For a word $x$ consisting of elements of $\Theta$, we write $\frk c_f(x)$ for the number of cheeseburger orders to the left of the leftmost flexible order in $\mcl R(x)$ (or the number of cheeseburger orders in $\mcl R(x)$ if $\mcl R(x)$ contains no flexible orders). We also let $\frk r(  x) :=  \frk c_f(  x)/\frk o( x) $, with $\frk o(x)$ as in~\eqref{eqn-theta-count-reduced} of Notation~\ref{def-theta-count}. 
\end{defn}

The main result of this section is the following proposition.

\begin{prop} \label{prop-wt-pi-exists}
Let $\ep_1>\ep_2>0$. and let $\pi_n = \pi_n(\ep_1)$ be as in Definition~\ref{def-end-event} There is a random time $\wt \pi_n  = \wt\pi_n(\ep_1,\ep_2) \in [1, J]_{\BB Z}$ with the following properties. 
\begin{enumerate}
\item We have $\wt\pi_n\geq \pi_n$ and $\mcl N_{\tb F}\left(X(-\wt\pi_n,-\pi_n)\right) = 0$. \label{item-wt-pi-mono}
\item Whenever $\pi_n < J$, we have 
\alb
&\mcl N_{\tb H}\left(X(- \pi_n, -1)\right)- \ep_2 n^{\xi/2}   \leq  \mcl N_{\tb H}\left(X(-\wt\pi_n, -1)\right)   \leq   \mcl N_{\tb H}\left(X(- \pi_n,-1)\right) \quad \op{and} \\ 
&\mcl N_{\tb C}\left(X(- \pi_n, -1)\right)  -\frac12 n^{\xi/2}  \leq  \mcl N_{\tb C}\left(X(-\wt\pi_n, -1)\right)   .  
\ale
\label{item-wt-pi-compare}
\item Let $x$ be any realization of $X_{-\wt\pi_n} \dots X_{-1}$ and let $\frk o(x)$ be as in~\eqref{eqn-theta-count-reduced} of Notation~\ref{def-theta-count}. The conditional law of $X_{-J} \dots X_{-|x|-1}$ given $\{X_{-\wt\pi_n} \dots X_{-1} = x\} $ is the same as the conditional law of $X_{-  J_{\frk o(x)}^H} \dots X_{-1}$ (defined as in~\eqref{eqn-J^H-def}) given the event that each cheeseburger in $X(-J_{\frk o(x)}^H , -1)$ is matched to a cheeseburger order in $x$ when we consider the reduced word $\mcl R\left(X(-J_{\frk o(x)}^H , -1) x\right)$.  \label{item-wt-pi-cond}
\item There is an $n_* \in\BB N$ such that for each $n\geq n_*$ and each $\zeta \geq n^{-\xi/2}$, we have (in the notation of Definition~\ref{def-c_f}) 
\[
\BB P\left(  \frk c_f\left(X(-\wt\pi_n ,-1) \right) \leq \zeta n^{\xi/2} \,|\, X_{-\pi_n} \dots X_{-1} \right) \BB 1_{(\pi_n <J)} \preceq \zeta^{100}  
\]
with the implicit constant deterministic and independent of $\zeta$ and $n$. \label{item-wt-pi-C} 
\end{enumerate}
\end{prop}

\begin{remark}
Condition~\ref{item-wt-pi-C} is the key property of $\wt\pi_n$. The other conditions in Proposition~\ref{prop-wt-pi-exists} hold with $  \pi_n$ in place of $\wt\pi_n$. See the discussion above for an explanation of why we need condition~\ref{item-wt-pi-C}.
\end{remark} 
  
Although $\wt\pi_n$ is not a stopping time for the word $X$ read backward, the following lemma allows us to compare the conditional law of $X_{-J} \dots X_{-1}$ given a realization of $X_{-\wt\pi_n} \dots X_{-1}$ to the conditional law we would get if $\wt\pi_n$ were in fact a stopping time. 

\begin{lem} \label{prop-last-abs-cont}
Let $\wt x$ be a realization of $X_{-\wt \pi_n}\dots X_{-1}$ for which $\wt\pi_n < J$. Let $\frk c_f(\wt x)$ and $\frk r(\wt x)$ be as in Definition~\ref{def-c_f} and let $\frk o(\wt x)$ be as in~\eqref{eqn-theta-count-reduced}.
The conditional law of the word $X_{-J} \dots X_{-\wt\pi_n-1}$ given $\{X_{-\wt \pi_n}\dots X_{-1} = x\}$ is absolutely continuous with respect to its conditional law given only $\{X_{-|\wt x|-1} \dots X_{-1}\}$, with Radon-Nikodym derivative bounded above by $\frk r(\wt x)^{-1 + o_{\frk r(\wt x)}(1)}$. Here the $o_{\frk r(\wt x)}(1)$ tends to zero as $\frk r(\wt x) \rta 0$, at a rate depending only on $\frk r(\wt x)$
\end{lem}
\begin{proof}
Fix a realization $\wt x$ as in the statement of the lemma.   
For $m \in \BB N$, let $ J_{\wt x ,m}^H$ be the smallest $j \geq |\wt x| + 1$ for which $X(-j   ,-|\wt x| -1)$ contains $m$ hamburgers and let $L_{\wt x , m}^H := d^*\left(X(-J_{\wt x ,m}^H , -|\wt x|-1)\right)$. Note that the conditional law of the pairs $(J_{\wt x , m}^H -|\wt x| , L_{\wt x , m}^H)$ given $\{X_{-|\wt x| }\dots X_{-1} = \wt x\}$ is the same as the law of the pairs $( J_{m}^H ,  L_{m}^H)$ of~\eqref{eqn-J^H-def}. 

By condition~\ref{item-wt-pi-cond} of Proposition~\ref{prop-wt-pi-exists}, the conditional law of $X_{-J} \dots X_{-|\wt x|-1}$ given $ \{X_{-\wt\pi_n }\dots X_{-1} = \wt x\}$ is the same as the conditional law of $X_{-J_{\wt x , \frk o(\wt x)}^H} \dots X_{-|\wt x|-1}$ given that every cheeseburger in $X(-J_{\wt x , \frk o(\wt x)}^H,-|\wt x|-1)$ is matched to a cheeseburger order in $X(-|\wt x| , -1)$. By Lemma~\ref{prop-J^H<J^C-prob} (proven just below), the conditional probability that this is the case given $\{ X_{-|\wt x |}\dots X_{-1} = \wt x\}$ is at least $\frk r(\wt x)^{ 1+o_{\frk r(\wt x)}(1)}$. The Radon-Nikodym derivative of the conditional law of $X_{-J} \dots X_{-|\wt x|-1}$ given $ \{X_{-\wt\pi_n }\dots X_{-1} = \wt x\}$ with respect to its conditional law given $ \{X_{-|\wt x|}\dots X_{-1} = \wt x\}$ is bounded above by the reciporical of this conditional probability. 
\end{proof}

\subsection{Probability of few cheeseburgers before a given number of hamburgers} 
\label{sec-J^H<J^C-prob}

In this section we will prove an estimate for the probability that $J_h^H < J_c^C$ for $h$ larger than $c$. We start by addressing the analogous problem for Brownian motion.

\begin{lem} \label{prop-bm-coord-hit}
Let $Z = (U,V)$ be a correlated Brownian motion as in~\eqref{eqn-bm-cov}. For $b > 0$, let $\tau_b^U$ be the smallest $t > 0$ for which $U(t) =b$ and for $\zeta >0$, let $\tau_\zeta^V$ be the smallest $t>0$ for which $V(t) = \zeta$. Then we have 
\eqb \label{eqn-bm-coord-hit}
\BB P\left( \tau_1^U < \tau_\zeta^V ,\,  V(\tau_1^U) \leq -1  \right) \succeq \zeta 
\eqe 
with the implicit constant depending only on $p$. Furthermore, for each $q\in (0,1)$, there exists $\delta_0 > 0$, depending only on $q$, such that for each $\zeta >0$,
\eqb \label{eqn-bm-coord-reg}
\BB P\left( V(\tau_1^U)  - \sup_{s \in [0,\tau_1^U]} V(s)      \leq -\delta_0 \,|\,  \tau_1^U < \tau_\zeta^V \right) \geq 1-q. 
\eqe  
\end{lem}
\begin{proof}
Let $L := \{1\} \times (-\infty,-1] $ and for $\zeta > 0$ let $L_\zeta' := \BB R \times \{\zeta\}$. 
Let
\eqbn
A := \left(  \begin{array}{cc}
1 &  - \frac{p}{1-p}   \\ 
0  &  \frac{\sqrt{1-2p}}{1-p}
\end{array}        \right) ,\qquad \wt Z  = (\wt U , \wt V):= A Z .
\eqen 
Then $\wt Z$ is a pair of independent Brownian motions. The event $\{\tau_1^U \leq \tau_\zeta^V ,\,  V(\tau_1^V) \leq -1\}$ is the same as the event that $\wt Z$ hits $A L$ before hitting $AL_\zeta'$. The set $AL_\zeta'$ is a horizontal line at distance proportional to $\zeta$ from the origin. Therefore, the probability that $\wt Z$ fails to hit $AL_\zeta'$ before time 1 is proportional to $\zeta$. By \cite[Theorem 2]{shimura-cone}, applied with angle $\pi$, the conditional law of $\wt Z$ given that it fails to hit $AL_\zeta'$ before time 1 converges as $\zeta\rta 0$ (with respect to the uniform topology) to the law of a random path in the lower half plane which has positive probability to hit $AL$ before time 1. This yields~\eqref{eqn-bm-coord-hit}.

To obtain~\eqref{eqn-bm-coord-reg}, we will first argue that for any $q\in(0,1)$, there exists a deterministic time $T>0$ such that for any $\zeta>0$, we have
\eqb \label{eqn-bm-coord-time}
\BB P\left( \tau_1^U \geq T \,|\,  \tau_1^U < \tau_\zeta^V \right) \geq \sqrt{1-q} . 
\eqe  
Let $a_\zeta$ be the distance from the point $A(L\cap L_\zeta')$ to the origin. Note that $a_\zeta$ is bounded below by a constant $a$ which is independent of $\zeta$. Let $\wt\tau_{a}$ be the first time $\wt U$ hits $a $. Then $\wt \tau_{a }$ is independent from $\wt V$ and we have $\tau_1^U \geq \wt \tau_{a }$. Therefore, for any $T > 0$ we have
\alb
\BB P\left( \tau_1^U < T \,|\,  \tau_1^U < \tau_\zeta^V \right) &\leq \frac{\BB P\left(\wt\tau_a < T \wedge \tau_\zeta^V \right)}{\BB P\left( \tau_1^U < \tau_\zeta^V \right)} \\
&\preceq \zeta^{-1}  \sum_{k= \lceil \log_2 T^{-1} \rceil}^\infty \BB P\left( \wt\tau_a \leq 2^{-k} ,\, \tau_\zeta^V \geq 2^{-k-1}    \right) ,
\ale
where in the last line we used~\eqref{eqn-bm-coord-hit}. By the Gaussian tail bound, 
\eqbn
\BB P\left(\wt\tau_a \leq 2^{-k}  \right) \preceq e^{-b 2^k}
\eqen
for $b$ a constant independent of $k$, $T$, and $\zeta$. Furthermore, by Brownian scaling we have
\alb
\BB P\left(\tau_\zeta^V \geq 2^{-k-1} \right) = \BB P\left(\tau_1^V \geq \zeta^{-2} 2^{-k-1} \right) \preceq \zeta  2^{ k/2}
\ale
where here we used that $\tau_1^V$ has the law of a constant multiple of a $1/2$-stable random variable. It follows that
\eqbn
\BB P\left( \tau_1^U < T \,|\,  \tau_1^U < \tau_\zeta^V \right) \preceq \sum_{k= \lceil \log_2 T^{-1} \rceil}^\infty 2^{k/2} e^{-b 2^k} ,
\eqen
which tends to 0 as $T\rta 0$. This implies~\eqref{eqn-bm-coord-time}. 

By \cite[Theorem 2]{shimura-cone} with angle $\pi$, for any $q \in (0,1)$ and $T > 0$, we can find $\delta_0 > 0$, depending only on $q$ and $T$, such that 
\eqbn
\BB P\left(    V(\tau_1^U)  - \sup_{s \in [0,\tau_1^U]} V(s)  \leq -\delta_0   \,|\, \tau_\zeta^V > T\right) \geq \sqrt{1-q} . 
\eqen
By combining this with~\eqref{eqn-bm-coord-time}, we infer~\eqref{eqn-bm-coord-reg}. 
\end{proof}

\begin{lem} \label{prop-J^H<J^C-prob}
For $m\in\BB N$, let $J_m^H$ (resp. $J_m^C$) be the smallest $j\in\BB N$ for which $X(-j,-1)$ contains $m$ hamburgers (resp. cheeseburgers).  
For $\ep >0$, we have
\eqbn
\BB P\left( J_m^H < J_{\lfloor \ep m\rfloor}^C \right) \succeq \ep^{ 1 + o_\ep(1)  }
\eqen
with the implicit constant depending only on $p$ and the $o_\ep(1)$ depending only on $\ep$ and $p$. 
\end{lem}
\begin{proof}
The proof is similar to some of the arguments in \cite[Section 2]{gms-burger-cone}. 
Fix $\delta > 0$, to be chosen later. For $k \in\BB N$ and $\ep > 0$, let $r_m^k(\ep) := \lfloor \delta^{-k} \ep m\rfloor$. Also let $\BB k(\ep) := \lfloor \frac{\log \ep }{\log \delta} \rfloor$ be the smallest $k\in\BB N$ for which $r_m^k(\ep) \geq m$. Let $E_m^k(\ep)$ be the event that $X(-J_{r_m^k(\ep)}^H , -J_{r_m^{k-1}(\ep)}^H-1)$ contains at least $r_m^k(\ep)$ cheeseburger orders and at most $r_m^{k-1}(\ep)$ cheeseburgers. Then we have
\eqb \label{eqn-J^H<J^C-contain}
\bigcap_{k=1}^{\BB k(\ep)} E_m^k(\ep) \subset \left\{J_m^H \geq J_{\lfloor \ep m\rfloor}^C \right\} .
\eqe 
By \cite[Theorem 2.5]{shef-burger} and Lemma~\ref{prop-bm-coord-hit}, we have
\eqbn
\BB P\left(E_m^k(\ep) \right) \succeq \delta + o_{\ep m}(1) ,
\eqen
with the implicit constant depending only on $p$ and the $o_{\ep m}(1)$ depending only on $\ep m$ (not on $k$). Hence there is a $b>0$, depending only on $p$, such that for any given $\delta >0$, we can find $n_* = n_*(\delta) \in \BB N$ such that whenever $\ep m \geq n_*$ and $k \in \BB N$, we have
\eqbn
\BB P\left( E_m^k(\ep) \right) \geq b \delta .
\eqen
By~\eqref{eqn-J^H<J^C-contain}, in this case we have
\eqbn
\BB P\left(J_m^H \geq J_{\lfloor \ep m\rfloor}^C \right) \geq (b \delta)^{\BB k(\ep)} \succeq b^{\frac{\log \ep}{\log \delta}} \ep .
\eqen
Since $\delta$ does not depend on $b$, by making $\delta$ arbitrarily small relative to $b$, we obtain the statement of the lemma.
\end{proof}

\subsection{Monotonicity conditioned on cheeseburgers being matched to cheeseburger orders}
\label{sec-F-event-mono}

For a word $  x = x_0\dots x_{|x|}$ consisting of elements of $\Theta$ and $m\in\BB N$, let $F_m( x)$ be the event that every cheeseburger in $X(-J_m^H , -1)$ is matched to a cheeseburger order in $x$ when we consider the reduced word $\mcl R(X(-J_m^H , -1) x )$. Equivalently, there is \emph{not} a $j \in [1,J_m^H]_{\BB Z}$ with the property that $X_{-j} = \tc C$; and on the event $\{X_0 \dots X_{|x|} = x\}$, we have $\phi(-j) \geq 0$ and either $X_{\phi(-j)} = \tb F$ or $\phi(-j) > |x|$. 

For $l \in\BB N$, let $f_l(x)$ be the $l$th smallest $i \geq 0$ such that $x_i =\tb F$ and $x_i$ does not have a match in $x$; or $f_l(x) = |x|+1$ if no such $i$ exists. We note that with $\frk c_f(x)$ as in Definition~\ref{def-c_f}, 
\eqbn
\frk c_f(x) = \mcl N_{\tb C}\left(\mcl R(x_0 \dots x_{f_1(x)-1})\right) .
\eqen

\begin{lem} \label{prop-F-event-mono}
Let $x = x_0 \dots x_{|x|}$ and $x' = x_0'\dots x_{|x'|}'$ be two words such that $\mcl R(x)$ and $\mcl R(x')$ contain no burgers; $x_{f_1(x)} \dots x_{|x|} = x_{f_1(x')}' \dots x_{|x'|}$; $\mcl N_{\tb H}\left( \mcl R(x_0\dots x_{f_1(x)-1}) \right) = \mcl N_{\tb H}\left(\mcl R(x_0' \dots x_{f_1(x')}') \right)$; and $\frk c_f(x) \leq \frk c_f(x')$. 
Then for each $m\in \BB N$, we have $F_m(x) \subset F_m(x')$. 
\end{lem}
\begin{proof}
Let $k  := \mcl N_{\tb F}(\mcl R(x))$ and $h := \mcl N_{\tb H}(\mcl R(x_0 \dots x_{f_1(x)-1}))$. By hypothesis the definitions of $k$ and $h$, are unchanged if we replace $x$ with $x'$.

We induct on $k$. The case $k=0$ is trivial. 
Now assume $k > 1$ and we have proven $F_m(x) \subset F_m(x')$ for words $x$ and $x'$ satisfying the hypothesis of the lemma with $\mcl N_{\tb F}(\mcl R(x)) < k$. If $F_m(x)$ occurs, then either $m \leq h$ and $J_{\frk c_f(x) +1}^C > J_m^H$; or $J_{h+1}^H < J_{\frk c_f(x)+1}^C$. In the former case, it is clear that $F_m(x')$ also occurs. Now assume we are in the latter case. Let $\wt x := X(-J_{h+1}^H , -1) x$ and $\wt x' := X(-J_{h+1}^H , -1) x'$. The $\tb F$ symbol corresponding to $x_{f_1(x)}$ is matched to $X_{-J_{h+1}^H}$ in $\mcl R(\wt x)$, and the word $X(-J_{h+1}^H , -1)$ contains no flexible orders. Therefore, we have $\wt x_{f_1(\wt x)} \dots \wt x_{|\wt x|} = x_{f_2(x)} \dots x_{|x|}$. Similarly, $\wt x'_{f_1(\wt x')} \dots  \wt x_{|\wt x'|}' = x_{f_2(x')}' \dots x_{|x'|}'$. By our hypothesis on $x$ and $x'$, we therefore have $\wt x_{f_1(\wt x)} \dots \wt x_{|\wt x|} = \wt x'_{f_1(\wt x')} \dots  \wt x_{|\wt x'|}'$. Furthermore, we have
\alb
\mcl N_{\tb H}\left(\mcl R(\wt x_0 \dots \wt x_{f_1(\wt x)-1}) \right)  = \mcl N_{\tb H}\left(\mcl R(x_{f_1(x )+1} \dots x_{f_2(x)-1}) \right) = \mcl N_{\tb H}\left(\mcl R(\wt x_0' \dots \wt x_{f_1(\wt x')-1}') \right) 
\ale
and since $J_{\frk c_f(x)+1}^C > J_{h+1}^H$, 
\alb
\mcl N_{\tb C}\left(\mcl R(\wt x_0 \dots \wt x_{f_1(\wt x)-1}) \right)  
&= \mcl N_{\tb C}\left(\mcl R(x_{f_1(x )+1} \dots x_{f_2(x)-1}) \right) + \frk c_f(x) - L_{h+1}^H  \\
&\leq \mcl N_{\tb C}\left(\mcl R(x_{f_1(x')+1}' \dots x_{f_2(x')-1}') \right) + \frk c_f(x') - L_{h+1}^H \\
&= \mcl N_{\tb C}\left(\mcl R(\wt x_0' \dots \wt x_{f_1(\wt x')-1}') \right)  .
\ale
Hence the words $\wt x$ and $\wt x'$ satisfy the hypotheses of the lemma.

If $F_m(x)$ occurs, then every cheeseburger in $X(-J_{m }^H , - J_{h+1}^H-1  )$ is matched to a cheeseburger order in $\wt x$ (here we use that $X(-J_{h+1}^H , -1)$ contains no $\tb F$'s). Therefore, $F_{m-h-1}(\wt x)$ occurs. Since $\mcl N_{\tb F}(\wt x)  = k-1$, the inductive hypothesis implies that $F_{m-h-1}(\wt x')$ occurs. Hence, we have $J_{h+1}^H < J_{\frk c_f(x)+1}^C \leq J_{\frk c_f(x')+1}^C$ and each cheeseburger in $X(-J_{m }^H , - J_{h+1}^H-1  )$ is matched to a cheeseburger order in $X(-J_{h+1}^H , -1) x'$. It follows that $F_m(x')$ occurs. 
That is, $F_m(x) \subset F_m(x')$.  
\end{proof}

\begin{lem} \label{prop-J^H-F-mono}
Fix a word $x = x_0 \dots x_{|x|}$ consisting of elements of $\Theta$ and write $h := \mcl N_{\tb H}(\mcl R(x_0 \dots x_{f_1(x)-1}))$. For any $l   \in \BB Z$, any $m\in [1,h]_{\BB Z}$, and any $m' \geq m$, we have
\eqbn
\BB P\left(L_m^H \leq l \,|\, F_{m'}(x) \right) \geq \BB P\left( L_m^H \leq l  \,|\, J_{\frk c_f(x)+1}^C > J_m^H    \right)  .
\eqen
\end{lem}
\begin{proof}
We have $F_{m'}(x) \subset \{J_{\frk c_f(x)+1}^C > J_m^H  \}$, so 
by Bayes' rule,
\begin{align} \label{eqn-J^H-F-mono-bayes}
&\BB P\left(L_m^H \leq l \,|\, F_{m'}(x) \right)  
  = \frac{\BB P\left( F_{m'}(x) \,|\, L_m^H \leq l ,\, J_{\frk c_f(x)+1}^C > J_m^H  \right) \BB P\left( L_m^H \leq l \,|\,  J_{\frk c_f(x)+1}^C > J_m^H \right) }{\BB P\left( F_{m'}(x) \,|\, J_{\frk c_f(x)+1}^C > J_m^H  \right) } .
\end{align}
On the event $\{J_{\frk c_f(x)+1}^C > J_m^H\}$, the word $\mcl R\left(X(-J_m^H , -1) x \right)$ contains precisely $h - m$ hamburger orders and $\frk c_f(x) - L_m^H$ cheeseburger orders to the left of its leftmost flexible order. By Lemma~\ref{prop-F-event-mono}, it follows that 
\eqbn
\BB P\left( F_{m'}(x) \,|\, L_m^H = l ,\, J_{\frk c_f(x)+1}^C > J_m^H  \right)
\eqen
is decreasing in $l$. In particular, 
\eqbn
\BB P\left( F_{m'}(x) \,|\, L_m^H \leq l ,\, J_{\frk c_f(x)+1}^C > J_m^H  \right) \geq \BB P\left( F_{m'}(x) \,|\, J_{\frk c_f(x)+1}^C > J_m^H  \right) .
\eqen
Combining this with~\eqref{eqn-J^H-F-mono-bayes} yields the statement of the lemma. 
\end{proof}

\subsection{Regularity conditioned on few cheeseburgers before a given number of hamburgers} 
\label{sec-J^H<J^C-reg}
 
In light of Lemma~\ref{prop-J^H-F-mono}, to study the conditional law of $X_{-J_m^H} \dots X_{-1}$ given $F_{m'}(x)$ (in the terminology of that lemma), it suffices to study the conditional law of $X_{-J_m^H} \dots X_{-1}$ given $J_m^H  < J_c^C$ for an appropriate choice of $c$. In this section, we will study this latter conditional law. In particular, we will prove the following.

\begin{lem} \label{prop-J^H<J^C-reg}
For each $q\in (0,1)$, there is a $\zeta>0$ and an $m_* = m_*(q,\zeta ) \in\BB N$ such that for each $h  \geq m_*$ and each $c\in\BB N$, 
\eqbn
\BB P\left( \mcl N_{\tb C}\left(X(-J_h^H , -1)\right) \geq \zeta h  \,|\,  J_h^H < J_c^C \right) \geq 1-q .
\eqen 
\end{lem}

In the case where $c \geq h$, we have $\BB P(J_h^H < J_c^C) \geq 1/2$, so the statement of the Lemma~\ref{prop-J^H<J^C-reg} is immediate from \cite[Theorem 2.5]{shef-burger}. Hence we can assume without loss of generality that $c<  h$. 

The proof of Lemma~\ref{prop-J^H<J^C-reg} in this case is similar to the argument of \cite[Section 3.4]{gms-burger-cone}. For $c  \in\BB N$ and $a\geq 1$, let $m_{c,a} := \lfloor a  c \rfloor$. Also let
\alb
F_{c,a}  := \left\{J_{m_{c,a}}^H < J_c^C \right\} ,\qquad  E_{c,a}(\zeta) := \left\{ \mcl N_{\tb C}\left(X(-J_{m_{c,a}}^H , -1)\right) > \zeta m_{c,a} \right\} \cap F_{c,a}. 
\ale
Roughly speaking, will prove by induction that for each $q\in (0,1)$, there exists a $\zeta> 0$ and $m_* \in\BB N$ such that whenever $m_{c,a} \geq m_*$, we have
\eqbn
\BB P\left(E_{c,a}(\zeta) \,|\, F_{c,a}  \right) \geq 1-q. 
\eqen

\begin{lem} \label{prop-J^H<J^C-uniform}
Let $q\in (0,1)$ and $\lambda \in (0,1/2)$. There is a $\delta_0 > 0$, depending only on $q$ and $\lambda$, with the following property. For each $\zeta >0$, there exists $m_* = m_*( \zeta, \delta_0 ,q , \lambda) \in\BB N$ such that for each $c\in\BB N$ and $a>0$ with $m_{c,a} \geq m_*$; each realization $x$ of $X_{-J_{m_{c,a}}^H} \dots X_{-1}$ for which $E_{c,a}(\zeta)$ occurs; and each $b \in \left[(1-\lambda)^{-1} a , \lambda^{-1} a \right]$, we have
\eqbn
\BB P\left(E_{c,b}(\delta_0 ) \,|\, X_{-J_{m_{c,a}}^H} \dots X_{-1} = x ,\, F_{c,b} \right) \geq 1-q .
\eqen
\end{lem}
\begin{proof}
For $x$ as in the statement of the lemma, we have $\mcl N_{\tb F}(\mcl R(x)) = 0$, so the conditional law of $X_{-J_{m_{c,b}}^H } \dots X_{-J_{m_{c,a}}^H -1}$ given $\{X_{-J_{m_{c,a}}^H} \dots X_{-1} = x   \}\cap F_{c,b}$ is the same as its conditional law given that $\mcl N_{\tc C}\left( X(-J_{m_{c,b}}^H ,-J_{m_{c,a}}^H -1)\right) \leq  \mcl N_{\tb C}\left(\mcl R(x)\right) + c$. By \cite[Theorem 2.5]{shef-burger}, the probability of this event is bounded below by a constant which depends on $\zeta$ and $\lambda$, but not on $c$. 

By \cite[Theorem 2.5]{shef-burger}, as $m\rta\infty$ we have
\eqbn
\left( m^{-2} J_m^H ,  - m^{-1} \mcl N_{\tb C}\left(X(-J_m^H ,-1)\right)    \right)  \rta \left(\tau_1^U , V(\tau_1^U) - \sup_{s\in [0,\tau_1^U]} V(s) \right)
\eqen
in law, with the objects on the right defined as in Lemma~\ref{prop-bm-coord-hit}. The statement of the lemma now follows from Lemma~\ref{prop-bm-coord-hit}. 
\end{proof}

\begin{lem} \label{prop-J^H<J^C-ratio}
For each $q , q_0 \in (0,1)$, $\lambda \in (0,1/2)$, and $\ep > 0$, there exists $\zeta_0  = \zeta_0(q,q_0 , \ep)  >0$ and $a_0 = a_0(q,q_0,\ep) > 0$ such that for each $\zeta \in (0,\zeta_0]$, there exists $m_* = m_*(q,q_0 , \lambda, \ep , \zeta)$ such that the following is true. Suppose we are given $c \in \BB N$ and $a \geq a_0$ such that $m_{c,a}  \geq m_*$ and $\BB P\left(E_{c,a} (\ep) \,|\, F_{c,a}  \right) \geq q_0$. Then whenever $b \in \left[(1-\lambda)^{-1} a , \lambda^{-1} a \right]$ and $b' \in \left[(1-\lambda)^{-1} b , \lambda^{-1} b \right]$,
\eqb \label{eqn-J^H<J^C-ratio}
\frac{  \BB P\left(  F_{c,b'}  \,|\,  E_{c,b}(\zeta)^c \cap F_{c,b} \right) }{\BB P\left( F_{c,b'}  \,|\,  E_{c,b}(\zeta) \right)} \leq q .
\eqe 
\end{lem}
\begin{proof}
By Lemma~\ref{prop-J^H<J^C-uniform}, we can find $\delta_0 > 0$, depending only on $p$, such that for each $\zeta > 0$, there exists $\wt m_* = \wt m_*(\zeta,\delta_0,\lambda)$ such that whenever $m_{c,a}  \geq \wt m_*$ and $b$ is as in the statement of the lemma, we have
\eqbn
\BB P\left( E_{c,b}(\delta_0 ) \,|\, E_{c,a} (\zeta) \cap F_{c,b}   \right)  \geq\frac12 .
\eqen
Hence if $m_{c,a}  \geq \wt m_*$, then
\begin{align}
\BB P\left( E_{c,b}(\delta_0) \,|\, E_{c,b}(\zeta) \right)   &\geq \BB P\left(E_{c,b}(\delta_0) \,|\ F_{c,b}  \right)\notag \\
&\geq \BB P\left(E_{c,b}(\delta_0) \,|\, E_{c,a} (\ep)    \cap F_{c,b} \right)  \BB P\left(  E_{c,a} (\ep) \,|\,  F_{c,b}\right) \notag \\
&\geq \frac12\BB P\left(  E_{c,a} (\ep) \,|\,  F_{c,b}\right)  . \label{eqn-delta0-split}
\end{align}
By Bayes' rule,
\begin{align}
\BB P\left(  E_{c,a} (\ep) \,|\,  F_{c,b}\right) &=  \frac{\BB P\left( F_{c,b} \,|\,  E_{c,a} (\ep)\right) \BB P\left(   E_{c,a}(\ep) \,|\, F_{c,a} \right) }{   \BB P\left( F_{c,b} \,|\,  E_{c,a}(\ep)\right) \BB P\left(   E_{c,a}(\ep) \,|\, F_{c,a}\right) + \BB P\left( F_{c,b}  \cap E_{c,a}(\ep)^c  \,|\, F_{c,a} \right)          }  . \label{eqn-ep-to-delta-bayes}
\end{align} 
By \cite[Theorem 2.5]{shef-burger} and our assumption on $\BB P\left(E_{c,a} (\ep) \,|\, F_{c,a}  \right)$, this quantity is bounded below by a constant depending only on $q_0$ and $\ep$ (not on $\zeta$). By~\eqref{eqn-delta0-split}, we arrive at
\eqbn  
\BB P\left( E_{c,b}(\delta_0) \,|\, E_{c,b}(\zeta) \right)  \succeq 1 .
\eqen
By combining this with \cite[Theorem 2.5]{shef-burger} we obtain that for $b'$ as in the statement of the lemma,
\eqb  \label{eqn-top-bound}
\BB P\left( F_{c,b'}  \,|\,  E_{c,b}(\zeta) \right) \geq \BB P\left(  F_{c,b'}    \,|\, E_{c,b}(\delta_0) \right) \BB P\left( E_{c,b}(\delta_0) \,|\, E_{c,b}(\zeta) \right)  \succeq 1 .
\eqe 
 
Next we consider the numerator in~\eqref{eqn-J^H<J^C-ratio}.  
Observe that if $ F_{c,b'}  \cap E_{c,b}(\zeta)^c  $ occurs, then $X(-J_{m_{c,b'}}^H , -J_{m_{c,b}}^H-1)$ contains at most $\zeta m_{c,b} + c$ cheeseburgers. To estimate the probability that this occurs, consider the correlated Brownian motion $Z = (U,V)$ of~\eqref{eqn-bm-cov}. By \cite[Theorem 2.5]{shef-burger}, for each $s > 0$, the probability that $X(-J_{m_{c,b'}}^H , -J_{m_{c,b}}^H-1)$ contains at most $s m_{c,b}  $ cheeseburgers converges as $m_{c,b} \rta\infty$ to the probability that $U$ reaches height $b' - b$ before $V$ reaches height $s b$. The probability that this occurs tends to 0 as $s \rta 0$, at a rate depending only on $\lambda$. It follows that for each $\alpha>0$, we can find $\zeta_0 > 0$ and $a_0 > 0$ (depending only on $\alpha$) such that for each $\zeta \in (0,\zeta_0]$, there exists $m_* = m_*(\zeta, \alpha , \ep) \geq \wt m_*$ such that whenever $a \geq a_0$ and $m_{c,a} \geq m_*$, we have
\eqb \label{eqn-bottom-bound}
 \BB P\left(  F_{c,b'}  \,|\,  E_{c,b}(\zeta)^c \cap F_{c,b} \right)  \leq \alpha .
\eqe 
We conclude by combining~\eqref{eqn-top-bound} and~\eqref{eqn-bottom-bound} and choosing $\alpha$ sufficiently small depending on $q$ and $\lambda$.  
\end{proof}

\begin{lem} \label{prop-J^H<J^C-induct}
Let $q , q_0\in (0,1)$ and $\lambda \in (0,1/2)$. There is a $\zeta_0 > 0$ and an $a_0 > 0$ (depending only on $q $ and $q_0$) such that for each $\zeta \in (0,\zeta_0]$ we can find $m_* = m_*(q,q_0, \zeta) \in \BB N$ with the following property. Suppose $c  \in \BB N$ and $a \geq a_0$ with $m_{c,a} \geq m_*$ and $\BB P\left(E_{c,a} (\zeta) \,|\, F_{c,a}  \right) \geq q_0$. Suppose also that $b \in \left[(1-\lambda)^{-1} a , \lambda^{-1}  a \right]$. Then $\BB P\left(E_{c,b} (\zeta) \,|\, F_{c,b}  \right) \geq  1-q$. 
\end{lem}
\begin{proof} 
Given $a$ and $b$ as in the statement of the lemma, let $\wt b := (a+b)/2$. 
By Lemma~\ref{prop-J^H<J^C-uniform} we can find $\zeta_0 > 0$ (depending only on $  q$ and $\lambda$) such that for $\zeta \in (0,\zeta_0]$ and $\ep \in (0,\zeta]$, there exists $\wt m_* = \wt m_*(\zeta , \ep,  q ,\lambda )\in \BB N$ such that if $m_{c,a} \geq \wt m_*$ and $b \in \left[(1-\lambda)^{-1} a , \lambda^{-1} a \right]$, then
\begin{align}\label{eqn-last-given-mid}
\BB P\left(E_{c,b}(\zeta) \,|\, E_{c,\wt b}(\ep) \cap F_{c,b} \right)  \geq  1- q  \quad \op{and} \quad \BB P\left(E_{c,\wt b}(\zeta) \,|\, E_{c,a}(\zeta) \cap F_{c,\wt b}   \right)  \geq  1- q   . 
\end{align} 
Henceforth fix $\zeta \in (0,\zeta_0]$. 

Fix $\alpha \in (0,1)$ to be chosen later (depending on $q, q_0 $, and $\zeta$). By Lemma~\ref{prop-J^H<J^C-ratio}, we can find $\zeta'  \in (0, \zeta]$ and $a_0 > 0$ (depending on $\alpha$, $q_0$, and $\zeta$) and $m_*   \geq \wt m_*$ (depending on $\alpha$, $q_0$, $\zeta'$, and $\zeta$) for which the following holds. If $a\geq a_0$, $m_{c,a} \geq m_*$, and $   \BB P\left(E_{c,a} (\zeta) \,|\, F_{c,a}  \right)  \geq q_0$, then 
\eqb \label{eqn-ratio-compare}
   \BB P\left(  F_{c,b}  \,|\,  E_{c,\wt b}(\zeta')^c \cap F_{c,\wt b} \right)    \leq  \alpha \BB P\left( F_{c,b}  \,|\,  E_{c,\wt b}(\zeta') \right) .
\eqe
In this case we have
\begin{align}
&\BB P\left(E_{c,b} (\zeta) \,|\, F_{c,b}  \right) \\
&\qquad\geq \frac{ \BB P\left(E_{c,b} (\zeta) \,|\, E_{c,\wt b }(\zeta')  \right)  \BB P\left(E_{c,\wt b}(\zeta') \,|\, F_{c,\wt b}\right)     }{ \BB P\left(F_{c,b}   \,|\, E_{c,\wt b }(\zeta')  \right)  \BB P\left(E_{c,\wt b}(\zeta') \,|\, F_{c,\wt b}\right)  + \BB P\left(F_{c,b}  \,|\, E_{c,\wt b }(\zeta')^c \cap F_{c,\wt b}  \right)  \BB P\left(E_{c,\wt b}(\zeta')^c \,|\, F_{c,\wt b}\right)        } \notag \\
&\qquad\geq \frac{ \BB P\left(E_{c,b} (\zeta) \,|\, E_{c,\wt b }(\zeta')  \right)}{    \BB P\left(F_{c,b}   \,|\, E_{c,\wt b }(\zeta')  \right)   } \times \frac{  \BB P\left(E_{c,\wt b}(\zeta') \,|\, F_{c,\wt b}\right)     }{   \BB P\left(E_{c,\wt b}(\zeta') \,|\, F_{c,\wt b}\right)  +  \alpha \BB P\left(E_{c,\wt b}(\zeta')^c \,|\, F_{c,\wt b}\right)        } .
\label{eqn-a-induct1}
\end{align}
By~\eqref{eqn-last-given-mid},
\alb
\frac{ \BB P\left(E_{c,b} (\zeta) \,|\, E_{c,\wt b }(\zeta')  \right)}{    \BB P\left(F_{c,b}   \,|\, E_{c,\wt b }(\zeta')  \right)   }
  = \BB P\left( E_{c,b} (\zeta)  \,|\,  E_{c,\wt b }(\zeta') \cap F_{c,b} \right) 
  \geq 1-  q .
\ale
Furthermore,
\eqb \label{eqn-a-lower}
\BB P\left(E_{c,\wt b}(\zeta') \,|\, F_{c,\wt b}\right) \geq \BB P\left(E_{c,\wt b}(\zeta') \,|\, E_{c,a}(\zeta)\cap F_{c,\wt b}\right) \BB P\left( E_{c,a}(\zeta)  \,|\,  F_{c,\wt b} \right) \geq (1- q ) \BB P\left( E_{c,a}(\zeta)  \,|\,  F_{c,\wt b} \right).
\eqe 
By Bayes' rule,
\begin{align}
 \BB P\left( E_{c,a}(\zeta)  \,|\,  F_{c,\wt b} \right) \notag \\
  \geq  \frac{  \BB P\left( F_{c,\wt b } \,|\,  E_{c,a}(\zeta)  \right) \BB P\left( E_{c,a}(\zeta)  \,|\,  F_{c,a} \right)       }{\BB P\left( F_{c,\wt b }   \,|\,  E_{c,a}(\zeta)  \right) \BB P\left( E_{c,a}(\zeta)  \,|\,  F_{c,a} \right)  + \BB P\left( F_{c,\wt b }  \,|\,  E_{c,a}(\zeta)^c \cap F_{c,a} \right)     }  .  
\end{align} 
By \cite[Theorem 2.5]{shef-burger} and our assumption on $\BB P\left( E_{c,a}(\zeta)  \,|\,  F_{c,a} \right)$, this quantity is at least a positive constant $c $ depending on $q_0$, $\lambda$ and $\zeta$ (but not on $\zeta'$). Therefore,~\eqref{eqn-a-lower} implies $\BB P\left(E_{c,\wt b}(\zeta') \,|\, F_{c,\wt b}\right) \geq (1-  q ) c $, so~\eqref{eqn-a-induct1} implies
\eqbn
\BB P\left(E_{c,b'} (\zeta) \,|\, F_{c,b'}  \right)  \geq \frac{(1-  q)^2 c }{(1-  q )c  + \alpha } .
\eqen
If we choose $\alpha$ sufficiently small relative to $c $ (and hence $\zeta'$ sufficiently small and $m_*$ sufficiently large), we can make this quantity as close to $1-q$ as we like. 
\end{proof}

In order to deal with the case of very small values of $c$, we will need the following elementary lemma.

\begin{lem} \label{prop-J^H-no-tube}
There is a constant $C > 0$, depending only on $p$, such that for each $m\in\BB N$ and each $\zeta >0$, 
\eqbn
\BB P\left( \mcl N_{\tb C}\left(X(-J_r^H , -1)\right) \leq \zeta m ,\,  \forall r \in [1,m]_{\BB Z} \right) \preceq e^{-C/\zeta}  
\eqen
with the implicit constant depending only on $p$. 
\end{lem}
\begin{proof}
Let $\BB k_\zeta := \lfloor 1/\zeta \rfloor$. For $m \in\BB N$ and $k\in [1,\BB k_\zeta]_{\BB Z}$, let $r_m^k(\zeta) := \lfloor k \zeta m\rfloor$ and let
\eqbn
E_m^k(\zeta) := \left\{  L_{\lfloor k \zeta m\rfloor}^H - L_{\lfloor (k-1) \zeta m\rfloor + 1}^H \geq \zeta m   \right\} ,
\eqen
where here we use the convention that $L_0^H = 0$. The events $E_m^k(\zeta)$ are independent. By~\cite[Lemma 2.1]{gms-burger-cone}, we can find $q \in (0,1)$, depending only on $p$, such that 
\eqbn
\BB P\left( E_m^k(\zeta)  \right) \geq q ,\quad \forall m\in\BB N ,\quad \forall \zeta>0 ,\quad \forall k \in [1,\BB k_\zeta]_{\BB Z} .
\eqen
On the event $E_m^k(\zeta)$, we have
\eqbn
\mcl N_{\tb C}\left(X(-\lfloor k \zeta m\rfloor , -1)\right) \geq \zeta m .
\eqen
Hence
\eqbn
\BB P\left( \mcl N_{\tb C}\left(X(-J_r^H , -1)\right) \leq \zeta m ,\,  \forall r \in [1,m]_{\BB Z} \right)  \leq \BB P\left(\bigcap_{k=1}^{\BB k_\zeta} E_m^k(\zeta) \right) \leq (1-q)^{\BB k_\zeta} .
\eqen
\end{proof}

\begin{proof}[Proof of Lemma~\ref{prop-J^H<J^C-reg}]
By Lemma~\ref{prop-J^H<J^C-induct}, for each $q \in (0,1/2)$ we can find $\zeta_0 >0$ and $a_0 > 0$ such that for each $\zeta \in (0,\zeta_0]$, there exists $\wt m_* = \wt m_*(q,  \zeta) \in \BB N$ with the following property. Suppose $c  \in \BB N$ and $a\geq a_0$ with $m_{c,a} \geq \wt m_*$ and $\BB P\left(E_{c,a} (\zeta) \,|\, F_{c,a}  \right) \geq 1/2$. Then for each $b \in \left[(1-\lambda)^{-1} a , \lambda^{-1}  a \right]$ we have $\BB P\left(E_{c,b} (\zeta) \,|\, F_{c,b}  \right) \geq  1-q$. By \cite[Theorem 2.5]{shef-burger}, we can find $\zeta  \in (0,\zeta_0]$ (depending on $a_0$ and $q$) such that for every $c\in \BB N$ and $a \in [1,a_0]$ we have $\BB P\left(E_{c,a} (\zeta ) \,|\, F_{c,a}  \right) \geq 1-q$. By induction, it follows that for every $c \geq \wt m_*$ and every $a \geq 1$ we have $\BB P\left(E_{c,a} (\zeta ) \,|\, F_{c,a}  \right) \geq  1- q$. Hence for $h\geq c \geq \wt m_* $, we have
\eqbn
\BB P\left( \mcl N_{\tb C}\left(X(-J_h^H , -1)\right) \geq \zeta  h  \,|\,  J_h^H < J_c^C \right) \geq 1-q .
\eqen

We still need to consider the case where $c$ is smaller than $\wt m_*$. For $h\in \BB N$, let $M_h$ be the smallest $m\in\BB N$ for which $\mcl N_{\tb C}\left(X(-J_m^H , -1)\right) \geq h^{1/3}$. By Lemma~\ref{prop-J^H-no-tube}, 
\eqbn
\BB P\left(M_h < h^{1/2} \right) = o_h^\infty(h) ,
\eqen
at a rate depending only on $p$. Hence for any $c\in\BB N$, 
\alb
 \BB P\left(M_h < h^{1/2} \,|\, J_h^H < J_c^C \right) \leq o_h^\infty(h) \BB P\left(J_h^H < J_1^C \right)^{-1} .
\ale
By, e.g., Lemma~\ref{prop-I-reg-var}, we have that $\BB P\left(J_h^H < J_1^C\right)$ is bounded below by some power of $h$. It follows that we can find $m_* \geq \wt m_*$ such that $m_*^{1/3} \geq \wt m_*$ and for $h\geq m_*$ and $c\in\BB N$, we have
\eqbn
\BB P\left(   M_h < h^{1/2}  \,|\, J_h^H < J_c^C \right) \geq  1 - q .
\eqen
Let $x$ be a realization of $X_{-J_{M_h}^H} \dots X_{-1}$ for which $M_h < h^{1/2}$ and $J_c^C > J_{M_h}^H$. The conditional law of $X_{-J_h^H} \dots X_{-J_{M_h}^H-1}$ given $\{ X_{-J_{M_h}^H} \dots X_{-1}\}$ and $\{J_h^H < J_c^C\}$ is the same as its conditional law given $\{ X_{-J_{M_h}^H} \dots X_{-1}\}$ and the event that $X(-J_h^H , -J_{M_h}^H-1)$ contains at most $c - d^*(\mcl R(x)) \geq h^{1/3} \geq \wt m_*$ cheeseburger orders. By the $c \geq \wt m_*$ case, we infer that there exists $\zeta>0$ such that for each $h\geq m_*$ and $c \leq h$, we have 
\eqbn
\BB P\left( \mcl N_{\tb C}\left(X(-J_h^H , -J_{M_h}^H -1)\right) \geq \zeta h    \,|\, M_h < h^{1/2} ,\, J_h^H < J_c^C \right) \geq 1-q .
\eqen
Hence for such an $h$ we have
\eqbn
\BB P\left( \mcl N_{\tb C}\left(X(-J_h^H , -1)\right) \geq \zeta h    \,|\,  J_h^H < J_c^C \right) \geq (1-q )^2 .
\eqen
Since $q$ was arbitrary, this completes the proof.
\end{proof}

\subsection{Proof of Proposition~\ref{prop-wt-pi-exists}}

To deduce Proposition~\ref{prop-wt-pi-exists} from the results of the preceding subsections, we will need two further lemmas.

\begin{lem} \label{prop-J^H-determined}
For $n\in\BB N$, let $\pi_n = \pi_n(\ep_1)$ be as in Definition~\ref{def-end-event}.  
Let $x = x_{-|x|} \dots x_{-1}$ be a word consisting of elements of $\Theta$ such that $\mcl R(x)$ contains no orders and with positive probability we have $ X_{-|x|} \dots X_{-1} = x$ and $\pi_n < |x|$. Then $\pi_n$ is determined by $x$ on this event, i.e.\ there exists $\frk j(x) \in [1,|x|]_{\BB Z}$ such that whenever  $ X_{-|x|} \dots X_{-1} = x$ and $\pi_n < |x|$, it holds that $\pi_n = \frk j(x)$. 
\end{lem}
\begin{proof}
Given $x$ as in the statement of the lemma, let $\frk j(x)$ be the smallest $j \in [1,|x|]_{\BB Z}$ such that $j \geq n - \ep_1 n^{\xi/2}$, there are at least $\ep_1  n^{\xi/2}$ hamburger orders to the left of the leftmost flexible order in $\mcl R(x_{-j} \dots x_{-1})$, and each cheeseburger in $\mcl R(x_{-|x| } \dots x_{-j-1})$ is matched to a cheeseburger order in $\mcl R(x_{-j} \dots x_{-1})$. Such a $j$ exists since we are assuming that for some $k\in\BB N$ it holds with positive probability we have $ X_{-\pi_n - k} \dots X_{-1} = x$. We claim that for each $k\in\BB N$ for which this is the case, we have $\frk j(x) = \pi_n$ on the event $\{X_{-\pi_n - k} \dots X_{-1} = x\}$. 

Indeed, suppose $X_{-\pi_n - k} \dots X_{-1} = x$. By definition, the time $\pi_n$ satisfies the defining properties of $\frk j(x)$. Since $\frk j(x)$ is the smallest time in $[1,|x|]_{\BB Z}$ satisfying these properties, it follows that $\pi_n \geq \frk j(x)$. If $\pi_n > \frk j(x)$, then by minimality of $\pi_n$ 
there is a cheeseburger in $X(-J,-\frk j(x) -1)$ which is either matched to a flexible order in $X(-\frk j(x) , -1)$ or does not have a match in $X(-\frk j(x),-1)$. By definition of $\frk j(x)$, this cheeseburger cannot appear in $X(-|x| , -\frk j(x) -1)$. Hence there is a cheeseburger in $X(-J,-|x|-1)$ which either does not have a match in $X(-|x| , -1)$ or is matched to a flexible order in $X(-\frk j(x) ,-1)$. Since $\pi_n \geq \frk j(x)$, this cheeseburger appears in $X(-J,-\pi_n-1)$, which contradicts the definition of $\pi_n$.
\end{proof}

\begin{lem} \label{prop-J^H-after-pi}
Let $J_{n,0}^H = \pi_n$ and for $m\in\BB N$ let $J_{n,m}^H$ be the smallest $j \geq \pi_n$ for which $X(-j,-\pi_n-1)$ contains $m$ hamburgers. Let $x=x_{-|x|}\dots x_{-1}$ be a realization of $X(-J_{n,m}^H , -1)$ for which $J>J_{n,m}^H$ (i.e.\ $\mcl R(x)$ contains no burgers). Let $\frk o(x)$ be as in~\eqref{eqn-theta-count-reduced}. The conditional law of $X_{-J} \dots X_{-|x|-1}$ given $\{X_{-J_{n,m}^H} \dots X_{-1}= x\} $ is the same as the conditional law of $X_{-J_{\frk o(x)}^H} \dots X_{-1}$ given the event $F_{\frk o(x)}(x)$ of Section~\ref{sec-F-event-mono}.
\end{lem}
\begin{proof}
For $m\in\BB N$, let $\wt J_{x,m}^H$ be the smallest $j\geq |x|+1$ for which $X(-j , -|x|-1)$ contains $m$ hamburgers. Also let $\frk j(x)$ be as in Lemma~\ref{prop-J^H-determined}, so that on the event $\{X_{-J_{n,m}^H} \dots X_{-1}= x\} $, we have $\pi_n = \frk j(x)$. 
 
The event $\{X_{- J_{n,m}^H} \dots X_{-1} = x  \}$ is the same as the event that $\{X_{-|x|} \dots X_{-1} = x\}$ and each $j \in [|x|+1 , J]_{\BB Z}$ such that $X_{-j} = \tc C$ and $-\phi(-j) \leq \frk j(x)$ is such that $X_{-j}$ is matched to a cheeseburger in $X(-\pi_n,-1) = \mcl R(x_{-\frk j(x)} \dots x_{-1})$. Since $X(- J_{n,m}^H , -\pi_n-1)$ contains no flexible orders, this is equivalent to the condition that $\{X_{-|x|} \dots X_{-1} = x\}$ and each cheeseburger in $X(-J,-|x|-1)$ is matched to a cheeseburger order in $X(-|x|,-1)$. If this is the case, then every hamburger order and every flexible order in $X(-|x|,-1)$ is matched to a hamburger in $X(-J,-|x|-1)$, so $J = \wt J_{x,\frk o(x)}^H$.

Conversely, if $X_{-|x|} \dots X_{-1} = x $ and every cheeseburger in $X(-\wt J_{x,\frk o(x)}^H,-|x|-1)$ matched to a cheeseburger order in $X(-|x|,-1)$, then the rightmost $\frk o(x)-1$ hamburgers in $X(-\wt J_{x,\frk o(x)}^H,-|x|-1)$ must be matched to the $\frk o(x)-1$ hamburger orders and flexible orders in $X(-|x|,-1)$. The leftmost hamburger in $X(-\wt J_{x,\frk o(x)}^H,-|x|-1)$ then has no match in $X(-|x|,-1)$ so we have $\wt J_{x,\frk o(x)}^H = J$. 

Thus, $\{X_{- J_{n,m}^H} \dots X_{-1} = x  \}$ is the same as the event that $\{X_{-|x|} \dots X_{-1} = x\}$ and every cheeseburger in $X(-\wt J_{x,\frk o(x)}^H,-|x|-1)$ matched to a cheeseburger order in $X(-|x|,-1)$. By translation invariance, this implies the statement of the lemma. 
\end{proof}

\begin{proof}[Proof of Proposition~\ref{prop-wt-pi-exists}]  
For $n,m\in\BB N$, let $J_{n,m}^H$ be the smallest $j\geq \pi_n$ such that $X(-j,-\pi_n-1)$ contains $m$ hamburgers. Also let $L_{n,m}^H := d^*\left(X(-J_{n,m}^H ,-\pi_n -1)\right)$. 

For $n  \in\BB N$ let $m_n^0 = 0$ and for $k\in\BB N$, inductively define 
\[
m_n^k := \lfloor 2^{ - k-1} (\ep_1 - \ep_2) \ep_2 n^{\xi/2} \rfloor   + m_n^{k-1}  .
\]
Also let $\BB k_n$ be the largest $k\in\BB N$ for which $2^{-k-1} (\ep_1 - \ep_2)\ep_2 n^{\xi/2} \geq n^{\xi/4}$. 

Fix $\delta \in (0,1/2)$ to be chosen later. For $k\in [0 ,\BB k_n]_{\BB Z}$, let 
\eqbn
E_n^k := \left\{ \frk c_f\left( X(-J_{n,m_n^k}^H ,- 1 )  \right) \geq \delta  2^{-k}  n^{\xi/2}   \right\} .
\eqen
Let $K_n$ be the minimum of $\BB k_n$ and the smallest $k\in [0,\BB k_n]_{\BB Z}$ for which $E_n^k$ occurs
and let 
\[
\wt\pi_n := J_{n,m_n^{K_n}}^H .
\]

It is clear that $\wt \pi_n \geq \pi_n$, $\mcl N_{\tb F}\left(X(-\wt\pi_n,-\pi_n)\right) = 0$, and $\mcl N_{\tb H}\left(X(-\wt\pi_n ,-1) \right) \leq \mcl N_{\tb H}\left(X(-\pi_n,-1)\right)$. Furthermore, 
\eqbn
m_n^{K_n} \leq \sum_{k=1}^{\BB k_n} 2^{ - k-1} \ep_2 n^{\xi/2} \leq \ep_2 n^{\xi/2} .
\eqen
Hence if $\pi_n < J$, we have
\eqbn
\mcl N_{\tb H}\left(X(-\wt\pi_n, -1)\right) \geq \mcl N_{\tb H}\left(X(- \pi_n, -1)\right)  -   \ep_2 n^{\xi/2}   .
\eqen
By definition of $\pi_n$, on the event $\{\pi_n < J\}$, there are at least $\ep_1 n^{\xi/2}$ hamburger orders to the left of the leftmost flexible order in $X(-\pi_n,-1)$. Since $X(-\wt\pi_n,-\pi_n-1)$ has at most $\ep_2 n^{\xi/2} \leq \ep_1 n^{\xi/2}$ hamburgers, it follows that this flexible order does not have a match in $X(-\wt\pi_n,-\pi_n-1)$. Therefore, 
\eqbn
\mcl N_{\tc C}\left(X(-\wt\pi_n,-\pi_n-1)\right) \leq \frk c_f\left(X(-\pi_n,-1)\right) .
\eqen
If in fact $\frk c_f\left(X(-\pi_n,-1)\right) \geq n^{\xi/2}/2$, we have $\wt\pi_n = \pi_n$. Otherwise, we have 
\eqbn
\mcl N_{\tb C}\left(X(-\wt\pi_n, -1)\right) \geq   \mcl N_{\tb C}\left(X(- \pi_n, -1)\right)  -     \mcl N_{\tc C}\left(X(-\wt\pi_n,-\pi_n-1)\right) \geq  \mcl N_{\tb C}\left(X(- \pi_n, -1)\right)  - \frac12 n^{\xi/2} .
\eqen
This verifies conditions~\ref{item-wt-pi-mono} and~\ref{item-wt-pi-compare} in the statement of the lemma. 

To obtain condition~\ref{item-wt-pi-cond}, let $x$ be a realization of $X_{-\wt\pi_n} \dots X_{-1}$. By Lemma~\ref{prop-J^H-determined}, the time $\pi_n$ is determined by the word $X_{-\wt\pi_n} \dots X_{-1}$, so also $K_n$ is determined by $X_{-\wt\pi_n} \dots X_{-1}$ (since $m_n^{K_n}$ is the number of hamburgers in $X(-\wt\pi_n,-\pi_n-1)$). Let $\frk k$ be the realization of $K_n$ corresponding to our given word $x$. 
Since $K_n$ is the minimum of $\BB k_n$ and the \emph{smallest} $k\in [0,\BB k_n]_{\BB Z}$ for which $E_n^k$ occurs, the event $\{X_{-\wt\pi_n} \dots X_{-1} = x\}$ is the same as the event that $\{X_{-J_{n,m_n^{\frk k} }^H} \dots X_{-1} = x\}$. Condition~\ref{item-wt-pi-cond} now follows from Lemma~\ref{prop-J^H-after-pi}. 

It remains to prove condition~\ref{item-wt-pi-C}.
Let $k\in [0,\BB k_n-1]_{\BB Z}$ and let $x^k$ be a realization of $X_{-J_{n,m_n^k}^H} \dots X_{- 1}$ such that the number of hamburger orders to the left of the leftmost flexible order in $\mcl R(x^k)$ is at least $m_n^{k+1} - m_n^k$. By definition, there are at least $\ep_1 n^{\xi/2}$ hamburger orders to the left of the leftmost flexible order in $X(-\pi_n,-1)$ on the event $\{\pi_n < J\}$ and we always have $\ep_1 n^{\xi/2} - m_n^k  < m_n^{k+1} - m_n^k$. Therefore if $\pi_n  <J$, the realization of $X_{-J_{n,m_n^k}^H} \dots X_{- 1}$ must satisfy the above condition. 
 
By Lemma~\ref{prop-J^H-after-pi}, the conditional law of $X_{-J} \dots X_{-J_{n,m_n^k}^H-1}$ given $\{X_{-J_{n,m_n^k}^H} \dots X_{- 1} = x^k\}$ is the same as the conditional law of $X_{-J_{\frk o(x^k)}^H} \dots X_{-1}$ given the event $F_{\frk o(x^k)}(x^k)$ of Section~\ref{sec-F-event-mono}. By Lemma~\ref{prop-J^H-F-mono}, for each $l\in\BB Z$ we have 
\eqb \label{eqn-m^k-compare}
\BB P\left( L_{n,m_n^{k+1}}^H - L_{n,m_n^k}^H    \leq l \,|\,    X_{-J_{n,m_n^k}^H} \dots X_{- 1} = x^k     \right) \geq \BB P\left( L_{m_n^{k+1} - m_n^k}^H  \leq l  \,|\,   J_{m_n^{k+1} - m_n^k}^H  < J_{\frk c_f(x^k)+1}^C   \right)  . 
\eqe 
By Lemma~\ref{prop-J^H<J^C-reg}, for each $q\in (0,1)$ we can find $\delta  >0$ and $n_* \in\BB N$ (depending only on $\ep_1$ and $\ep_2$) such that for any $n \geq  n_*$ and any choice of realization $x^k$ as above, we have
\eqb \label{eqn-m^k-delta}
\BB P\left( \mcl N_{\tb C}\left( X(-J_{m_n^{k+1} - m_n^k}^H  , -1\right)  \geq \delta 2^{-k}  n^{\xi/2}    \,|\,\mcl N_{\tc C}\left( X(-J_{m_n^{k+1} - m_n^k}^H  , -1\right)  \leq \frk c_f(x^k) \right) \geq 1-q. 
\eqe 
If $\mcl N_{\tb C}\left( X(-J_{m_n^{k+1} - m_n^k}^H  , -1\right)  \geq \delta  2^{-k}  n^{\xi/2} $ and $\mcl N_{\tc C}\left( X(-J_{m_n^{k+1} - m_n^k}^H  , -1\right)  \leq \frk c_f(x^k)$, then $L_{m_n^{k+1} - m_n^k}^H \leq    \frk c_f(x^k) - \delta  2^{-k}  n^{\xi/2} $.
By~\eqref{eqn-m^k-compare} and~\eqref{eqn-m^k-delta}, 
\eqbn
\BB P\left( L_{n,m_n^{k+1}}^H - L_{n,m_n^k}^H    \leq  \frk c_f(x^k) - \delta   2^{-k}  n^{\xi/2}  \,|\,    X_{-J_{n,m_n^k}^H} \dots X_{- 1} = x^k     \right) \geq 1-q . 
\eqen
If $L_{n,m_n^{k+1}}^H - L_{n,m_n^k}^H    \leq  \frk c_f(x^k) - \delta  2^{-k}  n^{\xi/2}$ and $ X_{-J_{n,m_n^k}^H} \dots X_{- 1} = x^k$, then 
\eqbn
\frk c_f\left(X(-L_{n,m_n^{k+1}}^H , -1)\right) = \frk c_f(x^k) - \left(L_{n,m_n^{k+1}}^H - L_{n,m_n^k}^H\right)  \geq \delta  2^{-k}  n^{\xi/2} .
\eqen
Note that here we use that $m_n^k < \ep_1 n^{\xi/2}$, so that the leftmost flexible order in $\mcl R(x^k)$ coincides with the leftmost flexible order in $X(-\pi_n,-1)$. 
Hence for this choice of $\delta$, we have
\eqbn
\BB P\left( E_n^{k+1}  \,|\,    X_{-J_{n,m_n^k}^H} \dots X_{- 1} = x^k     \right) \geq 1-q .
\eqen
Since each realization $x^k$ as above encodes a realization of $X_{-\pi_n} \dots X_{-1}$, it follows that
\eqb \label{eqn-small-k_*-prob}
\BB P\left(K_n \geq k \,|\, X_{-\pi_n} \dots X_{-1} \right) \BB 1_{(\pi_n <J)} \leq q^k  ,\quad \forall k \in [1,\BB k_n]_{\BB Z} .
\eqe 

For $\zeta \geq n^{\xi/2}$, let $k_n^\zeta := \BB k_n \wedge \lfloor \log_2 \frac{ \delta  }{\zeta}  - 1\rfloor $ be the minimum of $\BB k_n$ and the smallest $k\in\BB N$ for which $\delta 2^{-k} n^{\xi/2} \geq \zeta n^{\xi/2}$. By~\eqref{eqn-small-k_*-prob} and the definition of $\wt\pi_n$, 
\eqbn
\BB P\left(\frk c_f\left(X(-\wt\pi_n ,-1) \right) \leq   \zeta n^{\xi/2} \,|\, X_{-\pi_n} \dots X_{-1}   \right) \leq q^{k_n^\zeta}  \preceq (\zeta \wedge n^{-\xi/4})^{a(q)}
\eqen
with the implicit constant independent of $\zeta$ and $n$; and $a(q) \rta \infty$ as $q\rta 0$, at a rate depending only on $\ep_1$ and $\ep_2$. By choosing $q$ sufficiently small (and hence $\delta$ sufficiently small) we can arrange that $a(q) \geq  200$, which yields condition~\ref{item-wt-pi-C}. 
\end{proof}

\section{Local estimates for the last segment of the word}
\label{sec-end-local}

In this section we will use the following notation. Let $J$ be as in~\eqref{eqn-J-def}. Let $\mu$ be as in~\eqref{eqn-cone-exponent} and fix $\nu \in (1-\mu,1/2)$ and $\xi \in (2\nu , 1)$. We treat $\nu$ and $\xi$ as fixed throughout this section and do not make dependence on these parameters explicit. 

Fix $\ep_0 , \ep_1 , \ep_2  \in (0,1)$ with $\ep_1 > \ep_2$. For $l\in \BB N$, let $\mcl E_n^l = \mcl E_n^l(\ep_1)$ be the event of Definition~\ref{def-end-event}. Also let $\pi_n = \pi_n(\ep_1)$ be as in Definition~\ref{def-end-event} and let $\wt\pi_n = \wt\pi_n(\ep_1,\ep_2)$ be as in Proposition~\ref{prop-wt-pi-exists}.  
We will also need to consider the following event, which is illustrated in Figure~\ref{fig-box-event}. 
 
\begin{defn} \label{def-B_n^delta}
For $n\in\BB N$ and $\delta>0$, let $m_n^\delta := \lfloor (1-\delta) n\rfloor$. For $C>1$, let $\mcl B_n^\delta(C)$ be the event that the following is true:
\alb
&J > m_n^\delta  ,\quad \mcl N_{\tb F}\left(X(-m_n^\delta,-1)\right) \leq n^\nu,\quad\op{and} \\
&\mcl N_\theta\left(X(-m_n^\delta,-1)\right)  \in  \left[C^{-1} (\delta n)^{1/2} , C (\delta n)^{1/2} \right]_{\BB Z}     , \quad   \forall \theta \in \left\{\tb H,\tb C\right\}  .
\ale
\end{defn}

Recall the definition of the path $Z^n = (U^n,V^n)$ from~\eqref{eqn-Z^n-def}.  
Roughly speaking, $\mcl B_n^\delta(C)$ is the event that $Z^n(1-\delta)$ is where we expect it to be when we condition on $\mcl E_n^l$. The reader should note the similarity between the event of Definition~\ref{def-B_n^delta} and the event $\mcl B^\delta(C)$ of~\eqref{eqn-bm-B-def}.

\begin{figure}[ht!]
 \begin{center}
\includegraphics{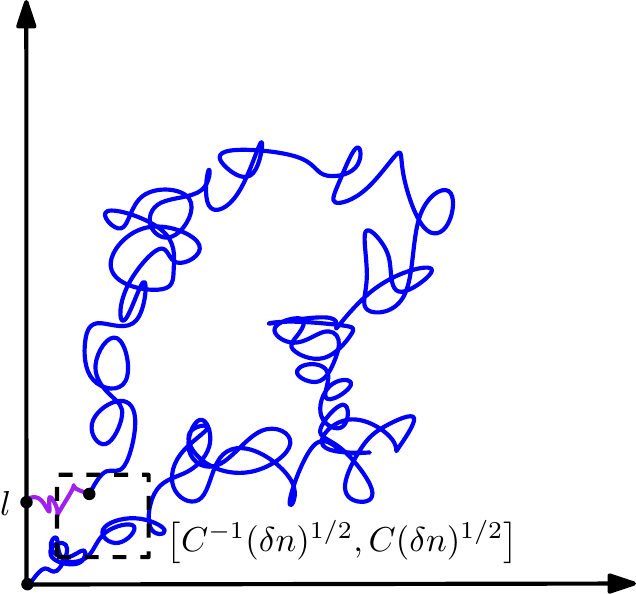} 
\caption{An illustration of the path $D$ of~\eqref{eqn-discrete-path} on the event $\mcl B_n^\delta(C) \cap \mcl E_n^l$ for $l\in \left[\ep_0 n^{\xi/2} , \ep_0^{-1} n^{\xi/2} \right]_{\BB Z}$, as defined in Definitions~\ref{def-end-event} and~\ref{def-B_n^delta}. By Lemma~\ref{prop-E^l-abs-cont}, the conditional law of the blue part of the path given $\mcl B_n^\delta(C) \cap \mcl E_n^l$ is mutually absolutely continuous with respect to its conditional law given only $\mcl B_n^\delta(C)$. By \cite[Theorem 4.1]{gms-burger-cone}, this latter conditional law converges as $n\rta\infty$ ($\delta$ fixed) to the conditional law of a correlated Brownian motion $\wh Z = (\wh U , \wh V)$ conditioned to stay in the first quadrant until time $1-\delta$ and conditioned on the event $\mcl B^\delta(C)$ of~\eqref{eqn-bm-B-def} in Section~\ref{sec-bm-cond}, defined with $1-\delta$ in place of 1. 
}\label{fig-box-event}
\end{center}
\end{figure}
 
The first main result of this subsection is a description of the conditional law of $X_{-m_n^\delta} \dots X_{-1}$ given $\mcl B_n^\delta(C) \cap \mcl E_n^l$ for $l\in \left[\ep_0 n^\xi , \ep_0^{-1} n^\xi\right]$, which will allow us to compare this law to the law of a correlated Brownian motion conditioned to stay in the first quadrant using \cite[Theorem 4.1]{gms-burger-cone}.

\begin{prop} \label{prop-E^l-abs-cont}
Fix $C>1$. For $\delta \in (0,1/2)$ and $n\in\BB N$, let $m_n^\delta $ and let $\mcl B_n^\delta(C)$ be as in Definition~\ref{def-B_n^delta}. 
For each $\delta \in (0,1/2)$, there exists $n_* = n_*(\delta,C,\ep_0,\ep_1)$ such that for each $n\geq n_*$ and each $l\in \left[\ep_0 n^{\xi/2} , \ep_0^{-1} n^{\xi/2}\right]_{\BB Z}$, the conditional law of $X_{-m_n^\delta} \dots X_{-1}$ given $\mcl E_n^l \cap \mcl B_n^\delta(C)$ is mutually absolutely continuous with respect to its conditional law given only $\mcl B_n^\delta( C) $, with Radon-Nikodym derivative bounded above and below by positive constants depending only on $C$, $\ep_0$, and $\ep_1$. 
\end{prop}
 
To complement Lemma~\ref{prop-end-box}, we have a result which tells us that if $C$ is large, then the event $\mcl B_n^\delta(C)$ is likely to occur when we condition on $\mcl E_n^l$. 
  
\begin{prop} \label{prop-end-box}
For each $q\in (0,1)$, there is a $C = C(q,\ep_0,\ep_1) >1$ such that the following is true.  
For each $\delta \in (0,1/2)$, there exists $n_* = n_*(\delta , C,q,\ep_0 , \ep_1) \in\BB N$ such that for each $n\geq n_*$ and each $l\in \left[\ep_0 n^{\xi/2} , \ep_0^{-1} n^{\xi/2}\right]_{\BB Z}$, we have 
\eqbn
\BB P\left(\mcl B_n^\delta(C)      \,\big|\, \mcl E_n^l    \right) \geq 1-q. 
\eqen
\end{prop}

The following result is needed to prove tightness of the conditional law of $Z^n|_{[0,2]}$ given $\{X(1,2n)=\emptyset\}$, plus absolute continuity with respect to Lebesgue measure of the law of a subsequential limiting path evaluated at a fixed time, in the next subsection.
 
\begin{prop} \label{prop-E^l-tight}
Define $Z^n = (U^n,V^n)$ as in~\eqref{eqn-Z^n-def}. For $\alpha>0$ and $n\in\BB N$, let $\wt G_n(\alpha,\zeta)$ be the event that the following is true. For each $s,t\in [0,1]$ with $|s-t| \leq \zeta$, we have $|Z^n(-s) - Z^n(-t)| \leq \alpha$. For each $\alpha>0$ and $q\in (0,1)$, there exists $\zeta = \zeta(q,\alpha,\ep_0,\ep_1) >0$ and $n_*  = n_*(\alpha,\zeta,\ep_0,\ep_1) \in\BB N$ such that for $n\geq n_*$ and $l\in \left[\ep_0 n^{\xi/2} , \ep_0^{-1} n^{\xi/2}\right]_{\BB Z}$, we have
\eqbn
\BB P\left(\wt G_n(\alpha,\zeta)\,|\, \mcl E_n^l\right) \geq 1- q.
\eqen
Furthermore, for each $t\in (0,1)$, each $q\in (0,1)$, and each set $A\subset [0,\infty)^2$ with zero Lebesgue measure, there exists $\beta = \beta(t,q,A,\ep_0,\ep_1) > 0$ and $n_* = n_*(  t,q , A , \ep_0,\ep_1) \in\BB N $ such that for each $n\geq n_*$ and $l\in \left[\ep_0 n^{\xi/2} , \ep_0^{-1} n^{\xi/2}\right]_{\BB Z}$, we have
\begin{align} \label{eqn-E^l-no-hit}
 \BB P\left(  \op{dist}( Z^n(t) , A )  < \beta  \,|\, \mcl E_n^l \right) \leq  q   .
\end{align}
\end{prop}

The idea of the proofs of Propositions~\ref{prop-E^l-abs-cont},~\ref{prop-end-box}, and~\ref{prop-E^l-tight} is to estimate $\BB P\left(\mcl E_n^l \,|\, X_{-m_n^\delta } \dots X_{-1}\right)$ when $l\in  [\ep_0 n^{\xi/2} , \ep_0^{-1} n^{\xi/2}]_{\BB Z}$ using the results of~\cite{gms-burger-local} and Section~\ref{sec-pi-reg}; and to use such estimates, together with Bayes' rule, to obtain statements about the conditional law of $X_{-m_n^\delta}\dots X_{-1}$ given $\mcl E_n^l$. 
 
In Section~\ref{sec-reg-var}, we will recall some basic facts about regularly varying functions and some results from \cite{gms-burger-cone,gms-burger-local} regarding regular variation of certain quantities associated with the word $X$. In Section~\ref{sec-end-lower}, we will prove a lower bound for the conditional probability of $\mcl E_n^l$ given a realization of $X_{-m_n^\delta} \dots X_{-1}$ for which $\mcl B_n^\delta(C)$ occurs. In Section~\ref{sec-end-upper}, we will prove our first upper bound for the conditional probability of $\mcl E_n^l$ given $X_{-N_n} \dots X_{-1}$, where $N_n$ is a stopping time for the word $X$, read backward. In Section~\ref{sec-end-upper-full}, we will use the upper bound of Section~\ref{sec-end-upper} to deduce two sharper upper bounds for the same probability. In Section~\ref{sec-end-reg}, we will prove a regularity estimate for the last segment of the word when we condition on $\mcl E_n^l$ and a realization of $X_{-m_n^\delta}\dots X_{-1}$. In Section~\ref{sec-end-properties}, we will combine the estimates of the earlier subsections to prove Propositions~\ref{prop-E^l-abs-cont},~\ref{prop-end-box}, and~\ref{prop-E^l-tight}. The reader may wish to read Section~\ref{sec-end-properties} before the other sections to get an idea for how all of the results of this section fit together.

\subsection{Regular variation} 
\label{sec-reg-var}

Here we recall some basic facts about regularly varying functions and some results from~\cite{gms-burger-cone,gms-burger-local} which we will need in the sequel.  
 
Let $\alpha > 0$. Recall that a function $f : [1,\infty) \rta (0,\infty)$ is called \emph{regularly varying of exponent $\alpha$} if for each $\lambda >0$, 
\eqbn
\lim_{t\rta\infty} \frac{f(\lambda  t)}{f(t)} =\lambda^{-\alpha} .
\eqen
A function $f$ is called \emph{slowly varying} if $f$ is regularly varying of exponent 0. A function $f$ is regularly varying of exponent $\alpha$ if and only if $f(t) = \psi(t) t^{-\alpha}$, where $\psi$ is slowly varying.  
See~\cite{reg-var-book} for more on regularly varying functions.

We recall the definition of the exponent $\mu$ from~\eqref{eqn-cone-exponent}. The following lemma is taken from \cite[Section A.2]{gms-burger-cone}.  

\begin{lem} \label{prop-I-reg-var}
Let $I$ be the smallest $i\in \BB N$ for which $X(1,i)$ contains an order. Then the law of $I$ is regularly varying with exponent $\mu$, i.e.\ there exists a slowly varying function $\psi_0 : [0,\infty) \rta (0,\infty)$ such that
\eqbn
\BB P\left(I > n \right) = \psi_0(n) n^{-\mu} ,\qquad \forall n\in\BB N .
\eqen
\end{lem}

The following is \cite[Lemma 4.1]{gms-burger-local}. 

\begin{lem} \label{prop-interval-reg-var}
Let $\psi_0$ be the slowly varying function from Lemma~\ref{prop-I-reg-var}. There is a slowly varying function $\psi_2$ such that for $n\in\BB N$ and $k \in [1 ,  n]_{\BB Z}$, we have
\eqbn
\BB P\left(\text{$\exists \, j \in [n- k ,n  ]_{\BB Z}$ such that $X(-j,-1)$ contains no orders}\right) = (1+o_{k/n}(1)) \psi_0(n) \psi_2(k)  (k/n)^\mu .
\eqen 
\end{lem} 

We will make use of the functions $\psi_0$ and $\psi_2$ throughout the remainder of this section. The reason why we use the notation $\psi_0$ and $\psi_2$ (and not $\psi_1$) is to be consistent with~\cite{gms-burger-local} (which also includes another slowly varying function $\psi_1$ which is not needed for the present paper).

\subsection{Lower bound}
\label{sec-end-lower}

In this subsection we will prove the following lower bound for the conditional probability of $\mcl E_n^l$ given a realization of $X_{-m_n^\delta} \dots X_{-1}$ for which $\mcl B_n^\delta(C)$ occurs. 

\begin{prop} \label{prop-end-lower}
For $n\in\BB N$, $\delta>0$ and $C>1$ let $m_n^\delta$ and $\mcl B_n^\delta(C)$ be as in Definition~\ref{def-B_n^delta}. Also let $\delta \in (0,1/2)$ and let $x$ be a realization of $X_{-m_n^\delta} \dots X_{-1}$ for which $\mcl B_n^\delta(C)$ occurs. 
Let $\psi_0$ and $\psi_2$ be the slowly varying functions from Lemmas~\ref{prop-I-reg-var} and~\ref{prop-interval-reg-var}. There is an $n_* = n_*(\delta , C, \ep_0,\ep_1) \in \BB N$ such that for each $n\geq n_*$, each $l \in \left[\ep_0 n^{\xi/2} , \ep_0^{-1} n^{\xi/2} \right]_{\BB Z}$, and each realization $x$ as above we have
\eqbn
\BB P\left( \mcl E_n^l   \,|\, X_{  - m_n^\delta } \dots X_{-1} =   x   \right) \succeq  \psi_0\left(\delta n \right) \psi_2(n^\xi) \delta^{-1-\mu} n^{-1 -  \mu + \xi  (\mu-1/2) }   
\eqen
with the implicit constant depending only on $\ep_0$, $\ep_1$, and $C$. 
\end{prop}
\begin{proof}
Throughout, we assume that $n$ is large enough that $C^{-1} \delta^{1/2} n^{1/2} \geq 100 n^{\xi/2}$. Fix a realization $x$ as in the statement of the proposition and an $l \in \left[\ep_0 n^{\xi/2} , \ep_0^{-1} n^{\xi/2} \right]_{\BB Z}$. Let $\frk h(x)$ and $\frk c(x)$ be as in~\eqref{eqn-theta-count-reduced} and let 
\[
\frk h_n(x) := \frk h(x) - \lfloor 2\ep_1 n^{\xi/2} + n^\nu \rfloor .
\]
For $m\in\BB N$, let $J_{x,m}^H$ be the smallest $j \geq |x|$ for which $X(-j,-|x|-1)$ contains $m$ hamburgers and let $L_{x,m}^H := d^*\left(X(-J_{x,m}^H , -|x|-1)\right)$. Let $\wt E_n (x)$ be the event that the following is true.
\begin{enumerate}
\item $J_{x,\frk h_n(x)}^H \in \left[ n- \ep_1^{-1} n^{\xi } , n-n^{\xi  } \right]_{\BB Z}$. 
\item $L_{x,\frk h_n(x)}^H \in \left[\frk c(x) - l  - 2 n^{\xi/2} , \frk c(x)  - l -n^{\xi/2} \right]_{\BB Z}$. \label{item-lower1-L}
\item $\mcl N_{\tc C}\left(X(-J_{x,\frk h_n(x)}^H  , -1)\right) \leq \frk c(x) - l -  n^{\xi/2}$.   
\end{enumerate}
By \cite[Proposition 4.4, Assertion 2]{gms-burger-local} (applied with $h = \frk h_n(x)$, $c = \frk c(x) - l - n^{\xi/2}$, $n-|x| \asymp \delta n$ in place of $n$, $k\asymp n^\xi$, and $r\asymp n^{\xi/2}$) we can find $n_*^0 = n_*^0(\ep_0,\ep_1,C) \in \BB N$ such that for $n\geq n_*^0$, 
\eqb \label{eqn-end-lower1}
\BB P\left(\wt E_n (x) \,|\, X_{  - m_n^\delta } \dots X_{-1} =   x \right) \succeq \psi_0(\delta n) \psi_2(n^\xi) \delta^{-1-\mu} n^{-(1-\xi)(1+\mu) } ,
\eqe 
with the implicit constant depending only on $\ep_0 ,\ep_1$, and $C$. 

Let $\wt J_{x }^C$ be the smallest $j \geq J_{x,\frk h_n(x)}^H + 1$ for which $X(-j,-J_{x,\frk h_n(x)}^H-1)$ contains $\lfloor (1/4) n^{\xi/2} \rfloor$ cheeseburgers. 
Let $\wt F_n(x)$ be the event that the following is true.
\begin{enumerate}
\item $\mcl N_{\tb H}\left(X(-\wt J_{x }^C , -J_{x,\frk h_n(x)}^H-1 )\right) \geq \ep_1 n^{\xi/2}$. 
\item $\wt J_{x }^C -  J_{x,\frk h_n(x)}^H \leq \frac12 n^{\xi }$. 
\item $|X(-\wt J_{x }^C , -J_{x,\frk h_n(x)}^H-1 )| \leq \frac12 n^{\xi/2}$.  \label{item-lower0-len}
\end{enumerate}
The event $\wt F_n(x)$ is independent from $\wt E_n (x)$ and by \cite[Theorem 2.5]{shef-burger}, by possibly increasing $n_*^0$ we can arrange that for $n\geq n_*^0$, 
\eqb  \label{eqn-end-lower2}
\BB P\left(\wt F_n(x) \,|\, \wt E_n (x) ,\, X_{  -  m_n^\delta } \dots X_{-1} =   x \right) \succeq 1, 
\eqe 
with the implicit constant depending only on $\ep_1$. 

Let $\wt x$ be a realization of $X_{-\wt J_{x }^C} \dots X_{-|x|-1}$ for which $\wt E_n (x) \cap \wt F_n(x)$ occurs and let $\frk o(\wt x x)$, $\frk c(\wt x x)$, and $\frk c_f(\wt x x)$ be as in~\eqref{eqn-theta-count-reduced} and Definition~\ref{def-c_f}, respectively. Then whenever $X_{-\wt J_{x  }^C} \dots X_{ -1} = \wt x  x$, we have the following.
\begin{enumerate}
\item $|\wt x x| \in \left[n-\ep_1^{-1} n^{\xi } , n- \frac12 n^{\xi}\right]_{\BB Z}$. \label{item-wt-x-time}
\item Since each of $X(-\wt J_{x }^C , -J_{x,\frk h_n(x)}^H-1)$ and $X(-J_{x,\frk h_n(x)}^H , -|x|-1)$ contains no flexible orders, also $\mcl R(\wt x)$ contains no flexible orders and $\mcl N_{\tb F}\left(\mcl R(\wt x x) \right) \leq \mcl N_{\tb F}\left(\mcl R(x)\right) \leq n^\nu$. \label{item-wt-x-F}
\item There are at least $\mcl N_{\tb H}\left(\mcl R(\wt x )\right) \geq   \ep_1 n^{\xi/2}$ hamburger orders to the left of the leftmost flexible order in $\mcl R(\wt x x)$. \label{item-wt-x-h_f}
\item We have 
\alb
\frk c(\wt x) &\geq \mcl N_{\tb C}\left(  X(-J_{x,\frk h_n(x)}^H  , -|x|-1) \right)  - \mcl N_{\tc C}\left(X(-\wt J_{x }^C , -J_{x,\frk h_n(x)}^H-1 ) \right)  \\
 &\geq         - L_{x,\frk h_n(x)}^H  -  \mcl N_{\tc C}\left(X(-J_{x,\frk h_n(x)}^H  , -1)\right) -  \frac14 n^{\xi/2} \geq \frac12 n^{\xi/2} ,
\ale
so there are at least $\frac12 n^{\xi/2} $ cheeseburger orders to the left of the leftmost flexible order in $\mcl R(\wt x x)$.\label{item-wt-x-c_f}
\item By condition~\ref{item-wt-x-h_f}, we have $ \ep_1 n^{\xi/2} \leq  \frk o(\wt x x) \leq \frk h(x) - \frk h_n(x) + n^\nu +|X(-\wt J_{x }^C , -J_{x,\frk h_n(x)}^H-1 )| \leq 2 n^{\xi/2}$. \label{item-wt-x-o}
\item By our choice of $l$ together with condition~\ref{item-lower1-L} in the definition of $\wt E_n(x)$ and condition~\ref{item-lower0-len} in the definition of $\wt F_n(x)$ we have $0 \leq 1 - \frk c(\wt x x)     \leq \ep_1^{-1} n^{\xi/2}$. \label{item-wt-x-dist} 
\end{enumerate}

For $m\in\BB N$, let $\wh J_{\wt x x,m}^H$ be the smallest $j\geq |\wt x x|$ for which $X(-j,-|\wt x x|-1)$ contains $m$ hamburgers, and set $\wh L_{\wt x x,m}^H := d^*\left( X(J_{\wt x x,m}^H , -|\wt x x|-1)\right)$. By observation~\ref{item-wt-x-c_f} above, if $X_{-\wt J_x^C} \dots X_{-1} = \wt x x$ and 
\alb
&\left(\wh J_{\wt x x,\frk o(\wt x x) }^H - |\wt x x|   , \wh L_{\wt x x,\frk o(\wt x x)}^H    \right) = \left(n-|\wt x x| ,  \frk c(\wt x x) - l\right)  \\
&\mcl N_{\tc C}\left(X(-\wh J_{\wt x x,\frk o(\wt x x)}^H  ,-   |\wt x x| -1 )  \right) \leq \frac12 n^{\xi/2}
\ale
then $J = \wh J_{x,\frk o(\wt x x)}^H$ and $\wh{\mcl E}_n^l$ occurs and every cheeseburger in $X(-\wh J_{x,\frk o(\wt x x)}^H  ,-   |\wt x x| -1 )$ is matched to a cheeseburger order in $\mcl R(\wt x x)$. By observations~\ref{item-wt-x-time} and~\ref{item-wt-x-h_f} above, the time $\wt J_{x }^C  =|\wt x x|$ satisfies the defining conditions of $\pi_n$. Therefore $\mcl E_n^l$ occurs. 
By observations~\ref{item-wt-x-o} and~\ref{item-wt-x-dist} together with \cite[Lemma 2.10]{gms-burger-local},  
\eqbn
\BB P\left( \mcl E_n^l \,|\, \wt E_n (x) ,\, \wt F_n(x) ,\, X_{  -  m_n^\delta} \dots X_{-1} =   x   \right) \succeq n^{-3\xi/2}. 
\eqen
 By combining this with~\eqref{eqn-end-lower1} and~\eqref{eqn-end-lower2}, we obtain the statement of the proposition.
\end{proof}
 
\subsection{Upper bound when the tip of the path is in the bulk}
\label{sec-end-upper}

In this subsection, we will prove an upper bound for the conditional probability of the event $\mcl E_n^l$ given a realization of $X_{- N_n } \dots X_{-1}$, where $N_n$ is a stopping time for the word $X$, read backward. This upper bound degenerates near the boundary of the first quadrant, but we will deduce from it various upper bounds without this drawback in Section~\ref{sec-end-upper-full} below. 

\begin{prop} \label{prop-end-upper}
Fix $\ol\xi \in (\xi ,1)$. Let $N_n$ be a stopping time for the word $x$, read backward, and let $x$ be a realization of $X_{- N_n} \dots X_{-1}$ such that $|x| < n - \ep_1^{-1} n^{\xi}$ and $\mcl R(x)$ contains no burgers, at least $ n^{\ol \xi/2}$ hamburger orders, at least $n^{\ol\xi/2}$ cheeseburger orders, and at most $n^\nu$ flexible orders. Let $\frk h(x)$ and $\frk c(x)$ be as in~\eqref{eqn-theta-count-reduced} and let 
\eqb \label{eqn-hc-for-end-upper}
  \frk m(x) := \frk h(x) \wedge \frk c(x)   .
\eqe 
Let $\psi_0$ and $\psi_2$ be the slowly varying functions from Lemmas~\ref{prop-I-reg-var} and~\ref{prop-interval-reg-var}. For each $l \in \left[\ep_0 n^{\xi/2} , \ep_0^{-1} n^{\xi/2} \right]_{\BB Z}$, we have
\eqbn
\BB P\left( \mcl E_n^l   \,|\, X_{  -  N_n} \dots X_{-1} =   x   \right) \preceq    \psi_0\left(\frk m(x)^2 \right) \psi_2(n^\xi) \frk m(x)^{-2 - 2\mu}  n^{\xi  (\mu-1/2) }  .
\eqen
with the implicit constant depending only on $\ol\xi$, $\ep_0$, and $\ep_1$.
\end{prop}
 
Recall the time $\wt\pi_n = \wt\pi_n(\ep_1,\ep_2)$ from Proposition~\ref{prop-wt-pi-exists}.  
We will prove Proposition~\ref{prop-end-upper} by considering the time intervals $\left[ N_n , \wt\pi_n\right]_{\BB Z}$ and $[\wt\pi_n , n]_{\BB Z}$ separately; estimating the former using \cite[Proposition 4.4]{gms-burger-local} and Proposition~\ref{prop-wt-pi-exists}; and estimating the latter using \cite[Proposition 2.2]{gms-burger-local}. 

\begin{lem} \label{prop-last-increment-upper}
Let $\wt x$ be a realization of $X_{-\wt \pi_n}\dots X_{-1}$ for which $\wt\pi_n < J$. Let $\frk c_f(\wt x)$ and $\frk r(\wt x)$ be as in Definition~\ref{def-c_f} and let $\frk o(\wt x)$ be as in~\eqref{eqn-theta-count-reduced}.
There are constants $a_0 , a_1 > 0$, depending only on $\ep_1$ and $\ep_2$ such that for each $n\in\BB N$, each realization $\wt x$ as in that lemma, and each $l \in\BB N$,
\alb
 \BB P\left( \mcl E_n^l  \,|\, X_{-\wt\pi_n} \dots X_{-1} = \wt x  \right) 
  \preceq   \exp\left(- \frac{a_0 \frk o(\wt x)^2 }{ n-|\wt x| }  - \frac{a_1 |\frk c(\wt x) - l|}{ (n-|\wt x|)^{1/2} }   \right)  \frk o(\wt x)^{-3  }  \frk r(\wt x)^{-1+o_{\frk r(\wt x)}(1)}  ,
\ale
with the implicit constants depending only on $\ep_1$ and $\ep_2$. 
\end{lem}
\begin{proof}
Fix a realization $\wt x$ as in the statement of the lemma. For $m\in\BB N$, let $J_{\wt x,m}^H$ be the smallest $j\geq |\wt x|+1$ such that $X(-j,-|\wt x|-1)$ contains $m$ hamburgers and let $L_{\wt x , m}^H := d^*\left(X(-J_{\wt x,m}^H,-1)\right)$. 

On the event $\mcl E_n^l$, every cheeseburger in $X(-J , -|\wt x|-1)$ is matched to a cheeseburger order in $X(-|\wt x| ,-1)$. It follows that
\eqbn
( J_{\wt x , \frk o(\wt x)}^H  , L_{\wt x , \frk o(\wt x)}^H ) = (n, l - \frk c(\wt x) ). 
\eqen
By Lemma~\ref{prop-last-abs-cont} together with the second equation of \cite[Lemma 2.8]{gms-burger-local} (applied with $ m = \frk o(\wt x)$, $k = n - |\wt x|$, and $R = |l - \frk c(\wt x) |$), we therefore have
\alb
&\BB P\left(  \mcl E_n^l   \,|\, X_{-\wt\pi_n} \dots X_{-1} = \wt x \right) \\
&\qquad = \BB P\left( ( J_{\wt x , \frk o(\wt x)}^H  , L_{\wt x , \frk o(\wt x)}^H ) = (n, l - \frk c(\wt x) )  \,|\, X_{-\wt\pi_n} \dots X_{-1} = \wt x \right)\\
&\qquad \leq \BB P\left( ( J_{\wt x , \frk o(\wt x)}^H  , L_{\wt x , \frk o(\wt x)}^H ) = (n, l - \frk c(\wt x) )  \,|\, X_{-|\wt x|} \dots X_{-1} = \wt x \right) \frk r(\wt x)^{- 1+o_{\frk r(\wt x)}(1)}  \\
&\qquad \preceq   \exp\left(- \frac{a_0 \frk o(\wt x)^2 }{ n-|\wt x| }  - \frac{a_1 |\frk c(\wt x) - l|}{ (n-|\wt x|)^{1/2} }   \right)  \frk o(\wt x)^{-3  } \frk r(\wt x)^{-1+o_{\frk r(\wt x)}(1)} ,
\ale
for constants $a_0 , a_1 > 0$ as in the statement of the lemma.
\end{proof}

Fix a small constant $\eta>0$. In what follows, we let 
\eqb \label{eqn-end-upper-k_n}
  \BB k_n^\eta  := \left\lfloor  \log_2 \left(  \eta n^{(\ol\xi-\xi)/2} \right) \right\rfloor  \quad \op{and}\quad \wt{\BB k}_n := \left\lfloor  \log_2 \left(  \eta n^{\xi/2} \right) \right\rfloor ,
\eqe 
be the smallest $k\in\BB N$ for which $2^k n^{\xi/2} \geq \eta n^{\ol\xi/2}$ (with $\ol\xi$ as in Proposition~\ref{prop-end-upper}) and the smallest $k\in\BB N$ for which $2^{-k} \leq n^{-\xi/2}$, respectively.

\begin{lem} \label{prop-pi-increment-upper}
Let $x$ be a realization of $X_{- N_n}\dots X_{-1}$ as in Proposition~\ref{prop-end-upper} and define $\frk h(x)$, $ \frk c(x)$, and $\frk m(x)$ as in Proposition~\ref{prop-end-upper}. 
For $k, l \in\BB N$, let $E_n^l(k  )$ be the event that $ \pi_n < J$,
\alb 
\mcl N_{\tb H}\left(X(- \pi_n,-1)\right)  \leq 2^{k } \ep_1 n^{\xi/2} , \quad \mcl N_{\tb C}\left(X(- \pi_n,-1)\right) - l   \leq 2^{k } n^{\xi/2}  , \quad \op{and} \quad \mcl N_{\tb F}\left(X(- \pi_n,-1)\right) \leq 2n^\nu .
\ale
There is an $\eta>0$, depending only on $\ep_0$, $\ep_1$, and $\ep_2$, such that for each $l\in \left[\ep_0 n^{\xi/2} , \ep_0^{-1} n^{\xi/2}\right]_{\BB Z}$ and each $k\in [1,\BB k_n^{2\eta}]_{\BB Z}$, we have
\begin{align} \label{eqn-pi-increment-upper}
\BB P\left(    E_n^l(k )    \,|\, X_{- N_n}\dots X_{-1} = x\right)  \preceq     \psi_0\left( \frk m(x)^2 \right)  \psi_2\left(2^{2 k }  n^{\xi } \right) \frk m(x)^{-2 - 2\mu}  2^{2(\mu+1)k   } n^{\xi(\mu+1)} 
\end{align}
with $\psi_0$ and $\psi_2$ as in Section~\ref{sec-reg-var} and the implicit constants depending only on $\ep_0,\ep_1$, and $\ep_2$. 
\end{lem}
\begin{proof}
Consider a realization $x'$ of $X_{-\pi_n} \dots X_{- N_n-1}$ for which $E_n^l(k)$ occurs when $X_{-\pi_n}\dots X_{-1} = x' x$. Since $ \pi_n < J$, the word $\mcl R(  x')$ contains at most $\frk c(x) + n^\nu$ cheeseburgers and at most $\frk h(x) + n^\nu$ hamburgers. Since $ \mcl N_{\tb C}\left(\mcl R(x'x)  \right) -l \leq 2^{k} n^{\xi/2}$, we have
\eqbn
d^*\left(\mcl R(  x ')  \right) \geq \frk c(x) - n^\nu - l - 2^{k}n^{\xi/2} . 
\eqen
Hence
\eqb \label{eqn-x'-small-C}
\mcl N_{\tb C}\left( \mcl R(  x ') \right) + \mcl N_{\tb F}\left( \mcl R(  x ') \right) \leq 3n^\nu + l + 2^{k} n^{\xi/2} \preceq 2^{k}n^{\xi/2} .
\eqe 
Furthermore, we have 
\eqb \label{eqn-x'-small-C-diff}
\left| \mcl N_{\tc C}\left( \mcl R(  x ')  \right) - \frk c(x)   \right| \leq n^\nu + l + 2^{k} n^{\xi/2} \preceq 2^{k}n^{\xi/2}   .
\eqe 
By definition of $E_n^l(k)$, we have
\eqb \label{eqn-x'-small-H}
   \mcl N_{\tb H}\left( \mcl R(x'x)  \right) + \mcl N_{\tb F}\left( \mcl R(x'x)  \right)      \preceq 2^{k} n^{\xi/2} ,
\eqe 
whence
\eqb \label{eqn-x'-small-H-diff}
 \mcl N_{\tb H}\left( \mcl R(x' )  \right) + \mcl N_{\tb F}\left( \mcl R(x' )  \right)      \preceq 2^{k} n^{\xi/2}  \quad \op{and} \quad \left| \frk h(x) -  \mcl N_{\tc H}\left( \mcl R(  x ')  \right) \right|  \preceq 2^{k}n^{\xi/2} .
\eqe 
By definition of $\pi_n$, we also have
\eqb \label{eqn-x'-small-time}
0\leq   n - |x'x|   \leq \ep_1^{-1} n^{\xi } \preceq  n^{\xi }  \leq 2^k n^{\xi } .
\eqe 
The implicit constants above depend only on $\ep_0,\ep_1$, and $\ep_2$. 
 
Since $\pi_n \geq n - \ep_1^{-1} n^{\xi/2}$, equations~\eqref{eqn-x'-small-C},~\eqref{eqn-x'-small-C-diff},~\eqref{eqn-x'-small-H}, and~\eqref{eqn-x'-small-H-diff} together imply that
$E_n^l(k)$ is contained in the event $E_{n ,k',r'}^{\frk h(x) , \frk c(x)}$ of \cite[Proposition 4.4]{gms-burger-local}, defined with  $\dots X_{-|x|-2} X_{-|x|-1}$ in place of $\dots X_{-2} X_{-1}$, for $k' \asymp 2^{2k} n^{\xi  }$ and $r' \asymp 2^k n^{\xi/2}$ (implicit constants depending only on $\ep_0$, $\ep_1$, and $\ep_2$).
By \cite[Proposition 4.4, Assertion 1]{gms-burger-local}, if we choose $\eta < (2b)^{-1}$ (say), then whenever $k\in [1,\BB k_n^{2\eta}]_{\BB Z}$ and $\frk m(x) \geq n^{\ol\xi/2}$, we have
\eqbn 
\BB P\left(E_n^l(k) \,|\, X_{- N_n}\dots X_{-1} =x   \right) \preceq  \psi_0\left( \frk m(x)^2 \right)  \psi_2\left(2^{2 k }  n^{\xi } \right) \frk m(x)^{-2 - 2\mu}  2^{2(\mu+1)k   } n^{\xi(\mu+1)} .
\eqen  
The statement of the lemma follows. 
\end{proof}

\begin{lem} \label{prop-mid-increment-upper}
Let $x$ be a realization of $X_{- N_n}\dots X_{-1}$ as in Proposition~\ref{prop-end-upper} and define $\frk h(x)$, $ \frk c(x)$, and $\frk m(x)$ as in Proposition~\ref{prop-end-upper}. 
For $k_1 , k_3, l \in\BB N$, let $\wt E_n^l(k_1,k_3 )$ be the event that $\wt\pi_n < J \wedge n$ and the following is true:
\alb 
&\mcl N_{\tb H}\left(X(-\wt\pi_n,-1)\right) \leq 2^{k_1} (\ep_1-\ep_2) n^{\xi/2} ,\quad \mcl N_{\tb C}\left(X(-\wt\pi_n,-1)\right) - l   \leq 2^{k_1 } n^{\xi/2} ,   \\
& \mcl N_{\tb F}\left(X(-\wt\pi_n,-1)\right) \leq n^\nu ,    \quad \op{and} \quad  \frk c_f\left(X(-\wt\pi_n,-1)\right) \leq 2^{-k_3 } n^{\xi/2} .
\ale
Let $\eta$ be chosen as in Lemma~\ref{prop-mid-increment-upper} and let $\BB k_n^\eta$ and $\wt{\BB k}_n$ be as in~\eqref{eqn-end-upper-k_n} for this choice of $\eta$. For each $l\in \left[\ep_0 n^{\xi/2} , \ep_0^{-1} n^{\xi/2}\right]_{\BB Z}$ and each $(k_1,k_3)    \in [1,\BB k_n^{ \eta}]_{\BB Z} \times [1,\wt{\BB k}_n]_{\BB Z}$, we have
\begin{align} \label{eqn-mid-increment-upper}
 \BB P\left(   \wt E_n^l(k_1,k_3 )    \,|\, X_{- N_n}\dots X_{-1} = x\right)  
  \preceq    \psi_0\left( \frk m(x)^2 \right)  \psi_2\left(2^{2 k_1}  n^{\xi } \right)  \frk m(x)^{-2 - 2\mu}  2^{2(\mu+1)k_1   -100k_3}   n^{\xi(\mu+1)} 
\end{align}
with $\psi_0$ and $\psi_2$ as in Section~\ref{sec-reg-var} and the implicit constants depending only on $\ep_0,\ep_1$, and $\ep_2$. 
\end{lem}
\begin{proof}
Let $E_n^l(k_1+1)$ be as in Lemma~\ref{prop-pi-increment-upper} with $k = k_1+1$. By condition~\ref{item-wt-pi-compare} of Proposition~\ref{prop-wt-pi-exists}, if $\wt E_n^l(k_1,k_3)$ occurs, then $E_n^l(k_1+1)$ occurs. By condition~\ref{item-wt-pi-C} of Proposition~\ref{prop-wt-pi-exists}, for $(k_1,k_3)  \in [1,\BB k_n^{ \eta}]_{\BB Z} \times [1,\wt{\BB k}_n]_{\BB Z}$ we have
\eqbn
\BB P\left( \frk c_f\left(X(-\wt\pi_n,-1)\right) \leq 2^{-k_3 } n^{\xi/2} \,|\,    E_n^l(k_1+1) ,\,  X_{- N_n}\dots X_{-1} = x  \right)  \preceq \left(2^{k_3} \wedge n\right)^{-100} .
\eqen
By combining this with the estimate of Lemma~\ref{prop-pi-increment-upper}, we obtain~\eqref{eqn-mid-increment-upper}.
\end{proof}

\begin{proof}[Proof of Proposition~\ref{prop-end-upper}]
For $k_1 , k_2 , k_3 , l \in \BB N$, let $\wh E_n^l(k_1,k_2,k_3)$ be the event that $\wt\pi_n < J\wedge n$ and the following is true:
\alb 
&\mcl N_{\tb H}\left(X(-\wt\pi_n,-1)\right) \in \left[ 2^{k_1-1}  (\ep_1 - \ep_2) n^{\xi/2} , 2^{k_1 }   (\ep_1 - \ep_2)  n^{\xi/2} \right]_{\BB Z} \\
&\mcl N_{\tb C}\left(X(-\wt\pi_n,-1)\right) - l   \in \left[ 2^{k_2-1} n^{\xi/2}   ,  2^{k_2} n^{\xi/2}\right]_{\BB Z} , \\
& \mcl N_{\tb F}\left(X(-\wt\pi_n,-1)\right) \leq 2n^\nu ,                 \quad \op{and} \\
& \frk c_f\left(X(-\wt\pi_n,-1)\right) \in \left[ 2^{-k_3   } n^{\xi/2} , 2^{k_3-1} n^{\xi/2}\right]_{\BB Z} .
\ale
Also let $\wh E_n^l(k_1, 0 ,k_3)$ be defined in the same manner as $\wh E_n^l(k_1,k_2,k_3)$ above but with $\left[-\ep_1^{-1} n^{\xi/2}  ,   n^{\xi/2}\right]_{\BB Z}$ in place of $\left[ 2^{k_2-1} n^{\xi/2}   ,  2^{k_2} n^{\xi/2}\right]_{\BB Z}$; and let $\wh E_n^l(k_1,k_2,0)$ be defined in the same manner as $\wh E_n^l(k_1,k_2,k_3)$ above but with  $\left[   n^{\xi/2} , \infty \right)_{\BB Z}$ in place of $\left[ 2^{-k_3  } n^{\xi/2} , 2^{-k_3} n^{\xi/2}\right]_{\BB Z}$. 

Let $\eta$ be as in Lemma~\ref{prop-pi-increment-upper} and let  
\alb
G_n := \left\{ \sup_{j\in [n-\ep_1^{-1} n^{\xi}  , n]_{\BB Z}} |X(-j , -n+ \ep_1^{-1} n^{\xi})| \leq \frac{\eta}{4} n^{\ol \xi/2} \right\} \cap \left\{ \sup_{j\in [n/2  , n]_{\BB Z}} \mcl N_{\tb F} \left(X(-j,-1)\right)    \leq 2 n^\nu  \right\} .
\ale
Since $\ol \xi > \xi$ and $\nu > 1-\mu$, \cite[Lemma 3.13]{shef-burger} together with \cite[Corollary 5.2]{gms-burger-cone} implies
\eqb \label{eqn-end-sup-bad}
\BB P\left(G_n^c \,|\, X_{-N_n} \dots X_{-1} = x\right) = o_n^\infty(n) .
\eqe 
 
By condition~\ref{item-wt-pi-mono} of Proposition~\ref{prop-wt-pi-exists}, on the event $\{\wt\pi_n < J\wedge n\}$ we always have $n-\wt\pi_n \leq \ep_1^{-1} n^{\xi } \leq |x|$. 
By Lemma~\ref{prop-last-increment-upper}, for $l \in \left[\ep_0 n^{\xi/2} , \ep_0^{-1} n^{\xi/2}\right]_{\BB Z}$ and $(k_1,k_2,k_3) \in [1,\BB k_n^{ \eta}]_{\BB Z}  \times [0,\BB k_n^{ \eta}]_{\BB Z}   \times [1,\wt{\BB k}_n]_{\BB Z}$, we have
\eqb \label{eqn-end-prob1}
\BB P\left( \mcl E_n^l  \,|\,   \wh E_n^l(k_1,k_2,k_3),\,      X_{- N_n} \dots X_{-1} =  x  \right)  
\preceq   \exp\left(-  a_0 2^{2 k_1}  -  a_1 2^{ k_2}  \right)  2^{-3k_1}  2^{(k_1 + k_3)(1+o_{k_1+k_3}(1))}   n^{-3\xi/2}    
\eqe
for positive constants $a_0$ and $a_1$ depending only on $\ep_0$, $\ep_1$, and $\ep_2$. By Lemma~\ref{prop-mid-increment-upper}, we have
\begin{align} \label{eqn-end-prob2}
&\BB P\left(  \wh E_n^l(k_1,k_2,k_3) \,|\,     X_{- N_n} \dots X_{-1} =  x   \right) \notag \\
&\qquad \preceq  \psi_0\left( \frk m(x)^2 \right)  \psi_2\left(2^{ 2 (k_1\vee k_2) } n^{\xi } \right)  \frk m(x)^{-2 - 2\mu}  2^{2(\mu+1)(k_1\vee k_2)  -100k_3 } n^{\xi(\mu+1)} .
\end{align}

If $n$ is chosen sufficiently large (depending only on $\ep_0$ and $\eta$), $\mcl E_n^l$ occurs for some $l\in \left[\ep_0 n^{\xi/2} , \ep_0^{-1} n^{\xi/2}\right]_{\BB Z}$, and $\wh E_n^l(k_1,k_2,k_3)$ occurs for some $(k_1,k_2) \notin [1,\BB k_n^{ \eta}]_{\BB Z}^2$, then $G_n^c$ occurs. 
By~\eqref{eqn-end-prob1} and~\eqref{eqn-end-prob2}, for each small $\alpha >0$ we have
\alb
&\frac{\BB P\left(  \mcl E_n^l \cap G_n \,|\, X_{- N_n} \dots X_{-1} =  x \right) }{  \psi_0\left( \frk m(x)^2 \right)  \psi_2\left( n^{\xi } \right)   \frk m(x)^{-2 - 2\mu}    n^{\xi(\mu-1/2)}  }     \\
&\qquad \preceq \sum_{k_1=1}^{\BB k_n^\eta} \sum_{k_2=0}^{\BB k_n^\eta} \sum_{k_3 =0}^{\wt{\BB k}_n} 
\frac{\psi_2\left(2^{ 2 (k_1\vee k_2) } n^{\xi } \right) }{ \psi_2\left(  n^{\xi } \right) }  2^{2(\mu+1)(k_1\vee k_2) - (2-\alpha) k_1    -(99-\alpha) k_3 }  \exp\left(-  a_0 2^{2 k_1}  -  a_1 2^{ k_2}  \right)  \\
&\qquad \preceq \sum_{k_1=1}^\infty \sum_{k_2=0}^\infty    2^{2(\mu+1 + \alpha)(k_1\vee k_2) - (2-\alpha) k_1    }  \exp\left(-  a_0 2^{2 k_1}  -  a_1 2^{ k_2}  \right)  \\
&\qquad \preceq     \sum_{k_1=1}^\infty \sum_{k_2 = k_1 }^\infty    2^{2(\mu+1 + \alpha)k_2  }  \exp\left(-  a_0 2^{2 k_1}  -  a_1 2^{ k_2}  \right)  
+\sum_{k_2=0}^\infty \sum_{k_1 = k_2 }^\infty    2^{2(\mu+1 + \alpha)k_1}  \exp\left(-  a_0 2^{2 k_1}  -  a_1 2^{ k_2}  \right)  \\
&\qquad \preceq     \sum_{k_1=1}^\infty  \exp\left(-  a_0 2^{2 k_1}    \right)  
+\sum_{k_2=0}^\infty  \exp\left(  -  a_1 2^{ k_2}  \right) 
\preceq 1 ,
\ale
with implicit constants depending only on $\ep_0$, $\ep_1$, $\ep_2$, and $\alpha$. By combining this with~\eqref{eqn-end-sup-bad}, we obtain the statement of the proposition. 
\end{proof}

We end this subsection by recording the following consequence of the proof of Proposition~\ref{prop-end-upper}, which will be used in Section~\ref{sec-end-reg} below.

\begin{cor} \label{prop-good-pi}
For $n  \in\BB N$ and $B> 1$, let 
\alb
&H_{n,1} (B) := \left\{B^{-1} n^{\xi/2} \leq  \mcl N_\theta \left(X(-\wt\pi_n , -1 )  \right) \leq B n^{\xi/2}  \quad \forall \theta\in \left\{\tb H ,\tb C\right\} \right\} \\
&H_{n,2}(B) := \left\{ \frk c_f\left(X(-\wt\pi_n,-1)\right) \geq B^{-1} n^{\xi/2} \right\}\\
&H_{n,3}(B)  :=  \left\{ n -\wt \pi_n \geq B^{-1} n^{\xi/2} \right\} \\
&F_{\wt\pi_n} := \left\{\mcl N_{\tb F}\left(X(-\wt\pi_n,-1)\right) \leq 2n^\nu \right\} \\
&H_n(B) :=\{\wt \pi_n < J\wedge n\} \cap H_{n,1}(B) \cap H_{n,2}(B) \cap F_{\wt\pi_n} .
\ale
Fix $C>1$ and $q \in (0,1)$. There is a $B>1$, depending only on $C$, $q$, $\ep_0$, $\ep_1$, and $\ep_2$ such that the following is true. Let $x$ be a realization of $X_{-m_n^\delta} \dots X_{-1}$ for which $\mcl B_n^\delta(C)$ occurs (notation as in Definition~\ref{def-B_n^delta}). Then for each $l\in \left[\ep_0 n^{\xi/2} , \ep_0^{-1} n^{\xi/2}\right]_{\BB Z}$, we have
\eqbn
\BB P\left(H_n(B) \,|\, X_{-m_n^\delta} \dots X_{-1} = x ,\, \mcl E_n^l \right) \geq 1-q. 
\eqen
\end{cor} 
\begin{proof}
Fix $\alpha >0$ to be chosen later, depending only on $\ep_0$, $\ep_1$, $\ep_2$, $C$, and $q$. 
Since $\mcl N_{\tb F}\left(\mcl R(x)\right) \leq n^\nu$, by \cite[Corollary 5.2]{gms-burger-cone} we have $\BB P(F_{\wt\pi_n}^c \,|\, X_{-m_n^\delta} \dots X_{-1} = x) = o_n^\infty(n)$, at a rate depending only on $\ep_0,\ep_1$, and $\ep_2$. Therefore, for each $\delta>0$ we can find $n_*^0 = n_*^0(\delta,\alpha \ep_0, \ep_1,\ep_2) \in \BB N$ such that for $n\geq n_*^0$, 
\eqb \label{eqn-good-pi-F-prob}
\BB P\left(F_{\wt\pi_n} \,|\, X_{-m_n^\delta} \dots X_{-1} = x ,\, \mcl E_n^l\right) \geq 1-\alpha. 
\eqe 

For $k_1,k_2,k_3\in\BB N$, define the events $\wh E_n^l(k_1,k_2,k_3)$ as in the proof of Proposition~\ref{prop-end-upper}. Also let $\BB k_n^\eta$ be defined as in~\eqref{eqn-end-upper-k_n} with $\eta$ as in Lemma~\ref{prop-pi-increment-upper}. By~\eqref{eqn-end-prob1} and~\eqref{eqn-end-prob2}, for $l \in \left[\ep_0 n^{\xi/2} , \ep_0^{-1} n^{\xi/2}\right]_{\BB Z}$ and $(k_1,k_2,k_3) \in [1,\BB k_n^{ \eta}]_{\BB Z}  \times [0,\BB k_n^{ \eta}]_{\BB Z}   \times [0,\infty)_{\BB Z}$, we have
\begin{align}  
&\frac{\BB P\left( \mcl E_n^l   \cap \wh E_n^l(k_1,k_2,k_3) \,|\,      X_{- m_n^\delta} \dots X_{-1} =  x  \right)}{  \psi_0\left( \delta n \right) \psi_2(n^{\xi})    \delta^{-1-\mu}   n^{-1-\mu + \xi(\mu-1/2)}   }   \notag  \\
&\qquad \preceq  \frac{    \psi_2\left(2^{ 2 (k_1\vee k_2) } n^{\xi } \right) }{\psi_2(n^\xi)}   \exp\left(-  a_0 2^{2 k_1}  -  a_1 2^{ k_2}  \right)  2^{2(\mu+1)(k_1\vee k_2) -2k_1   - 99k_3 +      (k_3 + k_1)( o_{k_1+k_3}(1))}     
\end{align}
for positive constants $a_0$ and $a_1$ depending only on $\ep_0$, $\ep_1$, and $\ep_2$. 

For $(\ol k_1,\ol k_2,\ol k_3) \in \BB N\times [0,\infty)_{\BB Z}^2$, let $\mcl K(\ol k_1,\ol k_2, \ol k_3)$ be the set of $(k_1,k_2,k_3) \in \BB N\times [0,\infty)_{\BB Z}^2$ for which either $k_1\geq \ol k_1$, $k_2 \geq \ol k_2$, or $k_3\geq \ol k_3$. 
By the calculation at the end of the proof of Proposition~\ref{prop-end-upper}, for each $\alpha > 0$ we can find $\ol k_1 , \ol k_2 , \ol k_3 \in \BB N$ and $n_*^1 = n_*^1 (\delta,\alpha,C,\ep_0,\ep_1,\ep_2) \geq n_*^0$ such that for $n\geq n_*^1$, we have
\alb
&\BB P\left( \mcl E_n^l   \cap \bigcup_{(k_1,k_2,k_3) \in \mcl K(\ol k_1,\ol k_2, \ol k_3)  }    \wh E_n^l(k_1,k_2,k_3) \,|\,      X_{- m_n^\delta} \dots X_{-1} =  x  \right) \\
&\qquad \preceq  \alpha \psi_0\left( \delta n \right)  \psi_2\left( n^{\xi } \right)    \delta^{-1-\mu}   n^{-1-\mu + \xi(\mu-1/2)}  
\ale
with the implicit constants depending only on $\ep_0,\ep_1$, $\ep_2$, and $C$. By combining this with Proposition~\ref{prop-end-lower} and possibly increasing $n_*^1$, we find that for $n\geq n_*^1$, 
\eqbn
 \BB P\left(   \bigcup_{(k_1,k_2,k_3) \in \mcl K(\ol k_1,\ol k_2, \ol k_3)  }    \wh E_n^l(k_1,k_2,k_3) \,|\, X_{- m_n^\delta} \dots X_{-1} =  x  ,\, \mcl E_n^l \right) \preceq \alpha .
\eqen
By definition of $\pi_n$ together with assertion~\ref{item-wt-pi-compare} of Proposition~\ref{prop-wt-pi-exists}, we always have $ \mcl N_{\tb H} \left(X(-\wt\pi_n , -1 )  \right) \geq (\ep_1-\ep_2) n^{\xi/2}$ on $\mcl E_n^l$.  
By combining this with~\eqref{eqn-good-pi-F-prob}, we infer that with $B_0 = (\ep_1-\ep_2)^{-1} \vee 2^{\ol k_1} \vee 2^{\ol k_2} \vee 2^{\ol k_3}$, 
\eqb \label{eqn-H1-H2-prob}
 \BB P\left(  \left(H_{n,1}(B_0 ) \cap H_{n,2}(B_0 ) \right)^c  \,|\,      X_{- m_n^\delta} \dots X_{-1} =  x  ,\, \mcl E_n^l \right) \preceq \alpha .
\eqe 
By Lemma~\ref{prop-last-increment-upper}, for each $\zeta > 0$ we have 
\[
\BB P\left(  \mcl E_n^l \,|\,  n-\wt\pi_n \leq \zeta n ,\, H_{n,1}(B_0 ) \cap H_{n,2}(B_0 ) \cap F_{\wt\pi_n} ,\, X_{-m_n^\delta} \dots X_{-1} = x  \right) \preceq \exp\left( -\frac{  a_0  }{ \zeta }\right) n^{-3\xi/2} 
\]
with the the constant $a_0 > 0$ and the implicit constant depending only on $B_0$, $C$, $\ep_0$, $\ep_1$, and $\ep_2$. By Lemma~\ref{prop-mid-increment-upper}, 
\eqbn
\BB P\left(  n-\wt\pi_n \leq \zeta n ,\, H_{n,1}(B_0 ) \cap H_{n,2}(B_0 ) \cap F_{\wt\pi_n} \,|\,  X_{-m_n^\delta} \dots X_{-1} = x  \right) \preceq \psi_0(\delta n) \psi_2(n^{\xi/2}) \delta^{-1-\mu} n^{-(1-\xi)(1+\mu)}
\eqen
with the implicit constant depending only on $B_0$, $C$, $\ep_0$, $\ep_1$, and $\ep_2$. By combining this with Proposition~\ref{prop-end-lower} and Bayes' rule, we find that if we choose $\zeta$ sufficiently small (depending on $B_0$, $C$, $\ep_0$, $\ep_1$, and $\ep_2$) and set $B := \zeta^{-1} \wedge B_0$, then for $n\geq n_*^1$, we have
\eqbn
\BB P\left(  H_n(B)^c \,|\,  X_{-m_n^\delta} \dots X_{-1} = x  \right) \preceq \alpha .
\eqen
We now obtain the statement of the corollary by choosing $n$ sufficiently large (depending only on $\ep_0,\ep_1,\ep_2$, $C$, and $\delta$) and $\alpha$ sufficiently small (depending only on $\ep_0,\ep_1,\ep_2$, $q$, and $C$). 
\end{proof}

\subsection{Stronger upper bounds}
\label{sec-end-upper-full}
 
The estimate of Proposition~\ref{prop-end-upper} is not quite enough to prove Proposition~\ref{prop-end-box}, as it degenerates when our realization $x$ is such that $(\frk h(x) , \frk c(x))$ is close to the boundary of the first quadrant (i.e.\ $\frk m(x)$ is small). In this subsection we will deduce from Proposition~\ref{prop-end-upper} two other upper bound which do not have this drawback, one with depends only on $n - |x|$ and one which depends on $\frk h(x) \vee \frk c(x)$. These two upper bounded will be used in Section~\ref{sec-end-properties} to prove Proposition~\ref{prop-end-box}. 

We start with the upper bound in terms of $n-|x|$. Note the similarity to the lower bound of Proposition~\ref{prop-end-lower}. 

\begin{prop} \label{prop-end-upper-n}
For $n\in\BB N$, let $N_n$ be a stopping time for $X$, read backward. 
Let $\psi_0$ and $\psi_2$ be the slowly varying functions as above. For each $\delta>0$, there exists $n_* = n_*(\delta,\ep_0,\ep_1)\in\BB N$ such that for each $n\geq n_*$; each $l \in \left[\ep_0 n^{\xi/2} , \ep_0^{-1} n^{\xi/2} \right]_{\BB Z}$; and each realization $x$ of $X_{-N_n  } \dots X_{-1}$ such that $|x| \in [(1-2\delta)n , (1-\delta)n ]_{\BB Z}$ and $\mcl R(x)$ contains at most $n^\nu$ flexible orders, we have 
\eqb \label{eqn-end-upper-n}
\BB P\left( \mcl E_n^l   \,|\, X_{  - N_n } \dots X_{-1} =   x   \right) \preceq    \psi_0\left( \delta n \right) \psi_2(n^\xi)\delta^{-1-\mu} n^{-1-\mu+ \xi(\mu-1/2)}  
\eqe 
with the implicit constant depending only on $\ep_0$ and $\ep_1$. 
\end{prop}
\begin{proof}
The proof is similar to that of condition~\ref{item-wt-pi-C} of Proposition~\ref{prop-wt-pi-exists}. The idea is to grow a little more of the word backward from time $-N_n$ and argue that with high probability, we soon reach a time for which $\mcl N_{\tb H}(X(-j,-1)) \wedge \mcl N_{\tb C}(X(-j,-1))$ is of order $(\delta n)^{1/2}$. We then condition on the word up to such a time and apply Proposition~\ref{prop-end-upper}.

Fix $ \ol\xi , \ol\xi' \in (\xi,1)$ with $\ol \xi < \ol \xi'$ and $\delta>0$. For $n\in\BB N$, let 
\[
\BB k_n   := \lfloor \log_2 (\delta   n^{ 1-\ol \xi' } )  \rfloor
\]
be the largest $k\in\BB N$ for which $2^{-k/2} (\delta n)^{1/2} \geq n^{\ol\xi'/2} $. 

Let $m_n^0 = N_n $ and for $k\in [1,\BB k_n]_{\BB Z}$, inductively define $m_n^k :=  \lfloor 2^{-k-1} \delta n \rfloor + m_n^{k-1} $. For $\zeta >0$ and $k\in[1,\BB k_n]_{\BB Z}$, let $E_n^k(\zeta)$ be the event that 
\eqbn
\mcl N_{\tb H}\left(X(-m_n^k ,-m_n^{k-1}-1)\right) \wedge \mcl N_{\tb C}\left(X(-m_n^k,-m_n^{k-1}-1)\right) \geq \zeta 2^{-k/2} (\delta n)^{1/2} .
\eqen
Let $K_n(\zeta)$ be the minimum of $\BB k_n+1$ and the smallest $k\in [1,\BB k_n]_{\BB Z}$ for which $E_n^k(\zeta) $ occurs and let $\wt N_n(\zeta) := m_n^{K_n(\zeta)}$. Note that if $N_n \leq (1-\delta)n$, then
\[
\wt N_n(\zeta) \leq N_n + \sum_{k=1}^{\BB k_n} 2^{-k-1} \delta n \leq (1-\delta/2) n .
\]

By \cite[Theorem 2.5]{shef-burger}, for each $s \in (0,1)$, we can find an $\zeta > 0$ (depending only on $s$) and an $ n_*^0 =  n_*^0(  \zeta , s , \delta ) \in \BB N$ such that for $n\geq n_*^0$ and $k\in[1 ,\BB k_n]_{\BB Z}$, we have 
\eqbn
\BB P\left(  E_n^k(\zeta) \right) \geq 1-s .
\eqen
Since the events $\{E_n^k(\zeta)\}$ are independent from each other and from $X_{-N_n} \dots X_{-1}$, for such a choice of $\zeta$ we have for $n\geq n_*^0$, $k\in [1, \BB k_n]_{\BB Z}$, and each realization $x$ as in the statement of the proposition that
\eqb \label{eqn-end-tube-prob}
\BB P\left(K_n(\zeta) \geq k   \,|\,   X_{  - N_n} \dots X_{-1} =   x       \right) \leq s^k .
\eqe 
In particular,
\eqb \label{eqn-end-tube-prob-long}
\BB P\left( K_n(\zeta) = \BB k_n + 1    \,|\, X_{  - N_n} \dots X_{-1} =   x  \right) \leq s^{\BB k_n} \leq \left(\delta n^{1-\ol \xi'}\right)^{-R}
\eqe  
where the exponent $R > 0$ can be made as large as we like by choosing $s$ (and hence also $\zeta$) sufficiently small, independently of $\delta$, $n$, and $x$.  

Let
\eqb \label{eqn-end-sup-F_n}
F_n^k := \left\{ \mcl N_{\tb F}\left(X(-m_n^k  ,-1)\right) \leq 2 n^\nu \right\}
\eqe 
with $\nu$ defined as in the beginning of this section. By \cite[Corollary 5.2]{gms-burger-cone}, for any realization $x$ as in the statement of the proposition we have
\eqb \label{eqn-end-sup-F_n-prob}
\BB P\left(\bigcup_{k=1}^{\BB k_n} F_n^k   \,|\, X_{  - N_n} \dots X_{-1} =   x   \right) = o_n^\infty(n)  
\eqe 
at a rate depending only on $p$. 
 
On the event $\{K_n(\zeta) = k\}$ for $k\in[1,\BB k_n]_{\BB Z}$, we have 
\eqbn
\mcl N_{\tb H}\left(X(-N_n(\zeta) ,- 1)\right) \wedge \mcl N_{\tb C}\left(X(-N_n(\zeta) ,-1)\right) \geq  \zeta 2^{-k/2}  (\delta n)^{1/2 }  .
\eqen
Let $n_*^1 = n_*^1(\zeta,s,\delta,\ol\xi',\ol\xi) \geq n_*^0$ be chosen so that $\zeta n^{\ol\xi'/2} \geq n^{\ol\xi' /2}$ for $n\geq n_*^1$. 
By Proposition~\ref{prop-end-upper} (applied with $\wt N_n(\zeta)$ in place of $N_n$) and our choice of $\BB k_n$, by possibly increasing $n_*^1$ we can arrange that for each $n\geq n_*^1$, each $k\in [1,\BB k_n]_{\BB Z}$, and each $l$ and $x$ as in the statement of the proposition we have
\alb
&\BB P\left(\mcl E_n^l \,|\, K_n(\zeta) = k   ,\,  F_n^k ,\,   X_{  - N_n} \dots X_{-1} =   x  \right) \\
&\qquad \preceq \psi_0\left( 2^{-k} \delta n  \right) \psi_2(n^{\xi/2}) 2^{2k(1+\mu  ) } \delta^{-1-\mu} n^{-1-\mu + \xi(\mu-1/2)}  
\ale
with the implicit constants depending only on $\zeta$, $\ol\xi, \ol\xi'$, $\ep_0$, and $\ep_1$. By combining this with~\eqref{eqn-end-tube-prob}, we infer
\begin{align} \label{eqn-prob-mclE-end-tube}
&\BB P\left(  \mcl E_n^l ,\,      K_n(\zeta) = k  ,\,  F_n^k  \,|\, X_{  - N_n} \dots X_{-1} =   x \right)\notag \\
&\qquad  \preceq  \psi_0\left( 2^{-k} \delta n  \right) \psi_2(n^{\xi/2}) s^k 2^{2k(1+\mu  ) } \delta^{-1-\mu} n^{-1-\mu + \xi(\mu-1/2)}   .
\end{align}
By~\eqref{eqn-end-tube-prob-long},~\eqref{eqn-end-sup-F_n-prob}, and~\eqref{eqn-prob-mclE-end-tube}, we find that if we choose $s  $ sufficiently small, depending only on $\ol \xi$ (and $\zeta$ sufficiently small depending on $s$), then there exists $q = q(s) \in (0,1)$ and $n_*^2  = n_*^2 ( s , \zeta , \delta , \ep_0 , \ep_1) \geq n_*^1$  such that for $n\geq n_*^2 $ and any $l$ and $x$ as in the statement of the proposition, it holds for each $k \in [1,\BB k_n]_{\BB Z}$ that
\eqbn
\BB P\left( \mcl E_n^l ,\,      K_n(\zeta) = k   \,|\, X_{  - N_n} \dots X_{-1} =   x \right)  \preceq  q^k \psi_0\left( 2^{-k} \delta n  \right) \psi_2(n^{\xi/2})   \delta^{-1-\mu} n^{-1-\mu + \xi(\mu-1/2)}  ,
\eqen
with the implicit constants depending only on $\ep_0$, $\ep_1$, $s$, and $\zeta$; and
\eqbn
\BB P\left( k_*^n(\zeta) = \BB k_n + 1    \,|\, X_{  - N_n} \dots X_{-1} =   x  \right) \leq  \delta^{-1-\mu} n^{-100}  .
\eqen
By summing over all $k\in [1,\BB k_n+1]_{\BB Z}$, we obtain the statement of the proposition for $n\geq n_*^2$. 
\end{proof}

Next we will prove a version of Lemma~\ref{prop-end-upper} with $\frk h(x) \vee \frk c(x)$ in place of $\frk h(x) \wedge \frk c(x)$, which will help us rule out the possibility that $D(X(-N_n,-1))$ is far from the origin. 

\begin{prop} \label{prop-end-upper-max}
Fix $\ol\xi \in (\xi,1)$ and for $n\in\BB N$, let $N_n$ be a stopping time for the word $X$, read backward. Let 
\eqbn
\frk M (x) := \frk h(x) \vee \frk c(x)  .
\eqen
There exists $n_* = n_*(\ep_0,\ep_1)$ such that for each $n\geq n_*$; each $l \in \left[\ep_0 n^{\xi/2} , \ep_0^{-1} n^{\xi/2} \right]_{\BB Z}$; and each realization $x$ of $X_{-N_n  } \dots X_{-1}$ such that $|x|  < n - \ep_1^{-1} n^\xi$, $\frk M(x) \geq n^{\ol\xi/2}$, and $\mcl N_{\tb F}(\mcl R(x)) \leq n^\nu$, we have
\eqb \label{eqn-end-upper-max}
\BB P\left( \mcl E_n^l   \,|\, X_{  - N_n } \dots X_{-1} =   x   \right) \preceq    \psi_0\left(\frk M (x)^2 \right) \psi_2(n^\xi) \frk M (x)^{-2 -2\mu} n^{ \xi(\mu-1/2)}  
\eqe 
with the implicit constant depending only on $\ep_0$ and $\ep_1$. 
\end{prop}
\begin{proof}
Fix a realization $x$ as in the statement of the proposition. Suppose $\frk h(x) \geq \frk c(x)$, so that $\frk M(x) = \frk h(x) \geq n^{\ol\xi/2}$ (the case where we have the reverse inequality is treated similarly). For $m\in\BB N$, let $J_{x,m}^H$ be the smallest $j \in\BB N$ for which $X(-j,-|x|-1)$ contains $m$ hamburgers and let $L_{x,m}^H := d^*\left(X(-J_{x,m}^H , -|x| -1)\right)$. Let $h_x^0 = 0$ and for $k\in\BB N$, inductively define $h_x^k := \lfloor 2^{-k} \frk h(x) \rfloor + h_x^{k-1}$. For $\zeta >0$, let 
\eqbn
E_x^k := \left\{  \mcl N_{\tb C}\left( X(-J_{x,h_x^k}^H ,-J_{x,h_x^k}^H - 1   )  \right) \geq 2^{-k} \frk h(x)  \right\}  .
\eqen
The proposition is proven by choosing $\zeta>0$ small and stopping at (roughly speaking) the smallest $k$ for which $E_x^k$ occurs. The rest of the argument is very similar to the proofs of Propositions~\ref{prop-wt-pi-exists} and~\ref{prop-end-upper-n}, so we omit the details. 
\end{proof}

\subsection{Regularity for the last segment of the word}
\label{sec-end-reg}

In this subsection we will prove a regularity result for the last segment of the word conditioned on the event $\mcl E_n^l$. This result is needed in the proof of Proposition~\ref{prop-E^l-tight} to show that it is unlikely that the path $Z^n$ of~\eqref{eqn-Z^n-def} moves a large amount during the time interval $[1-\delta ,1]$ (i.e.\ that the purple segment in Figure~\ref{fig-box-event} has huge diameter) when we condition on $\mcl E_n^l$. 
  
\begin{prop} \label{prop-end-reg}
Let $m_n^\delta$ and $\mcl B_n^\delta(C)$ be as in Definition~\ref{def-B_n^delta}. For each $C>1$ and each $q \in (0,1)$, there is an $A > 0$, depending only on $C$, $q$, $\ep_0$, and $\ep_1$ such that for each $\delta>0$, there exists $n_* = n_*(\delta,C,q,\ep_0,\ep_1) \in \BB N$ such that for each $n\geq n_*$, each $l\in \left[\ep_0 n^{\xi/2} , \ep_0^{-1} n^{\xi/2}\right]_{\BB Z}$, and each realization $x$ of $X_{-m_n^\delta} \dots X_{-1}$ for which $\mcl B_n^\delta(C)$ occurs, we have
\eqbn
\BB P\left(\sup_{j\in [m_n^\delta +1 , n]_{\BB Z}} |X(-j,-m_n^\delta-1)| \leq A (\delta n)^{1/2} \,\big|\, X_{-m_n^\delta} \dots X_{-1} = x  ,\, \mcl E_n^l \right) \geq 1-q. 
\eqen
\end{prop}

\begin{lem} \label{prop-good-pi-reg}
Fix $C>1$. Also let $\delta>0$ and let $x$ be a realization of $X_{-m_n^\delta} \dots X_{-1}$ for which $\mcl B_n^\delta(C)$ occurs. For $A > 0$, let
\eqbn
G_n(x; A) :=  \left\{\sup_{j\in [|x|+1, \wt\pi_n]_{\BB Z}} |X(-j,-|x|-1)| \leq A (\delta n)^{1/2} \right\} .
\eqen
For $B>1$, let $H_n(B)$ be the event of Corollary~\ref{prop-good-pi}. There exists $B_0 > 0$ such that for each $q \in (0,1)$ and each $B\geq B_0$, there exists $A = A(q,B,C,\ep_0,\ep_1,\ep_2)$ such that for each $\delta > 0$, 
we can find $n_* = n_*(\delta,A,B,C,\ep_0,\ep_1,\ep_2) \in \BB N$ such that for each $n\geq n_*$ and each realization $x$ as above, we have
\eqbn
\BB P\left( G_n(x;A) \,|\, H_n(B) ,\, X_{-m_n^\delta} \dots X_{-1} = x  \right) \geq 1-q  .
\eqen
\end{lem}
\begin{proof}
Fix a realization $x$ as in the statement of the lemma. By Proposition~\ref{prop-end-lower} and Corollary~\ref{prop-good-pi}, we can find $B   > 1$ such that for each $\delta>0$, there exists and $n_*^0 = n_*^0( \delta, B,C,\ep_0,\ep_1,\ep_2) \in\BB N$ such that for $n\geq n_*^0$ and $l\in \left[\ep_0 n^{\xi/2} , \ep_0^{-1}n^{\xi/2}\right]_{\BB Z}$, we have
\eqbn
\BB P\left(H_n(B) \cap \mcl E_n^l \,|\, X_{-m_n^\delta} \dots X_{-1} = x\right) \succeq \psi_0(\delta n) \psi_2(n^{\xi}) \delta^{-1-\mu} n^{-1-\mu + \xi(\mu-1/2)}
\eqen
with the implicit constants depending only on $B$, $C$, $\ep_0$, $\ep_1$, and $\ep_2$.
By Lemma~\ref{prop-last-increment-upper}, 
\eqbn
\BB P\left(\mcl E_n^l \,|\,  H_n(B) ,\,   X_{-m_n^\delta} \dots X_{-1} = x \right) \preceq n^{-3\xi/2} 
\eqen
with the implicit constant depending only on $B$. It follows that for $n\geq n_*^0$, we have
\begin{align} \label{eqn-H(B)-succeq}
\BB P\left(H_n(B) \,|\, X_{-m_n^\delta} \dots X_{-1} = x\right) &= \frac{\BB P\left(H_n(B) \cap \mcl E_n^l \,|\, X_{-m_n^\delta} \dots X_{-1} = x\right)}{\BB P\left(\mcl E_n^l \,|\,  H_n(B) ,\,   X_{-m_n^\delta} \dots X_{-1} = x \right)} \notag \\
& \succeq \psi_0(\delta n) \psi_2(n^{\xi}) \delta^{-1-\mu} n^{-(1-\xi)(1+\mu)} .
\end{align}
Henceforth fix $B > 1$ for which~\eqref{eqn-H(B)-succeq} holds. 

For $\wh B > 1$, let $\wh E_n(x;\wh B)$ be the event that there is a $j\in [n+1,n+n^\xi]_{\BB Z}$ such that the following is true:
\begin{align} \label{eqn-end-corner-event}
&\mcl N_{\tc H}\left(X(-j,-|x| -1)\right) \in \left[\frk h(x) - \wh Bn^{\xi/2} , \frk h(x) +\wh Bn^{\xi/2}  \right]_{\BB Z}, \notag \\
& \mcl N_{\tc C}\left(X(-j,-|x| -1)\right) \in \left[\frk c(x) - \wh Bn^{\xi/2} , \frk c(x) + \wh Bn^{\xi/2} \right]_{\BB Z}, \notag \\
&\mcl N_{\tb H}\left( X(-j,-|x| -1) \right) + \mcl N_{\tb F}\left( X(-j,-|x| -1) \right)  \leq \wh Bn^{\xi/2} ,\quad \op{and} \notag \\
&  \mcl N_{\tb C}\left( X(-j,-|x| -1) \right) + \mcl N_{\tb F}\left(X(-j,-|x| -1)\right)   \leq \wh Bn^{\xi/2} .
\end{align}
We claim that there exists $\wh B > 1$, depending only on $B$, $C$, $\ep_0$, $\ep_1$, and $\ep_2$ such that if $n$ is chosen sufficiently large, then for any realization $x'$ of $X_{-\wt\pi_n} \dots X_{-m_n^\delta-1}$ for which $H_n(B)$ occurs when $X_{-\wt\pi_n} \dots X_{-1} = x' x$, we have 
\eqb \label{eqn-wh-E-given-x'}
\BB P\left(\wh E_n(x;\wh B) \,|\, X_{-\wt\pi_n} \dots X_{-1} = x' x \right) \succeq 1,
\eqe 
with the implicit constant depending only on $B$, $C$, $\ep_0$, $\ep_1$, and $\ep_2$. To see this, let $\wh F_n(x'x) = \wh F_n(x'x; B')$ be the event that  
\alb
&\sup_{j\in [|x'x| , n+n^\xi]_{\BB Z}} |X(-j,-|x'x|-1)| \leq 2B n^{\xi/2} ,\quad \mcl N_{\tc H}\left(X(-n-n^\xi , -|x'x|-1)\right) \geq B n^{\xi/2} + 2n^\nu  , \quad \op{and} \\
&\mcl N_{\tc C}\left(X(-n-n^\xi , -|x'x|-1)\right) \leq B^{-1} n^{\xi/2} - 2n^\nu .
\ale
By condition~\ref{item-wt-pi-mono} of Proposition~\ref{prop-wt-pi-exists}, we have $|x'x| \geq n-\ep_1^{-1} n^{\xi }$. By \cite[Theorem 2.5]{shef-burger}, it follows that we can find $  n_*^1 = n_*^1( \delta,B,C,\ep_0,\ep_1,\ep_2) \geq n_*^0$ such that for each $n\geq n_*^1$ and each choice of realizations $x'$ and $x$ as above, we have
\eqbn
\BB P\left(\wh F_n(x'x) \,|\, X_{-|x'x| }\dots X_{-1} = x'x\right) \succeq 1, 
\eqen
with the implicit constant depending only on $B$ and $\ep_1$. By definition of $H_n(B)$, if $\wh F_n(x'x)$ occurs and $X_{-|x'x| }\dots X_{-1} = x'x$, then $J \leq n+n^\xi$ and each cheeseburger in $X(-n-n^\xi , -|x'x|-1)$ is matched to a cheeseburger order in $X(-|x'x| ,-1)$. 
By condition~\ref{item-wt-pi-cond} of Proposition~\ref{prop-wt-pi-exists}, we therefore have 
\eqbn
\BB P\left(\wh F_n(x'x) \,|\, X_{-\wt\pi_n}\dots X_{-1} = x'x\right) \succeq 1, 
\eqen
with the implicit constant depending only on $B$ and $\ep_1$. On the other hand, the definition of $H_n(B)$ implies that $\mcl R(x')$ has at most $B  n^{\xi/2} +n^\nu$ hamburger orders plus flexible orders; and between $\frk h(x) - B n^{\xi/2}$ and $\frk h(x) +  n^\nu$ hamburgers. Similar statements hold for cheeseburger orders and hamburger orders. Consequently, if $n$ is chosen sufficiently large (depending only on $B$) then if $\wh F_n(x'x)$ occurs and $X_{-\wt\pi_n}\dots X_{-1} = x'x$, we have that $\wh E_n(x;\wh B)$ occurs for $\wh B$ slightly larger than $2B + 1$. This yields~\eqref{eqn-wh-E-given-x'}. 

Now fix $\alpha >0$ to be chosen later, depending only on $q$, $B$, and $C$. By \cite[Proposition 4.4, Assertion 3]{gms-burger-local} we can find $A>0$, depending only on $\alpha$, $B$, and $C$, such that for each $\delta>0$ there exists $n_*^2 = n_*^2( \alpha,\delta,B,C,\ep_0,\ep_1,\ep_2) \geq n_*^1$ such that for $n\geq n_*^2$ and each realization $x$ as in the statement of the lemma, we have
\eqbn \label{eqn-reg-given-wh-E}
\BB P\left(  G_n(x;A)^c \,|\, \wh E_n(x;\wh B) ,\, X_{-m_n^\delta} \dots X_{-1} =  x\right) \leq \alpha .
\eqen 
By \cite[Proposition 4.4, Assertion 1]{gms-burger-local},
\eqb \label{eqn-bad-wh-E-prob}
\BB P\left(G_n(x;A)^c \cap \wh E_n(x;\wh B)\,|\,   X_{-m_n^\delta} \dots X_{-1} =   x \right) \preceq \alpha \psi_0(\delta n) \psi_2(n^{\xi}) \delta^{-1-\mu} n^{-(1-\xi)(1+\mu)}
\eqe 
with the implicit constants depending only $\wh B$ and $C$. 

By combining~\eqref{eqn-H(B)-succeq},~\eqref{eqn-wh-E-given-x'} and~\eqref{eqn-bad-wh-E-prob}, we infer that for $n\geq n_*^2$,
\alb
&\BB P\left( G_n(x;A)^c  \,|\, H_n(B) ,\, X_{-m_n^\delta} \dots X_{-1} = x\right) \\
&\quad \leq \frac{\BB P\left(G_n(x;A)^c \cap \wh E_n(x;\wh B)\,|\,   X_{-m_n^\delta} \dots X_{-1} =   x \right) }{ \BB P\left(\wh E_n(x;\wh B) \,|\,  G_n(x;A)^c ,\, H_n(B),\,  X_{-m_n^\delta} \dots X_{-1} =   x \right)  \BB P\left(H_n(B) \,|\, X_{-m_n^\delta} \dots X_{-1} = x\right)    } \ 
 \preceq \alpha ,
\ale
with the implicit constants depending only on $B$, $C$, $\ep_0$, $\ep_1$, and $\ep_2$. We now conclude by choosing $\alpha$ smaller than $q$ divided by this implicit constant. 
\end{proof}

\begin{proof}[Proof of Proposition~\ref{prop-end-reg}]
By Corollary~\ref{prop-good-pi}, we can find $B  = B(C,\alpha,\ep_0,\ep_1) > 0$ and $n_*^0 = n_*^0(\delta,C,\alpha,\ep_0,\ep_1) \in \BB N$ such that for $n\geq n_*^0$ and $x,l$ as in the statement of the proposition, we have
\eqb \label{eqn-H_n-given-E}
\BB P\left(H_n(B) \,\big|\, X_{-m_n^\delta} \dots X_{-1} = x  ,\, \mcl E_n^l \right) \geq 1- q/2 , 
\eqe 
with $H_n(B)$ as in Corollary~\ref{prop-good-pi}. Henceforth fix such a $B$. 
 
Fix $\alpha>0$ to be determined later depending on $C$ and $q$. 
By Lemma~\ref{prop-good-pi-reg}, we can find $A_0 > 0$, depending only on $B$, $C$, $\alpha$, $\ep_0$, and $\ep_1$ and $n_*^1  = n_*^1(\delta ,C,\alpha,\ep_0,\ep_1) \geq n_*^0$ such that for $n\geq n_*^1$, 
\eqbn
\BB P\left( \sup_{j\in [m_n^\delta+1, \wt\pi_n]_{\BB Z}} |X(-j,-|x|-1)| \leq A_0 (\delta n)^{1/2}        \,\big|\, H_n(B),\, X_{-m_n^\delta} \dots X_{-1} = x    \right) \geq 1-\alpha .
\eqen
By definition of $\wt\pi_n$, we have $0\leq n-\wt\pi_n \leq \ep_1^{-1} n^{\xi/2}$ on $\mcl E_n^l$. By \cite[Lemma 3.13]{shef-burger}, for any given $A\geq A_0$ we can find $n_*^2  = n_*^2(\delta ,A,C,\alpha,\ep_0,\ep_1) \geq n_*^2$ such that for $n\geq n_*^2$, 
\eqb \label{eqn-G_n-given-E}
\BB P\left( \wh G_n(x;A)      \,\big|\, H_n(B),\, X_{-m_n^\delta} \dots X_{-1} = x    \right) \geq 1- 2\alpha ,
\eqe 
where
\eqbn
\wh G_n(x;A) := \left\{ \sup_{j\in [|x|+1, n]_{\BB Z}} |X(-j,-|x|-1)| \leq A (\delta n)^{1/2}   \right\} .
\eqen

By~\eqref{eqn-H_n-given-E},  
\begin{align} \label{eqn-end-reg-split}
\BB P\left(\wh G_n(x;A)^c \,|\, X_{-m_n^\delta} \dots X_{-1} = x ,\, \mcl E_n^l \right)
 \leq \BB P\left(\wh G_n(x;A)^c \cap H_n(B)   \,|\,    X_{-m_n^\delta} \dots X_{-1} = x ,\, \mcl E_n^l \right) + \frac{q}{2} . 
\end{align}
By Bayes' rule and~\eqref{eqn-G_n-given-E},
\begin{align} \label{eqn-end-reg-bayes}
&\BB P\left(\wh G_n(x;A)^c    \,|\,  H_n(B) ,\,   X_{-m_n^\delta} \dots X_{-1} = x ,\, \mcl E_n^l \right) \notag \\
&\qquad = \frac{  \BB P\left( \mcl E_n^l \,|\,    \wh G_n(x;A)^c,\, H_n(B) , \,       X_{-m_n^\delta} \dots X_{-1} = x   \right) \BB P\left(     \wh G_n(x;A)^c \,|\, H_n(B) , \,       X_{-m_n^\delta} \dots X_{-1} = x   \right)             }{ \BB P\left(\mcl E_n^l \,|\, H_n(B) ,\, X_{-m_n^\delta} \dots X_{-1} = x\right)   } \notag \\
&\qquad \leq \frac{ 2 \alpha \BB P\left( \mcl E_n^l \,|\,    \wh G_n(x;A)^c,\, H_n(B) , \,       X_{-m_n^\delta} \dots X_{-1} = x   \right)          }{ \BB P\left(\mcl E_n^l \,|\, H_n(B) ,\, X_{-m_n^\delta} \dots X_{-1} = x\right)   }  .
\end{align}
By Lemma~\ref{prop-last-increment-upper},
\eqbn
\BB P\left( \mcl E_n^l \,|\,    \wh G_n(x;A)^c,\, H_n(B) , \,       X_{-m_n^\delta} \dots X_{-1} = x   \right) \preceq n^{-3\xi/2}
\eqen
with the implicit constants depending only on $B$, $C$, $\ep_0$, $\ep_1$, and $\ep_2$. By assertion~\ref{item-wt-pi-cond} of Proposition~\ref{prop-wt-pi-exists} together with \cite[Lemma 2.9]{gms-burger-local}, we can find $n_*^3 (\delta ,A,C,\alpha,\ep_0,\ep_1) \geq n_*^2$  such that for $n\geq n_*^3$ and $x,l$ as in the statement of the proposition, 
\eqbn
\BB P\left(\mcl E_n^l \,|\, H_n(B) ,\, X_{-m_n^\delta} \dots X_{-1} = x\right) \succeq n^{-3\xi/2}  ,
\eqen
with the implicit constants depending only on $B$, $C$, $\ep_0$, $\ep_1$, and $\ep_2$.
By~\eqref{eqn-end-reg-bayes},
\eqbn
\BB P\left(\wh G_n(x;A)^c    \,|\,    H_n(B)  ,\,   X_{-m_n^\delta} \dots X_{-1} = x ,\, \mcl E_n^l \right) \preceq \alpha ,
\eqen
with the implicit constants depending only on $B$, $C$, $\ep_0$, $\ep_1$, and $\ep_2$. By choosing $\alpha$ sufficiently small relative to $q$ (and hence $A$ sufficiently large) and combining this with~\eqref{eqn-end-reg-split}, we obtain the statement of the proposition.
\end{proof}

\subsection{Proofs of Propositions~\ref{prop-E^l-abs-cont},~\ref{prop-end-box}, and~\ref{prop-E^l-tight}}
\label{sec-end-properties}

In this subsection we combine the earlier results of this section to prove the main results stated at the beginning of the section. First we deduce Proposition~\ref{prop-E^l-abs-cont} from Propositions~\ref{prop-end-lower} and~\ref{prop-end-upper}. 

\begin{proof}[Proof of Proposition~\ref{prop-E^l-abs-cont}]
By Propositions~\ref{prop-end-lower} and~\ref{prop-end-upper}, we can find $n_*^0 = n_*^0(\delta,C,\ep_0,\ep_1) \in \BB N$ such that for each $n\geq n_*^0$, each $l \in \left[\ep_0 n^{\xi/2} , \ep_0^{-1} n^{\xi/2}\right]_{\BB Z}$, and any two realizations $x$ and $x'$ of $X_{-m_n^\delta} \dots X_{-1}$ for which $\mcl B_n^\delta (C)  $ occurs,
\eqbn
\BB P\left(\mcl E_n^l \,|\, X_{-m_n^\delta} \dots X_{-1} = x \right) \asymp \BB P\left(\mcl E_n^l \,|\, X_{-m_n^\delta} \dots X_{-1} = x' \right)
\eqen
with the implicit constants depending only on $C$, $\ep_0$, and $\ep_1$. The statement of the proposition now follows from Bayes' rule. 
\end{proof}

Next we deduce Proposition~\ref{prop-end-box} from Propositions~\ref{prop-end-lower},~\ref{prop-end-upper-n}, and~\ref{prop-end-upper-max}.

\begin{proof}[Proof of Proposition~\ref{prop-end-box}]
Throughout the proof, we require implicit constants to depend only on $\ep_0$ and $\ep_1$. Fix $q\in (0,1)$. For $n\in\BB N$ and $\delta>0$, let $m_n^\delta := \lfloor (1-\delta) n \rfloor$, as in Definition~\ref{def-B_n^delta}. 
For $n\in\BB N$, let  
\eqbn
F_n := \left\{ \sup_{j\in [n/2,n]_{\BB Z}} \mcl N_{\tb F}\left(X(-j,-1)\right) \leq n^\nu \right\} .
\eqen
By \cite[Corollary 5.2]{gms-burger-cone}, we have $\BB P(F_n) =o_n^\infty(n)$, so for sufficiently large $n\in\BB N$ (depending only on $\ep_0$ and $\ep_1$) it holds for each $l\in \left[\ep_0 n^{\xi/2} , \ep_0^{-1} n^{\xi/2}\right]_{\BB Z}$ that
\eqb \label{eqn-F_n-box-prob}
\BB P\left( F_n \,|\, \mcl E_n^l  \right) \geq 1-\frac{q}{3} . 
\eqe 
 
Now fix $\alpha>0$ to be chosen later, depending only on $q$, $\ep_0$, and $\ep_1$. 
For $\zeta > 0$ and $n\in\BB N$, let 
\eqbn
\ul G_n^\delta(\zeta) := F_n\cap \{J > m_n^\delta\} \cap \left\{  \mcl N_{\tb H} \left(X(-m_n^\delta,-1)\right) \wedge \mcl N_{\tb C}\left(X(-m_n^\delta,-1)\right) \leq \zeta (\delta n)^{1/2} \right\}  .
\eqen
By \cite[Theorem 4.1]{gms-burger-cone}, we can find $\zeta  > 0$, depending only on $\alpha$, such that for each $\delta \in (0,1/2)$ there exists $n_*^0 = n_*^0(\delta, \alpha , \zeta) \in\BB N$ such that for $n\geq n_*^0$, we have
\eqbn
\BB P\left( \ul G_n^\delta(\zeta) \,|\, \ul G_n^\delta(1)  \right) \leq  \alpha ,\qquad \BB P\left( \ul G_n^\delta(1) \setminus \ul G_n^\delta(1/2)  \,|\, \ul G_n^\delta(1)  \right) \succeq 1 .
\eqen 
By Proposition~\ref{prop-end-upper-n}, we can find $n_*^1(\delta,\ep_0,\ep_1) \geq n_*^0$ such that for $n\geq n_*^1$, 
\eqbn
\BB P\left( \mcl E_n^l   \,|\,  \ul G_n^\delta(\zeta)  \right) \preceq   \psi_0\left( \delta n \right) \psi_2(n^\xi)\delta^{-1-\mu} n^{-1-\mu+ \xi(\mu-1/2)}  .
\eqen 
By combining this with Proposition~\ref{prop-end-lower}, possibly increasing $n_*^1$, and applying Bayes' rule, we obtain
\begin{align} \label{eqn-end-box-bayes}
\BB P\left( \ul G_n^\delta(\zeta)  \,|\, \mcl E_n^l \right) &\leq \BB P\left( \ul G_n^\delta(\zeta) \,|\, \mcl E_n^l \cap \ul G_n^\delta(1) \right) \notag\\
&\leq \frac{  \BB P\left( \mcl E_n^l   \,|\,  \ul G_n^\delta(\zeta)   \right) \BB P\left( \ul G_n^\delta(\zeta)  \,|\, \ul G_n^\delta(1) \right)  }{ \BB P\left( \mcl E_n^l   \,|\,  \ul G_n^\delta(1) \setminus \ul G_n^\delta(1/2) \right)  \BB P\left( \ul G_n^\delta(1) \setminus \ul G_n^\delta(1/2)  \,|\, \ul G_n^\delta(1)  \right)    }   
 \preceq \alpha .
\end{align}
By choosing $\alpha$ sufficiently small, depending only on $q$, $\ep_0$, and $\ep_1$ (and hence $\zeta$ sufficiently small), we can arrange that 
\eqb \label{eqn-far-from-bdy-prob}
\BB P\left( \ul G_n^\delta(\zeta)  \,|\, \mcl E_n^l \right) \leq \frac{q}{3}  .
\eqe 

To ensure that the word $X(-m_n^\delta,-1)$ does not contain too many orders when we condition on $\mcl E_n^l$, for $C>1$ define
\eqbn
\ol G_n^\delta(C) := F_n\cap \left\{J > m_n^\delta ,\,  \mcl N_{\tb H} \left(X(-m_n^\delta,-1)\right) \vee \mcl N_{\tb C}\left(X(-m_n^\delta,-1)\right) \geq C (\delta n)^{1/2} \right\}  .
\eqen
By Proposition~\ref{prop-end-upper-max}, for sufficiently large $n$ (depending only on $\delta$, $\ep_0$, and $\ep_1$), we have
\eqbn
\BB P\left(\mcl E_n^l \,|\, \ol G_n^\delta(C) \right) \preceq  \psi_0\left(C^2 \delta  n  \right) \psi_2(n^\xi) C^{-2 -2\mu} \delta^{-1-\mu} n^{ -1-\mu + \xi(\mu-1/2)}   .
\eqen
By the same argument used in~\eqref{eqn-end-box-bayes}, we can find $n_*^2  = n_*^2(\delta,\ep_0,\ep_1) \in\BB N$ such that for $n\geq n_*^2$, we have
\eqbn
\BB P\left(  \ol G_n^\delta(C) \,|\,  \mcl E_n^l  \right) \preceq C^{-2-2\mu + o_C(1)} .
\eqen 
By combining this with~\eqref{eqn-F_n-box-prob} and~\eqref{eqn-far-from-bdy-prob} and choosing $C$ sufficiently large (depending only on $q$, $\ep_0$, and $\ep_1$), we obtain the statement of the proposition with $C \vee \zeta^{-1}$ in place of $C$ and an appropriate choice of $n_* \geq n_*^1 \vee n_*^2$. 
\end{proof}

Finally, we prove Proposition~\ref{prop-E^l-tight} using Propositions~\ref{prop-E^l-abs-cont},~\ref{prop-end-box}, and~\ref{prop-end-reg}.

\begin{proof}[Proof of Proposition~\ref{prop-E^l-tight}]
Let $m_n^\delta $ and $\mcl B_n^\delta(C)$ be as in Definition~\ref{def-B_n^delta}. 
By Lemma~\ref{prop-end-box} together with~\cite[Lemma 2.8]{gms-burger-cone}, we can find $C = C(q,\ep_0,\ep_1)$ such that for each $\delta \in (0,1/2)$, there exists $n_*^0 = n_*^0( C,q,\delta,\ep_0,\ep_1)$ such that for $n\geq n_*^0$ and each $l\in \left[\ep_0 n^{\xi/2} , \ep_0^{-1} n^{\xi/2}\right]_{\BB Z}$, we have 
\eqb \label{eqn-tight-box}
\BB P\left( \mcl B_n^\delta(C)  \,\big|\, \mcl E_n^l    \right) \geq 1-q. 
\eqe 

By Proposition~\ref{prop-end-reg}, there is an $A > 0$, depending only on $C$, $q$, $\ep_0$, and $\ep_1$ such that for each $\delta>0$, there exists $n_*^1 = n_*^1(\delta,C,q,\ep_0,\ep_1) \geq n_*^0$ such that for each $n\geq n_*^1$, each $l\in \left[\ep_0 n^{\xi/2} , \ep_0^{-1} n^{\xi/2}\right]_{\BB Z}$, and each realization $x$ of $X_{-m_n^\delta} \dots X_{-1}$ for which $\mcl B_n^\delta(C)$ occurs, we have
\eqb \label{eqn-tight-sup}
\BB P\left(\sup_{j\in [m_n^\delta +1 , n]_{\BB Z}} |X(-j,-m_n^\delta -1)| \leq A (\delta n)^{1/2} \,\big|\, X_{-m_n^\delta} \dots X_{-1} = x  ,\, \mcl E_n^l \right) \geq 1-q. 
\eqe 
Henceforth assume we have chosen $\delta$ in such a way that $A\delta^{1/2} \leq \alpha/2$. 

By Proposition~\ref{prop-E^l-abs-cont}, by possibly increasing $n_*^1$ we can arrange that for each $l$ as in the statement of the lemma, the conditional law of $X_{-m_n^\delta} \dots X_{-1}$ given $\mcl E_n^l \cap \mcl B_n^\delta(C)$ is mutually absolutely continuous with respect to its conditional law given only $\mcl B_n^\delta(C) $, with Radon-Nikodym derivative bounded above and below by positive constants depending only on $C$, $\ep_0$, and $\ep_1$.  

By \cite[Theorem 4.1]{gms-burger-cone}, the conditional laws of $Z^n(-\cdot)|_{[ 0,  1-\delta  ]}$ given $\mcl B_n^\delta(C)$ converge (in the uniform topology) as $n\rta\infty$ to the law of a correlated Brownian motion $\wh Z^\delta$ with variances and covariances as in~\eqref{eqn-bm-cov} conditioned to stay in the first quadrant until time $1-\delta$ and satisfy $\wh Z^\delta(1-\delta) \in \left[C^{-1} \delta^{1/2} , C \delta^{1/2} \right]^2$ (see Section~\ref{sec-bm-def} for a precise definition of this conditioned Brownian motion). Since $\wh Z^\delta$ is a.s.\ continuous, it follows that we can find $\zeta>0$, depending only on $\delta,C,q,\alpha,\ep_0$, and $\ep_1$ and
$n_*^2 = n_*^2(\delta,C,q,\alpha,\ep_0,\ep_1) \geq n_*^1$ such that for $n\geq n_*^2$, we have
\eqb \label{eqn-tight-main}
\BB P\left(\wt G_n^\delta(\alpha,\zeta) \,|\, \mcl B_n^\delta(C) \cap \mcl E_n^l \right) \geq 1-q
\eqe 
where $\wt G_n^\delta(\alpha,\zeta)$ is the event that for each $s,t\in [0,1-\delta]$ with $|s-t| \leq \zeta$, we have $|Z^n(-s) - Z^n(-t)| \leq \alpha/2$.  

If $\wt G_n^\delta(\alpha,\zeta)$ occurs and the event of~\eqref{eqn-tight-sup} occurs, then $\wt G_n(\alpha,\zeta)$ occurs. It now follows from~\eqref{eqn-tight-box},~\eqref{eqn-tight-sup}, and~\eqref{eqn-tight-main} that for $n\geq n_*^2$, 
\eqbn
\BB P\left(\wt G_n(\alpha,\zeta) \,|\, \mcl E_n^l\right) \geq (1-q)(1-2q) .
\eqen
Since $q$ is arbitrary this proves the first statement of the lemma.

To obtain the second statement of the proposition, fix $t,q \in (0,1)$ and an open set $A\subset [0,\infty)^2$ with zero Lebesgue measure. Suppose we have chosen $C$ as above for our given choice of $q$, $1-\delta > t$, and $n_*^1$ as above for this choice of $\delta$ and $C$. Since the law of $\wh Z^\delta(t)$ (as defined just above) has a bounded density with respect to Lebesgue measure (see \cite[Section 3.2]{shimura-cone} for an explicit formula for this density), for each $q' \in (0,q]$, we can find $\beta = \beta (t,q',C,\delta)$ such that with $\wt{\mcl B}^\delta(C)$ as in~\eqref{eqn-bm-wt-B-def}, we have 
\eqbn
\BB P\left(\wh Z^\delta(t) \in B_{2\beta }(A) \,|\, \wt{\mcl B}^\delta(C)\right) \leq q' .
\eqen
By \cite[Theorem 4.1]{gms-burger-cone}, the conditional law of $Z^n(t)$ given $\mcl B_n^\delta(C)$ converges as $n\rta\infty$ to the conditional law of $\wh Z^\delta(t)$ given $\wt{\mcl B}^\delta(C)$. Therefore, if we choose $q'$ smaller than $q/2$ times the reciprocal of the implicit constant in Proposition~\ref{prop-E^l-abs-cont}, then we can find $n_*^3 = n_*^3( \delta,C,t,q,A,  \ep_0, \ep_1) \geq n_*^1$ such that~\eqref{eqn-E^l-no-hit} holds for each $n\geq n_*^3$ and $l \in \left[\ep_0  n^{\xi/2} , \ep_0^{-1} n^{\xi/2}\right]_{\BB Z}$.  
\end{proof}

\section{Conclusion of the proofs}
\label{sec-conclusion}

In this section we will combine the results of the previous subsections to prove our main results, Theorem~\ref{thm-main} and Theorem~\ref{thm-cone-limit-finite}. 

We will use the following notation. Let $J$ be as in~\eqref{eqn-J-def} and let $\mu$ be as in~\eqref{eqn-cone-exponent}. Fix $\nu \in (1-\mu,1/2)$ and $\xi \in (2\nu , 1)$.

We recall the definitions of $Z^n= (U^n, V^n)$ and of $Z = (U,V)$ from Section~\ref{sec-burger-prelim}. We also recall the definitions of the event $\mcl E_n^l(\ep_1)$ for $\ep_1 >0$ and $n , l \in\BB N$ from Definition~\ref{def-end-event}; and the event $\mcl G_n' = \mcl G_n'(\ep_0,\ep_1)$ from Lemma~\ref{prop-G'-event}. 
For $\ep_0 > 0$, let $\iota_0^n(\ep_0)$ be defined as in Section~\ref{sec-big-loop-setup}.  

This section is structured as follows. In Section~\ref{sec-tightness}, we will prove tightness of the conditional laws of $Z^n|_{[0,2]}$ given $\{X(1,2n) = \emptyset\}$. In Section~\ref{sec-disk-conv}, we will prove a scaling limit result conditioned on the event $\wh{\mcl E}_n^l$ of Definition~\ref{def-end-event}, plus several regularity conditions, for $l$ at distance of order $n^{1/2}$ from 0 (rather than of order $n^{\xi/2}$, as was considered in Section~\ref{sec-end-local}). In Section~\ref{sec-macroscopic-cone}, we will prove that if we condition on $\mcl E_n^l$ for $l \asymp n^{\xi/2}$, then with high conditional probability there is a $j\in [1,n]_{\BB Z}$ with $X_{-j}= \tb F$, $-\phi(-j) - j$ a little bit less than $n$, and $|X(\phi(-j) , -j)| \asymp n^{1/2}$ with the property that the event of Section~\ref{sec-disk-conv} occurs with $X_{\phi(-j)} \dots X_{-j}$ in place of $X_{-n} \dots X_{-1}$. In Section~\ref{sec-main-proof}, we will deduce Theorem~\ref{thm-main} from the results of the preceding subsections together with the results of Section~\ref{sec-big-loop} and Lemma~\ref{prop-bm-excursion} from Appendix~\ref{sec-bm-cond}. In Section~\ref{sec-cone-proof}, we will complete the proof of Theorem~\ref{thm-cone-limit-finite}.

\subsection{Tightness}
\label{sec-tightness}

In this subsection we establish tightness of the conditional law of $Z^n|_{[0,2]}$ given $\{X(1,2n) = \emptyset\}$.

\begin{prop} \label{prop-tight}
For $n\in\BB N$, the family of conditional laws of $Z^n|_{[0,2]}$ given $\{X(1,2n) = \emptyset\}$ is tight in the topology of uniform convergence. Furthermore, if $\wt Z$ is a path whose law is a weak subsequential limit of the conditional laws of $Z^n|_{[0,2]}$ given $\{X(1,2n) = \emptyset\}$, then for each $t\in (0,2)$, the law of $\wt Z(t)$ is absolutely continuous with respect to Lebesgue measure on $[0,\infty)^2$ (so in particular $\BB P(\wt Z(t) \in (0,\infty)^2) = 1$).
\end{prop}
\begin{proof}
For $\zeta > 0$ and $k\in\BB N$, let $G_n^k(\zeta)$ be the event that the following is true. For each $s,t\in[0,2]$ with $|s-t| \leq \zeta$, we have $|Z^n(s) - Z^n(t)| \leq 2^{-k}$. By the Arz\'ela-Ascoli theorem, to prove tightness it suffices to find for each $q\in (0,1)$ a sequence $\zeta_k \rta 0$, independent of $n$, such that for each $n\in\BB N$, 
\eqb \label{eqn-tight-union}
\BB P\left(\bigcup_{k=1}^\infty G_n^k(\zeta_k) \right) \geq 1-q .
\eqe 

To this end, suppose given $k\in\BB N$ and $q\in (0,1)$. Also fix $q_k' > 0$ to be determined later, depending only on $q$ and $k$. By Lemma~\ref{prop-G'-event}, we can find $\ep_0 , \ep_1 \in (0,1)$ depending only on $q_k'$ and $n_*^0 = n_*^0( q_k') \in \BB N$ such that for $n\geq n_*^0$, 
\eqb \label{eqn-tight-G_n'-prob}
\BB P\left(\mcl G_n' \,|\, X(1,2n) =\emptyset \right) \geq 1-q_k' .
\eqe 
where $\mcl G_n' = \mcl G_n'(\ep_0,\ep_1)$ is defined as in Lemma~\ref{prop-G'-event}. 

By Lemma~\ref{prop-E^l-equiv} and Proposition~\ref{prop-E^l-tight}, we can find $  \zeta_k \in \BB N$, depending on $\ep_0$, $\ep_1$, $k$, and $q_k'$; and $n_*^1 = n_*^1( \ep_0,\ep_2,k,q_k')  \geq n_*^0$ such that for each $n\geq n_*^1$ and each $(m,l)$ as above, we have
\eqb \label{eqn-tight-mid}
\BB P\left(\wt G_n^k( \zeta_k)\,|\, \mcl G_n' \right) \geq 1-q_k'
\eqe 
where $\wt G_n^k( \zeta_k)$ is the event that $|Z^n(-s) - Z^n(-t)| \leq 2^{-k-1}$ whenever $s,t\in [\phi(\iota_0^n)/n , \iota_0^n/n]_{\BB Z}$ with $|s-t| \leq  \zeta_k$. 

On the event $\mcl G_n' $, we have $|\phi(\iota_0^n)| \leq   n^{\xi/2}$ and $n-\iota_0^n \leq n^{\xi/2}$. By \cite[Lemma 3.13]{shef-burger} together with \cite[Proposition 2.13]{gms-burger-cone}, we can find $n_*^2  = n_*^2( \ep_0,\ep_2,k,q_k')$ such that the conditional probability given $\{X(1,2n) = \emptyset\}$ that $\mcl G_n'(\ep_0,\ep_1)$ occurs and 
\eqb \label{eqn-bad-ends}
\left( \sup_{j\in [1,\phi(\iota_0^n)]_{\BB Z}} |X(1,j)|\right)  \vee \left( \sup_{j\in [ \iota_0^n , 2n]_{\BB Z}} |X(1,j)|\right) \geq 2^{-k-1} n^{1/2} 
\eqe 
is at most $ q_k'$. If~\eqref{eqn-bad-ends} does not hold and $\wt G_n^k( \zeta_k)$ occurs, then $G_n^k(\zeta_k)$ occurs. Hence~\eqref{eqn-tight-G_n'-prob} and~\eqref{eqn-tight-mid} imply that for $n\geq n_*^2$, 
\eqbn
\BB P\left(G_n^k(\zeta_k) \,|\, X(1,2n) = \emptyset\right) \geq 1 - (1-q_k')^2 - q_k' .
\eqen
Choose $q_k'$ sufficiently small that this quantity is at least $1-q 2^{-k}$. 
Since there are only finitely many $n\leq n_*^2$, by shrinking $\zeta_k$ we can arrange that
\[
\BB P\left(G_n^k( \zeta_k) \,|\, X(1,2n) = \emptyset\right) \geq  1- q 2^{-k} ,\qquad \forall n\in\BB N.
\]
Thus~\eqref{eqn-tight-union} holds.

To obtain the last statement, fix a weak subsequential limit of the conditional laws of $Z^n$ given $\{X(1,2n) = \emptyset\}$, a time $t  \in (0,2)$, and an open set $A\subset [0,\infty)^2$ with zero Lebesgue measure. Let $\wt Z$ be distributed according to our given limiting law. Choose $\ep_0 ,\ep_1 > 0$ and $n_*^0 \in\BB N$ such that~\eqref{eqn-tight-mid} holds with $ q$ in place of $q_k'$. Then choose $\beta =\beta(q,t,A,\ep_0,\ep_1) > 0$ and $n_*^1 = n_*^1(q,t,A,\ep_0,\ep_1)\geq n_*^0$ such that~\eqref{eqn-E^l-no-hit} of Proposition~\ref{prop-E^l-tight} holds with $2\beta$ in place of $\beta$ and $t/2$ in place of $t$. Finally, use \cite[Lemma 3.13]{shef-burger} to choose $n_*^2 = n_*^2(q,t,\beta,\ep_0,\ep_1) \geq n_*^1$ such that for $n \geq n_*^2$, 
\eqbn
\BB P\left( \sup_{s \in [t  -  n^{\xi-1} , t   + n^{\xi-1}]_{\BB Z}} |Z^n( s) - Z^n(t)| > \beta  \,|\,  X(1,2n) = \emptyset \right) \leq q .
\eqen
Then for $n\geq n_*^2$, it holds that
\begin{align}  
&\BB P\left( \op{dist}(Z^n(t) , A) < \beta  \,|\,  X(1,2n) = \emptyset \right) \leq 3 q .
\end{align}
Since $q$ is arbitrary, $\BB P(\wt Z(t) \in A) = 0$. 
\end{proof}

\subsection{Convergence conditioned on an exit position at macroscopic distance from 0}
\label{sec-disk-conv}

In this section we will define a regularized version of the event $\wh{\mcl E}_n^l$ of Definition~\ref{def-end-event} for $l\asymp n^{1/2}$, which we call $\ul{\mcl E}_n^l(\alpha)$. Roughly speaking, we will then prove the following. Suppose we fix $z_1,z_2 \in (0,\infty)^2$ and $0<s_1<s_2<1$. Then for sufficiently large $n$ and sufficiently small $\alpha$, the conditional law of $Z^n|_{[s_1,s_2]}$ given $\ul{\mcl E}_n^l(\alpha)$ and the event that $Z^n(s_1) \approx z_1$ and $Z^n(s_2) \approx z_2$ is close to the law of a Brownian bridge from $z_1$ to $z_2$ conditioned to stay in the first quadrant (see Proposition~\ref{prop-2pt-cond-conv} below for a precise statement). This statement together with Lemma~\ref{prop-bm-excursion} will eventually be used to identify the law of a subsequential limit of the conditional laws of $Z^n$ given $\{X(1,2n) = \emptyset\}$. 

We now proceed to define our events. We recommend that the reader consult Figure~\ref{fig-macro-cone-event} for an illustration before reading the formal definition of the events.

\begin{figure}[ht!]
 \begin{center}
\includegraphics{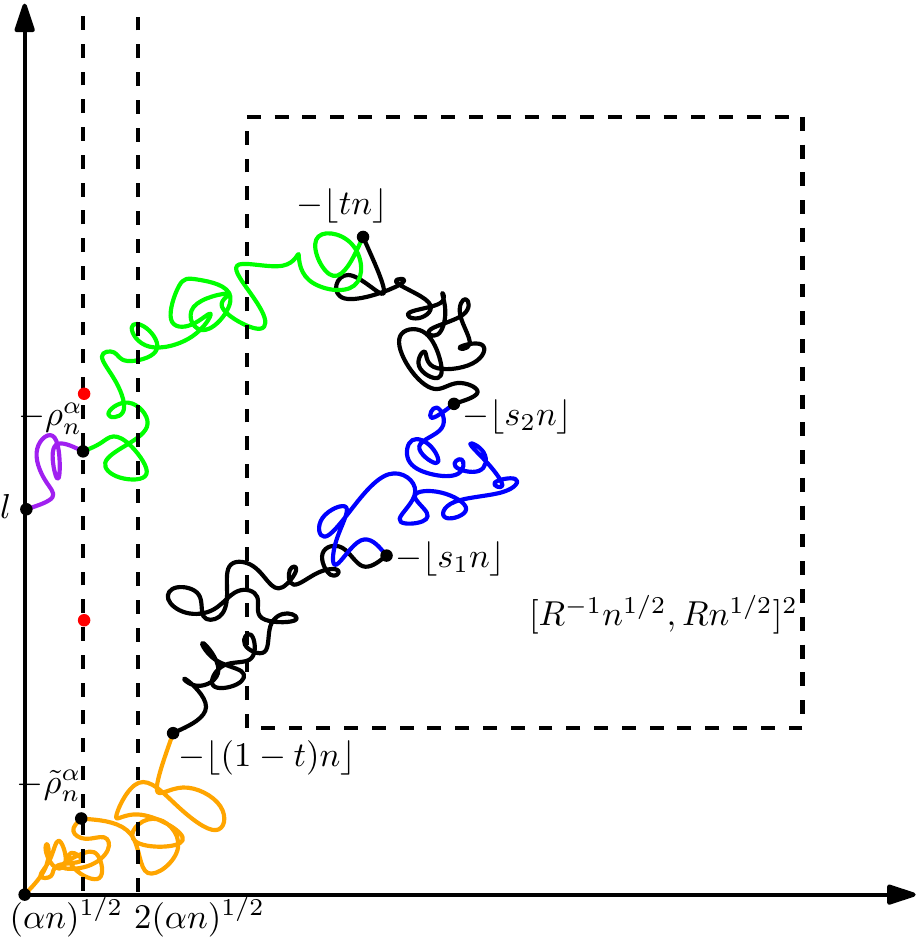} 
\caption{An illustration of the path $D$ of~\eqref{eqn-discrete-path} when the event $\ul{\mcl E}_n^l$ of Section~\ref{sec-disk-conv} occurs. The event $G_n(\alpha)$ concerns the part of the path before the green segment. In particular, the path is required to stay in the first quadrant until time $-\lfloor t n \rfloor$, to stay to the right of the vertical line at distance $2(\alpha n)^{1/2}$ from the origin between times $-\lfloor (1-t) n\rfloor$ and $-\lfloor t n \rfloor$, and to be in the rectangle $\left[R^{-1} n^{1/2} , R n^{1/2} \right]_{\BB Z}^2$ at time $-\lfloor t n \rfloor$; and the path at time $-\wt\rho_n^\alpha$ is required to be within distance at most $R^{-1} n^{1/2}$ of the horizontal axis. The event $\wt H_n^{ l}(\alpha)$ specifies additional regularity conditions for the green part of the path and the time $-\rho_n^\alpha$. The event $H_n^l(\alpha)$ is the intersection of $\wt H_n^l(\alpha)$ and the event that the diameter of the purple segment of the path is at most a small constant times $n^{1/2}$. On the event $ H_n^l(\alpha)$, each cheeseburger in the reduced word corresponding to the purple part of the path is matched to a cheeseburger order. This fact allows us to apply the results of \cite[Section 2]{gms-burger-local} to estimate the conditional probability of the event $\ul{\mcl E}_n^l(\alpha)$ that $  H_n^l(\alpha)$ occurs and the path exits the first quadrant at the position $(0,l)$ (which lies at distance proportional to $n^{1/2}$ from the origin). Proposition~\ref{prop-2pt-cond-conv} is proven by showing that if we condition on $\ul{\mcl E}_n^l(\alpha)$ and the approximate locations of the path at times $-\lfloor t n\rfloor$, $-\lfloor s_1 n \rfloor$, and $-\lfloor s_2 n \rfloor$, then the law of the blue part of the path is approximately that of a correlated Brownian bridge conditioned to stay in the first quadrant. 
}\label{fig-macro-cone-event}
\end{center}
\end{figure}

For $n\in\BB N$, $t \in (1/2,1)$, and $\alpha>0$ let 
\eqb \label{eqn-rho-def}
\rho_n^\alpha = \rho_n^\alpha(t) := J\wedge n \wedge \inf\left\{ j \geq t n \,:\,   \mcl N_{\tb H}\left(X(-j,-1)\right) \leq (\alpha n)^{1/2}  \right\} . 
\eqe
Let
\eqbn
O_n^\alpha = O_n^\alpha(t)   := \mcl N_{\tb H}\left(X(-\rho_n^\alpha(t) ,-1)\right) + \mcl N_{\tb F}\left(X(-\rho_n^\alpha(t) ,-1)\right) + 1 ,\quad Q_n^\alpha = Q_n^\alpha(t)   := \mcl N_{\tb C}\left(X(-\rho_n^\alpha(t) ,-1)\right) .
\eqen 
Also let 
\eqb \label{eqn-wt-rho-def}
 \wt\rho_n^\alpha = \wt\rho_n^\alpha(t) := \lfloor (1-t) n \rfloor \wedge  \sup\left\{ j  \geq 1 \,:\, \mcl N_{\tb H}\left(X(-j,-1)\right) \geq (\alpha n)^{1/2} \right\}  .
\eqe

For $t\in (1/2,1)$, $\alpha>0$, and $R>1$, let $G_n(\alpha)  = G_n(\alpha; t, R)$ be the event that the following is true.
\begin{enumerate}
\item $J > t n $. \label{item-mid-event-size}
\item $  R^{-1} n^{1/2} \leq \mcl N_\theta\left(X(-t n , -1)\right) \leq R n^{1/2}   $ for each $\theta\in \{\tb H,\tb C\}$.  \label{item-mid-event-end}
\item $\inf_{j \in [(1-t) n , t n ]_{\BB Z}} \mcl N_{\tb H}\left(X(-j,-1)\right) \geq 2(\alpha n)^{1/2}$. \label{item-mid-event-inf}
\item $\mcl N_{\tb C}\left( X(-\wt\rho_n^\alpha(t) ,-1) \right) \leq   R^{-1} n^{1/2}$. \label{item-mid-event-start}
\item $\sup_{j \in [1,t n ]_{\BB Z}} \mcl N_{\tb F}\left(X(-j,-1)\right) \leq n^\nu$. \label{item-mid-event-F}
\end{enumerate}
For $t\in (1/2,1)$, $\alpha>0$, $R>1$, and $l \in\BB N$, let $\wt H_n^{ l}(\alpha) = \wt H_n^{ l}(\alpha;t, R)$ be the event that the following is true.
\begin{enumerate}
\item $G_n(\alpha;t ,R)$ occurs. \label{item-rho-event-G}
\item $n - \rho_n^\alpha(t) \in \left[R^{-1} \alpha n , R \alpha n \right]_{\BB Z}$. \label{item-rho-event-n} 
\item $Q_n^\alpha(t) \in \left[l - R (\alpha n)^{1/2} , l + R (\alpha n)^{1/2} \right]$. \label{item-rho-event-C}
\end{enumerate}
Let 
\eqb \label{eqn-rho-event-sup}
 H_n^l(\alpha) =  H_n^{ l}(\alpha;t ,R) := \wt H_n^{ l}(\alpha;t ,R) \cap \left\{ \mcl N_{\tc C}\left(X(-J , -\rho_n^\alpha(t)\right) \leq R^{-1} n^{1/2} \right\}  .
\eqe 
Let $\wh{\mcl E}_n^l$ be the event that $J = n$, $X(-J) = \tc H$, and $|X(-J,-1)| =l$, as in Definition~\ref{def-end-event}; and let
\eqb \label{eqn-ul-E^l-def}
\ul{\mcl E}_n^{l}(\alpha)  = \ul{\mcl E}_n^{l}(\alpha;t, R) := \wh{\mcl E}_n^l \cap  H_n^{ l}(\alpha;t, R).
\eqe

Our reason for including several of the conditions in the definitions of the events of this subsection is to make the following lemma true. 

\begin{lem} \label{prop-rho-match}
Define the time $\rho_n^\alpha = \rho_n^\alpha(t)$ and the events as in the beginning of this subsection. Fix $t \in (1/2,1)$ and $R>1$. There exists $\alpha_* = \alpha_*( R)$  such that for each $\alpha \in (0,\alpha_*]$, there exists $n_* = n_*(\alpha ,  R)\in\BB N$ such that for $\alpha \in (0,\alpha_*]$, $n\geq n_*$, and $l \geq 3 R^{-1} n^{1/2}$, the following holds. On the event $H_n^{ l}(\alpha) = H_n^{ l}(\alpha;t,R)$, each cheeseburger in $X(-J,-\rho_n^\alpha -1)$ is matched to a cheeseburger order in $X(-\rho_n^\alpha ,-1)$. 
\end{lem}

Lemma~\ref{prop-rho-match} will allow us to apply the estimates of \cite[Section 2]{gms-burger-local} to the word $X_{-J} \dots X_{-\rho_n^\alpha-1}$, without having to worry about the matches of $\tb F$'s. 

\begin{proof}[Proof of Lemma~\ref{prop-rho-match}]
Assume $H_n^{ l}(\alpha)$ occurs. 
By definition of $\rho_n^\alpha$, the word $X(\phi(-\rho_n^\alpha ) , -1)$ contains at most $\alpha n^{1/2} $ hamburger orders. By condition~\ref{item-mid-event-inf} in the definition of $G_n(\alpha)$ together with the definitions of $\rho_n^\alpha$ and $\wt\rho_n^\alpha$, it follows that $\rho_n^\alpha$ is a hamburger and $-\phi(-\rho_n^\alpha) = \wt\rho_n^\alpha$. By condition~\ref{item-mid-event-start} in the definition of $G_n(\alpha)$, we have $\mcl N_{\tb C}\left( X(-\wt\rho_n^\alpha , -1) \right) \leq R^{-1} n^{1/2}$. If we choose $\alpha $ small enough that $R\alpha^{1/2} \leq R^{-1}$ and $n$ sufficiently large depending on $\alpha$, then for $l\geq 3 R^{-1} n^{1/2}$ we have by condition~\ref{item-rho-event-C} in the definition of $\wt H_n^{ l}(\alpha)$ that
\eqbn
\mcl N_{\tb C}\left(X(-\rho_n^\alpha , -\wt\rho_n^\alpha )\right) \geq Q_n^\alpha  - \mcl N_{\tb C}\left( X(-\wt\rho_n^\alpha , -1) \right) \geq R^{-1} n^{1 /2} .
\eqen
The word $X(-\rho_n^\alpha , -\wt\rho_n^\alpha )$ contains no flexible orders. 
By the definition~\eqref{eqn-rho-event-sup} of $H_n^{ l}(\alpha)$, we infer that each cheeseburger in $X(-J,-\rho_n^\alpha-1)$ is matched to a cheeseburger order in $X(-\rho_n^\alpha,\phi(-\rho_n^\alpha)  )$, as required.
\end{proof}

The main result of this subsection is the following proposition. 
  
\begin{prop} \label{prop-2pt-cond-conv}
Fix $t\in (1/2,1)$, $R>1$, $b \in (3R^{-1} ,1)$, $q\in (0,1)$, $s_1 < s_2 \in (1-t,t)$, and $z_1,z_2 \in (0,\infty)^2$. For $\zeta>0$ and $n\in\BB N$, let $\mcl P_n^{s_1,s_2}(z_1,z_2; \zeta)$ be the event that the following is true:
\eqb \label{eqn-mclP-event}
 |Z^n(s_1) - z_1| \leq \zeta n^{1/2} ,\quad |Z^n(s_2) - z_2| \leq \zeta n^{1/2} ,\quad \op{and} \quad Z^n([s_1,s_2]) \subset (0,\infty)^2  .
\eqe  
Also define the objects and events as in the beginning of this subsection. 
There exists $\zeta_*   > 0$ (depending on the above parameters) such that for each $\zeta \in (0,\zeta_*]$ there exists $\alpha_* > 0$ (depending on $\zeta$ and the above parameters) such that for each $\alpha \in (0,\alpha_*]$, there exists $n_*\in\BB N $ (depending on $\alpha$, $\zeta$, and the above parameters) with the following property. For each $n\geq n_*$ and each $l\in [b n^{1/2} , b^{-1} n^{1/2}]_{\BB Z}$, the Prokhorov distance between the conditional law of $Z^n(-\cdot)|_{[s_1,s_2]}$ given $ \mcl P_n^{s_1,s_2}(z_1,z_2; \zeta) \cap \ul{\mcl E}_n^l(\alpha)$ and the law of a two-dimensional Brownian bridge from $z_1$ to $z_2$ in time $s_2-s_1$ with variances and covariances as in~\eqref{eqn-bm-cov} conditioned to stay in the first quadrant is at most $q$. 
\end{prop}

The idea of the proof of Proposition~\ref{prop-2pt-cond-conv} is to show that if we are given two realizations $x$ and $x'$ of $X_{-\lfloor t n \rfloor} \dots X_{-1}$ for which $G_n(\alpha)$ occurs and the corresponding realizations of $Z^n(t)$ are close together, then the conditional probabilities of $\ul{\mcl E}_n^l(\alpha)$ given $\{X_{-\lfloor t n \rfloor} \dots X_{-1} = x\}$ and $\{X_{-\lfloor t n \rfloor} \dots X_{-1} = x'\}$ are close together. We then apply Bayes' rule to invert the conditioning, and finally condition on $\mcl P_n^{s_1,s_2}(z_1,z_2; \zeta)$ for $s_1 < s_2 \in (1-t,t)$.

\begin{lem} \label{prop-rho-abs-cont}
Define the times $\rho_n^\alpha = \rho_n^\alpha(t)$ and $\wt\rho_n^\alpha = \wt\rho_n^\alpha(t)$ and the events as at the beginning of this subsection. Fix $t \in (1/2,1)$ and $R>1$. For each $q\in (0,1)$, there exists $\alpha_* = \alpha_*(q, R)>0$ and $\zeta_0 = \zeta_0(q, R) > 0$ such that for each $\alpha \in (0,\alpha_*]$, there exists $n_* = n_*(\alpha,q, R)$ with the following property. Let $\alpha \in (0,\alpha_*]$, $n\geq n_*$, and $l  \geq  3 R^{-1} n^{1/2}$. Suppose given two such realizations $x$ and $x'$ for which (in the notion of~\eqref{eqn-theta-count-reduced}),
\eqb \label{eqn-rho-abs-cont-close}
\left| |x | - |x '| \right| \leq 2\zeta_0  \alpha n ,\quad \left| \frk o(x) - \frk o(x') \right| \leq 2 \zeta_0 (\alpha n)^{1/2} , \quad \left| \frk c(x) - \frk c(x') \right| \leq 2\zeta_0 (\alpha n)^{1/2}  .
\eqe 
Then
\eqb \label{eqn-rho-abs-cont}
1-q \leq \frac{ \BB P\left(  \ul{\mcl E}_n^l(\alpha) \,|\, X_{-\rho_n^\alpha} \dots X_{-1} = x \right) }{ \BB P\left(  \ul{\mcl E}_n^l(\alpha) \,|\, X_{-\rho_n^\alpha} \dots X_{-1} = x' \right) } \leq (1-q)^{-1}  .
\eqe 
\end{lem}
\begin{proof}
Let $x$ be a realization of $X_{-\rho_n^\alpha} \dots X_{-1}$ for which $\wt H_n^{ l}(\alpha)$ occurs. By definition of $\wt H_n^l(\alpha)$, if $n$ is chosen large enough that $\alpha n^{1/2} \geq n^\nu$, then
\eqbn
\left(n-|x| , \frk o(x) ,  l -\frk c(x) \right) \in \left[ R^{-1} \alpha n ,  R \alpha n \right]_{\BB Z} \times \left[\frac12 (\alpha n)^{1/2} ,  2(\alpha n)^{1/2} \right]_{\BB Z} \times \left[-R(\alpha n)^{1/2} , R(\alpha n)^{1/2} \right]_{\BB Z} .
\eqen
For $m\in\BB N$, let $J_{x,m}^H$ be the smallest $j\geq |x|+1$ for which $X(-j,-|x|-1)$ contains $m$ hamburgers and let $L_{x,m}^H = d^*(X(-J_{x,m}^H ,-|x|-1))$. 
Lemma~\ref{prop-rho-match} implies that we can find $\wt\alpha_* = \wt\alpha_*( R)$ such that for each $\alpha \in (0,\alpha_*]$ there exists $\wt n_* = \wt n_*(\alpha ,  R)$ such that if $n\geq n_*$ and $l\geq 3R^{-1} n^{1/2}$, then on the event $ \ul{\mcl E}_n^l(\alpha)$ each cheeseburger in $X(-J,-\rho_n^\alpha-1)$ is matched to a cheeseburger order in $X(-\rho_n^\alpha,-1)$. Hence if $X_{-\rho_n^\alpha} \dots X_{-1} = x$, then $\ul{\mcl E}_n^l(\alpha)$ occurs if and only if the following is true:
\alb
(J_{x,\frk o(x)}^H , L_{x,\frk o(x)}^H) = (n-|x| , l - \frk c(x) ) \quad \op{and} \quad 
\sup_{i \in [|x|+1,J_{x,\frk o(x)}^H ]_{\BB Z}} |X(-i,-|x|-1)| \leq R^{-1} n^{1/2} .
\ale
By \cite[Proposition 2.2 and Lemma~2.9]{gms-burger-local},
\eqbn
\BB P\left(  \ul{\mcl E}_n^l(\alpha) \,|\, X_{-\rho_n^\alpha} \dots X_{-1} = x \right) = (1+o_\alpha(1)) \frk o(x)^{-3} \left(g\left(\frac{n-|x|  }{\frk o(x)^2} , \frac{l - \frk c(x)}{\frk o(x)}\right) + o_{\alpha n}(1)\right) ,
\eqen
where
\eqbn
g(t,v) =   a_0 t^{-2} \exp\left(-\frac{a_1 + a_2  (v+a_3)^2  }{ t} \right) 
\eqen
for constants $a_0,a_1,a_2 ,a_3 > 0$ depending only on $p$; 
the $o_\alpha(1)$ is uniform in the choice of $n$ and $l$ as in the statement of the lemma;
and the $o_{\alpha n}(1)$ is depends on $\alpha$ and $n$ only through $\alpha n$. 
The function $(s ,u,v) \mapsto \alpha^{3/2} u^{-3} g\left(\frac{ s^2 }{u^2  }, \frac{v}{u}\right)$ is continuous on $\left[R^{-1} \alpha  , R\alpha  \right] \times \left[(1/2) \alpha^{1/2} , 2\alpha^{1/2}  \right] \times \left[-R \alpha^{1/2} ,R\alpha^{1/2}  \right]$, with modulus of continuity depending only on $R$; and is bounded above and below on this set by finite positive constants depending only on $R$ and $b$. 

It follows that we can find $\alpha_* = \alpha_*(q,  R) \geq \wt\alpha$ such that for $\alpha \in (0,\alpha_*]$, there exists $n_* = n_*(\alpha, b,R) \geq \wt n_*$ such that for each $\alpha \in (0, \alpha_*]$, each $n\geq n_*$, and each $l$ and $x$ as above,
\eqb \label{eqn-P-to-g}
1-q \leq \frac{\BB P\left(  \ul{\mcl E}_n^l(\alpha) \,|\, X_{-\rho_n^\alpha} \dots X_{-1} = x \right)}{\frk o(x)^{-3}g\left(\frac{n-|x|  }{\frk o(x)^2} , \frac{l - \frk c(x)}{\frk o(x)}\right)    } \leq (1-q)^{-1} .
\eqe 
Furthermore, we can find $\zeta_0 > 0$, depending only on $R$ and $b$, such that for any $\alpha \in (0,\alpha_*]$, $n\geq n_*$, $l\geq 3 R^{-1} n^{1/2}$, and any two realizations $x$ and $x'$ of $X_{-\rho_n^\alpha} \dots X_{-1}$ for which $\wt H_n^l(\alpha)$ occurs and~\eqref{eqn-rho-abs-cont-close} holds,
we have
\eqb \label{eqn-g-to-g'}
1-q \leq \frac{   \frk o(x)^{-3}g\left(\frac{n-|x|  }{\frk o(x)^2} , \frac{l - \frk c(x)}{\frk o(x)}\right)   }{   \frk o(x')^{-3}g\left(\frac{n-|x'|  }{\frk o(x')^2} , \frac{l - \frk c(x')}{\frk o(x')}\right)   } \leq (1-q)^{-1} .
\eqe
By combining~\eqref{eqn-P-to-g} and~\eqref{eqn-g-to-g'}, we obtain~\eqref{eqn-rho-abs-cont} with $(1-q)^2$ in place of $1-q$. Since $q$ is arbitrary we conclude.
\end{proof}

For our next lemma, we will consider the following event. 
For $(h,c) \in \BB N^2$, let
\eqb \label{eqn-mid-box-event}
B_n^{ h,c}(\zeta) := \left\{\left|\mcl N_{\tb H}\left(X(-t n,-1)\right) - h \right| \leq \zeta n^{1/2} , \:   \left|\mcl N_{\tb C}\left(X(-t n,-1)\right) - c \right| \leq \zeta n^{1/2}     \right\}  .
\eqe

\begin{lem} \label{prop-mid-abs-cont}
Define the time $\rho_n^\alpha = \rho_n^\alpha(t)$ and the events as at the beginning of this subsection and the event $B_n^{h,c}(\zeta)$ as in~\eqref{eqn-mid-box-event}. Fix $t \in (1/2,1)$ and $R>1$. For each $q\in (0,1)$, there exists $\alpha_* = \alpha_*(q,t, R)  >0$ such that for each $\alpha \in (0,\alpha_*]$, there exists $\zeta_*  = \zeta_*(\alpha,q,t, R) > 0$ such that for each $\zeta \in (0,\zeta_*]$, there exists $n_* = n_*(\zeta,\alpha,q,t, R)$ with the following property. 
For each $n \geq n_*$, $l\in \left[3R^{-1} n^{1/2} , 3R n^{1/2}\right]_{\BB Z}$, and $(h,c) \in  \left[R^{-1}   n^{1/2} , R n^{1/2}\right]_{\BB Z}^2$, the conditional law of $X_{-\lfloor t n \rfloor- 1} \dots X_{-1}$ given $ B_n^{ h,c}( \zeta) \cap  \ul{\mcl E}_n^{ l}(\alpha) $ is absolutely continuous with respect to its conditional law given only $B_n^{ h,c}( \zeta) \cap G_n(\alpha)$, with Radon-Nikodym derivative bounded below by $1-q$ and above by $(1-q)^{-1}$.  
\end{lem}
\begin{proof}
Let $\alpha_* = \alpha_*(q,R)$ and $\zeta_0 = \zeta_0(q, R)$ be chosen so that the conclusion of Lemma~\ref{prop-rho-abs-cont} holds with our given $q$, $t$, and $R$. Henceforth fix $\alpha\in (0,\alpha_*]$ and let $n_*^0 = n_*^0(\alpha,q, R)$ be as in Lemma~\ref{prop-rho-abs-cont} for this choice of $\alpha$. 

Given $n\in\BB N$ and $l \in \left[3R^{-1} n^{1/2} , 3R n^{1/2}\right]_{\BB Z}$, choose a finite collection $\mcl A_n^l  $ of disjoint subsets of $\BB Z^3$ such that 
\eqbn
\bigsqcup_{A\in\mcl A_n^l} A = \left[   n - R \alpha n , n - R^{-1} \alpha n  \right]_{\BB Z} \times \left[R^{-1} (\alpha n)^{1/2} ,  2(\alpha n)^{1/2} \right]_{\BB Z} \times \left[l-R(\alpha n)^{1/2} , l+R(\alpha n)^{1/2} \right]_{\BB Z}   
\eqen
and each $A\in \mcl A_n^l$ is a box in the integer lattice (i.e.\ a product of three discrete invervals) with dimensions at least (resp. at most) $ \frac18 \zeta \alpha n \times \frac18 \zeta (\alpha n)^{1/2} \times \frac18 \zeta(\alpha n)^{1/2}$ (resp. $ \frac14 \zeta \alpha n \times \frac14 \zeta (\alpha n)^{1/2} \times \frac14 \zeta(\alpha n)^{1/2}$).

Fix $q' \in (0,1)$, to be chosen later. 
Let $Z = (U,V)$ be a correlated two-dimensional Brownian motion as in~\eqref{eqn-bm-cov} and for $r_1,r_2 > 0$, let $\tau^\alpha(r_1,r_2)$ be the smallest $s \in [0,1-t]$ for which either $U(s) \leq -r_1 + \alpha$ or $V(s) \leq -r_2$ (or $\tau^\alpha(r_1,r_2) = 1-t$ if no such $s$ exists).   

It follows from \cite[Theorem 2.5]{shef-burger} and conditions~\ref{item-mid-event-end} and~\ref{item-mid-event-F} in the definition of $G_n(\alpha)$ that we can find $n_*^1 = n_*^1(q',\alpha, q, R) \geq n_*^0$ such that the following is true. Let $n\geq n_*^1$, $l \in \left[3R^{-1} n^{1/2} , 3R n^{1/2}\right]_{\BB Z}$, and let $x$ be any realization of $X_{-\lfloor t n \rfloor} \dots X_{-1}$ for which $G_n(\alpha)$ occurs. 
Then the Prokhorov distance between the conditional law given $\{X_{-\lfloor t n \rfloor} \dots X_{-1} = x\}$ of the pair 
\eqbn
\left( \left(Z^n(-\cdot - t) - Z^n(t) \right)|_{[0,1-t]} , n^{-1} \rho_n^\alpha - t\right) 
\eqen
and the law of the pair
\eqbn
\left(\wh Z|_{[0,1-t]} , \tau^\alpha\left( \frac{\frk h(x)}{n^{1/2}} , \frac{ \frk c(x) }{n^{1/2}}  \right) \right)
\eqen
is at most $q' $. Furthermore, by possibly increasing $n_*^1$ we can arrange that for each $n\geq n_*^1$, each $l,x$ as above, and each $A\in \mcl A_n^l$,
\eqbn
\BB P\left( (\rho_n^\alpha   , O_n^\alpha , Q_n^\alpha) \in A  \,|\,  X_{-\lfloor t n \rfloor} \dots X_{-1} = x   \right) \succeq 1
\eqen
with the implicit constant depending only on $\alpha$, $\zeta_0$, $t$, and $R$. 

It is clear that the law of $\tau^\alpha(r_1,r_2)$ depends continuously on $r_1$ and $r_2$. Hence we can find $\zeta_* = \zeta_*(\alpha,\zeta_0,q, R)$ with the following property. If $\zeta \in (0,\zeta_*]$ and we choose $q'$ sufficiently small, depending on $\alpha, q ,\zeta $, and $R$, then for each $n\geq n_*^1$; each $l \in \left[3R^{-1} n^{1/2} , 3R n^{1/2}\right]_{\BB Z}$; and any two realizations $x$ and $x'$ of $X_{-\lfloor t n \rfloor} \dots X_{-1}$ for which $G_n(\alpha)$ occurs, $|\frk h(x) - \frk h(x')| \leq 2\zeta n^{1/2}$, and $|\frk c(x) - \frk c(x') |\leq 2\zeta n^{1/2}$, it holds for each $A \in \mcl A_n^l$ that
\eqb \label{eqn-mid-to-rho-cont}
1-q \leq \frac{\BB P\left( (\rho_n^\alpha   , O_n^\alpha , Q_n^\alpha) \in A  \,|\,  X_{-\lfloor t n \rfloor} \dots X_{-1} = x   \right) }{\BB P\left( (\rho_n^\alpha   , O_n^\alpha , Q_n^\alpha) \in A  \,|\,  X_{-\lfloor t n \rfloor} \dots X_{-1} = x '  \right) } \leq (1-q)^{-1} .
\eqe 
Since we chose $\alpha_*$ and $\zeta_0$ so that the conclusion of Lemma~\ref{prop-rho-abs-cont} holds, we also have
\eqb \label{eqn-rho-to-E^l-cont}
1-q \leq \frac{\BB P\left( \ul{\mcl E}_n^l(\alpha)  \,|\,  (\rho_n^\alpha   , O_n^\alpha , Q_n^\alpha) \in A ,\,  X_{-\lfloor t n \rfloor} \dots X_{-1} = x   \right) }{\BB P\left(  \ul{\mcl E}_n^l(\alpha)     \,|\,  (\rho_n^\alpha   , O_n^\alpha , Q_n^\alpha) \in A ,\,  X_{-\lfloor t n \rfloor} \dots X_{-1} = x '  \right) } \leq (1-q)^{-1} .
\eqe 
By multiplying~\eqref{eqn-mid-to-rho-cont} and~\eqref{eqn-rho-to-E^l-cont}, then summing over all $A\in\mcl A_n^l$, we obtain
\eqb \label{eqn-mid-cont-compare}
(1-q)^2 \leq \frac{\BB P\left(\ul{\mcl E}_n^l(\alpha)  \,|\,   X_{-\lfloor t n \rfloor} \dots X_{-1} = x\right)}{   \BB P\left(\ul{\mcl E}_n^l(\alpha)  \,|\,   X_{-\lfloor t n \rfloor} \dots X_{-1} = x ' \right)     } \leq (1-q)^{-2} .
\eqe 

Now suppose we are given $(h,c) \in \left[R^{-1} n^{1/2} , R  n^{1/2}\right]_{\BB Z}^2$ and a realization $x$ as above for which $B_n^{h,c}(\zeta) \cap G_n(\alpha)$ occurs. By averaging~\eqref{eqn-mid-cont-compare} over all realizations $x'$ for which $B_n^{h,c}(\zeta) \cap G_n(\alpha)$ occurs, we obtain
\eqb \label{eqn-mid-cont-compare'}
(1-q)^2 \leq \frac{\BB P\left(\ul{\mcl E}_n^l(\alpha)  \,|\,   X_{-\lfloor t n \rfloor} \dots X_{-1} = x\right)}{   \BB P\left(\ul{\mcl E}_n^l(\alpha)  \,|\,        B_n^{h,c}(\zeta) \cap G_n(\alpha) \right)     } \leq (1-q)^{-2} .
\eqe 
By Bayes' rule,
\alb
& \BB P\left(     X_{-\lfloor t n \rfloor} \dots X_{-1} = x \,|\,     B_n^{h,c}(\zeta) \cap \ul{\mcl E}_n^l(\alpha)   \right) \\
 &\qquad = \frac{  \BB P\left(\ul{\mcl E}_n^l(\alpha)  \,|\,   X_{-\lfloor t n \rfloor} \dots X_{-1} = x\right) \BB P\left(   X_{-\lfloor t n \rfloor} \dots X_{-1} = x \,|\,  B_n^{h,c}(\zeta) \cap G_n(\alpha) \right)  }{   \BB P\left(\ul{\mcl E}_n^l(\alpha)  \,|\,        B_n^{h,c}(\zeta) \cap G_n(\alpha) \right)   } .
\ale
By combining this with~\eqref{eqn-mid-cont-compare'} we obtain the statement of the lemma with $(1-q)^2$ in place of $1-q$. Since $q$ is arbitrary we conclude. 
\end{proof}

For our next lemma we will need analogues of some of the above events for a correlated two-dimensional Brownian motion. Namely, for $t \in (1/2,1)$ let $\wh Z^t = (\wh U^t , \wh V)$ be a two-dimensional Brownian motion with variances and covariances as in~\eqref{eqn-bm-cov} conditioned to stay in the first quadrant until time $t$. 
For $t \in (1/2, 1)$ and $R>1$ let $G(\alpha) = G(\alpha;t,R)$ be the event that the following is true.
\begin{enumerate}  
\item $\inf_{s \in [(1-t)  , t ] } \wh U(t) \geq 2\alpha^{1/2}$. \label{item-bm-mid-event-inf}
\item Let $\wt\sigma^{t,\alpha}$ be the largest $s\in [0,1-t]$ for which $\wh U(s) \geq\alpha^{1/2}$, or $\wt\sigma^{t,\alpha}=0$ if no such $s$ exists. Then $\wh V(s) \leq   R^{-1} $. \label{item-bm-mid-event-start} 
\end{enumerate}
For $(u,v)\in (0,\infty)^2$ and $\zeta>0$, let 
\eqb \label{eqn-bm-mid-box-event}
\wh B^{u,v}(\zeta) = \wh B^{u,v}(\zeta;t) := \left\{|\wh Z(t) - u| \leq \zeta ,\, |\wh Z(t) - v|\leq \zeta     \right\} .
\eqe 
 
\begin{lem} \label{prop-box-cond-conv}
Define the event $\ul{\mcl E}_n^l(\alpha)$ as in~\eqref{eqn-ul-E^l-def}, the event $B_n^{h,c}(\zeta)$ as in~\eqref{eqn-mid-box-event}, and the Brownian motion events as above.
Suppose given $t\in (1/2,1)$, $R>1$, $b \in (3R^{-1} ,1)$, and $q\in (0,1)$. There exists $\alpha_* = \alpha_*(q,t,b,R)  >0$ such that for each $\alpha \in (0,\alpha_*]$, there exists $\zeta_*  = \zeta_*(\alpha,q,t,b, R) > 0$ such that for each $\zeta \in (0,\zeta_*]$ there exists $n_* = n_*(\zeta,\alpha,q,t,b, R)$ with the following property. For each $n\geq n_*$, each $l\in [b n^{1/2} , b^{-1} n^{1/2}]_{\BB Z}$, and each $(h,c) \in [R^{-1} n^{1/2} , R n^{1/2}]_{\BB Z}$, the Prokhorov distance between the conditional law of $Z^n(-\cdot)|_{[0,t]}$ given $  B_n^{h,c}(\zeta) \cap \ul{\mcl E}_n^l(\alpha) $ and the conditional law of $\wh Z|_{[0,t]}$ given $\wh B^{h/n^{1/2} , c/n^{1/2} }(\zeta) \cap G(\alpha)$ is at most $q$. 
\end{lem}
\begin{proof}
For each $  \alpha , \zeta > 0$, we have
\eqbn
\inf_{(u,v) \in [R^{-1} , R]^2} \BB P\left(\wh B^{u,v}(\zeta) \cap G(\alpha) \right) > 0 .
\eqen
By \cite[Lemma 2.8 and Theorem 4.1]{gms-burger-cone}, for each $\alpha,\zeta>0$ and $q\in (0,1)$ there exists a positive integer $n_* = n_*(\alpha,\zeta,q,t,R)  $ such that for each $n\geq n_*$ and each $(h,c) \in [R^{-1} n^{1/2} , R n^{1/2}]_{\BB Z}^2$, the Prokhorov distance between the conditional law of $Z^n(-\cdot)|_{[0,t]}$ given $ B_n^{h,c}(\zeta) \cap G_n(\alpha)$ and the conditional law of $\wh Z|_{[0,t]}$ given $\wh B^{h/n^{1/2} , c/n^{1/2} }(\zeta) \cap G (\alpha)$ is at most $q$. Note that $Z^n(-\cdot)|_{[0,t]}$ is not determined by $X_{-\lfloor t n \rfloor} \dots X_{-1}$ due to the presence of flexible orders in $X(-tn,-1)$; however, $n^{-1/2}$ times the number of such flexible orders tends to zero in law as $n\rta\infty$ by \cite[Lemma 2.8]{gms-burger-cone}. 
We conclude by combining this with Lemma~\ref{prop-mid-abs-cont}. 
\end{proof}

\begin{proof}[Proof of Proposition~\ref{prop-2pt-cond-conv}]
For $z_1,z_2\in (0,\infty)^2$, let $\ul{\BB P}^{z_1,z_2}_{s_2-s_1}$ be the law of a two-dimensional Brownian bridge from $z_1$ to $z_2$ in time $s_2-s_1$ with variances and covariances as in~\eqref{eqn-bm-cov} conditioned to stay in the first quadrant. 

Since $1-t  < s_1$, it follows from the Markov property and the definition of $G(\alpha)$ that for each $z_1,z_2\in (0,\infty)^2$, each $(u,v) \in \left[R^{-1} , R\right]$, and each $\wt\zeta > 0$, the regular conditional law of $\wh Z|_{[s_1,s_2]}$ given $\{\wh Z(s_1) =z_1 ,\, \wh Z(s_2)=z_2,\, \wh Z(t)  = (u,v)\}$ and the event $  G(\alpha)$ is given by
\eqb \label{eqn-bridge-cond}
\ul{\BB P}^{z_1,z_2}_{s_2-s_1}\left(\cdot \,|\, \inf_{s \in [0,s_2-s_1]} U(s) \geq 2\alpha^{1/2} \right).
\eqe  
As $\zeta \rta 0$, the regular conditional law of $\wh Z|_{[s_1,s_2]}$ given $\{\wh Z(t)  = (u,v)\}$ and the event $\{|\wh Z(s_1) - z_1 | \vee |\wh Z(s_2)- z_2| \leq \zeta \} \cap G(\alpha)$ converges to its regular conditional law given $\{\wh Z(t)  = (u,v) , \, \wh Z(s_1) = z_1,\, \wh Z(s_2)  = z_2\}$ and the event $ G(\alpha)$, uniformly over all $(u,v) \in [R^{-1} , R]^2$. 
Hence we can find $\zeta_* = \zeta_*( q ,s_1,s_2,z_1,z_2 ,t,b, R ) > 0$ such that for each $\zeta\in (0,\zeta_*]$, and $\wt \zeta >0$, the Prokhorov distance between the conditional law of $\wh Z|_{[s_1,s_2]}$ given $\{|\wh Z(s_1) - z_1 | \vee |\wh Z(s_2)- z_2| \leq \zeta \}\cap \wh B^{u,v}(\wt\zeta ) \cap G(\alpha)$ and the law~\eqref{eqn-bridge-cond} is at most $q/3$, where here $\wh B^{u,v}(\wt\zeta)$ is as defined in~\eqref{eqn-bm-mid-box-event}.

Furthermore, we have  
\eqbn
\lim_{\alpha\rta 0}  \ul{\BB P}^{z_1,z_2}_{s_2-s_1}\left(\inf_{s \in [0,s_2-s_1]} U(s) \geq 2\alpha^{1/2}\right) = 1 ,
\eqen
where $U$ denotes the first coordinate of the Brownian bridge. Therefore, we can find $\wt\alpha_*  =\wt\alpha_*(q ,s_1,s_2,z_1,z_2 ) > 0$ such that 
\[
\wt\alpha_* \leq \frac12 \min_{i\in\{1,2\}} \op{dist}(z_i , \partial (0,\infty)^2)
\]
 and for each $\alpha \in (0,\wt\alpha_*]$, the Prokhorov distance between the law $\ul{\BB P}^{z_1,z_2}_{s_2-s_1}$ and the law~\eqref{eqn-bridge-cond} is at most $q/3$.

Henceforth fix $\zeta \in (0,\zeta_*]$. We have 
\eqbn
  \inf_{  (u,v) \in \left[R^{-1} , R\right]} \inf_{\alpha \in (0, \wt\alpha_*]} \inf_{\wt\zeta > 0} \BB P\left( |\wh Z(s_1) - z_1 | \vee |\wh Z(s_2)- z_2| \leq \zeta  \,|\, \wh B^{u,v}(\wt\zeta  ) \cap G(\alpha)\right) > 0 .
\eqen
It follows that there exists $\wt q_\zeta = \wt q_\zeta( q , s_1,s_2,z_1,z_2,t,b,R)$ such that whenever $(u,v) \in \left[R^{-1} , R\right]$, $\alpha \in (0,\wt\alpha_*]$, $\wt\zeta >0$, and $n , l\in\BB N$ are such that the Prokhorov distance between the conditional law of $Z^n(-\cdot)|_{[0,t]}$ given $  B_n^{h,c}(\zeta) \cap \ul{\mcl E}_n^l(\alpha) $ and the conditional law of $\wh Z|_{[0,t]}$ given $\wh B^{u,v}(\wt\zeta  ) \cap G(\alpha)$ is at most $\wt q_\zeta$, then the Prokhorov distance between the conditional law of $Z^n(-\cdot)|_{[s_1,s_2]}$ given $ \mcl P_n^{s_1,s_2}(z_1,z_2; \zeta) \cap \ul{\mcl E}_n^l(\alpha)$ and the conditional law of $\wh Z|_{[s_1,s_2]}$ given $\{|\wh Z(s_1) - z_1 | \vee |\wh Z(s_2)- z_2| \leq \zeta \}\cap \wh B^{u,v}(\wt\zeta ) \cap G(\alpha)$ is at most $q/3$.  

By Lemma~\ref{prop-box-cond-conv}, we can find $\alpha_* = \alpha_*(\zeta ,q , s_1,s_2,z_1,z_2,t,b,R) \in (0,\wt\alpha_*]$ such that for each $\alpha \in (0,\alpha_*]$, there exists $\wt \zeta  = \wt\zeta(\alpha,\zeta ,q , s_1,s_2,z_1,z_2,t,b,R) >0$ and $n_* = n_*(\zeta ,q , s_1,s_2,z_1,z_2,t,b,R) \geq \wt n_*$ such that for each $n\geq n_*$, the Prokhorov distance between the conditional law of $Z^n(-\cdot)|_{[0,t]}$ given $  B_n^{h,c}(\zeta) \cap \ul{\mcl E}_n^l(\alpha) $ and the conditional law of $\wh Z|_{[0,t]}$ given $\wh B^{u,v}(\wt\zeta  ) \cap G(\alpha)$ is at most $\wt q_\zeta$. By the triangle inequality and our choices of $\zeta_*$ and $\alpha_*$, we obtain the statement of the lemma.
\end{proof}

\subsection{Existence of a macroscopic cone interval}
\label{sec-macroscopic-cone}

In order to apply the results of Section~\ref{sec-disk-conv} to study the conditional law of $X_1\dots X_{2n}$ given $\{X(1,2n) =\emptyset\}$, we need to show that with high conditional probability given $\{X(1,2n) =\emptyset\}$, the event $\ul{\mcl E}_n^l(\alpha;t,R)$ occurs with some large sub-word of $X_1\dots X_{2n}$ in place of $X_{-n}\dots X_{-1}$. By Lemmas~\ref{prop-G'-event} and~\ref{prop-E^l-equiv}, it suffices to prove this when we instead condition on the event $\mcl E_n^l = \mcl E_n^l(  \ep_1)$ of Definition~\ref{def-end-event} for $l \in \left[\ep_0 n^{\xi/2} , \ep_0^{-1}\right]$ and $\ep_0,\ep_1 > 0$ small but fixed. This is the main purpose of the present subsection. 
 
For $n\in\BB N$, $\beta \in (0,1/2)$, and $   b\in (0,1)$ let $\iota_n^\beta = \iota_n^\beta(b)$ be the the smallest $j \in  [1, J\wedge \beta n]_{\BB Z}$ such that the following is true:
\begin{enumerate}
\item $X_{-j} = \tb F$, $X_{\phi(-j)} = \tc H$, and $-\phi(-j) \in [(1-\beta)n  , (1-b\beta) n]_{\BB Z}$;
\item $  |X(-j,\phi(-j))| \in \left[ b n^{1/2} , b^{-1} n^{1/2} \right]_{\BB Z}$;
\end{enumerate}
or $\iota_n^\beta(b) = J \wedge \lfloor \beta n \rfloor$ if no such $j$ exists.
Let
\eqbn
M_n^\beta  = M_n^\beta(b) :=     -\phi(-\iota_n^\beta) - \iota_n^\beta  ,\quad L_n^\beta = L_n^\beta(b) := |X(  \phi(-\iota_n^\beta) , -\iota_n^\beta,)| .
\eqen
For $n , l\in\BB N$, $\alpha,\beta >0$, $b>0$, $t\in (1/2,1-\beta)$, and $R>1$, let $\mcl H_n (\alpha,\beta) = \mcl H_n(\alpha,\beta;t,b,R)$ be the event that $\iota_n^\beta <  J \wedge \lfloor \beta n \rfloor$ and the event $\ul{\mcl E}^{L_n^\beta}_{M_n^\beta}(\alpha ; t' ,R)$ of Section~\ref{sec-disk-conv} occurs with $M_n^\beta$ in place of $n$; $L_n^\beta$ in place of $l$; the word $X_{\phi(-\iota_n^\beta)} \dots X_{ -\iota_n^\beta }$ in place of the word $X_{-M_n^\beta } \dots X_{-1}$; and $t'$ chosen so that $t' M_n^\beta +\iota_n^\beta = t n$.  
See Figure~\ref{fig-cone-prob-event} for an illustration. 

\begin{figure}[ht!]
 \begin{center}
\includegraphics{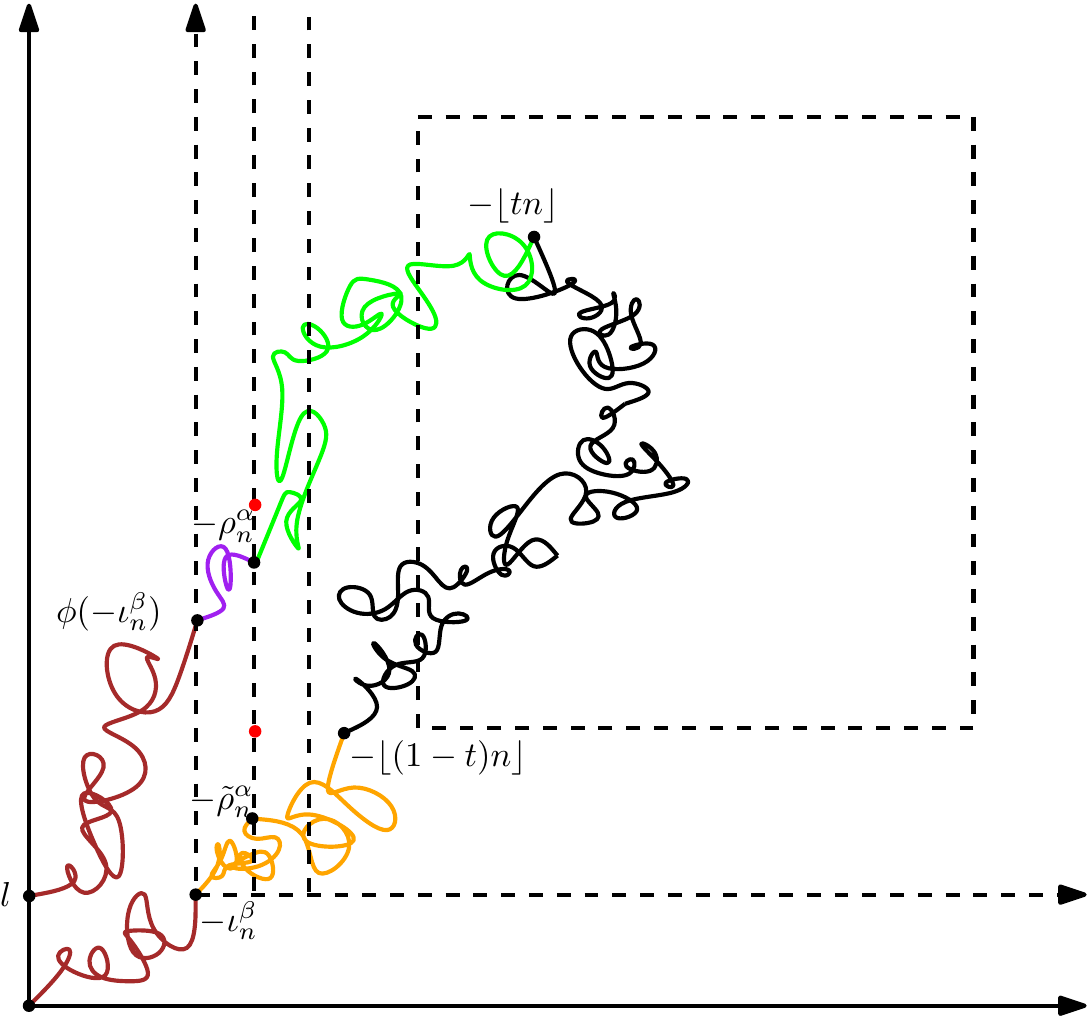} 
\caption{An illustration of the path $D$ of~\eqref{eqn-discrete-path} when the event $\mcl H_n (\alpha,\beta) \cap \mcl E_n^l$ occurs for $l \in \left[\ep_0 n^{\xi/2} , \ep_0^{-1} n^{\xi/2} \right]_{\BB Z}$. On this event the path exits the first quadrant at distance proportional to $n^{\xi/2}$ from the origin, but there is a large $\pi/2$-cone interval (corresponding to the word $X_{\phi(-\iota_n^\beta)} \dots X_{-\iota_n^\beta}$, and to the part of the path not in brown in the figure) for which the event $\ul{\mcl E}_{M_n^\beta}^{L_n^\beta}(\alpha)$ of Section~\ref{sec-disk-conv} (depicted in Figure~\ref{fig-macro-cone-event}) occurs.
}\label{fig-cone-prob-event}
\end{center}
\end{figure}

The reason for our interest in the event $\mcl H_n(\alpha,\beta)$ is the following lemma.

\begin{lem} \label{prop-H-equiv}
Let $x$ and $x'$ be realizations of $X_{-\iota_n^\beta } \dots X_{-1}$ and $X_{-n}\dots X_{\phi(-\iota_n^\beta)-1}$, respectively, and let $l' \in\BB N$. Suppose
\eqbn
\BB P\left( X_{-\iota_n^\beta} \dots X_{-1} =x ,\, X_{-n}\dots X_{\phi(-\iota_n^\beta)-1} = x',\, L_n^\beta  = l',\, \mcl H_n(\alpha,\beta)    \right)  >0 .
\eqen
Let $m := n-|x|-|x'|$. The conditional law of $ X_{\phi(-\iota_n^\beta)} \dots X_{-\iota_n^\beta}$ given 
\[
\left\{ X_{-\iota_n^\beta} \dots X_{-1} =x ,\, X_{-n}\dots X_{\phi(-\iota_n^\beta)-1} = x',\, L_n^\beta= l',\, \mcl H_n(\alpha,\beta) \right\}
\]
is the same as the conditional law of $X_{-m} \dots X_{-1}$ given the event $\ul{\mcl E}_m^{l'}(\alpha)$ of Section~\ref{sec-disk-conv}. 
\end{lem}
\begin{proof}
This is proven in exactly the same manner as Lemma~\ref{prop-E^l-equiv}.
\end{proof}

The main goal of this subsection is to show that for any given $t$ and $\beta$, the event $\mcl H_n(\alpha,\beta; t,b,R)$ occurs with high conditional probability given the event $\mcl E_n^l$ of Definition~\ref{def-end-event} when $R$ is large and $\alpha$ is small.

\begin{prop} \label{prop-H-prob}
Define the events $\mcl H_n(\alpha,\beta) = \mcl H_n(\alpha,\beta;t,b,R)$ as in the discussion just above. For $l\in\BB N$ and $\ep_1 > 0$, let $\mcl E_n^l = \mcl E_n^l(\ep_1)$ be the event of Definition~\ref{def-end-event}. For $n\in\BB N$, $\delta>0$, and $C>1$, also let $\mcl B_n^\delta(C)$ be the event of Definition~\ref{def-B_n^delta}. 
Suppose given $\ep_0 ,\ep_1>0$, $C>1$, $q\in (0,1)$, $\beta > 0$, and $t \in (1/2,1-\beta)$. There exists $R = R(q,\beta,t,C,\ep_0,\ep_1)>1$, $b = b(q,\beta,t,C,\ep_0,\ep_1)\geq (3R)^{-1}$, and $\alpha_*  = \alpha_*(q,\beta,t,C,\ep_0,\ep_1)$ such that for each $\alpha \in (0,\alpha_*]$, there exists $\delta_* = \delta_*(\alpha,q,\beta,t,C,\ep_0,\ep_1)$ such that for $\delta \in (0,\delta_*]$, there exists $n_* = n_*(\delta,\alpha,q,\beta,t,C,\ep_0,\ep_1) \in\BB N$ such that for each $n\geq n_*$ and each $l\in \left[\ep_0 n^{\xi/2} ,\ep_0^{-1} n^{\xi/2}\right]_{\BB Z}$, 
\eqbn
\BB P\left(\mcl H_n(\alpha,\beta) \,|\, \mcl B_n^\delta(C) \cap \mcl E_n^l \right) \geq 1-q. 
\eqen
\end{prop}

The proof of Proposition~\ref{prop-H-prob} proceeds as follows. We first prove in Lemma~\ref{prop-dot-bm-prob} a lower bound for the probability of an analogue of the event $\mcl H_n(\alpha,\beta) $ for a correlated two-dimensional Brownian motion $\dot Z$ started from 0 conditioned to stay in the first quadrant until time 1 and satisfy $\dot Z(1)  = 0$. In Lemma~\ref{prop-hat-bm-prob}, we use Proposition~\ref{prop-bm-limit} to transfer this to an estimate for the probability of an analogous event for a Brownian motion $\wh Z^\delta$ conditioned to stay in the first quadrant until time $1-\delta$, and conditioned on the event $\mcl B^\delta(C)$ of~\eqref{eqn-bm-B-def} for $C$ fixed and $\delta$ small. In Lemma~\ref{prop-H-given-B}, we take a scaling limit (using~\cite[Corollary 5.9]{gms-burger-cone}) to obtain an estimate for the probability of the event $\mcl H_n(\alpha,\beta)$ conditioned on the event $\mcl B_n^\delta(C)$ of Proposition~\ref{prop-E^l-abs-cont}. We then deduce Proposition~\ref{prop-H-prob} from Proposition~\ref{prop-E^l-abs-cont}. 

We first need to define the analogue of the event $\mcl H_n(\alpha,\beta)$ for a Brownian motion. For this purpose it will be convenient to introduce the following terminology, which we take from \cite[Appendix A]{gms-burger-cone}.

\begin{defn}\label{def-cone-time-forward}
A time $t$ is called a \emph{(weak) forward $\pi/2$-cone time} for a function $Z = (U,V) : \BB R \rta \BB R^2$ if there exists $t' > t$ such that $U_s \geq U_{t }$ and $V_s \geq V_{t }$ for $s\in ([t    , t']$. Equivalently, $Z([t,t'])$ is contained in the ``cone" $Z_{t } + \{z\in \BB C : \op{arg} z \in [0,\pi/2]\}$. We write $ \ol v_Z(t)$ for the infimum of the times $t'$ for which this condition is satisfied, i.e. $\ol v_Z(t)$ is the exit time from the cone. 
\end{defn} 

Let $\dot Z = (\dot U , \dot V)$ be a two-dimensional Brownian motion started from 0 with variances and covariances as in~\eqref{eqn-bm-cov} conditioned to stay in the first quadrant until time 1 and satisfy $\dot Z(1) = 0$, as in Lemma~\ref{prop-bm-excursion}. For $\beta \in (0,1/2)$ and $b\in (0,1)$, let $\tau^\beta = \tau^\beta(b)$ be the the smallest left forward $\pi/2$-cone time $s$ for $\dot Z$ such that 
\eqb \label{eqn-tau^beta-def}
s \in [0,\beta] ,\quad  \ol v_{\dot Z}(s) \in [1-\beta,1-b\beta] ,\quad \op{and}\quad \dot V(\ol v_{\dot Z}(s)) - \dot V(s)  \geq b ;
\eqe
or $\tau^\beta = \beta$ if no such $s$ exists. Let
\eqbn
\BB t_\beta := \ol v_{\dot Z}(\tau^\beta) - \tau^\beta .
\eqen

For $\alpha>0$ and $t\in (1/2, 1-\beta)$, let 
\alb
&\sigma^\alpha = \sigma^{\alpha,\beta}(t,b) := \ol v_{\dot Z}(\tau^\beta) \wedge \inf\left\{s \geq t \,:\,  \dot U(s) \leq   (\alpha \BB t_\beta)^{1/2} +    \dot U(\tau^\beta) \right\} \\
&\wt\sigma^\alpha = \wt\sigma^{\alpha,\beta}(t,b) :=    (1-t)    \wedge  \sup \left\{s\in [\tau^\beta , 1-t] \,:\, \dot U(s) \geq  (\alpha \BB t_\beta)^{1/2} +    \dot U(\tau^\beta) \right\} .
\ale 

For $\alpha >0$, $\beta \in (0,1/2)$, $t\in (1/2,1-\beta)$, $b > 0$, and $R>1$, let $\dot{\mcl H}(\alpha,\beta) = \dot{\mcl H}(\alpha,\beta;t,b,R)$ be the event that the following is true.
\begin{enumerate}
\item $\tau^\beta  < \beta$. \label{item-bm-bubble-exist}
\item $\dot Z(t) - \dot Z(\tau^\beta) \in \left[R^{-1}\BB t_\beta^{1/2} , R \BB t_\beta^{1/2}\right]^2 $, $\inf_{s\in [1-t,t]} (\dot U(s) - \dot U(\tau^\beta)  ) \geq 2(\alpha \BB t_\beta)^{1/2}$, and $\dot V(\wt\sigma^\alpha) - \dot V(\tau^\beta) \leq R^{-1} \BB t_\beta^{1/2}$. \label{item-bm-bubble-mid}
\item $  \sigma^\alpha  \in \left[\ol v_{\dot Z}(\tau^\beta) - R  \alpha \BB t_\beta , \ol v_{\dot Z}(\tau^\beta) -R^{-1}\alpha \BB t_\beta\right]$ and $\dot V(\sigma^\alpha)   \in  \left[\dot V(\ol v_{\dot Z}(\tau^\beta)) - R(\alpha \BB t_\beta)^{1/2} , \dot V(\ol v_{\dot Z}(\tau^\beta)) + R(\alpha \BB t_\beta)^{1/2} \right]$. \label{item-bm-bubble-end}
\item $ \inf_{s\in [\sigma^\alpha , \ol v_{\dot Z}(\tau^\alpha)]} \left( \dot V(s)  -\dot V(\ol v_{\dot Z}(\tau^\beta)) \right)   \geq  - R^{-1} \BB t_\beta$. \label{item-bm-bubble-sup} 
\end{enumerate}

\begin{remark}
Condition~\ref{item-bm-bubble-mid} (resp.~\ref{item-bm-bubble-end},~\ref{item-bm-bubble-sup}) in the definition of $\dot{\mcl H}(\alpha,\beta)$ is a continuum analogue of the conditions in the definition of the event $G_n(\alpha)$ (resp. $\wt H_n^l(\alpha)$; $H_n^l(\alpha)$) of Section~\ref{sec-disk-conv} but with $\tau^\beta - \ol v_{\dot Z}(\tau^\beta)$ playing the role of $n$ and $\dot V(\ol v_{\dot Z}(\tau^\beta)) - \dot V(\tau^\beta)$ playing the role of $l$. 
The reason why the intervals in the definition of $\dot{\mcl H}(\alpha,\beta)$ are scaled by $\BB t_\beta$ or $\BB t_\beta^{1/2}$ is as follows. We want the event $\dot{\mcl H}(\alpha,\beta)$ to be in some sense the ``scaling limit" of the events $\mcl H_n(\alpha,\beta)$ defined above. The event $\mcl H_n(\alpha,\beta)$ is defined so that the event $\ul{\mcl E}^{L_n^\beta}_{M_n^\beta}(\alpha ; t' ,R)$ of Section~\ref{sec-disk-conv} occurs. In particular, this latter event is defined using scales of size $M_n^\beta$ or $(M_n^\beta)^{1/2}$ instead of size $n$ or $n^{1/2}$. The quantity $\BB t_\beta$ is the continuum analogue of the quantity $M_n^\beta$.  
\end{remark}

\begin{lem} \label{prop-dot-bm-prob}
Let $\dot Z  $ be as above. Suppose given $q\in (0,1)$, $\beta\in (0,1/2)$, and $t \in (1/2,1 -\beta)$. There exists $R = R(q,\beta,t)$, $b = b(q,\beta,t) \geq 3R^{-1}$, and $\alpha_*  = \alpha_*(q,\beta,t)$ such that for each $\alpha \in (0,\alpha_*]$, we have
\eqbn
 \BB P  \left(\dot{\mcl H}(\alpha,\beta)  \right) \geq 1-q .
\eqen 
\end{lem}
\begin{proof}
Since $\dot Z$ is a.s.\ continuous and  a.s.\ does not hit $\partial (0,\infty)^2$ except at its endpoints, $\dot Z([1-t,t])$ is a.s.\ bounded and a.s.\ lies at positive distance from $(0,\infty)^2$. Hence we can find some $R_0 >1$, depending only on $t$, such that
\eqb \label{eqn-dotZ-R0}
\BB P\left( \dot Z([1-t,t]) \subset [R_0^{-1} , R_0] \right) \geq 1-q. 
\eqe
 
We will next argue that $\dot{\BB P}(\tau^\beta(b)  < \beta)$ is close to 1 if $b$ is chosen sufficiently small. We do this using \cite[Theorem 1.1]{sphere-constructions}.
Let $\kappa$ and $p$ be related as in~\eqref{eqn-p-kappa}. Let $(\BB C , h)$ be a $4/\sqrt{\kappa}$-quantum sphere and let $\eta$ is a whole-plane space-filling $\op{SLE}_{\kappa}$ from $-\infty$ to $\infty$ independent from $h$ and parametrized by quantum mass with respect to $h$. 
By~\cite[Theorem 1.1]{sphere-constructions}, the left and right quantum boundary lengths $(L,R)$ of $\eta$ with respect to $h$ evolve as a constant multiple of $\dot Z$. It is a.s.\ the case that the time reversal of $\eta$ (which has the same law as $\eta$) forms a counterclockwise ``bubble" which disconnects 0 from $\infty$ and has quantum mass at least $1-\beta$ and positive quantum boundary length. The set of times during which $\eta$ is filling in such a bubble is equal to $[s , \ol v_{(L,R)}(s)]$ for some left forward $\pi/2$-cone time $s$ for $(L,R)$ with $\ol v_{(L,R)}(s) - s = 1-\beta$, and the quantum boundary length of this bubble is equal to $\ol v_{(L,R)}(s) - s$ for this $\pi/2$-cone time $s$ (c.f.~\cite[Section 9]{wedges}). Hence, for each $\beta>0$ there a.s.\ exists some $b$ for which $  \tau^\beta(b) < \beta$, so we can find some $b$ for which this holds with probability as close to 1 as we like. By shrinking $b$ and using continuity, we can arrange that in fact 
\eqb \label{eqn-dotZ-b}
\BB P\left(  \tau^\beta  < \beta ,\,  |\dot Z(\tau^\beta )|  \leq (2 R_0)^{-1} \right) \geq 1-q 
\eqe 
with $R_0$ as in~\eqref{eqn-dotZ-R0}

If $\beta < 1-t$, then $\dot U(s) - \dot U(\tau^\beta )$ is a.s.\ positive for $s\in [1-t,t]$ on the event $\{\tau^\beta < \beta\}$. It follows that if $\beta < 1-t$, then there exists $\alpha_0 > 0$, depending only on $t$ and $\beta$, such that
\eqb \label{eqn-dotZ-alpha0}
\BB P\left( \inf_{s\in [1-t,t]} (\dot U(s) - \dot U(\tau^\beta )  ) \geq 2(\alpha_0 \BB t_\beta)^{1/2}   \right) \geq 1- q. 
\eqe

Equations~\eqref{eqn-dotZ-b} and~\eqref{eqn-dotZ-alpha0} give us a lower bound on the probability of conditions~\ref{item-bm-bubble-exist} and~\ref{item-bm-bubble-mid} in the definition of $\dot{\mcl H}(\alpha,\beta)$, except for the condition that $\dot V(\wt\sigma^\alpha) - \dot V(\tau^\beta ) \leq R^{-1} \BB t_\beta^{1/2}$. We will next find an $R$ for which conditions~\ref{item-bm-bubble-end} and~\ref{item-bm-bubble-sup} hold with high probability, and finally argue that $\dot V(\wt\sigma^\alpha) - \dot V(\tau^\beta ) \leq R^{-1} \BB t_\beta^{1/2}$ with high probability when $\alpha$ is small. 
 
Let $E_0$ be the union of the events of~\ref{eqn-dotZ-R0},~\eqref{eqn-dotZ-b}, and~\eqref{eqn-dotZ-alpha0}, so that $\BB P(E_0) \geq 1-3q$. On the event $E_0$, $\ol v_{\dot Z}(\tau^\beta )  $ (resp. $\sigma^\alpha$) is the smallest $s \geq t$ for which $\dot U(s) \leq \dot U(\tau^\beta )$ (resp. $\dot U(s) \leq (\alpha \BB t_\beta)^{1/2} + \dot U(\tau^\beta )$).

By Assertion~\ref{item-bm-limit-density} of Proposition~\ref{prop-bm-limit}, the regular conditional law of $\dot Z|_{[t,1-b\beta]}$ given almost any realization $\frk z = (\frk u , \frk v)$ of $\dot Z|_{[0,t]}$ for which $\frk z(t) \in \left[R_0^{-1} , R_0\right]^2$ is absolutely continuous with respect to the law of a Brownian motion as in~\eqref{eqn-bm-cov} started from $\frk z(t)$ and run for $1-b\beta-t$ units of time. Now suppose given realizations $\frk s$ of $\tau^\beta$ and $\frk z = (\frk u , \frk v)$ of $\dot Z|_{[0,t]}$ for which $E_0$ can possibly occur. Then the regular conditional law of $\dot Z|_{[0,t]}$ given $\{(\tau^\beta , \dot Z|_{[0,t]}) = (\frk s , \frk z)\}$ and the event $E_0$ is the same as its conditional law given the event $G(\frk s , \frk z)$ that $\ol v_{\dot Z}(\tau^\beta) \in [1-\beta,1-b\beta]$ and $\dot V(\ol v_{\dot Z}(\tau^\beta)) - \frk v(\frk s) \geq b$. 

Since $\frk z(t) \in \left[R_0^{-1} , R_0\right]^2$ and $|\frk z(\frk s)| \leq (2R_0)^{-1}$, we infer from the last assertion of Proposition~\ref{prop-bm-limit} that the conditional probability of $G(\frk s , \frk z)$ given $E_0$ and $\{(\tau^\beta , \dot Z|_{[0,t]}) = (\frk s , \frk z)\}$ is a.s.\ bounded below by a constant depending only on $R_0$, $\alpha_0$, and $b$. Furthermore, by our above description of $\sigma^\alpha$ and $\ol v(\tau^\beta)$ on the event $E_0$, as $R\rta \infty$, the regular conditional probability given $E_0$ and $\{(\tau^\beta , \dot Z|_{[0,t]}) = (\frk s , \frk z)\}$ that either 
\eqbn
\sup_{s\in [\sigma^\alpha , \ol v_{\dot Z}(\tau^\beta)]} |\dot Z(s) - \dot Z(\sigma^\alpha)| > R(\alpha \BB t_\beta)^{1/2} \quad \op{or} \quad \ol v_{\dot Z}(\tau^\beta) - \sigma^\alpha \notin [R^{-1} \alpha \BB t_\beta , R\alpha \BB t_\beta ] 
\eqen
 tends to zero, uniformly over all $\alpha \in (0,\alpha_0]$. 
 It follows that we can find $R = R(q,t,\alpha_0,b,R_0) \geq 3b^{-1}$ and $\alpha_*^1 = \alpha_*^1(q,t,R , b , R_0) \in (0,\alpha_0]$ such that for $\alpha\in (0,\alpha_*^1]$, the conditional probability given $E_0$ that conditions~\ref{item-bm-bubble-end} and~\ref{item-bm-bubble-sup} in the definition of $\dot{\mcl H}(\alpha,\beta)$ occur is at least $1-q$. 

Finally, by continuity, we can find $\alpha_*^2 = \alpha_*(q,t,\alpha_0,b,R) \in (0,\alpha_*^1]$ such that for $\alpha \in (0,\alpha_*^2]$, we have
\eqbn
\BB P\left( \dot V(\wt\sigma^\alpha) - \dot V(\tau^\beta) \leq  R^{-1} \BB t_b\eta \right) \geq 1-q. 
\eqen
By combining our above observations, we obtain that for $b$ and $R$ as above,
\eqbn
 \BB P  \left( \dot{\mcl H}(\alpha,\beta)  \right) \geq (1-3q)(1-q) -q .
\eqen
Since $q$ is arbitrary this implies the statement of the lemma.
\end{proof}

\begin{lem} \label{prop-hat-bm-prob}
For $\delta>0$ let $\wh Z^\delta = (\wh U^\delta , \wh V^\delta)$ be a Brownian motion with variances and covariances as in~\eqref{eqn-bm-cov} conditioned to stay in the first quadrant until time $1-\delta$. Also let $\wt{\mcl B}^\delta(C)$ be defined as in~\eqref{eqn-bm-wt-B-def} and define the events $\mcl H^\delta(\alpha,\beta) = \mcl H^\delta(\alpha,\beta;t,b,R)$ in the same manner as the events $\dot{\mcl H}(\alpha,\beta)$ of Lemma~\ref{prop-dot-bm-prob} but with $\wh Z^\delta$ in place of $\dot Z$. Suppose given $C>1$, $q\in (0,1)$, $\beta > 0$, and $t \in (1/2,1-\beta)$. There exists $R = R(q,\beta,t,C)>1$, $b = b(q,\beta,t,C)\geq (3R)^{-1}$, and $\alpha_*  = \alpha_*(q,\beta,t,C)$ such that for each $\alpha \in (0,\alpha_*]$, there exists $\delta_* = \delta_*(\alpha,q,\beta,t,C)$ such that for $\delta \in (0,\delta_*]$, 
\eqbn
\BB P \left( \mcl H^\delta(\alpha,\beta)  \,|\, \wt{\mcl B}^\delta(C) \right) \geq 1-q .
\eqen
\end{lem}
\begin{proof} 
The event $\mcl H^\delta(\alpha,\beta)$ is measurable with respect to $\sigma(\wh Z^\delta|_{[0,1-b\beta]})$. By Remark~\ref{remark-bm-limit-Z^delta} and assertion~\ref{item-bm-limit-strong} of Proposition~\ref{prop-bm-limit}, the conditional probability of $\mcl H(\alpha,\beta)$ given $\wt{\mcl B}^\delta(C)$ converges to $\BB P\left(\dot{\mcl H}(\alpha,\beta) \right)$ as $\delta\rta 0$. The statement of the lemma now follows from Lemma~\ref{prop-dot-bm-prob}.
\end{proof}
 
\begin{lem} \label{prop-H-given-B}
Define the events $\mcl H_n(\alpha,\beta) = \mcl H_n(\alpha,\beta;t,b,R)$ as in the beginning of this subsection. For $n\in\BB N$, $\delta>0$, and $C>1$, also let $\mcl B_n^\delta(C)$ be the event of Definition~\ref{def-B_n^delta}. 
Suppose given $C>1$, $q\in (0,1)$, $\beta > 0$, and $t \in (1/2,1-\beta)$. There exists $R = R(q,\beta,t,C)>1$, $b = b(q,\beta,t,C)\geq (3R)^{-1}$, and $\alpha_*  = \alpha_*(q,\beta,t,C)$ such that for each $\alpha \in (0,\alpha_*]$, there exists $\delta_* = \delta_*(\alpha,q,\beta,t,C)$ such that for $\delta \in (0,\delta_*]$, there exists $n_* = n_*(\delta,\alpha,q,\beta,t,C) \in\BB N$ such that for $n\geq n_*$, 
\eqbn
\BB P\left(\mcl H_n(\alpha,\beta) \,|\, \mcl B_n^\delta(C) \right) \geq 1-q. 
\eqen
\end{lem}
\begin{proof}
Fix $\delta \in (0,1/2)$ to be chosen later. Let $\wh Z^\delta$, $\wt{\mcl B}^\delta(C)$, and $\mcl H^\delta(\alpha,\beta)$ be as in Lemma~\ref{prop-H-given-B}.
Define the times $\tau^\beta$, $\BB t_\beta$, and $\sigma^\alpha$ as in the discussion above Lemma~\ref{prop-dot-bm-prob} with $\wh Z^\delta$ in place of $\dot Z$. We first claim that the conditional joint law given $\{J > (1-\delta) n\}$ of the triple 
\eqb \label{eqn-n-triple}
\left(Z^n|_{[-(1-\delta),0]} ,\, n^{-1} \iota_n^\beta , \, -n^{-1} \phi(-\iota_n^\beta) \BB 1_{\left(\iota_n^\beta <   \lfloor (1-\beta) n \rfloor \right)} \right)
\eqe
converges as $n\rta \infty$ to the joint law of the triple 
\eqb \label{eqn-bm-triple}
\left(\wh Z^\delta(-\cdot)|_{[-(1-\delta),0]} , \, \tau^\beta ,\, \ol v_{\wh Z^\delta}(\tau^\beta) \BB 1_{(\tau^\beta < \beta  )} \right) . 
\eqe

Since $\tau^\beta$ is the \emph{smallest} forward $\pi/2$ for $\wh Z^\delta$ satisfying~\eqref{eqn-tau^beta-def} (with $\dot Z$ in place of $\wh Z^\delta$), continuity of $\wh Z^\delta$ implies that on the event $\{\tau^\beta < \beta\}$, there a.s.\ exists a (random) interval $I  \subset [0,1]$ with rational endpoints and a time $a  \in I \cap \BB Q$ such that $\tau^\beta$ is the smallest forward $\pi/2$-cone time $s$ for $\wh Z^\delta$ in $I $ with $a  \in [s , \ol v_{\wh Z^\delta}(s)]$. 

For $n\in\BB N$ let $\wh X^n$ be a word sampled from the conditional law of $X$ given $\{J > (1-\delta) n\}$ and let $\wh Z^n$ be the path constructed from as in~\eqref{eqn-Z^n-def} with $\wh X^n$ in place of $X$. By condition 4 of \cite[Corollary 5.9]{gms-burger-cone} (which is the analogue of Theorem~\ref{thm-cone-limit-finite} for words conditioned to have no burgers), we can find a coupling of the sequence $(\wh X^n)$ with the path $\wh Z^\delta$ such that on the event $\{\tau^\beta <\beta\}$, the following holds a.s. 
Let $I$ and $a$ be as above and let $\wt\iota_n^\beta$ be the smallest $i \in \BB N$ such that $n^{-1} i \in I$, $\wh X_{-i}^n = \tb F$, and $a n \in [i , -\phi(-i)]$. Then we have
\eqb \label{eqn-wh-Z-coupling}
\text{$\wh Z^n\rta \wh Z^\delta(-\cdot)$ uniformly in $[-1,0]$},\quad   n^{-1} \wt \iota_n^\beta \rta \tau^\beta \quad\op{and}\quad  - n^{-1} \phi(-\wt \iota_n^\beta) \rta \ol v_{\wh Z^\delta}(\tau^\beta) .
\eqe 
Henceforth fix a coupling satisfying~\eqref{eqn-wh-Z-coupling}. We claim that if we define $\iota_n^\beta$ with the word $\wh X^n$ in place of the word $X$, then a.s.\ 
\eqb \label{eqn-iota-conv}
n^{-1} \iota_n^\beta \rta \tau^\beta \quad\op{and}\quad  - n^{-1} \phi(- \iota_n^\beta) \BB 1_{(\iota_n^\beta < \beta n )} \rta  \ol v_{\wh Z^\delta}(\tau^\beta) \BB 1_{(\tau^\beta < \beta  )}  .
\eqe  
This will prove our claim at the beginning of the proof. 

First suppose $\tau^\beta  <\beta$. It is a.s.\ the case that $\tau^\beta \in (0,\beta)$, $\ol v_{\wh Z^\delta}(\tau^\beta) \in (1-\beta,1-b\beta)$, and $\wh V(\ol v_{\wh Z^\delta}(\tau^\beta)) - \wh V(\tau^\beta) > b$. Therefore, in a coupling satisfying~\eqref{eqn-wh-Z-coupling}, it is a.s.\ the case that for sufficiently large $n\in\BB N$, we have $\wt\iota_n^\beta \in [1,\beta n]_{\BB Z}$, $-  \phi(-\wt\iota_n^\beta) \in [(1-\beta) n , (1-b\beta) n]_{\BB Z}$, and $|\wh X^n(\phi(-\wt\iota_n^\beta) , -\wt\iota_n^\beta)| \geq b n^{1/2}$, where $\wh X^n(\cdot,\cdot)$ is defined as in~\eqref{eqn-X(a,b)} with $\wh X^n$ in place of $X$. By minimality of $\iota_n^\beta$ for sufficiently large $n\in\BB N$ we have $\wt\iota_n^\beta \geq \iota_n^\beta$ and $\iota_n^\beta < \lfloor (1-\beta) n\rfloor$. By~\eqref{eqn-wh-Z-coupling},
$\limsup_{n\rta\infty} n^{-1} \iota_n^\beta \leq \tau^\beta$.
Similarly, $\liminf_{n\rta\infty} (- n^{-1} \phi(-\wt\iota_n^\beta)) \geq \ol v_{\wh Z^\delta}(\tau^\beta) $.

On the other hand, suppose $s$ is a subsequential limit of the times $-n^{-1} \iota_n^\beta$ (defined with $\wh X^n$ in place of $X$). By~\cite[Lemma 5.7]{gms-burger-cone}, we have that $s$ is a forward $\pi/2$-cone time for $\wh Z^\delta$ and $-n^{-1} \phi(-\iota_n^\beta) \rta \ol v_{\wh Z^\delta}(s)$. By definition of $\iota_n^\beta$ we have that $s$ satisfies~\eqref{eqn-tau^beta-def}, so by minimality of $\tau^\beta$ we have $s\leq \tau^\beta$. By combining this with our above observation, $s =  \tau^\beta$. We thus obtain~\eqref{eqn-iota-conv} in the case $\tau^\beta < \beta$. 

We will now complete the proof of~\eqref{eqn-iota-conv} by showing that $  \iota_n^\beta = \lfloor (1-\beta)n \rfloor$ for sufficiently large $n$ whenever $\tau^\beta = \beta$. If not, then by the same argument as above, a subsequential limit $s$ of the times $-n^{-1} \iota_n^\beta$ satisfies~\eqref{eqn-tau^beta-def}, whence $\tau^\beta < \beta$, which yields a contradiction. This completes the proof of our claim regarding convergence in law. 

Our above claim implies that for each fixed $\delta \in (0,1/2)$, the conditional joint law of the triple~\eqref{eqn-n-triple} given $\mcl B_n^\delta(C)$ converges as $n\rta \infty$ to the conditional joint law of the triple~\eqref{eqn-bm-triple} given $\wt{\mcl B}^\delta(C)$. This convergence together with Lemma~\ref{prop-hat-bm-prob} yields the statement of the lemma.
\end{proof}

\begin{proof}[Proof of Proposition~\ref{prop-H-prob}]
Combine Proposition~\ref{prop-E^l-abs-cont} and Lemma~\ref{prop-H-given-B}.
\end{proof}

\subsection{Proof of Theorem~\ref{thm-main}}
\label{sec-main-proof}

The main input in the proof of Theorem~\ref{thm-main} is the following proposition, whose proof requires most of the preceding results in the paper. See Figure~\ref{fig-whole-path-detail} for an illustration.

\begin{figure}[ht!]
 \begin{center}
\includegraphics{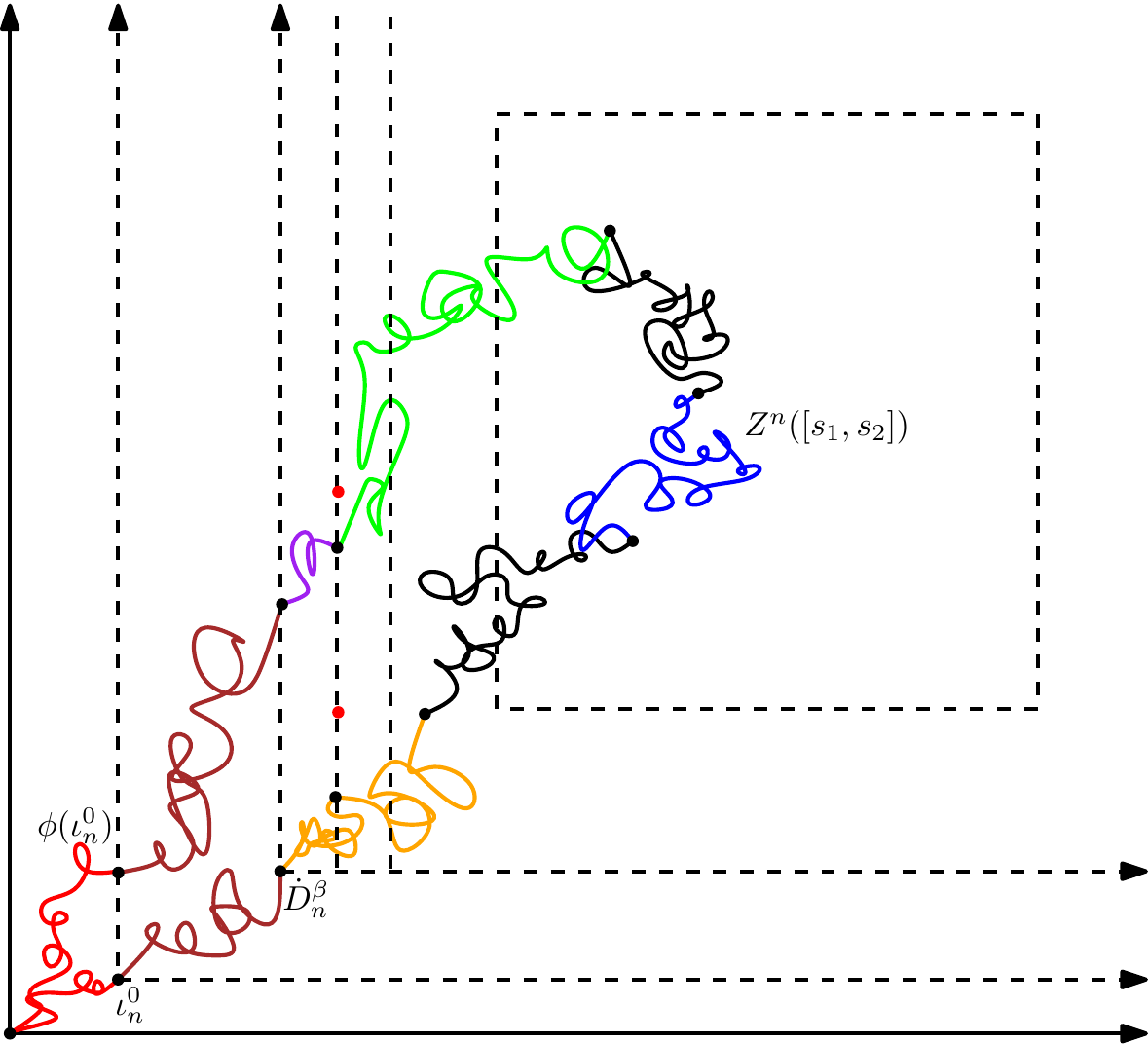} 
\caption{An illustration of the path $D$ of~\eqref{eqn-discrete-path} when the event $ \mcl G_n' \cap \dot{\mcl H}_{M_n}(\alpha,\beta)\cap \dot{\mcl B}_{M_n}^\delta(C)$ in the proof of Proposition~\ref{prop-ll-bridge} occurs; equivalently, the event event $\mcl G_n'$ of Lemma~\ref{prop-G'-event} occurs and the event depicted in Figure~\ref{fig-cone-prob-event} occurs for the part of the curve not in red. By Lemma~\ref{prop-G'-event}, the conditional probability given $\{X(1,2n)=\emptyset\}$ of the event $\mcl G_n'$ is close to 1 when $\ep_0$ and $\ep_1$ are small and $n$ is large. By Lemmas~\ref{prop-E^l-equiv} and~\ref{prop-end-box}, the conditional probability given $\mcl G_n'$ of the event $\dot{\mcl B}_{M_n}^\delta(C)$ is close to 1 when $C$ is large, $\delta$ is small, and $n$ is large. By Proposition~\ref{prop-H-prob}, the conditional probability given $\mcl G_n' \cap \dot{\mcl B}_{M_n}^\delta(C)$ of the event $\dot{\mcl H}_{M_n}(\alpha,\beta)$ is close to 1 when $n$ is large. By Proposition~\ref{prop-2pt-cond-conv}, the conditional law of the blue segment of the curve given $ \mcl G_n' \cap \dot{\mcl H}_{M_n}(\alpha,\beta)\cap \dot{\mcl B}_{M_n}^\delta(C)$ and approximate realizations of the endpoints of this segment is close to the law of a Brownian bridge conditioned to stay in the first quadrant. Combining these statements proves Proposition~\ref{prop-ll-bridge}. 
}\label{fig-whole-path-detail}
\end{center}
\end{figure}

\begin{prop} \label{prop-ll-bridge}
Fix $s_1,s_2 \in (0,1)$, $q \in (0,2)$, and $z_1,z_2 \in (0,\infty)^2$. There is a $\zeta_* =\zeta_*(q,s_1,s_2,z_1,z_2) \in \BB N$ such that for each $\zeta \in (0,\zeta_*]$, there exists $n_* = n_*(\zeta,q,s_1,s_2,z_1,z_2)$ such that the following is true. For each $n\geq n_*$, the Prokhorov distance between the conditional law of $Z^n|_{[s_1,s_2]}$ given the event $\mcl P_n^{s_1,s_2}(z_1,z_2;\zeta)$ of~\eqref{eqn-mclP-event}; and the law of a two-dimensional Brownian bridge from $z_1$ to $z_2$ in time $s_2-s_1$ with variances and covariances as in~\eqref{eqn-bm-cov} conditioned to stay in the first quadrant is at most $q$. 
\end{prop}
\begin{proof}
For $z_1,z_2\in (0,\infty)^2$, let $\ul{\BB P}^{z_1,z_2}_{s_2-s_1}$ be the law of a two-dimensional Brownian bridge from $z_1$ to $z_2$ in time $s_2-s_1$ with variances and covariances as in~\eqref{eqn-bm-cov} conditioned to stay in the first quadrant, as in the proof of Lemma~\ref{prop-2pt-cond-conv}. By considering translations of a path with the law $\BB P_{s_2-s_1}^{z_1,z_2}$, we can find $\eta   = \eta(q,s_1,s_2,z_1,z_2) > 0$ with 
\[
\eta \leq \frac12 \min_{i\in \{1,2\}} \op{dist}(z_i ,\partial (0,\infty)^2)
\]
 such that for each $w  \in (0,\eta]^2$, the Prokhorov distance between the law $\ul{\BB P}_{s_2-s_1}^{z_1,z_2}$ and the law
\eqb \label{eqn-bridge-cond-eta}
\ul{\BB P}^{z_1+w,z_2+w}_{s_2-s_1}\left(\,\cdot \,|\, Z([0,s_2-s_1]) \subset w + (0,\infty)^2  \right)
\eqe  
is at most $q$. 

By Lemma~\ref{prop-G'-event}, we can find $\ep_0 , \ep_1 \in (0,1)$ depending only on $q $ and $n_*^0 = n_*^0( q ) \in \BB N$ such that for $n\geq n_*^0$, 
\eqb \label{eqn-ll-G'-prob}
\BB P\left(\mcl G_n' \,|\, X(1,2n) =\emptyset \right) \geq 1-q ,
\eqe
where $\mcl G_n' = \mcl G_n'(\ep_0,\ep_1)$ is as in Lemma~\ref{prop-G'-event}. Let $\iota_0^n = \iota_0^n(\ep_0)$ be as in Section~\ref{sec-big-loop-setup}. Also let 
\eqb  \label{eqn-end-data}
M_n := \iota_0^n -\phi(\iota_0^n) ,\quad L_n := |X(\phi(\iota_0^n) , \iota_0^n)| .
\eqe 

For $\delta>0$ and $C>1$, let $\dot{\mcl B}_{M_n}^\delta(C)$ be defined in the same manner as the event $\mcl B_{M_n}^\delta(C)$ of Definition~\ref{def-B_n^delta} but with $M_n$ in place of $n$ and the word $X_{ \iota_0^n(\ep_0)} \dots X_{\phi(\iota_0^n(\ep_0))}$ in place of the word $X_{-M_n} \dots X_{-1}$. 
By Lemma~\ref{prop-E^l-equiv} and Proposition~\ref{prop-end-box}, there exists $C = C(q,\ep_0,\ep_1) > 0$ such that for each $\delta \in (0,1/2)$, there exists $n_*^1 = n_*^1(\delta,q,\ep_0,\ep_1) \geq n_*^0$ such that for $n\geq n_*^1$, 
\eqb \label{eqn-ll-B-prob}
\BB P\left(\dot{\mcl B}_{M_n}^\delta(C) \,|\, \mcl G_n' \right) \geq 1-q .
\eqe
 
For $\beta \in (0,1/2)$, $t\in (1/2,1-\beta)$, $R > 1$, $b \in (0,1)$, and $\alpha>0$ let $\dot{\mcl H}_{M_n}(\alpha,\beta) = \dot{\mcl H}_{M_n}(\alpha , \beta ; t, b,R)$ be defined in the same manner as the event $\mcl H_{M_n}(\alpha,\beta)$ of Section~\ref{sec-macroscopic-cone} but with $M_n$ in place of $n$ and the word $X_{ \iota_0^n(\ep_0)} \dots X_{\phi(\iota_0^n(\ep_0))}$ in place of the word $X_{-M_n} \dots X_{-1}$. Also let $\dot\iota_{M_n}^\beta = \dot\iota_{M_n}^\beta(b)$ be defined in the same manner as the time $\iota_{M_n}^\beta$ of Section~\ref{sec-macroscopic-cone} but with $M_n$ in place of $n$ and the word $X_{ \iota_0^n(\ep_0)} \dots X_{\phi(\iota_0^n(\ep_0))}$ in place of the word $X_{-M_n} \dots X_{-1}$. 

By Proposition~\ref{prop-tight}, \cite[Corollary 5.2]{gms-burger-cone}, and the lower bound in~\cite[Theorem~ 1.10]{gms-burger-local}, we can find $\beta = \beta(\eta, s_1,s_2,\ep_0) \in (0,s_1)$ such that for each $\delta \in (0,1/2)$, there exists $ n_*^2 = n_*^2(\delta,q,\eta,\beta,\ep_0,\ep_1) \in\BB N$ such that for each $n\geq n_*^2$, 
\begin{align} \label{eqn-beta-cone-tip}
&\BB P\left(\sup_{j \in [1, 2\beta n + \ep_0^{-1} n^{\xi/2} ]_{\BB Z}} |X(2n-j,2n)| \leq \frac{\eta}{2} n^{1/2} \right) \geq 1-q/2 \notag\\
&\BB P\left(  \sup_{j \in [1, 2\beta n + \ep_0^{-1} n^{\xi/2} ]_{\BB Z}} \mcl N_{\tb F}\left( X(2n-j,2n) \right) \leq n^\nu  \,|\, X(1,2n) =\emptyset\right) \geq 1-q/2 .
\end{align}
Henceforth fix such a $\beta$.  Also choose $t\in (1/2,1-\beta)$ such that $1-t  <  s_1 < s_2 < t$. By Lemmas~\ref{prop-E^l-equiv} and~\ref{prop-H-prob}, there exists $R = R(q,\beta,t_1,C,\ep_0,\ep_1)>1$, $b = b(q,\beta,t,C,\ep_0,\ep_1)\geq (3R)^{-1}$, and $\wt\alpha_*  = \wt\alpha_*(q,\beta,t,C,\ep_0,\ep_1)$ such that for each $\alpha \in (0,\wt\alpha_*]$, there exists $\delta_* = \delta_*(\alpha,q,\beta,t,C,\ep_0,\ep_1)$ such that for $\delta \in (0,\delta_*]$, there exists $n_*^3 = n_*^3(\delta,\alpha,q,\beta,t,C,\ep_0,\ep_1) \geq n_*^2$ such that for each $n\geq n_*^3$ 
\eqbn
\BB P\left(\dot{\mcl H}_{M_n}(\alpha,\beta) \,|\, \mcl B_n^\delta(C) \cap \mcl G_n' \right) \geq 1-q .
\eqen
By~\eqref{eqn-ll-B-prob}, in fact
\eqb \label{eqn-ll-H-prob}
\BB P\left(\dot{\mcl H}_{M_n}(\alpha,\beta) \,|\,   \mcl G_n' \right) \geq (1-q)^2 .
\eqe 
Henceforth fix such an $r$ and such a $b$. 
Define
\begin{align} \label{eqn-dot-data}
&\dot D_n^\beta := - D\left(X(2n - \iota_0^n - \dot\iota_n^\beta , 2n)\right) ,\quad \dot W_n^\beta := n^{-1/2} \dot D_n^\beta, \notag \\
& \dot F_n^\beta := \left\{\mcl N_{\tb F}\left(X(2n - \iota_0^n - \dot\iota_n^\beta , 2n) \right) \leq n^\nu \right\} .
\end{align}  
By~\eqref{eqn-beta-cone-tip}, the conditional probability given $\{X(1,2n)=\emptyset\}$ that $\mcl G_n' \cap \dot{\mcl H}_{M_n}(\alpha,\beta)$ occurs and either coordinate of $\dot D_n^\beta$ is larger than $\frac12 \eta$ is at most $q$.  

By Lemma~\ref{prop-H-equiv} and Proposition~\ref{prop-2pt-cond-conv}, 
there exists $\zeta_*  = \zeta_*( q,s_1,s_2,z_1,z_2,t,b,R) > 0$ such that for each $\zeta \in (0,\zeta_*]$ there exists $\alpha_* = \alpha_*(\zeta, q,s_1,s_2,z_1,z_2,t,b,R  ) \in (0,\wt\alpha_*]$ such that for each $\alpha \in (0,\alpha_*]$, there exists $\delta = \delta(\alpha,\zeta,q,s_1,s_2,z_1,z_2,t,b,R  )$ and $n_*^4 = n_*^4( \delta,\alpha,\zeta, q,s_1,s_2,z_1,z_2,t,b,R,\ep_0,\ep_1)$ with the following property. For each $n\geq n_*^4$, the Prokhorov distance between the conditional law of $Z^n |_{[s_1,s_2]}$ given $\dot D_n^\beta$ and the event 
\[
 \mcl P_n^{s_1,s_2}(z_1,z_2; \zeta) \cap  \dot{\mcl H}_{M_n}(\alpha,\beta) \cap \mcl G_n' \cap \dot F_n^\beta
 \]
and the conditional law of a two-dimensional Brownian bridge from $z_1 + \dot W_n^\beta$ to $z_2 + \dot W_n^\beta$ in time $s_2-s_1$ with variances and covariances as in~\eqref{eqn-bm-cov} conditioned to stay in the translated first quadrant $\dot W_n^\beta + (0,\infty)^2$ is at most $q$. 

By our choice of $\eta$, on the event that each coordinate of $\dot D_n^\beta$ is at most $\frac12 \eta$ and $\dot F_n^\beta$ occurs, the Prokhorov distance between this conditional law and the law $\ul{\BB P}_{s_2-s_1}^{z_1,z_2}$ is at most $q$ (here the event $\dot F_n^\beta$ is used to ensure that $Z^n(2- n^{-1} \iota_0^n - n^{-1} \dot\iota_n^\beta)$ is close to $n^{-1/2} \dot D_n^\beta$). By~\eqref{eqn-ll-G'-prob},~\eqref{eqn-beta-cone-tip}, and~\eqref{eqn-ll-H-prob}, it holds with conditional probability at least $(1-q)^2 - 2q$ given $\{X(1,2n)=\emptyset\}$ that $\dot{\mcl H}_{M_n}(\alpha,\beta) \cap \mcl G_n'$ occurs, each coordinate of $\dot D_n^\beta$ is at most $\frac12 \eta$, and $\dot F_n^\beta$ occurs. Since $q$ is arbitrary the statement of the proposition follows. 
\end{proof}

\begin{proof}[Proof of Theorem~\ref{thm-main}]
By Proposition~\ref{prop-tight}, from any sequence of integers tending to $\infty$, we can extract a subsequence along which the conditional laws of $Z^n|_{[0,2]}$ given $\{X(1,2n)=\emptyset\}$ converge to the law of a random continuous path $\wt Z  : [0,2]\rta [0,\infty)^2$ such that the law of $\wt Z(t)$ is absolutely continuous with respect to Lebesgue measure on $[0,\infty)^2$ for each $t\in (0,2)$. Henceforth fix such a subsequential limit $\wt Z$. 

By Proposition~\ref{prop-ll-bridge}, for each $q\in (0,1)$, $s_1<s_2 \in (0,2)$, and $z_1,z_2 \in (0,\infty)^2$, there exists $\zeta_* = \zeta_*(q)$ such that for each $\zeta \in (0,\zeta_*]$, the Prokhorov distance between the conditional law of $\wt Z|_{[s_1,s_2]}$ given $\{|\wt Z(s_1) - z_1| \vee |\wt Z(s_2) - z_2| \leq \zeta\}$ and the law of a two-dimensional Brownian bridge from $z_1$ to $z_2$ in time $s_2-s_1$ with variances and covariances as in~\eqref{eqn-bm-cov} conditioned to stay in the first quadrant is at most $q$. Denote this latter law by $\ul{\BB P}_{s_2-s_1}^{z_1,z_2}$. Also let $\wt{\BB P}^{z_1,z_2}_{s_1,s_2}$ be the regular conditional law of $\wt Z|_{[s_1,s_2]}$ given $\{\wt Z(s_1) = z_1 ,\, \wt Z(s_2) = z_2\}$. 
We claim that for Lebesgue almost every $z_1,z_2 \in (0,\infty)^2$, we have 
\eqb \label{eqn-cond-law-agree}
 \wt{\BB P}^{z_1,z_2}_{s_1,s_2} = \ul{\BB P}_{s_2-s_1}^{z_1,z_2} .
\eqe
Once~\eqref{eqn-cond-law-agree} is established, the theorem follows from the uniqueness statement of Lemma~\ref{prop-bm-excursion} together with the second statement of Proposition~\ref{prop-tight}. 

To obtain~\eqref{eqn-cond-law-agree}, fix $N\in\BB N$ and times $t_1,\dots , t_N \in [s_1,s_2]$. Also let $\mcl C$ be the space of continuous functions $\BB R^N\rta \BB R$ which vanish at $\infty$, with the supremum norm. The space $\mcl C$ is separable, so we can find a countable collection $\{f_k\}_{k\in\BB N}\subset \mcl C$ which is dense in $\mcl C$. 
Let $\wt{\BB E}_{s_1,s_2}^{z_1,z_2}$ and $\ul{\BB E}^{z_1,z_2}_{s_2-s_1}$ denote the expectations corresponding to $\wt{\BB P}_{s_1,s_2}^{z_1,z_2}$ and $\ul{\BB P}^{z_1,z_2}_{s_2-s_1}$, respectively. For $k\in\BB N$, let $\Psi_k , \Phi_k : [0,\infty)^2 \times [0,\infty)^2 \rta \BB R^2$ be defined by 
\eqbn
\Psi_k(z_1,z_2) := \wt{\BB E}_{s_1,s_2}^{z_1,z_2}\left(f_k\left(\wt Z(t_1),\dots , \wt Z(t_N)   \right) \, \right),\quad \Phi_k(z_1,z_2) := \ul{\BB E}^{z_1,z_2}_{s_2-s_1}\left(f_k\left( Z(t_1),\dots ,   Z(t_N)\right) \right) .
\eqen
By the above statement regarding Prokhorov distances, it holds for each $z_1,z_2 \in (0,\infty)^2$ that
\eqbn
\lim_{\zeta\rta 0} \int_{B_\zeta(z_1)} \int_{B_\zeta(z_2)} \Psi_k(w_1,w_2) \, dw_1 \, dw_2 = \Phi_k(z_1,z_2) ,\quad \forall k\in\BB N .
\eqen
By the Lebesgue differentiation theorem, for almost every $z_1,z_2 \in (0,\infty)^2$ we have
\eqbn
 \Psi_k(z_1,z_2)   = \Phi_k(z_1,z_2) ,\quad \forall k\in\BB N .
\eqen
For every such $z_1,z_2$, the law of $( \wt Z(t_1),\dots , \wt Z(t_N))$ under $ \wt{\BB P}^{z_1,z_2}_{s_1,s_2}$ agrees with the law of $(   Z(t_1),\dots ,  Z(t_N))$ under $  \ul{\BB P}_{s_2-s_1}^{z_1,z_2}$. For almost every $z_1,z_2\in (0,\infty)^2$, these two laws agree for every finite vector of times $(t_1 , \dots , t_N) \subset \BB Q \cap [s_1,s_2]$. It follows that~\eqref{eqn-cond-law-agree} holds. 
\end{proof}

\subsection{Proof of Theorem~\ref{thm-cone-limit-finite}}
\label{sec-cone-proof}

In this subsection we will deduce Theorem~\ref{thm-cone-limit-finite} from Theorem~\ref{thm-main}, the earlier results of this paper, and the results of \cite{gms-burger-cone}. 

\begin{lem} \label{prop-cone-detect-path}
Fix $a \in (-\infty,0)$ and $r>0$. 
For $n\in\BB N$, let $\wh \iota_n^{a,r}$ be the minimum of 0 and the smallest $i\in\BB N$ such that $X_i = \tb F$, $i \geq a n$, and $i - \phi(i) \geq r n - 1$. For $t \in \BB R$, let
\eqb \label{eqn-cone-detect-path}
\ol U_r^n(t) := U^n( t) - \inf_{s\in [t-r,t]} U^n(s) ,\quad \ol V_r^n(t) := V^n( t) - \inf_{s\in [t-r,t]} V^n(s) ,\quad \ol Z_r^n(t) := (\ol U_r^n(t) , \ol V_r^n(t)) .
\eqe 
For each $\alpha >0$ and $q\in (0,1)$, there exists $\zeta =\zeta(\alpha,q,a,r) > 0$ and $n_* =n_*(\alpha,q,a,r) \in\BB N$ such that for each $n\geq n_*$, we have
\eqbn
\BB P\left( \inf_{t \in [a , n^{-1} \wh\iota_n^{a,r} - \alpha ]_{\BB Z}} |\ol Z_r^n(t)| \geq \zeta  \,|\, J > n  \right) \geq 1- q .
\eqen
\end{lem}
\begin{proof}
Let $\wh Z =(\wh U ,\wh V)$ have the law of a correlated Brownian as in~\eqref{eqn-bm-cov} conditioned to stay in the first quadrant until time $-1$ when run backward. Let $\wh \tau^{a,r}$ be the minimum of 0 and the smallest $\pi/2$-cone time $t$ for $\wh Z$ such that $t\geq a$ and $t - v_Z(t) \geq r$. Also let $\ol U$, $\ol V$, and $\ol Z_r$ be as in~\eqref{eqn-cone-detect-path} with $\wh Z$ in place of $Z^n$. Observe that the zeros of $\ol Z_r$ are precisely the $\pi/2$-cone times $t$ for $\wh Z$ with $t -  v_Z(t) \geq r$. Therefore, $\wh\tau^{a,r}$ can equivalently be defined as the minimum of 0 and the smallest $t\geq a$ for which $\ol Z_r(t) = 0$. Hence, for each $\alpha>0$ and $q\in (0,1)$ there exists $\zeta =\zeta(\alpha,q,a,r) > 0$ such that
\eqb \label{eqn-pre-cone-inf}
\BB P\left( \inf_{t \in [a , \wh\tau^{a,r} - \alpha ]_{\BB Z}} |\ol Z_r (t)| \geq 2\zeta  \,|\, J > n  \right) \geq 1- \frac{q}{2}  .
\eqe 

By condition 5 of \cite[Corollary 5.9]{gms-burger-cone}, the joint conditional law of the pair $\left(Z^n , n^{-1} \wh\iota_n^{a,r} \right)$ given $\{J > n\}$ converges as $n\rta \infty$ to the joint law of the pair $\left(\wh Z ,  \wh\tau^{a,r}\right)$. The statement of the lemma now follows from~\eqref{eqn-pre-cone-inf}.
\end{proof}

\begin{lem} \label{prop-E^l-cone-detect}
Suppose we are in the setting of Lemma~\ref{prop-cone-detect-path} with $a \in (-1,0)$ and $r \in (0,1)$. For each $\alpha >0$ and $q\in (0,1)$, there exists $\zeta =\zeta(\alpha,q,a,r,\ep_0,\ep_1) > 0$, $\beta =\beta(\alpha,q,a,r,\ep_0,\ep_1) \in (0,1/2)$, and $n_* =n_*(\alpha,q,a,r,\ep_0,\ep_1) \in\BB N$ such that for each $n\geq n_*$ and each $l\in \left[\ep_0 n^{\xi/2} , \ep_0^{-1} n^{\xi/2}\right]_{\BB Z}$, we have
\eqbn
\BB P\left(    -(1-\beta) n  \leq  \phi(\wh\iota_n^{a,r})  \leq   \wh\iota_n^{a,r} \leq  - \beta n   , \, \inf_{t \in [a , n^{-1} \wh\iota_n^{a,r} - \alpha ]_{\BB Z}} |\ol Z_r^n(t)| \geq \zeta  \,|\, \mcl E_n^l(\ep_1)  \right) \geq 1- q,
\eqen
with $\mcl E_n^l(\ep_1)$ as in Definition~\ref{def-end-event}.
\end{lem}
\begin{proof}
For $\delta>0$, let $\wh Z_\delta$ have the law of a correlated Brownian as in~\eqref{eqn-bm-cov} conditioned to stay in the first quadrant until time $-(1-\delta)$ when run backward. Let $\wh \tau^{a,r}_\delta$ be defined as in the proof of Lemma~\ref{prop-cone-detect-path} with $\wh Z_\delta$ in place of $\wh Z$. Also fix $q' \in (0,1)$ to be chosen later. For $\delta > 0$ and $C>1$, let
\eqbn
\wh{\mcl B}^\delta(C) := \left\{\wh Z_\delta(-1+\delta ) \in \left[C^{-1 }\delta^{1/2} , C\delta^{1/2}\right]_{\BB Z}\right\} .
\eqen
We claim that for each $C>1$ we can find $\beta = \beta(q',C,a,r,\ep_0,\ep_1) \in (0, 1/4)$ and $\delta = \delta(q',C,a,r,\ep_0,\ep_1) \in (0, \beta/2)$ such that 
\eqb  \label{eqn-cone-exists}
\BB P\left( -(1- \beta/2) \leq      \wh\tau^{a,r}_\delta  \leq   v_{\wh Z^\delta}( \wh\tau^{a,r}_\delta ) \leq -\beta/2      \,|\, \wh{\mcl B}^\delta(C)  \right) \geq 1-q' .
\eqe 
This follows from the same argument used in the proof of Proposition~\ref{prop-dot-bm-prob}. To be more precise, let $\dot Z$ be as in Lemma~\ref{prop-dot-bm-prob}. The argument used in that lemma (involving CLE on a quantum sphere) shows that with high probability there is a $\pi/2$-cone interval for $\dot Z$ of time length slightly less than 2, so the analogue of~\eqref{eqn-cone-exists} for $\dot Z$ is true. By Proposition~\ref{prop-bm-limit} the analogue of~\eqref{eqn-cone-exists} for the path $\wh Z$ conditioned on $\mcl B^\delta(C)$ is true for small $\delta$. By the argument of Lemma~\ref{prop-hat-bm-prob} we then obtain~\eqref{eqn-cone-exists}. 

Since $\BB P\left(\wh{\mcl B}^\delta(C)\right) > 0$, it follows from Lemma~\ref{prop-cone-detect-path} and \cite[Corollary 5.9]{gms-burger-cone} that for each $\alpha >0$ and $C>1$, there exists $\zeta =\zeta(\alpha,q',C,\delta ,\beta, a,r,C) > 0$ and $n_* =n_*(\alpha,q',C,\delta ,\beta, a,r,C)  \in\BB N$ such that for each $n\geq n_*$, we have
\eqbn
\BB P\left( -(1-\beta) n  \leq  \phi(\wh\iota_n^{a,r})  \leq   \wh\iota_n^{a,r} \leq  - \beta n   , \, \inf_{t \in [a , n^{-1} \wh\iota_n^{a,r} - \alpha ]_{\BB Z}} |\ol Z_r^n(t)| \geq \zeta \,|\, \mcl B_n^\delta(C) \right) \geq 1- 2q'.
\eqen
We conclude by combining this with Propositions~\ref{prop-E^l-abs-cont} and~\ref{prop-end-box}.
\end{proof}

\begin{lem} \label{prop-cone-conv-finite}
For $(a,r) \in (0,2)^2$, let $\tau^{a,r}$, $\iota_n^{a,r}$, and $\tau_n^{a,r}$ be defined as in Condition~\ref{item-cone-limit-stopping} of Theorem~\ref{thm-cone-limit-finite}. 
Let $\mcl A$ be a finite collection of pairs $(a,r) \in (0,2)^2$. As $n\rta\infty$, the joint conditional law given $\{X(1,2n)=\emptyset\}$ of
\eqbn
\left(Z^n|_{[0,2]} ,\, \{\tau_n^{a,r}\}_{(a,r) \in\mcl A}\right)  
\eqen
converges to the joint law of 
\eqbn
\left(\dot Z ,\, \{\tau^{a,r}\}_{(a,r) \in\mcl A}\right) .
\eqen 
\end{lem}
\begin{proof}
By Lemma~\ref{prop-G'-event}, we can find $\ep_0 , \ep_1 \in (0,1)$ depending only on $q $ and $n_*^0 = n_*^0( q ) \in \BB N$ such that for $n\geq n_*^0$, 
\eqbn
\BB P\left(\mcl G_n' \,|\, X(1,2n) =\emptyset \right) \geq 1-q ,
\eqen
where $\mcl G_n' = \mcl G_n'(\ep_0,\ep_1)$ is as in Lemma~\ref{prop-G'-event}. Let $\iota_0^n = \iota_0^n(\ep_0)$ be as in Section~\ref{sec-big-loop-setup}. If we define $\wh \iota_n^{a,r}$ as in Lemma~\ref{prop-cone-detect-path} with the word $\dots X_{ \iota_0^n-2} X_{ \iota_0^n-1}$ in place of the word $\dots X_{-2} X_{-1}$, then on the event $\{\wh\iota_n^{a,r} < 0\} \cap \mcl G_n'$, it holds that $\iota_n^{a,r} = \iota_0^n + \wh \iota_n^{a,r}$.  

By Lemmas~\ref{prop-E^l-equiv} and~\ref{prop-E^l-cone-detect}, for each $\alpha>0$ we can find $\zeta =\zeta(\alpha,q,\mcl A ,\ep_0,\ep_1) > 0$, $\beta =\beta(\alpha,q,\mcl A,\ep_0,\ep_1) \in (0,1/2)$, and $n_*^1 =n_*^1(\alpha,q,\mcl A,\ep_0,\ep_1) \geq n_*^0$ such that for each $n\geq n_*^1$ and each $(a,r) \in \mcl A$, 
\eqb \label{eqn-cone-detect-G'}
\BB P\left(  \beta n \leq \phi( \iota_n^{a,r}) \leq    \iota_n^{a,r} \leq 2n- \beta n     , \, \inf_{t \in [a , \tau_n^{a,r} - \alpha  ]_{\BB Z}} |\ol Z_r^n(t)| \geq \zeta  \,|\, \mcl G_n'  \right) \geq 1- q ,
\eqe 
where here $\ol Z_r^n$ is as in~\eqref{eqn-cone-detect-path}. 

For $n\in\BB N$, let $\dot X^n$ be a random word with the conditional law of $Z^n|_{[0,2]}$ given $\{X(1,2n) =\emptyset\}$ and let $\dot Z^n$ be the corresponding path, as in~\eqref{eqn-Z^n-def}. Assume that we have defined the times $\tau_{n_k}^{a,r}$ with respect to the word $\dot X^n$. 
By Theorem~\ref{thm-main}, the Prokhorov theorem, and the Skorokhod theorem, for any sequence of integers tending to $\infty$, we can find a subsequence $n_k\rta\infty$, a family of random times $\{\wt\tau^{a,r}\}_{(a,r) \in \mcl A}$ coupled with $\dot Z$, and a coupling of $\left(\dot X^{n_k} ,\, \{\tau_{n_k}^{a,r}\}_{(a,r) \in\mcl A}\right)$ with $\left(\dot Z  ,\, \{\wt\tau^{a,r}\}_{(a,r) \in\mcl A}\right)$ such that a.s.\ $\dot Z^{n_k} \rta \dot Z$ uniformly on $[0,2]$ and $\tau_{n_k}^{a,r} \rta \wt\tau^{a,r}$ for each $(a,r) \in \mcl A$. 

We will complete the proof by showing that in fact $\wt\tau^{a,r} =\tau^{a,r}$ for each $(a,r) \in \mcl A$. By \cite[Lemma 5.7]{gms-burger-cone}, each $\wt\tau^{a,r}$ is a $\pi/2$-cone time for $\dot Z$ with $\wt\tau^{a,r} - v_{\dot Z}(\wt\tau^{a,r}) \geq r$ and $\wt\tau^{a,r} \geq a$. It follows that a.s.\ $\wt\tau^{a,r} \geq \tau^{a,r}$. 

Now suppose by way of contradiction that $\BB P(\wt\tau^{a,r} > \tau^{a,r}) > 0$ for some $(a,r) \in \mcl A$. Then we can find deterministic $\alpha_0 >0 $ and $q_0\in (0,1)$ such that
\eqb \label{eqn-wt-tau>tau}
\BB P\left( \wt\tau^{a,r} \geq \tau^{a,r} + 2\alpha_0 \right) \geq 2q_0
\eqe 
Let $\ol Z_r$ be defined as in~\eqref{eqn-cone-detect-path} with $\dot Z$ in place of $Z^n$. Since $\ol Z_r(\tau^{a,r}) = 0$, it follows from~\eqref{eqn-wt-tau>tau} and our choice of coupling that for each $\zeta > 0$, there exists $k_* \in\BB N$ such that for $k\geq k_*$, we have
\eqbn
\BB P\left(  \inf_{t\in [a ,  \tau_n^{a,r} - \alpha_0  ]} |\dot Z^n(t)| < \zeta     \right) \geq  q_0 .
\eqen
This contradicts~\eqref{eqn-cone-detect-G'}. We conclude that $\wt\tau^{a,r} = \tau^{a,r}$ a.s., which concludes the proof.
\end{proof}

\begin{proof}[Proof of Theorem~\ref{thm-cone-limit-finite}]
By Lemma~\ref{prop-cone-conv-finite}, the finite-dimensional marginals of the joint conditional law given $\{X(1,2n)=\emptyset\}$ of
\eqbn
\left(Z^n|_{[0,2]} ,\, \{\tau_n^{a,r}\}_{(a,r) \in\mcl Q \times (\mcl Q \cap (0,\infty)) }\right)  
\eqen
converge as $n\rta\infty$ to the corresponding finite-dimensional marginals of the joint law of 
\eqbn
\left(\dot Z ,\, \{\tau^{a,r}\}_{(a,r) \in\mcl Q \times (\mcl Q \cap (0,\infty)) }\right) .
\eqen 
The statement of the theorem follows from exactly the same argument used in the proof of \cite[Theorem 1.9]{gms-burger-cone}.
\end{proof}

\bibliography{cibiblong,cibib} 

\def\cprime{$'$}
\begin{thebibliography}{MWW14b}

\bibitem[Ald91a]{aldous-crt1}
D.~Aldous.
\newblock The continuum random tree. {I}.
\newblock {\em Ann. Probab.}, 19(1):1--28, 1991. \MR{1085326 (91i:60024)}

\bibitem[Ald91b]{aldous-crt2}
D.~Aldous.
\newblock The continuum random tree. {II}. {A}n overview.
\newblock In {\em Stochastic analysis ({D}urham, 1990)}, volume 167 of {\em
  London Math. Soc. Lecture Note Ser.}, pages 23--70. Cambridge Univ. Press,
  Cambridge, 1991. \MR{1166406 (93f:60010)}

\bibitem[Ald93]{aldous-crt3}
D.~Aldous.
\newblock The continuum random tree. {III}.
\newblock {\em Ann. Probab.}, 21(1):248--289, 1993. \MR{1207226 (94c:60015)}

\bibitem[Ber07]{bernardi-maps}
O.~Bernardi.
\newblock Bijective counting of tree-rooted maps and shuffles of parenthesis
  systems.
\newblock {\em Electron. J. Combin.}, 14(1):Research Paper 9, 36 pp.
  (electronic), 2007. \MR{2285813 (2007m:05125)}

\bibitem[Ber08]{bernardi-sandpile}
O.~Bernardi.
\newblock Tutte polynomial, subgraphs, orientations and sandpile model: new
  connections via embeddings.
\newblock {\em Electron. J. Combin.}, 15(1):Research Paper 109, 53, 2008,
  \arxiv{math/0612003}. \MR{2438581 (2009f:05110)}

\bibitem[BGT87]{reg-var-book}
N.~H. Bingham, C.~M. Goldie, and J.~L. Teugels.
\newblock {\em Regular variation}, volume~27 of {\em Encyclopedia of
  Mathematics and its Applications}.
\newblock Cambridge University Press, Cambridge, 1987. \MR{898871 (88i:26004)}

\bibitem[BLR15]{blr-exponents}
N.~{Berestycki}, B.~{Laslier}, and G.~{Ray}.
\newblock {Critical exponents on Fortuin--Kastelyn weighted planar maps}.
\newblock {\em ArXiv e-prints}, February 2015, \arxiv{1502.00450}.

\bibitem[CC13]{cc-pos-bridge}
F.~Caravenna and L.~Chaumont.
\newblock An invariance principle for random walk bridges conditioned to stay
  positive.
\newblock {\em Electron. J. Probab.}, 18:no. 60, 32, 2013,
  \arxiv{math/0602306}. \MR{3068391}

\bibitem[{Che}15]{chen-fk}
L.~{Chen}.
\newblock {Basic properties of the infinite critical-FK random map}.
\newblock {\em ArXiv e-prints}, February 2015, \arxiv{1502.01013}.

\bibitem[DMS14]{wedges}
B.~{Duplantier}, J.~{Miller}, and S.~{Sheffield}.
\newblock {Liouville quantum gravity as a mating of trees}.
\newblock {\em ArXiv e-prints}, September 2014, \arxiv{1409.7055}.

\bibitem[DS11]{shef-kpz}
B.~Duplantier and S.~Sheffield.
\newblock Liouville quantum gravity and {KPZ}.
\newblock {\em Invent. Math.}, 185(2):333--393, 2011, \arxiv{1206.0212}.
  \MR{2819163 (2012f:81251)}

\bibitem[DW15a]{dw-cones}
D.~Denisov and V.~Wachtel.
\newblock Random walks in cones.
\newblock {\em Ann. Probab.}, 43(3):992--1044, 2015, \arxiv{1110.1254}.
  \MR{3342657}

\bibitem[DW15b]{dw-limit}
J.~{Duraj} and V.~{Wachtel}.
\newblock {Invariance principles for random walks in cones}.
\newblock {\em ArXiv e-prints}, August 2015, \arxiv{1508.07966}.

\bibitem[FK72]{fk-cluster}
C.~M. Fortuin and P.~W. Kasteleyn.
\newblock On the random-cluster model. {I}. {I}ntroduction and relation to
  other models.
\newblock {\em Physica}, 57:536--564, 1972. \MR{0359655 (50 \#12107)}

\bibitem[Gar11]{garbit-cone-walk}
R.~Garbit.
\newblock A central limit theorem for two-dimensional random walks in a cone.
\newblock {\em Bull. Soc. Math. France}, 139(2):271--286, 2011,
  \arxiv{0911.4774}. \MR{2828570 (2012e:60095)}

\bibitem[GM15a]{gwynne-miller-inversion}
E.~Gwynne and J.~Miller.
\newblock Conformal invariance of whole-plane {$\op{CLE}_\kappa$} for {$\kappa
  \in (4,8)$}.
\newblock In preparation, 2015.

\bibitem[GM15b]{gwynne-miller-cle}
E.~Gwynne and J.~Miller.
\newblock Convergence of the topology of critical {Fortuin-Kasteleyn} planar
  maps to that of {CLE$_\kappa$} on a {Liouville} quantum surface.
\newblock In preparation, 2015.

\bibitem[GMS15]{gms-burger-cone}
E.~{Gwynne}, C.~{Mao}, and X.~{Sun}.
\newblock {Scaling limits for the critical Fortuin-Kasteleyn model on a random
  planar map I: cone times}.
\newblock {\em ArXiv e-prints}, February 2015, \arxiv{1502.00546}.

\bibitem[Gri06]{grimmett-fk}
G.~Grimmett.
\newblock {\em The random-cluster model}, volume 333 of {\em Grundlehren der
  Mathematischen Wissenschaften [Fundamental Principles of Mathematical
  Sciences]}.
\newblock Springer-Verlag, Berlin, 2006. \MR{2243761 (2007m:60295)}

\bibitem[GS15]{gms-burger-local}
E.~{Gwynne} and X.~{Sun}.
\newblock {Scaling limits for the critical Fortuin-Kastelyn model on a random
  planar map II: local estimates and empty reduced word exponent}.
\newblock {\em ArXiv e-prints}, May 2015, \arxiv{1505.03375}.

\bibitem[KN04]{kager-nienhuis-guide}
W.~Kager and B.~Nienhuis.
\newblock {A guide to stochastic L\"owner evolution and its applications}.
\newblock {\em Journal of Statistical Physics}, 115(5-6):1149--1229, 2004,
  \arxiv{math-ph/0312056}.

\bibitem[KW14]{werner-sphere-cle}
A.~{Kemppainen} and W.~{Werner}.
\newblock {The nested simple conformal loop ensembles in the Riemann sphere}.
\newblock {\em ArXiv e-prints}, February 2014, \arxiv{1402.2433}.

\bibitem[{Le }14]{legall-sphere-survey}
J.-F. {Le Gall}.
\newblock {Random geometry on the sphere}.
\newblock {\em ArXiv e-prints}, March 2014, \arxiv{1403.7943}.

\bibitem[LG13]{legall-uniqueness}
J.-F. Le~Gall.
\newblock Uniqueness and universality of the {B}rownian map.
\newblock {\em Ann. Probab.}, 41(4):2880--2960, 2013, \arxiv{1105.4842}.
  \MR{3112934}

\bibitem[Mie09]{miermont-survey}
G.~Miermont.
\newblock Random maps and their scaling limits.
\newblock In {\em Fractal geometry and stochastics {IV}}, volume~61 of {\em
  Progr. Probab.}, pages 197--224. Birkh\"auser Verlag, Basel, 2009.
  \MR{2762678 (2012a:60017)}

\bibitem[Mie13]{miermont-brownian-map}
G.~Miermont.
\newblock The {B}rownian map is the scaling limit of uniform random plane
  quadrangulations.
\newblock {\em Acta Math.}, 210(2):319--401, 2013, \arxiv{1104.1606}.
  \MR{3070569}

\bibitem[MS12a]{ig1}
J.~{Miller} and S.~{Sheffield}.
\newblock {Imaginary Geometry I: Interacting SLEs}.
\newblock {\em ArXiv e-prints}, January 2012, \arxiv{1201.1496}.

\bibitem[MS12b]{ig2}
J.~{Miller} and S.~{Sheffield}.
\newblock {Imaginary geometry II: reversibility of SLE$_\kappa(\rho_1;\rho_2)$
  for $\kappa \in (0,4)$}.
\newblock {\em ArXiv e-prints}, January 2012, \arxiv{1201.1497}.

\bibitem[MS12c]{ig3}
J.~{Miller} and S.~{Sheffield}.
\newblock {Imaginary geometry III: reversibility of SLE$_\kappa$ for $\kappa
  \in (4,8)$}.
\newblock {\em ArXiv e-prints}, January 2012, \arxiv{1201.1498}.

\bibitem[MS13a]{ig4}
J.~{Miller} and S.~{Sheffield}.
\newblock {Imaginary geometry IV: interior rays, whole-plane reversibility, and
  space-filling trees}.
\newblock {\em ArXiv e-prints}, February 2013, \arxiv{1302.4738}.

\bibitem[MS13b]{qle}
J.~{Miller} and S.~{Sheffield}.
\newblock {Quantum Loewner Evolution}.
\newblock {\em ArXiv e-prints}, December 2013, \arxiv{1312.5745}.

\bibitem[MS15a]{tbm-characterization}
J.~{Miller} and S.~{Sheffield}.
\newblock {An axiomatic characterization of the Brownian map}.
\newblock {\em ArXiv e-prints}, June 2015, \arxiv{1506.03806}.

\bibitem[MS15b]{lqg-tbm1}
J.~{Miller} and S.~{Sheffield}.
\newblock {Liouville quantum gravity and the Brownian map I: The QLE(8/3,0)
  metric}.
\newblock {\em ArXiv e-prints}, July 2015, \arxiv{1507.00719}.

\bibitem[MS15c]{lqg-tbm2}
J.~{Miller} and S.~{Sheffield}.
\newblock {Liouville quantum gravity and the Brownian map II: geodesics and
  continuity of the embedding}.
\newblock In preparation, 2015.

\bibitem[MS15d]{lqg-tbm3}
J.~{Miller} and S.~{Sheffield}.
\newblock {Liouville quantum gravity and the Brownian map III: the embedding is
  determined}.
\newblock In preparation, 2015.

\bibitem[MS15e]{sphere-constructions}
J.~{Miller} and S.~{Sheffield}.
\newblock {Liouville quantum gravity spheres as matings of finite-diameter
  trees}.
\newblock {\em ArXiv e-prints}, June 2015, \arxiv{1506.03804}.

\bibitem[Mul67]{mullin-maps}
R.~C. Mullin.
\newblock On the enumeration of tree-rooted maps.
\newblock {\em Canad. J. Math.}, 19:174--183, 1967. \MR{0205882 (34 \#5708)}

\bibitem[MWW14a]{mww-extremes}
J.~{Miller}, S.~S. {Watson}, and D.~B. {Wilson}.
\newblock {Extreme nesting in the conformal loop ensemble}.
\newblock {\em ArXiv e-prints}, December 2014, \arxiv{1401.0217}.

\bibitem[MWW14b]{mww-nesting}
J.~{Miller}, S.~S. {Watson}, and D.~B. {Wilson}.
\newblock {The conformal loop ensemble nesting field}.
\newblock {\em ArXiv e-prints}, December 2014, \arxiv{1401.0218}.

\bibitem[She09]{shef-cle}
S.~Sheffield.
\newblock Exploration trees and conformal loop ensembles.
\newblock {\em Duke Math. J.}, 147(1):79--129, 2009, \arxiv{math/0609167}.
  \MR{2494457 (2010g:60184)}

\bibitem[{She}10]{shef-zipper}
S.~{Sheffield}.
\newblock {Conformal weldings of random surfaces: SLE and the quantum gravity
  zipper}.
\newblock {\em ArXiv e-prints}, December 2010, \arxiv{1012.4797}.

\bibitem[{She}11]{shef-burger}
S.~{Sheffield}.
\newblock {Quantum gravity and inventory accumulation}.
\newblock {\em ArXiv e-prints}, August 2011, \arxiv{1108.2241}.

\bibitem[Shi85]{shimura-cone}
M.~Shimura.
\newblock Excursions in a cone for two-dimensional {B}rownian motion.
\newblock {\em J. Math. Kyoto Univ.}, 25(3):433--443, 1985. \MR{807490
  (87a:60095)}

\bibitem[Shi91]{shimura-cone-walk}
M.~Shimura.
\newblock A limit theorem for two-dimensional random walk conditioned to stay
  in a cone.
\newblock {\em Yokohama Math. J.}, 39(1):21--36, 1991. \MR{1137264 (93b:60153)}

\bibitem[SW12]{shef-werner-cle}
S.~Sheffield and W.~Werner.
\newblock Conformal loop ensembles: the {M}arkovian characterization and the
  loop-soup construction.
\newblock {\em Ann. of Math. (2)}, 176(3):1827--1917, 2012, \arxiv{1006.2374}.
  \MR{2979861}

\end{thebibliography}
\bibliographystyle{hmralphaabbrv}

\end{document}